\documentclass{amsbook}
\usepackage{amssymb,amsthm,mathtools,dsfont,nicefrac,upgreek, empheq}%packages for additional symbols and math typesetting
\usepackage{graphicx}
\usepackage{wrapfig}
\usepackage[utf8]{inputenc}
\usepackage[T1]{fontenc}
\usepackage[all]{xy}
\usepackage{enumerate}%customise enumerate and itemize environments
\usepackage[pdftitle={Explicit $GL(2)$ trace formulas and uniform, mixed Weyl laws},
  pdfauthor={Marc R. Palm},
  pdfsubject={Mathematics}]{hyperref}%create internal and external hyperlinks in PDF files
\usepackage[lite]{amsrefs}%for references
\usepackage{microtype}%microtypography for pdflatex avoids most bad boxes
\usepackage{pdfpages}
\usepackage{xcolor} 
\usepackage{currvita}
 \usepackage{lipsum}
\usepackage{amsthm}
\usepackage{microtype}
\usepackage{bm,bbm}
\usepackage{enumitem}
\usepackage{pdfpages} 
\renewcommand{\PrintDOI}[1]{\href{http://dx.doi.org/#1}{DOI #1}%
 \IfEmptyBibField{pages}{, (to appear in print)}{}}

\numberwithin{section}{chapter}
\numberwithin{equation}{section}

 \newtheorem{theorem}{Theorem}[section]

\newtheorem{thm}[theorem]{Theorem}
\newtheorem*{thmu}{Theorem}
\newtheorem{defnthm}[theorem]{Definition-Theorem}
\newtheorem{proposition}[theorem]{Proposition}
\newtheorem{corollary}[theorem]{Corollary}

\newtheorem{lemma}[theorem]{Lemma}

\theoremstyle{definition}
 \newtheorem{defn}[theorem]{Definition}
  \newtheorem*{defnu}{Definition}

\newtheorem{example}[theorem]{Example}
\newtheorem{examples}[theorem]{Examples}
 \theoremstyle{remark}
\newtheorem*{remark}{Remark}

 \newtheorem{claim}{Claim}
\newcommand{\sma}{\left(\begin{smallmatrix}}
\newcommand{\smz}{\end{smallmatrix}\right)}
\newcommand{\pma}{\begin{pmatrix}}
\newcommand{\pmz}{\end{pmatrix}}

\definecolor{leichtgrau}{gray}{.90}
\newcommand*\mygreybox[1]{%
\colorbox{leichtgrau}{\hspace{1em}#1\hspace{1em}}}

\let\emptyset\varnothing

\newcommand{\sign}{\textup{sign}}
\newcommand{\bA}{\mathds{A}}   
\newcommand{\bF}{\mathds{k}}                               %Adeles
\renewcommand{\o}{\mathfrak{o}}
\newcommand{\p}{\mathfrak{p}}
\newcommand{\n}{\mathfrak{n}}
 \newcommand{\w}{\mathfrak{w}}
\newcommand{\F}{\mathds{F}_v}
\newcommand{\Fq}{\mathds{F}_q}  
\newcommand{\E}{\mathds{E}_w}

\newcommand{\im}{\textup{i}}
\newcommand{\e}{\textup{e}}        

\newcommand{\cond}{\textup{cond}}

\newcommand{\bC}{\mathbb{C}}
\newcommand{\bR}{\mathbb{R}}
\newcommand{\bZ}{\mathbb{Z}}
\newcommand{\bN}{\mathbb{N}}
\newcommand{\bQ}{\mathbb{Q}}

\newcommand{\GL}{\textup{GL}}
 \newcommand{\PGL}{\textup{PGL}}
 \newcommand{\PSL}{\textup{PSL}}
\newcommand{\SL}{\textup{SL}}

       \newcommand{\U}{\textup{U}}
 \newcommand{\SU}{\textup{SU}}
 
  \renewcommand{\O}{\textup{O}}
 \newcommand{\SO}{\textup{SO}}

 \newcommand{\Z}{\textup{Z}}    %%%%%center 
  \newcommand{\B}{\textup{B}}        %%%%%%Borel
     \newcommand{\M}{\textup{M}}    
         \newcommand{\N}{\textup{N}}

\renewcommand{\Re}{\textup{Re}}
 \renewcommand{\Im}{\textup{Im}}

\renewcommand{\d}{\;\textup{d}}

 \newcommand*{\blank}{\textup{\textvisiblespace}}

\newcommand{\mA}{\mathcal{A}}

\newcommand{\mF}{\mathcal{F}}

\newcommand{\mH}{\mathcal{H}}
\newcommand{\mmH}{\underline{\mathcal{H}}}

\newcommand{\mJ}{\mathcal{J}}
\newcommand{\mK}{\mathcal{K}}
\newcommand{\mL}{\mathcal{L}}
\newcommand{\mM}{\mathcal{M}}
\newcommand{\mN}{\mathcal{N}}
\newcommand{\mO}{\mathcal{O}}
\newcommand{\mP}{\mathcal{P}}

\newcommand{\mS}{\mathcal{S}}

\newcommand{\mSH}{\mathcal{SH}}
\newcommand{\fS}{\mathfrak{S}}

\newcommand{\fU}{\mathfrak{U}}

\newcommand{\fX}{\mathfrak{X}}

\newcommand{\St}{\textup{St}}

\DeclareMathOperator{\tr}{\textup{tr}}
\newcommand{\vol}{\textup{vol}}
 \renewcommand{\det}{\textup{det} \;}

 \newcommand{\bH}{\mathbb{H}} 
\newcommand{\Cinf}{\textup{C}^\infty} 
\newcommand{\Ccinf}{\textup{C}_{c}^\infty}

  \newcommand{\dr}{\textup{d}_{\bR}^+} 
    \newcommand{\dpr}{\textup{d}_{\bR}^\times} 

 \DeclareMathOperator{\Ind}{\textup{Ind}}
 
\DeclareMathOperator{\Res}{\textup{Res}}
\DeclareMathOperator{\ind}{\textup{c-ind}}

\DeclareMathOperator{\infl}{\textup{infl}}    %inflation
\DeclareMathOperator{\Sym}{\textup{Sym}}    %symmetric

 \DeclareMathOperator{\Endo}{\textup{End}}
 \DeclareMathOperator{\Hom}{\textup{Hom}}

\newcommand{\resprod}{\mathop{\prod\nolimits'}}
\newcommand{\resotimes}{\mathop{\bigotimes\nolimits'}}
\newcommand{\uF}{\underline{F}}

\makeindex

\begin{document}
\pagenumbering{roman}
%\preprint{APS/Fluct-QChains} 
 \begin{titlepage}%%%%%%%%%Titelseite
\begin{center}

\textsc{\LARGE Explicit $\GL(2)$ trace formulas and uniform,~mixed Weyl laws}

\vspace{2cm}

\includegraphics[width=60mm]{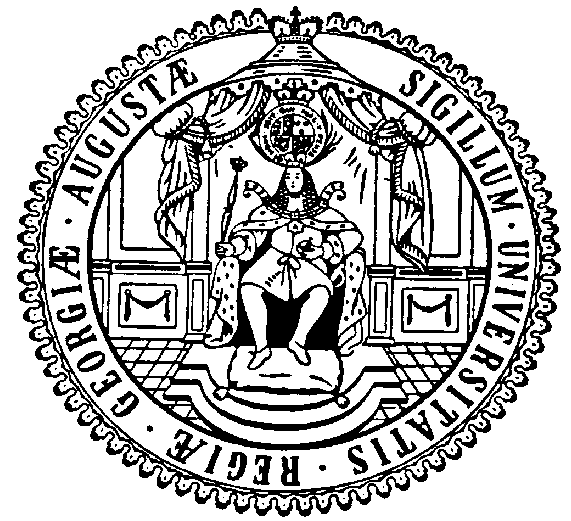}

\vspace{1.5cm}

\normalsize
Dissertation\\
zur Erlangung des mathematisch-naturwissenschaftlichen Doktorgrades\\
\large ``Doctor rerum naturalium'' \\
\normalsize der Georg-August-Universit\"at G\"ottingen\\
im Promotionsprogramm der PhD School of Mathematical Science (SMS)\\
der Georg-August University School of Science (GAUSS).\\

\vspace{1cm}
\small{vorgelegt von}\\

\large{Marc R. Palm} \\%\footnote{Mathematisches Institut, Bunsenstra{\ss}e 3-5, 37073 G{\"o}ttingen, Email: \href{mailto:palm@uni-math.gwdg.de}{palm@uni-math.gwdg.de}.} \\
\normalsize aus Trier \\
G\"ottingen, 2012. 

\vspace{1cm}

\small{Betreuer:}\\

\large Prof. Dr. Valentin Blomer \\
\normalsize und \\
\large Prof. Dr. Ralf Meyer

\vspace{1cm}
\end{center}
\end{titlepage}

\tableofcontents
\section*{Abstract}
\small 
This thesis provides an explicit, general trace formula for the Hecke and Casimir eigenvalues of $\GL(2)$-automorphic representations over a global field. In special cases, we obtain Selberg's original trace formula \cite{Selbergtrace}. Computations for the determinant of the scattering matrices, the residues of the Eisenstein series, etc. are provided. The first instance of a mixed, uniform Weyl law for every algebraic number field is given as standard application. ``Mixed'' means that automorphic forms with preassigned discrete series representation at a set of real places are counted. ``Uniform'' indicates that the estimates implicitly depend on the number field, but not on the congruence subgroup. The method of proof relies on a suitable partition of the cuspidal, automorphic spectrum, and the explication of the non-invariant Arthur trace formula via Bushnell and Kutzko's theory of types. A pseudo matrix coefficient for each local, square-integrable representation, i.e., the discrete series, the supercuspidal, or the Steinberg representation, is constructed explicitly.

\normalsize
\chapter*{Introduction}
Mathematicians are interested  --- among other things --- in classifying all elements/structures with specific properties. Often the classification is achieved by relating them to elements/structures of a different kind.

In certain cases, this is a hopeless task. For example, a useful ``classification'' of all prime numbers is out of reach, and we are, for example, happy with an asymptotic description (the Prime Number Theorem).

Before diving into technicalities, let us high-light two related ``classification'' problems:
\begin{itemize}
\item
The first problem under consideration originates from spectral theory and geometric analysis. Let $(M,g)$ be a closed Riemannian manifold. What are the eigenvalues and the eigenfunctions of the Laplace operator?
\item
The second problem is that of determining how many holomorphic functions live on a compact Riemann surface?
\end{itemize}
We will consider a finite-volume, locally symmetric space with singulatities, and study a hybrid version of the above problems. That is, we analyze certain functions, which are holomorphic in some of the variables, and real-analytic only in other variables. 

The underlying group structure of a locally symmetric space allows us to appeal to the machinery of representation theory. We can derive so called trace formulas - an identity between geometric data (conjugacy classes) and spectral data (eigenvalues). 

The aim of this thesis is to refine tools --- the trace formulas --- for a classification of the joint eigenfunctions of certain families of commuting operators, namely, a classification of all $\GL(2)$-automorphic forms over a global field.

A final classicification in general is unachievable with current technology. Some partial results exist. In the function field case, a useful classification of $\GL(2)$-automorphic forms in terms of Galois-representations was achieved by Drinfeld. In the number field setting, we expect such a correspondence only for a proper subset of all $\GL(2)$-automorphic forms. Special cases are known. But in general, the situation is similar to that of the prime numbers. We can only count $\GL(2)$-automorphic forms asymptotically (Weyl laws). 

I extent, refine and generalize some of the presently known Weyl laws for $\GL(2)$ in this PhD thesis. The techniques lie at the interface of arithmetic, harmonic analysis and representation theory, and are mostly long computations.

The reader unfamiliar with Weyl laws might start with the surveys of Iwaniec \cite{Iwaniec:Spectra}, M\"uller \cite{Mueller:Weyllaw}, and Sarnak \cite{Sarnak:Spectra}. For example,
M\"uller \cite{Mueller:Weyllaw} adresses the spectral analysis of compact Riemannian manifolds. I will motivate as well the subject historically by starting with the classical example of Maass cusp forms.

\section*{Spectral analysis of Hecke Maass cusp forms}
\subsection*{Hecke Maass cusp forms and their $L$-functions}
A Hecke Maass cusp form of weight zero and $\SL_2(\bZ)$ is a smooth, complex-valued function $f$ on
$$\mathbb{H} \coloneqq \{ z \in \bC:  \Im z >0 \}$$
such that $f$ is
\begin{enumerate}
\item a modular form; for all elements $\sma a & b \\ c & d\smz \in \SL_2(\bZ)$, we have that 
\[ f\left( \frac{az+b}{cz +d} \right) =  f(z),\]
\item square integrable; 
\[ \left\| f \right\|^2 \coloneqq \int\limits_{\SL_2(\bZ)\backslash \mathbb{H}} \left| f(x+\im y ) \right|^2 \frac{\d x \d y}{y^2} < \infty,\]
\item a cusp form; for all $y>0$, we have that 
\[ \int\limits_{0}^{1} f(x + \im y) \d x = 0, \] 
\item a Maass form; an eigenfunction of $\Delta \coloneqq -y^2 \left( \frac{\partial}{\partial x^2} + \frac{\partial}{\partial y^2} \right)$
\[ \Delta f = \lambda_\infty f,\]
\item a Hecke form; an eigenfunction of the Hecke operators 
\[T_n f =\lambda_n f,\]
 where the Hecke operator is defined for every integer $n \in \bZ$
\[ T_n f(z) = n^{-1/2}\sum_{d|n} \sum_{b=0}^{d-1} f\left(\frac{nz+bd}{d^2}\right) . \]
\end{enumerate}
Let $X(\SL_2(\bZ))$ denote the collection of (suitably normalized) Hecke Maass cusp forms, and let $L^2_0(\SL_2(\bZ))$ be the $\mL^2$-completion of the vector space generated by the Hecke Maass functions. The set $X(\SL_2(\bZ))$ is an orthonormal basis of $L^2_0(\SL_2(\bZ))$, and $L^2_0(\SL_2(\bZ))$ is the space of measurable functions $q$ with $\left\|q\right\| < \infty$ and satisfying the cuspidality condition, because the Hecke and Laplace operators commute. These operators also commute with the involution $f(z) = f( -\overline{z})$. We say that a Maass cusp form $f$ is even [odd] if 
\[ f(x+\im y) = f(-x + \im y) \qquad [f(x+\im y) = -f(-x + \im y)].  \] 
Set $\epsilon =0$ [$\epsilon=1$] if $f$ is even [odd]. The modularity and the cuspidality condition (1) and (2) ensure that the Mellin transform of $f$
\[ L(s,f) \coloneqq \frac{1}{2} \int\limits_{0}^\infty f( \im y) y^{s-1/2} \frac{\d y}{y},\]
which we refer to as $L$-function of $f$,  has an analytic continuation with a functional equation \cite{Bump:Auto}*{Proposition 1.9.1, pg.107}
\[ L(s,f) =(-1)^{\epsilon} L(1-s,f)   .\]
Conditions $(3)$ and $(4)$ ensure that $L(s,f)$ has an Euler product factorization
\begin{align*} L(s,f) = &  \uppi^{-s} \Upgamma\left( \frac{s+\epsilon +\sqrt{\frac{1}{4} -\lambda_\infty}}{2} \right) \Upgamma\left( \frac{s+\epsilon -\sqrt{\frac{1}{4}-\lambda_\infty}}{2} \right) \\ & \quad \times \prod\limits_{p \textup{ prime number}} \left( 1 - \lambda_p p^{-s} + p^{-2s} \right)^{-1}. \end{align*}
The values $\lambda_\infty$ and $\lambda_p$ thus uniquely determine the factors in the Euler product, and are therefore very important arithmetic quantities, about which we know very little. For example, it remains an open conjecture whether the Laplace eigenvalues occur with single multiplicity (see Luo \cite{Luo:Weyl}). 

\subsection*{Spectral analysis} 
Hecke Maass cusp forms can be defined for certain finite index subgroups $\Gamma$ of $\SL_2(\bZ)$, but the definitions $(1)$ -- $(5)$ become more involved \cite{Hejhal2}. In 1949, Maass realized that the Artin-$L$-function of certain Galois representations is identical to the $L$-function of a Hecke Maass cusp of weight zero and level  
\[ \Gamma_0(N) = \left\{ \sma a & b \\  c & d \smz \in \SL_2(\bZ): c \in N \right\}\]
 with Laplace eigenvalue $\lambda_\infty = \frac{1}{4}$ \cite{Bump:Auto}*{Theorem 1.9.1, pg.112}. In 1956, Selberg \cite{Selbergtrace} answered the question as to whether there exist other examples of Hecke Maass cusp forms by providing an explicit formula for
\begin{align*}
 \sum\limits_{ f \in X(\SL_2(\bZ)) } h\left(\sqrt{\frac{1}{4} - \lambda_\infty}\right),
\end{align*}
where $h$ is the Fourier transform of a compactly supported function. The explicit formula consists of distributions associated to conjugacy classes in $\SL_2(\bZ)$, Eisenstein series, their residues, and one-dimensional representations.

Because $h$ is an entire function, we cannot determine the location of single Laplace eigenvalues, but we can prove an asymptotic formula
\[    \left\{ \lambda_\infty(f) \leq T : f \in X(\SL_2(\bZ))  \right\} = \frac{\vol(\SL_2(\bZ) \backslash \bH)}{4 \uppi} T       + \textup{o} (T).\]
Let $\Gamma$ be a congruence subgroup of $\SL_2(\bZ)$, i.e., containing
\[ \Gamma(N) = \left\{ \gamma \in \SL_2(\bZ) : \gamma \equiv \sma 1 & 0 \\ 0 & 1 \smz \bmod N \right\},\]
 and let $X(\Gamma)$ be the set of normalized Hecke Maass cusp forms of weight zero and level $\Gamma$.  
One purpose of this thesis is to provide the following folklore asymptotic law for the Laplace eigenvalues:
\begin{thmu}[Uniform Weyl law~\ref{cor:weyler}]
For $T \geq 1$, we have 
\begin{align*}  \left\{ \lambda_\infty(f) \leq T : f \in X(\Gamma)  \right\}  = &\frac{[ \SL_2(\bZ) : \Gamma]}{12}  T   \\&  + d_\Gamma \sqrt{T} \log T   +  \mO\left( [ \SL_2(\bZ) : \Gamma] \sqrt{T} \right).\end{align*}
The constant $d_\Gamma$ depends upon $\Gamma$ and is bounded by the index $[ \SL_2(\bZ) : \Gamma]$. The implied constant in the error term is absolute.
\end{thmu}
Although there are slightly sharper bounds in the $T$-aspect of $\mO( \sqrt{T} / \log T)$ (see Randol \cite{Randol}), no uniform bound (with proof) in the level aspect seems to be available in the literature. In particular cases, the proof is an exercise in bookkeeping if one applies H\"ormander's method (as interpreted by Lapid/M\"uller \cite{LapidMueller:SLN}) and Huxley's computation of the scattering matrix \cite{Huxley:Scatt}. 

A similar uniform bound over an arbitrary number field with the same precision is achieved in this thesis. Indeed, the main application is a uniform, mixed Weyl law. Mixed Weyl laws are asymptotic formulas for automorphic forms with preassigned discrete series at some archimedean places. The proof is fairly standard as soon as an explicit $\GL(2)$ trace formula is found, which generalizes the Selberg trace formula. 
 
 \subsection*{Trace formulas}
Although this problem in spectral analysis was the motivating problem to develop general, explicit $\GL(2)$-trace formulas, the main emphasis of my thesis is the derivation of such formulas. The Weyl law is merely a standard consequence.  

Let us shortly review the literature available. In the topic of explicit $\SL(2)$ or $\GL(2)$ trace formulas, the treatments of special questions are scattered throughout the literature.

 For expository simplicity, I have only stated the definition of an automorphic function for weight zero. In general, one can allow arbitrary real weight and also introduce unitary, finite dimensional representations of $\Gamma$. Trace formulas under these modifications have been derived by both Hejhal \cite{Hejhal1}, \cite{Hejhal2} and Venkov \cite{Venkov}. The formulas are often not as explicit as one would ultimately like. For example, the distribution associated to the Eisenstein series depends on the computation of a scattering matrix. Both Hejhal \cite{Hejhal:Scatt} and Huxley \cite{Huxley:Scatt} have provided these for the most interesting congruence subgroups.

From a representation theoretic point of view, it becomes clear that one can isolate the holomorphic modular forms of weight $k \geq 2$. Selberg derived formulas for the traces of Hecke operators acting on the family of holomorphic modular functions of weight $k \geq 3$. Eichler \cite{Eichler} was able to treat $k \geq 2$.  Hijikata \cite{Hijikata} and Oesterl\'e \cite{Oesterle:Thesis} have generalized these formulas for holomorphic automorphic functions of weight $k\geq 2$. 

The above theory is related to $\mathbb{Q}$ as we will soon see; the field $\mathbb{Q}$ can be replaced by an arbitrary global field. The trace formula for Laplace operators has only been derived for totally real fields by Efrat \cite{Efrat}, and for quadratic imaginary fields by Tanigawa \cite{Tanigawa}, Szmidt \cite{Szmidt}, Bauer \cite{Bauer1}, \cite{Bauer2}, Elstrodt, Grunewald and Mennicke \cite{Elstrodt}. In this generality, the scattering matrices have been computed for $\SL_2(\o_{\uF})$, where $\o_{\uF}$ is the ring of integers of a general number field $F$, by Sarnak and Efrat \cite{EfratSarnak:Scatt}, Sorensen \cite{Sorensen}, and Masri \cite{Masri:Scatt}. The trace formula for the Hecke eigenvalue has been obtained by Shimizu for holomorphic automorphic forms in several variables over a totally real field \cite{Shimizu:Hecke}. Bruggeman and Miatello \cite{BruggemanMiatello} provide a closely related trace formula - the Bruggeman-Kuznetsov formula - in the general number field setting.

\section*{Spectral analysis of ad\`elic automorphic forms}

 \subsection*{Features of the ad\`elic setting}
The computations in this thesis aim to derive similar formulas for a global field, i.e., for either a general algebraic number field or a global function field. I have chosen to compute in an ad\`elic framework which has several advantages:
\begin{itemize}
\item the trace formula already exists; the non-invariant Arthur trace formula only needs to be specialized
\item the relation to the group and its representation theory is more transparent
\item the fudge factors in the trace formula and the determinant of the scattering matrices are computable by exploiting the product structure of the ad\`eles
\item the definition of the Hecke operators becomes more directly related to representation theory
\item the computations are largely independent of the specific global field under consideration
\item we can argue with $\SL_2(\bQ)$ instead of $\Gamma$, where we have a better description of the conjugacy classes
\end{itemize}
The Arthur trace formula can only handle the analysis related to congruence subgroups $\Gamma$ and treat integer weights.

Tamagawa \cite{Tamagawa:Trace}, \cite{Tamagawa:Zeta} studied the trace formula in an ad\`elic setting for a division algebra. Jacquet and Langlands \cite{JacquetLanglands} derived a similar, but more complicated formula for $\GL(2)$ (see also the expos\'e of Jacquet and Gelbart \cite{GelbartJacquet}). Arthur generalized this theory to a general reductive group in the number field setting \cite{Arthur:Rank1}, \cite{Arthur:Notes}. Laumon \cite{Laumon1}, \cite{Laumon2} and Lafforgue \cite{Lafforgue} developed the trace formula for $\GL(n)$ in the function field setting.

The coarse Arthur trace formula is only useful to mathematicians with extensive training in the representation theory and harmonic analysis of reductive groups.  Various authors have succeeded in making the Arthur trace formula more directly applicable. Duflo/Labesse \cite{DufloLabesse} and Knightly/Li \cite{KnightlyLi} have shown how to derive Selberg's trace formula for the Hecke eigenvalue from the Arthur trace formula, reproving trace formulas for the Hecke operator for $\GL(2)$-holomorphic forms over $\mathbb{Q}$ under various restrictions on weight and level. More general formulas for Hecke eigenvalues in this spirit were obtained by Arthur \cite{Arthur:Hecke}, who derived them from his invariant trace formula \cite{Arthur:Invariant}. For the analysis of Maass wave forms, Reznikov \cite{Reznikov:Scatt} has studied the Laplace eigenvalues via the non-invariant trace formula and indicated a possible way to calculate the determinant of the scattering matrix.  The non-invariant trace formula for $\GL(n)$ over $\mathbb{Q}$ has been used by Lapid and M\"uller \cite{LapidMueller:SLN} for the same purpose in higher rank. Both parties are allowed to remain fairly vague, because their only purpose is to find a Weyl law. No specialization of the Arthur trace formula to the Selberg trace formula for the Laplace eigenvalues has been carried out thoroughly in the literature.

My thesis provides a more flexible approach than in \cite{DufloLabesse}, \cite{KnightlyLi}, \cite{Reznikov:Scatt}. Similar to these expositions, I will derive explicit formulas from the coarse trace formula in the context of $\GL(2)$. My approach is more closely related to the local representation theory at each place, requires restrictions on neither weight nor level, works in the context of a general global field, and can simultaneously treat Maass cusp  and holomorphic modular forms. This requires an explicit construction of the test function largely motivated by the theory of types at the non-archimedean places. This theory of types is due to Bushnell and Kutzko \cite{BushnellKutzko:GLNopen}, and is a representation theoretic refinement of Atkin-Lehner theory \cite{AtkinLehner}, or rather Casselman's representation-theoretic interpretation \cite{Casselman:Restr}. The article \cite{Reznikov:Scatt} suggests appealing quite heavily to the representation theory of groups like $\SL_2(\mathbb{Z}/N)$, and relies conceptually on the induction-by-steps procedure(see Venkov and Zograf \cite{VenkovZograf}). Bushnell and Kutzko's theory asserts that only a few representations of  $\SL_2(\mathbb{Z}/N)$ are actually necessary to understand the cuspidal automorphic spectrum in general. An efficient management of the congruence subgroups/congruence representation theory is one of the main computational difficulties in the theory of automorphic forms. As opposed to abelian class field theory, the two-dimensional Galois representations are not directly related to representation of 
\[ \GL_2(\widehat{\bZ}) \coloneqq \lim\limits_{\substack{\leftarrow \\ N}}\GL_2(\mathbb{Z}/N) .\]
Only certain subsets of each category are related (see Paskunas \cite{Vytas:Unicity} for a mathematical statement). 

The representation theory of $\GL_2(\widehat{\bZ}) $ is well-understood, evidenced by Stasinski's efficient treatment \cite{Stasinski:Smooth}. However, the approach proposed here is more likely to generalize to $\GL(n)$ because the unitary dual of $\GL_n(\widehat{\bZ} )$ has not been classified.

\subsection*{From classical automorphic functions to ad\`elic automorphic functions}
Before shifting to the number field setting, let us briefly review how classical automorphic functions are related to ad\`elic automorphic functions. I follow Reznikov \cite{Reznikov:Scatt}.

A function $f: \bH \rightarrow \bC$ of weight zero lifts to a function $F$ on $\SL_2(\bR)$ via the homeomorphism $\mathbb{H} \cong \SL_2(\bR) / \SO(2)$. This lift depends on a fixed point $z_0 \in \bH$. We choose $z_0 =\im$, and then
\[ F\left( \sma a & b \\ c & d \smz \right) = f\left( \frac{a \im + b}{c \im + d} \right).\]
We have obtained a function on a Lie group. If $f$ satisfies (1), then $F$ is a function on $\Gamma \backslash \SL_2(\bR)$. 

 In order to switch to an ad\`elic group, we must assume that $\Gamma$ is a congruence subgroup, say, containing $\Gamma(N)$. In the process, we will apply strong approximation, that is, for every compact, open subgroup $U \subset \prod\limits_{p \; \textup{prime}} \SL_2(\bZ_p)$, we have a surjection
\[ \SL_2(\bQ) \cdot \SL_2(\bR) \cdot U = \SL_2(\bA).\]
Consider $\Gamma /\Gamma(N)$ as a subgroup of $\SL_2(\bZ / N)$. We obtain by the Chinese remainder theorem and the representability of the varieties (over $\bZ$) an isomorphism 
\[ \SL_2(\bZ / N) \cong \prod\limits_{p^k || N} \SL_2(\bZ / p^k). \]
Denote by $(\Gamma /\Gamma(N))_p$ the projection onto one component $\SL_2(\bZ/ p^k)$, and define $U_p$ as its pullback along 
\[ \SL_2 (\bZ_p) \twoheadrightarrow \SL_2(\bZ / p^k).\]
 Let $\bA_\bQ$ be the ring of ad\`eles of $\bQ$. Set 
\[ U = \prod\limits_p U_p,\] 
then we have both $U \cap \SL_2(\bQ) =\Gamma$ and a homeomorphism
\[ \Gamma \backslash \bH \cong \Gamma \backslash \SL_2(\bR) / \SO(2) \cong \SL_2(\bQ) \backslash \SL_2(\bA) / U \times \SO(2).\]
This provides the lift to a function on $\SL_2(\bA)$. Weights can now be easily introduced via a one-dimensional representation of $\SO(2)$. For any
integer $k \in \bZ$, define
\[ \epsilon_k: \SO(2) \rightarrow \bC^1, \qquad  \sma \cos \theta & \sin \theta \\ - \sin\theta & \cos \theta \smz \mapsto \e^{\im \theta}.\]

The reader who wishes to see this worked out in the number field setting, possibly allowing finite-dimensional representations of congruence subgroups, may refer to \cite{Reznikov:Scatt}.

\subsection*{Definition of ad\`elic automorphic functions}
A cuspidal automorphic function on $\SL(2)$ over $\bQ$ of level $K_f \coloneqq \prod\limits_p \SL_2(\bZ_p)$ and integer weight $k$ is a smooth function 
\[ f : \SL_2(\bA) \rightarrow \bC, \]
such that $f$ is
\begin{enumerate}[label=(\Roman*)]
\item modular;  for all elements $\gamma  \in \SL_2(\bQ)$ and $(u, u_\infty) \in K_f \times \SO(2)$, we have 
\[ f\left( \gamma g (u,u_\infty) \right) =  f(g)\epsilon_k(u_\infty),\]
\item cuspidal;  for all $g \in \SL_2(\bA_\bQ)$, we have 
\[ \int\limits_{\N(\bQ) \backslash \N(\bA)} f(ng) \d x = 0,\]
\item an eigenfunction of the Casimir operator, or alternatively, an eigenvector of the commutative $*$-algebra
\[ \mH(\SL_2(\bR), \epsilon_k) =\left\{ \phi \in \Ccinf(\SL_2(\bR)) : \phi(k_1 g k_2) = \epsilon_{k}(k_1k_2)\phi(g) ,\; k_i \in \SO(2)\right\},  \]
which acts as
  \[ T_{\phi_\infty} f(x) \coloneqq \int\limits_{\SL_2(\bR)} \phi_\infty(g^{-1}) f(gx) \d g, \]
\item an eigenfunction of the Hecke operators, that is, an eigenvector of  
\[ T_{\phi_p} f(x) \coloneqq \int\limits_{\SL_2(\bQ_p)} \phi_p(g^{-1}) f(gx) \d g\]
for all $\phi_p \in \Ccinf( \SL_2( \bQ_p) // \SL_2(\bZ_p))$.
\end{enumerate}
For expository simplicity, we have restricted ourselves to cuspidal automorphic functions of level $K_f$ so far, because the Hecke operators are fairly complicated to define for open subgroups of $K_f$.

\subsection*{Relation between automorphic functions and automorphic representations}
Consider $\mL^2_0( \SL_2(\bQ) \backslash \SL_2(\bA_\bQ))$ as the vector space of measurable functions $f : \SL_2(\bA) \rightarrow \bC$, such that both
\[ \int\limits_{\SL_2(\bQ) \backslash \SL_2(\bA_\bQ)}  \left| f(g)  \right|^2 \d g < \infty, \]
and condition (II) hold. With the right translation by elements of $\SL_2(\bA_\bQ)$, the space $\mL^2_0( \SL_2(\bQ) \backslash \SL_2(\bA_\bQ))$ becomes a unitary representation of $\SL_2(\bA_\bQ)$.
The irreducible subrepresentations of $\SL_2(\bA_\bQ)$ are called irreducible, cuspidal automorphic representations.

Conditions (III) and (IV) ensure that a cuspidal automorphic function $f$ is a matrix coefficient of a unique irreducible, supercuspidal automorphic representation $\pi$ by the converse of Schur's lemma. 

Let $U$ be an open subgroup. An irreducible, cuspidal automorphic representation has a matrix coefficient, which is an automorphic cuspidal representation of weight $k$ and level $U$ if and only if the restriction of $\pi$ to $U \times \SO(2)$ contains $\mathds{1}_U \otimes \epsilon_k$, where $\mathds{1}_U$ is the trivial representation of $U$.

\subsection*{The Arthur trace formula}
According to Flath \cite{Flath}, irreducible, supercuspidal automorphic representations factor into a tensor product 
\[ \pi = \pi_\infty \otimes \bigotimes_p \pi_p\] 
of irreducible, infinite-dimensional, unitary representations $\pi_p$ of $\SL_2(\bQ_p)$ for all primes $p$ and $\pi_\infty$ of $\SL_2(\bR)$. We refer to them as (local) factors of $\pi$. The automorphic $L$-function essentially depends only on $\pi$, its Euler factors on the local factors $\pi_p$ and $\pi_\infty$. 

We choose a finite set $S$ of prime numbers, containing all the primes $p$ with $U_p \neq \SL_2(\bZ_p)$. Choose compactly supported, smooth functions $\phi_p \in \Ccinf(\SL_2(\bQ_p))$ with 
\[ \phi_p(u_1 g_p u_2) = \phi_p(g_p) \qquad u_1,u_2 \in U_p, \; g_p \in \SL_2(\bQ_p),\]
and $\phi_\infty  \in \Ccinf( \SL_2(\bR))$ with
\[ \phi_\infty(k_1 g_\infty k_2) = \epsilon(k_1k_2) \phi_\infty(g_\infty) , \qquad k_1, k_2 \in \SO(2),\; g_\infty \in \SL_2(\bR).\]
To an irreducible, unitary representation $\pi_v$ of $G_v=\SL_2(\bQ_p)$ or $G_v=\SL_2(\bR)$, we can associate a functional
\[ \tr \pi : \Ccinf(G_v) \rightarrow \bC,\]
because the operator
\[ \pi(\phi) : v \mapsto \int\limits_G \phi(g) \pi(g) v \d g \]
is trace class. The functional depends in an obvious manner upon the choice of Haar measure $\d g$ on $G_v$.

The Arthur trace formula provides in this context a formula for
\[ \sum\limits_{\pi = \pi_\infty \otimes \bigotimes_p \pi_p} \tr \pi_\infty(\phi_\infty) \cdot \prod\limits_{p \in S} \tr \pi_p(\phi_p),\]
where $\pi$ runs through the irreducible, cuspidal automorphic forms whose restriction to $U \times \SO(2)$ contains $\mathds{1}_U \otimes \epsilon_k$. The formula is given in terms of conjugacy classes of $\SL_2(\bQ)$, Eisenstein series, their residues, and one-dimensional representations.

\subsection*{$\GL(2)$-automorphic representations}
In this thesis, I prefer to work with $\GL(2)$ instead of $\SL(2)$. This is for technical convenience only. Some of the reasons for this are:
\begin{itemize}
 \item  the representation theory of $\GL_2(\bQ_p)$ is easier to describe
 \item  the conjugacy classes in $\GL_2(\bQ)$ are described by the theory of the rational canonical form
 \item  the Arthur trace formula is much better covered for $\GL(2)$ in the literature
\end{itemize}
Bushnell and Kutzko explain the differences between $\GL(n)$ and $\SL(n)$ for non-archimedean fields in \cite{BushnellKutzko:SLN1} and \cite{BushnellKutzko:SLN2}. Labesse and Langlands \cite{LabesseLanglands} have addressed the differences between $\SL(2)$ and $\GL(2)$ on the global level.

Studying $\GL(2)$ comes with an inconvenient feature, namely, a large center. This is technically annoying, but innocent. Every irreducible representation always has a central character by Schur's lemma.

Let $\uF$ be a global field with ring of ad\`eles $\bA$ and set of valuations $\mS$. We fix a unitary Hecke character $\chi: \uF^\times \backslash \bA^\times \rightarrow \bC^1$, consider it as character of the center $\Z(\bA)$ of $\GL_2(\bA)$, and choose a right invariant quotient measure $\textup{d} \dot{x}$ on $\GL_2(\uF) \backslash \GL_2(\bA)$. 

An irreducible, cuspidal automorphic representation is an irreducible subrepresentation of the right regular representation $g :f(x) \mapsto f(xg)$ on the space
\begin{align*} L_0^2(\chi) \coloneqq  \Big\{ f: \GL_2(\bA) \rightarrow \bC \colon & \; \textup{ for all } z \in \Z(\bA), \gamma \in \GL_2(\uF), g \in \GL_2(\bA) \\
& f(z \gamma g) = \chi(z) f(g),  \\
&\vspace{2cm}  \int\limits_{\uF \backslash \bA} f\left( \sma 1 &t\\ 0 & 1\smz g \right) \d t = 0 , \\ 
& \int\limits_{\GL_2(\uF) \Z(\bA) \backslash \GL_2(\bA)} \left| f(x) \right|^2 \d \dot{x} < \infty \Big\}. \end{align*}

\subsection*{Similarity classes of $\GL(2)$-automorphic representations and spectral parameters}
In order to explicate the Arthur trace formula, we partition the irreducible, cuspidal automorphic representations in the smallest entities, disjointly analyzable by purely local methods. 
For this, we define an equivalence relation whose equivalence classes will be called \textbf{similarity classes}\index{similarity classes} for lack of a better name. The local $L$-factors and the root numbers of a similarity class look almost identical.

Let us formulate the equivalence relation in explicit language. Let $\F$ be a local field, and let $K_v$ be a maximal compact subgroup of $\GL_2(\F)$, such that $\GL_2(\F) = \sma * & * \\ 0& * \smz K_v$. For every local character $(\mu_{1,v}, \mu_{2,v})$ of $\F^1 = \{ x \in \F : \left|x \right|_v = 1\}$  and every complex value $s \in \bC$, define the principal series representation or parabolic induction\index{principal series representation or parabolic induction} as the right regular representation
\begin{align*} \mJ_v(\mu_{1,v}, \mu_{2,v}, s)  \coloneqq \Big\{ f: \GL_2(\F) \rightarrow \bC & \textup{ with }\int\limits_{ K_v} \left| f(k_v) \right|^2 \d k_v < \infty \textup{ and}\\ 
                                                                               &    f\left( \sma a & x \\ 0 & b \smz g \right) = \mu_{1,v}(a) \mu_{2,v}(b) \left| \frac{a}{b}\right|_v^{s+1/2}  f(g)   \Big\} . \end{align*}

Two irreducible, cuspidal automorphic representations $\pi_j$ for $j=1,2$ with factorization $\otimes_v \pi_{j,v}$ are similar if for all places $v \in \mS$ either
\begin{itemize}
 \item $\pi_{1,v}$ and $ \pi_{2,v}$ have square integrable matrix coefficients, and $\pi_{1,v}$ is isomorphic to $\pi_{2,v}$, or
 \item there exist two (possibly distinct) complex values $s_1,s_2\in \bC$, such that $\pi_{1,v}$ is isomorphic to $\mJ(\mu_{v}, s_1)$, and $\pi_{2,v}$ is isomorphic to~$\mJ(\mu_{v}, s_2)$.
\end{itemize} 
\subsection*{The Hecke operator on a similarity class}
The first advantage of considering similarity classes are the easily definable Hecke operators.

 Let $\mu_v$ be a one-dimensional representation of $\F^1 = \o_v^\times$ with conductor $\p_v^N$, then the Hecke operators are defined as
\[ T_\phi f(x) \coloneqq \int\limits_{\GL_2(\bQ_p)} \phi(g) f(g^{-1}x) \d g \]
for all $\phi \in \Ccinf(\GL_2(\bQ_p))$ with
\[ \phi\left( \sma a_1 & b_1 \\ c_1 & d_1 \smz g  \sma a_2 & b_2 \\ c_2 & d_2 \smz \right) = \mu(a_1a_2) \phi(g) , \qquad \textup{for } \sma a_j & b_j  \\ c_j & d_j \smz \in \GL_2(\o_v), \; c_j \in \p_v^N.\]
These operators build a family of commutative operators, which follows from a multiplicity-one-property given by the analysis of Bushnell and Kutzko's types \cite{BushnellKutzko:GLNopen}, the Godement's principle \cite{Godement:Spherical}, and a rephrasing of Casselman's interpretation \cite{Casselman:Restr} of Atkin-Lehner theory \cite{AtkinLehner}. These operators vanish on all but the similarity classes with elements $\pi = \otimes_v \pi_v$ such that $\pi_v \cong  \mJ(\mu_v, 1, s)$ for some $s \in \bC$.

For the square integrable representations, the convolution operators with\\ pseudo~matrix coefficients have similar properties.

\section*{The main result: An explicit trace formula}
\subsection*{Spectral parameters}
Let $\fX$ be a similarity class of irreducible, cuspidal automorphic representations. We write $\mS^{ps}(\fX)$ or $\mS^{sq}(\fX)$ for the places at which the factor of elements from $\fX$ are principal series representation or square integrable representations.
Similarly, we define $\mS_\infty^{ps}(\fX), \mS_\bR^{ps}(\fX), \mS_\bC^{ps}(\fX), \mS_f^{ps}(\fX)$ as the set of archimedean, real, complex, non-archimedean places in $\mS^{ps}(\fX)$. Because at almost all places $v \in \mS$, a cuspidal automorphic representation has a factor isomorphic to \( \mJ_v(1,1,s_v)\), we know that $\mS^{ps}(\fX)$ is cofinite, that is, $\mS - \mS^{ps}(\fX)$ is finite. All other sets are allowed to be empty. A square integrable representation of $\GL_2(\bC)$ does not exist, hence $\mS_\bC^{ps}(\fX) = \mS_\bC(\fX)$.

For every element $\pi \in \fX$ and every valuation $v \in \mS^{ps}(\fX)$, define $s_v(\pi) \in \bC$, such that $\pi_v \cong \mJ_v(\mu_{1,v}, \mu_{2,v} ,s_v(\pi))$  for two one-dimensional representations $\mu_1, \mu_2$ of $\F^1$. The complex number $s_v(\pi)$ is called the \textbf{spectral parameter}\index{spectral parameter} of $\pi$ at $v$. From them, one can directly compute the eigenvalues of the Hecke operators, the eigenvalues of the Laplace operators, and the Euler factors of the associated automorphic $L$-function. 

\subsection*{Relation to Laplace and Hecke eigenvalues} 
In the context of $\SL_2(\bQ)$, the relations between spectral parameters and the Laplace or Hecke eigenvalues of the Hecke Maass cusp $f_\pi$ forms are
\begin{align*}
 \lambda_\infty(f_\pi) &= \frac{1}{4} -s_\infty(\pi)^2,\\
 \lambda_p(f_\pi) & = p^{s_p(\pi)}+ p^{-s_p(\pi)}.
\end{align*}

\subsection*{The constructed test function and the cuspidal spectrum}
Let $\fX$ be a similarity class of automorphic representations. For every $v  \in \mS_{\infty}^{ps}(\fX)$,  choose an arbitrary even function $h_v : \bR \rightarrow \bC$, which is the Fourier transform of a compactly supported function. 

We fix an arbitrary finite (possibly empty) subset $\mS^{\bm H}$ of $\mS_f^{ps}(\fX)$, at which we want to examine the spectral parameters or equivalently the Hecke eigenvalues.

I will construct an explicit test function $\phi= \otimes \phi_v$ depending on the above data, such that
\begin{align*}\sum\limits_{\pi = \otimes_v \pi_v} \prod\limits_v \tr \pi_v( \phi_v)   =  \sum\limits_{\pi  \in \fX} \quad \prod\limits_{v \in \mS_{\infty}^{ps}(\fX)} h_v( \im s_v(\pi))  \quad  \prod\limits_{v \in \mS^{\bm H}}  \left( q_v^{s_v(\pi)} + q_v^{-s_v(\pi)} \right), \end{align*}
with $q_v$ being the residue characteristic of $\F$. The first product is to be interpreted as 
 \[  \prod\limits_{v \in \mS_{\infty}^{ps}(\fX)} h_v( \im s_v(\pi))  = 1 \]
if $\mS_\infty^{ps}(\fX) = \emptyset$ is the empty set, for example, in the global function field setting. 

This is the content of Section~\ref{section:globaltest}. The above construction is shown in full generality. For the remaining parts, I restrict myself to the analysis of equivalence classes, which are unramified at the complex places only. The reader might excuse this, as the classical references restrict to this case as well. The coarse Arthur trace formula provides an identity
\[  \sum\limits_{\pi = \otimes_v \pi_v} \prod\limits_v \tr \pi_v( \phi_v) = J_1( \phi) +  \dots+ J_{\textup{Eis}}(\phi) + \dots\] 
I wish to indicate in some examples how to express $J_*(\phi)$ only in terms of $h_v$, of course depending on the similarity class $\fX$ and the set $\mS^{\bm H}$.

\subsection*{Example 1: The identity distribution}
The identity distribution vanishes if $\mS^{\bm H}$ is chosen to be non-empty. Otherwise, we obtain
\begin{align*}
 J_1(\phi) & =\vol( \GL_2(\uF) \Z(\bA) \backslash \GL_2(\bA))  \prod\limits_{v \in \mS} \phi_v(1) \\ 
       & =  \vol( \GL_2(\uF) \Z(\bA) \backslash \GL_2(\bA)) \times \prod\limits_{v \in \mS_f} C_v  \times \prod\limits_{\substack{ v \in \mS^{sq}_\bR \\ \pi_v \textup{ weight k}\geq 2}}   \frac{k-1}{8 \uppi}    \\
         & \qquad  \times    \prod\limits_{\substack{ v \in S_\bR \\ \pi_v \cong \mJ( 1,1,s_v(\pi)) }}   \frac{1}{4 \uppi}   \int\limits_{\bR} h_v(r) r \tanh(\uppi r) \d r     \\
          & \qquad  \times    \prod\limits_{\substack{ v \in S_\bR \\ \pi_v \cong \mJ( \sign,1,s_v(\pi)) }}    \frac{1}{4 \uppi}  \int\limits_{\bR} h_v(r) r \coth(\uppi r) \d r     \\
           & {} \qquad  \times \qquad    \prod\limits_{\substack{ v \in S_\bC }}    \frac{1}{4 \uppi}  \int\limits_{\bR} h_v(r) r^2  \d r  .
\end{align*}
The numbers $C_v$ are positive scalars depending only on the factor of $\pi\in \fX$ at~$v$, for example, $C_v$ is one if $\pi_v \cong \mJ(1,1,s_v(\pi))$. This is Theorem~\ref{thm:globalidentity}.
\subsection*{Example 2: The Eisenstein distribution and the scattering matrix}
The notion of a scattering matrix only makes sense if automorphic representations in $\fX$ have no square integrable factors. This is related to the Jacquet-Langlands correspondence \cite{JacquetLanglands}, Weyl's observation that the Laplace operators of compact Riemannian manifolds admit discrete spectrum, and the existence of so-called simple trace formulas by Flicker and Kazhdan \cite{FlickerKazhdan}.

If we assume that $\pi \in \fX$ has no square integrable factor, we can define a Hecke operator $\mu: \uF \backslash \bA^1 \rightarrow \bC$ with $\mu = \otimes_v \mu_v$, such that $\pi_v \cong \mJ_v(\mu_v,1,s_v(\pi))$ (possibly only after twisting by a one-dimensional representation).
In this case, we define for $\Re s >0$
\[ \Lambda_\fX (s) = \prod\limits_{\substack{ v \in \mS_f \\ \mu_v =1 }} (1-q_v^{-s})^{-1}  \prod\limits_{v \in \mS_\infty}  L_v( s, \mu_v),\]                                                                               
which differs from the zeta function of $\uF$ only by a finite number of factors, thus admitting a meromorphic continuation and some sort of functional equation. In the function field setting, we again interpret the product over the empty sets as being one. 
The role of $\Lambda_\fX(s)$ is parallel to that of the scattering matrix in classical treatments. The reader observes that $\Lambda_\fX$ only depends on the characters at the archimedean places, and all non-archimedean Euler factors at which $\mu$ is ramified are skipped. 

The distribution of the Eisenstein series evaluates in the number field setting as
\[   J_{\textup{Eis}} (\phi) = \frac{1}{4 \uppi}  \int\limits_{\Re s = 0}     \prod\limits_{v \in \mS_{\infty}(\fX)} h_v( \im s)  \times  \prod\limits_{v \in \mS^{\bm H}}  \left( q_v^{s} + q_v^{-s} \right)  \frac{\partial \log}{\partial s} \frac{\Lambda_\fX(2s)}{\Lambda_\fX(2s+1)} \d s,\]
and for a function field with field of constants $\mathds{F}_{\bm q}$ as
\[   J_{\textup{Eis}} (\phi) = \frac{\log(\bm q)}{4 \uppi}  \int\limits_{0}^{\frac{2 \uppi \im}{  \log(\bm q)}}     \prod\limits_{v \in \mS^{\bm H}} \left( q_v^{s} + q_v^{-s} \right) \frac{\partial \log}{\partial s} \frac{\Lambda_\fX(2s)}{\Lambda_\fX(2s+1)}   \d s.\]
With the exception of the full modular group, the comparison with the classical formulas for the scattering matrix as given by Hejhal \cite{Hejhal:Scatt} and Huxley \cite{Huxley:Scatt} is difficult. For example, for $\SL_2(\bZ)$ we obtain precisely what Selberg does (see Iwaniec \cite{Iwaniec:Spectral}*{Theorem 10.2}).

The Eisenstein distribution vanishes if the elements in $\fX$ have at least one supercuspidal factor or at least two square integrable factors. If there is only one square integrable, non-supercuspidal factor, the Eisenstein distribution is explicitly computed. More precise statements can be found in Theorem~\ref{thm:globalEis}. 
\section*{Application of the trace formulas}
The Arthur, Eichler-Selberg, and Selberg trace formulas have many applications. I focus here on only one standard application, namely, how to use the trace formula to count automorphic representations.  
This requires the choice $\mS^{\bm H} = \emptyset$.
\subsection*{Dimension formulas}
For $\mS_\infty^{ps}(\fX)$ to be empty, $\uF$ must either be a global function field or a totally real number field. The number of isomorphism classes in $\fX$ is finite, but can become arbitrarily large. 
The trace formula provides the dimension in terms of elliptic elements. I remain at this stage fairly vague, but further investigations are planned.

\subsection*{Mixed, uniform Weyl laws}
If $\mS_\infty^{ps}(\fX)$ is non-empty, then $\uF$ is an algebraic number field. Let $r_1$ [$r_2$] be the number of real [complex] places inside $\mS_\infty^{ps}(\fX)$. 
We prove three Weyl laws.
\begin{thmu}[See Theorem~\ref{thm:weyl}]        \mbox{}
\begin{itemize}
 \item If $\uF = \bQ$, and all factors in $\fX$ are principal series representations, the following asymptotic holds  for $T\geq 1$:
         \begin{align*} \mbox{} \\ &  \# \{ \pi \in \fX: s_\bR(\pi) \leq T \}   \\
\mbox{}\\
                                     & \qquad =C_\fX T^2 -  \frac{2}{\uppi} T \log T  + \mO( C_\fX T ).
         \end{align*}
 \item If $\uF$ is an  algebraic number field, and every element in $\fX$ has a square integrable local factor,  the following asymptotic holds  for $T_v \geq 1$:
        \begin{align*} \mbox{}  \\ \# &\{ \pi \in \fX: s_v(\pi) \leq T_v \textup{ for all } v\in S_\infty^{ps}(\fX) \} \\
\mbox{} \\
 &=C_\fX  \prod\limits_{v \in \mS_\bR^{ps}}T_v^2  \prod\limits_{v \in \mS_\bC^{ps}}T_v^3  +  \mO_{\uF}\left( C_\fX \sum_{u \in S_\infty^{ps}} \frac{1}{T_u} \prod\limits_{v \in \mS_\bR^{ps}}T_v^2  \prod\limits_{v \in \mS_\bC^{ps}}T_v^3 \right).\end{align*}
 \item If $\uF$ is an algebraic number field, the following asymptotic holds for $T_v \geq 1$:
 \begin{align*} \mbox{}  \\ \# &\{ \pi \in \fX: s_v(\pi) \leq T_v \textup{ for all } v\in S_\infty^{ps}(\fX) \} \\ 
   \mbox{}\\
                        &=C_\fX  \prod\limits_{v \in \mS_\bR^{ps}}T_v^2  \prod\limits_{v \in \mS_\bC^{ps}}T_v^3  \\   
\mbox{}\\
                         &\qquad   +  \mO_{\uF}\left( C_\fX \sum_{u \in S_\infty^{ps}} \frac{1}{T_u} \prod\limits_{v \in \mS_\bR^{ps}}T_v^2  \prod\limits_{v \in \mS_\bC^{ps}}T_v^3 + \sum_{w \in S_\bR^{ps}}   \log T_w \prod\limits_{v \in \mS_\infty} T_v\right).
 \end{align*}
\end{itemize}
The constant $C_\fX$ depends upon the number field and the similarity class $\fX$. The implied constant in $\mO_{\uF}$ depend on $\uF$ explicitly.
\end{thmu}

The constant $C_\fX$ is explicitly provided in dependence upon $\fX$. The first and second asymptotic have partial intersections. The last bound is more general than the previous bounds, but also comes with a larger error term.  It comes as a surprise that one cannot easily isolate $T_w \log T_w$-terms as soon as the number field has two archimedean places. In the complex case, these terms are absorbed into the error term $T_w^2$.  All provided bounds are absolute for fixed $\uF$. 

This is the first explicit Weyl law with a sharp bound for $\GL(2)$ beyond the case $\mathbb{Q}$ (see \cite{Elstrodt} for a discussion). The vanishing of the intermediate terms has also been obtained over $\mathbb{Q}$ by Risager \cite{Risager} in the more difficult, classical, congruence-subgroup-setting. This is not unexpected, since the Jacquet-Langlands correspondence \cite{JacquetLanglands} asserts that the Weyl law should be that of a compact surface. For fields with complex places, the $T \log T$-terms get absorbed by the error term $T^2$. The bounds are close to optimal in the sense that the Weyl law is almost as sharp as in the easier case of a closed Riemannian manifold, where the Weyl law was given by Avakumovi\'c \cite{Avakumovic}: 
\[ \frac{\vol(X)}{(4\uppi)^{\dim(X)/2} \Gamma(\dim(X)/2 -1)} T^{\dim(X)} + \mO( T^{\dim(X) -1} ).\]
This bound is sharp for spheres (see M\"uller \cite{Mueller:Weyllaw} for an overview). Slightly sharper bounds with an error term $\mO( T / \log T)$ in the $T$-aspect exist in our setting over $\bQ$, but one has to appeal to the Selberg Zeta function. I do not know of any serious consequence of this sharper saving. Over compact, hyperbolic, arithmetic Riemannian surfaces, lower bounds exist as well $\Omega( T^{1/2} / \log(T))$ by Hejhal \cite{Hejhal1}*{pg.303}, and a slightly sharper lower bound has been recently proven for $\SL_2(\bZ)$ by Li and Sarnak \cite{LiSarnak}, but these are the best bounds available over $\mathbb{Q}$.

 No error terms have been obtained in the literature for any algebraic number field other  than $\mathbb{Q}$. Reznikov \cite{Reznikov:Scatt} has provided the main term in the special case $\mS_\infty = \mS_\infty^{ps}(\fX)$ for $\SL(2)$ over a general number field. His methods are also certainly capable of deriving sharp bounds in this particular case.

I also want to mention that an upper bound is known in great generality due to work of Donnelly \cite{Donnelly}, and the main term by work of Lindenstrauss and Venkatesh \cite{LindenstraussVenkatesh:Weyl} for split adjoint semisimple groups over $\mathbb{Q}$, and with a similar non-uniform error term for $\SL(n)$ over $\mathbb{Q}$ by work of Lapid and M\"uller \cite{LapidMueller:SLN}.

The above asymptotic is the first instance of a mixed Weyl law with discrete series representations at some archimedean places and principal series representations at others. Even the main terms were unknown in this case.
In principle, by appealing to simple trace formulas, one could argue that this does not come unexpected in many (but certainly not all) cases.

My bound is uniform in the $\fX$-aspect and only implicitly depends upon the number field $\uF$. I could not find even a single bound of this form for $\bQ$ or a compact surface in the literature, that is, a bound which is uniform in the level aspect. The same bound can be obtained for a compact hyperbolic Riemann surface $\vol(X) T^2 + \mO(\vol(X) T)$ in families of coverings by appealing to the classical Selberg trace formula and the H\"ormander's method (see section two in Lapid and M\"uller \cite{LapidMueller:SLN}). In the congruence subgroup setting over $\bQ$, the same method works with explicit knowledge of the scattering matrix.  A uniform bound $\mO( V_\fX T / \log T)$ seems possible in this particular case, when one appeals to Jorgenson and Kramer's uniform prime geodesic theorem \cite{JorgensonKramer} and applies a contour integration of the Selberg Zeta function.      

Sharper local bounds in the level aspect have been obtained by Duke \cite{Duke:Dimension} and  Michel/Venkatesh \cite{MichelVenkatesh:Weight1} in some cases for
\[ \# \{ \pi \in \fX: s_v(\pi) = 0  \textup{ for all } v\in S_\infty^{ps}(\fX) \} .\]
One should also compare the uniformity with the uniform Riemann-van-Mangoldt estimates for $L$-functions (see Iwaniec/Kowalski \cite{IwaniecKowalski}{Theorem 5.8, pg.104}). By using the analogy between $L$-functions and Selberg zeta functions, it is suggested that the uniform bound is hard to beat. Perhaps slightly improved bounds are possible by appealing to Littlewood's work on the Riemann hypothesis, but no significant savings can be obtained along these lines. 
\section*{Outline and organization of the material}
\subsection*{Chapter 1:}
I have elected to prove the trace formula in the first chapter modulo the local harmonic analysis, which is provided only in the last three chapters. Without it, the computations in those chapters would seem meaningless. After a short mentioning of standard facts about cuspidal automorphic representations, I construct the required test functions explicitly in dependence of the functions $(h_v)_v$, the similarity class $\fX$, and the set $\mS^{\bm H}$ (Section 1.3). I introduce the Arthur trace formula merely as an identity of distributions in Section 1.4:
\begin{align*}
& \sum\limits_{\pi  \in \fX} \quad \prod\limits_{v \in \mS_{\infty}^{ps}(\fX)} h_v( \im s_v(\pi))  \quad  \prod\limits_{v \in \mS^{\bm H}}  \left(  q^{s_v(\pi)} + q^{-s_v(\pi)} \right) =    \\
 & J_1(\phi_{h, \fX}) + J_{\textup{par}}(\phi_{h, \fX}) + \sum\limits_{\alpha \in F^\times - 1} J_{\alpha}(\phi_{h, \fX}) + \sum\limits_{\gamma \textup{ ell.}} J_\gamma (\phi_{h, \fX}) +   \dots
\end{align*}
I provide explicit formulas for the  $J_{*}(\phi_{h, \fX})$ in Sections 1.5--1.11 by factoring them into local distributions. The local distributions are computed in chapter seven, eight, and nine, depending on whether local means real, complex, or non-archimedean.

\subsection*{Chapter 2:}
 I derive the Selberg trace formula for $\GL_2(\bZ) \backslash \GL_2(\bR) /\Z(\bR) \O(2) $, and briefly indicate the general situation. This should be sufficient to indicate the translation between the similarity class setting and the congruence subgroup setting in general, especially because I have already stressed this in the transfer between classical and ad\`elic $\SL(2)$ automorphic functions mentioned above.

\subsection*{Chapter 3:}   I address the counting problem and, in particular, prove the Weyl law.

\subsection*{Chapters 4--6:} I reprove some general results about locally compact groups, which seem to be available only either in the Lie or totally disconnected group setting. With the theory provided, I give an algebraic formalism to compute the trace distribution of an admissible representation, which consists of either parabolic inductions (Theorem~\ref{thm:jacquettrace}), subquotients of parabolic inductions (Corollary~\ref{cor:charvanish}), or compactly induced representations from a subgroup compact modulo the center (Theorem~\ref{thm:frobdc}). This motivates the construction of the test functions.
                                                                                                                         
\subsection*{Chapters 7--9:}  I provide the local harmonic analysis which was required for the trace formula. This relies heavily on chapters four through six. All of this is ground work, but the devil is in the details. For the chapters dealing with $\GL_2(\bR)$ and $\GL_2(\bC)$, I refer mainly to the monographs of Hejhal \cite{Hejhal1}, \cite{Hejhal2} and Elstrodt, Grunewald, and Mennicke \cite{Elstrodt}, despite the fact that these do not address representation theory. For the non-archimedean case, I sketch the theory of types via Clifford theory and develop the computations from there. A particular eye-catcher is the formula for the elliptic orbital integral~\ref{thm:ell}, which begs for a generalization to higher rank.

\section*{Potential further investigations and applications}
\subsection*{Unramified factors at complex places} The analysis of automorphic forms with ramified principal series representations as complex factors remains almost unaddressed. The non-spherical situation is not studied in the literature to the extent needed to generalize my analysis in this case, mostly because the Abel inversion formula becomes very complicated (see Brummelhuis and Koornwinder \cite{Koornwinder:SL2C}). In particular, my construction of the corresponding test functions suggest that new and  interesting phenomena occur. The construction is parallel to that for the discrete series representation at the real places, and the trace formula should thus become very simple.
\subsection*{Non-square-free Hecke operators}
The above trace formula only addresses traces of square-free Hecke operators. Although this restriction can be justified by the fact that these operators generate all Hecke operators and completely determine the local representation theory, amplification methods in analytic number theory require a trace formula for all Hecke operators. The omitted computations do not appear very difficult in this particular case.
\subsection*{Equidistribution of Hecke operators}
I have not addressed the distribution of the Hecke eigenvalues at all. In Knightly and Li \cite{KnightlyLi}*{Chapter 29}, following Conrey-Duke-Farmer \cite{Conrey:Hecke} and Serre \cite{Serre:Hecke}, it is demonstrated how one can prove the Sato-Tate conjecture for unramified Hecke operators via the Eichler-Selberg trace formula for weight $k \geq 3$. Sarnak has addressed the question for Maass functions \cite{Sarnak:Hecke}. A natural question is whether the Sato-Tate conjecture holds for similarity classes as well.
\subsection*{Generalization to higher rank}
As mentioned, much of my computations generalize to higher rank, in particular, the construction of the test functions. At the very least, explicit formulas for $\GL(n)$ are desirable. Some formulas have been obtained by Venkov \cite{Venkov:SL3} and Wallace \cite{Wallace:SL3}, but the approach via the Arthur trace formula seems more profitable in a congruence setting. The computational difficulties will involve sophisticated integral calculus at the archimedean places and counting problems at the non-archimedean places. Both of these difficulties should be addressed via the use of root systems, infinitesimal methods, and the Bruhat-Tits building. These methods have been avoided in the present treatment of $\GL(2)$, since I was interested in an elementary presentation close to the classical references. This luxury is too expensive in higher rank.

\section*{Acknowledgements}
I would like to express my gratitude to both my advisors, Prof. Valentin Blomer and Prof. Ralf Meyer. 
Without their knowledge, advice, time, and patience, this thesis would not have been possible.

I am indebted to many of my colleagues at the Mathematical Institute G\"ottingen and the Internet community around Mathoverflow for interesting and fruitful mathematical discussions.

I am thankful to my family and friends for their continuing support.

During the preparation of this thesis, I was supported by the Research Training Group 1493 ``Mathematical Structures in Modern Quantum Physic'',
and indirectly by the Volkswagenstiftung and the European Research Council via grants of Prof. Blomer.

\cleardoublepage

\pagenumbering{arabic}

\chapter*{Part I --- The $\GL(2)$ trace formula} 
In this chapter, we will state and derive the explicit trace formula from the Arthur trace formula. The local analysis will be postponed to later chapters, and for the time being, we will take the local results for granted.
Because the Arthur trace formula has been treated in various notes by Jacquet and Langlands \cite{JacquetLanglands}, Gelbart and Jacquet \cite{GelbartJacquet},  Gelbart \cite{Gelbart}, Laumon \cite{Laumon2}, Shokranian \cite{Shokranian}, Laumon \cite{Laumon2}, Arthur \cite{Arthur:Notes}, and Knightly and Li \cite{KnightlyLi}, it is  introduced merely as a cumbersome identity of distributions and treated as a black-box tool. In this thesis, we apply the Arthur trace formula with a fixed central character as given \cite{JacquetLanglands}, \cite{GelbartJacquet}, and \cite{Gelbart}. 

The reader who wants to get the basic idea is encouraged to read chapters nine and eleven in Deitmar and Echterhoff's textbook \cite{DeEc}. The book explains the trace formula of Tamagawa \cite{Tamagawa:Trace}, Gangolli and Warner \cite{Gangolli:Trace}, and Wallach \cite{Wallach:Trace}, which is valid for a cocompact lattice in a locally compact group and derives the Selberg trace formula for a compact Riemann surface from there. Despite the fact that the non-compact, ad\`elic and non-spherical situation is more delicate, I derive explicit trace formulas from the Arthur trace formula via the same principle.

\chapter{The explicit $\GL(2)$ trace formula}

  \newcommand{\uo}{\underline{\o}_F}
 \section{Notation}
\subsection{Global fields and local fields}
 Let $\uF$ be a global field, which is by definition either \index{$\uF$}
\begin{itemize}
 \item  an algebraic number field, that is, a finite field extension of the rational numbers $\mathbb{Q}$, or 
 \item an algebraic function field, that is, a finite field extension of the rational polynomials $\mathbb{F}_{q_F}(T)$ over the finite field of constants of cardinality $q_F$.\index{$q_F$}
\end{itemize} 
Consider \cite{ArtinWhaples} for an axiomatic approach. Let $\uo$ be the ring of integers of $\uF$. \index{$\o_{\uF}$}
 
Valuations of $\uF$ are denoted by $v$ and, if non-archimedean, are always normalized, such that they surject onto $\mathbb{Z}$. Let $\mS$ be the set of valuations of $\uF$. We define subsets of $\mS$: \index{$\mS, \mS_\infty, \mS_\bR, \mS_\bC, \mS_f$}
\begin{align*}
\mS_\infty &\coloneqq \{ \textup{ archimedean valuations of }\uF \} \\
\mS_\bR & \coloneqq \{ \textup{ real valuations of } \uF \}  \\
\mS_\bC & \coloneqq \{ \textup{ complex valuations of } \uF \}  \\
\mS_f     & \coloneqq \{ \textup{ non-archimedean valuations of } \uF \}  \end{align*}
Let $\F$ be the completion of $\uF$ at $v$ and let $\left| \blank \right|_v$ denote its norm. \index{$\F, v$} 
 
If $v$ is a non-archimedean field, we denote by \index{$\o_v, \p_v, q_v,\w_v , \psi_v$}
\begin{align*}
\o_v & \qquad -\textup{ the ring of integers of }\F,\\
\p_v & \qquad -\textup{ the maximal ideal of }\o_v, \\
q_v  &  \qquad -\textup{ the residue characteristic, i.e., the cardinality of }\o_v/\p_v, \\
\left| x \right|_v \coloneqq q_v^{v(x)} & \qquad -\textup{ norm of an element }x \in \F,\\
\w_v &  \qquad -\textup{ an a priori fixed uniformizer, i.e.,  a generator of }\p_v,   \\
\psi_v &  \qquad -\textup{ an a priori fixed one-dimensional representation of }\F, \\
          &{} \qquad\qquad \textup{such that $\psi \big|_{\o_v} =1$ and $\psi \big|_{\p^{-1}_v} \neq1$.}
\end{align*}

Let $\F$ be any completion of $\uF$, then define \index{$\F^1$}
\[ \F^1 =\{ x \in \F^\times : \left| x \right|_v =1 \}.\]
Any one-dimensional representation $\chi :\F^\times \rightarrow \bC^\times$ decomposes according to
\[  \chi(x) = \chi_{a}(x) \left| x \right|^s \]
for a unique character $\chi_{a} $ of $\F^\times$, such that the image of $\chi_{a}\big|_{\F^1}$ and $\chi_{a}$ coincide, and for a complex value $s \in \bC$. The value $s$ is unique if $\F$ is archimedean, and unique modulo $\frac{2 \uppi \im}{\log q_v}$ if $\F$ is non-archimedean with residue characteristic $q_v$. We say that $\chi$ is algebraic if $s$ can be chosen as zero.\index{algebraic character of $\F^\times$}

\subsection{Ad\`eles and Id\`eles}
The ring $\bA \coloneqq \bA_{\uF}$ of ad\`eles is defined as the set \index{$\bA$}
\begin{align*} \bA \coloneqq \Big\{ a =(a_v)_{v \in \mS}&  :   a_v \in \F \textup{ and } \\
                           & a_v \in \o_v \textup{ for all but finitely many } v\in \mS_f \Big\}. \end{align*}
With the restricted product topology (see Example~\ref{ex:restricted}) the ring $\bA$ becomes a locally compact commutative ring. 
Let $\bA^\times$ denote the group of invertible elements in $\bA$, sometimes referred to as the group of id\`eles. We have \index{$\bA^\times$}
\[ \bA^\times \coloneqq \left\{ a =(a_v)_{v \in \mS}  :   a_v \in \F^\times \textup{ and } a_v \in \o_v^\times \textup{ for all but finitely many } v\in \mS_f \right\}.\]

The field $\uF$ embeds diagonally into $\bA_{\uF}$
\[ \iota : \uF \hookrightarrow \bA, \qquad x \mapsto (x)_{v \in \mS}.\]
We will drop $\iota$ from the notation, and consider $\uF$ as a subring of $\bA$ and $\uF^\times$ as subgroup of $\bA^\times$.
\begin{thm}[\cite{Tate:Thesis}]
The group $\uF \backslash \bA$ is compact.
\end{thm}

 We define an ad\`elic norm map on $\bA^\times$ by  \index{$\left\| x \right\|_\bA$} 
\[   \left\| (a_v)_v  \right\|_\bA  \coloneqq \prod\limits_{v \in \mS} \left| a_v \right|_v.\] 
\begin{thm}[Product Formula \cite{Tate:Thesis}]\label{thm:productformula}
For $x \in \uF^\times$, we have that $\left\| x \right\|_\bA=1$. 
\end{thm}

The group $\uF^\times$ is a cocompact lattice in \index{$\bA^1$}
\[ \bA^1 \coloneqq \{ x \in \bA : \left\|x\right\|_{\bA} = 1 \}.\]
\begin{thm}[\cite{Tate:Thesis}]
The group $\uF^\times \backslash \bA^1$ is compact.
\end{thm}

A Hecke quasi character is a one-dimensional representation 
\[ \chi :\uF^\times \backslash \bA^\times \rightarrow \bC^\times.\] 
There exists a complex value $s_\chi$ and a unique one-dimensional representation  
\[ \chi_a: \uF^\times \backslash \bA^1 \rightarrow \bC^1,\]
seen as character of $\uF^\times \backslash \bA^\times$, such that
\[ \chi(x) = \chi_a (x) \left\| x \right\|^{s}.\]
If $\uF$ is an algebraic number field, the value $s_\chi$ is unique. If $\uF$ is a global function field, whose field of constants has cardinality $q_F$, then the value $s_\chi$ is unique modulo $\frac{2 \uppi \im}{\log q_F}.$
We say that $\chi$ is an algebraic Hecke character if $s_\chi$ can be chosen as zero. \index{algebraic Hecke character}
\begin{thm}[\cite{Tate:Thesis}]
An algebraic Hecke character $\chi$ factors into a tensor product 
\( \chi =\otimes_v \chi_v\) of algebraic one-dimensional representations $\chi_v$ of $\F^\times$, of which all but finitely many are trivial.
\end{thm}

\subsection{Additive Haar measures on $\F$ and $\bA$}
We fix an additive character $\psi =\otimes_v \psi_v$ of $\uF \backslash \bA$, and choose the additive Haar measure $\textup{d}_v^+ x$ of $\F$ such that it is self-dual with respect to $\psi_v$, i.e., for all places $v$ and for all functions $f_v \in \Ccinf(\F)$, we have
\[ \int\limits_{\F} \int\limits_{\F}  \psi_v(x_v y_v)  f_v(x_v) \d_v^+ x_v \d_v^+ y_v = f_v(0).\]
The Haar measure $\textup{d}_\bA^+ x$ on $\bA$ is the product measure $\bigotimes_v \d_+ x$. This measure is self-dual with regards to $\psi$, i.e., for all function $f \in \Ccinf(\bA)$, we have
\[ \int\limits_{\bA} \int\limits_{\bA}  \psi(xy)  f(x) \d_\bA^+ x \d_\bA^+ y = f(0).\]
We choose an additive character $\psi_v$ with $\psi_v |_{\o_v}=1$ and $\psi_v |_{\p_v}\neq1$. In this manner, the compact subgroup $\o_v$ has unit measure.  
\subsection{Multiplicative Haar measures on $\F$ and $\bA$}
Given the additive Haar measure for every local field, we define the Haar measure $\textup{d}_v^\times x$ on $\F^\times$ via the local zeta function \cite{Tate:Thesis}*{pages 316--322}\index{local zeta function$\zeta_v$}
\[ \zeta_v(s) \coloneqq    \begin{cases} ( 1 - q_v^{-s})^{-1}, & v \textup{ non-archimedean}, \\ 
                                                     \uppi^{-s/2} \Upgamma(s/2) , & v\textup{ real},\\
                                                    (2\uppi)^{-s} 2\Upgamma(s), & v \textup{ complex}. 
                           \end{cases}\]
For each function $f_v \in \Ccinf(\F^\times)$, we define
\[ \int\limits_{\F^\times} f_v(x_v) \d_v^\times x_v = \frac{1}{\zeta_v(1)} \int\limits_{\F^\times} f_v(x_v) \frac{1}{\left| x_v \right|_v}\d_v^+ x_v.\]  
We have for real and complex valuations $\zeta_v(1)=1$. Let $\zeta_{\uF}$ be the completed zeta function of the global field $\uF$
\[ \zeta_{\uF}(s) \coloneqq \prod\limits_{v} \zeta_v(s) \qquad \Re s> 1. \]
The zeta function has a meromorphic extension the whole complex plane with a simple pole at $s=1$. We define $\lambda_{-1}$ and $\lambda_0$ as the coefficients in the Laurent expansion at $s =1$
\[ \zeta_{\uF}(s) = \frac{\lambda_{-1}}{s-1} + \lambda_0 + \mO(s-1).\]

The Haar measure $\d^\times_\bA x$ of $\bA^\times$ is given as the normalized product measure  
\[\d^\times_\bA x \coloneqq \frac{1}{\lambda_{-1}} \bigotimes_v \textup{d}_v^\times x_v.\]

\subsection{The group $\GL(2)$ and its subgroups}
For any unital ring $R$, we define the groups           \index{groups $\GL_2(R), \B(R), \M(R), \Z(R), \N(R)$}
\begin{align*}
\GL_2(R) &\coloneqq \left\{ \sma a & b \\  c & d \smz : a,b,c,d \in R \textup{ with }ad-bc  \in R^\times \right\}, \\
\B(R)& \coloneqq  \left\{ \sma a & b \\  0 & d \smz : b \in R, a,d  \in R^\times \right\}, \\
\M(R)& \coloneqq  \left\{ \sma a & 0 \\  0 & d \smz : a,d  \in R^\times \right\}, \\
\Z(R)& \coloneqq  \left\{ \sma a & 0 \\  0 & a \smz : a  \in R^\times \right\}, \\
\N(R) &\coloneqq  \left\{ \sma 1 & b \\  0 & 1 \smz : b  \in R \right\}.
\end{align*}
and for any valuation, we define the compact group    \index{compact groups $K_v, \U(2), \O(2), \GL_2(\o_v), \bm K$}
\[ K_v   \coloneqq \begin{cases} \GL_2(\o_v), & v \textup{ non-archimedean}, \\
                          \O(2), & v \textup{ real}, \\
                          \U(2), & v \textup{ complex}, \end{cases} \qquad \bm K = \prod\limits_v K_v.
\]
Compact groups conventionally carry a probability Haar measure. This convention differs from \cite{GelbartJacquet} (see Equation 7.6 and 7.7 on page 242).\footnote{The normalization is absolutely convergent, since the root numbers $\epsilon(1, 1, \psi_v)=1$ at all non-archimedean places \cite{Bump:Auto}*{Proposition 3.1.9, page 274}.}
The Haar measure on $\M(\F), \Z(\F)$ and $\N(\F)$ are normalized, such that we have integral identities for $f \in \mL^1(\M(\F)), h \in \mL^1(\Z(\F))$, and $g \in \mL^1(\N(\F))$
\begin{align*}
\int\limits_{\M(\F)} f(m) \d m &= \int\limits_{\F^\times} \int\limits_{\F^\times} f\left( \sma m_1 & 0 \\ 0 & m_2 \smz \right) \d^\times_v  m_1 \d^\times_v m_2, \\
\int\limits_{\Z(\F)} h(z) \d z &= \int\limits_{\F^\times}  h\left( \sma z & 0 \\ 0 &z \smz \right) \d^\times_v z, \\
\int\limits_{\N(\F)} g(n) \d n &= \int\limits_{\F^\times} g\left( \sma 1 & n \\ 0 & 1 \smz \right) \d^+_v n.
\end{align*}

We normalize the Haar measures on $\GL_2(\F)$ and $\Z(\F) \backslash \GL_2(\F)$ \cite{DeEc}*{Propsosition 1.5.5, page 25}, in such a way to produce integral identities for $f \in \mL^1(\GL_2(\F))$ and $h \in \mL^1(\Z(\F) \backslash \GL_2(\F))$
\begin{align*}
 \int\limits_{\GL_2(\F)} f(g) \d g &=    \int\limits_{\M(\F)}\int\limits_{\N(\F)} \int\limits_{K_v} f(mnk) \d m \d n \d k,\\
  \int\limits_{\Z(\F) \backslash \GL_2(\F)} f(g) \d g &=    \int\limits_{\F^\times} \int\limits_{\F} \int\limits_{K_v} f\left( \sma m & 0 \\ 0 & 1 \smz\sma 1 & n \\ 0 & 1\smz k \right) \d_v^\times m \d_v^+ n \d k .
\end{align*}

We normalize the Haar measure on $\Z(\bA) \backslash \GL_2(\bA)$, such that for $\phi = \bigotimes_v \phi_v$
\begin{align}\label{eqintad}  
 \int\limits_{\Z(\bA) \backslash \GL_2(\bA)}\phi(g) \d g &=    \prod\limits_{v}    \int\limits_{\Z(\F) \backslash \GL_2(\F)} \phi_v(g_v) \d g_v \\ 
                                                                             &=    \prod\limits_{v}    \int\limits_{\F^\times} \int\limits_{\F} \int\limits_{K_v} \phi_v\left( \sma m_v & 0 \\ 0 & 1 \smz\sma 1 & n_v \\ 0 & 1\smz k_v \right) \d_v^\times m_v \d_v^+ n_v \d k_v  .      \end{align}

\section{Cuspidal automorphic representations}\label{section:globalcuspidal}

 \subsection{Unitary representations}
\begin{defn}
For an algebraic Hecke character $\chi : \uF^\times \backslash \bA^\times \rightarrow \bC^1$, we define the Hilbert space 
\[\bm \mL^2(\chi) \coloneqq \bm \mL^2\left( \GL_2(\uF) \backslash\GL_2(\bA), \chi\right) \] 
as the space of measurable functions $f : \GL_2(\bA) \rightarrow \bC,$
such that
\begin{enumerate}[font=\normalfont]
 \item for all $z \in \bA^\times$, $\gamma\in \GL_2(\uF)$ and $g \in \GL_2(\bA)$, we have 
 \[ f\left( \gamma \sma z& 0 \\ 0 & z \smz g \right) = \chi(z) f(g),\] 
and
 \item for a $\GL_2(\bA)$-right invariant measure $\textup{d} \dot{g}$ on $\GL_2(\uF)\Z(\bA) \backslash \GL_2(\bA)$, we have 
\[ \int\limits_{\GL_2(\uF)\Z(\bA) \backslash \GL_2(\bA)} \left| f(g) \right|^2 \d \dot{g} < \infty.\]
\end{enumerate}
 \end{defn}
\begin{defn}
We define the subspace
\[ \bm \mL^2_0(\chi) \coloneqq \bm \mL^2_0\left( \GL_2(\uF) \backslash\GL_2(\bA), \chi\right) \] 
of functions, which additionally satisfy: 

(3) for any invariant measure $\textup{d} \dot n $ on $\N(\uF) \backslash \N(\bA)$ and all element $g\in \GL_2(\bA)$, we have
\[ \int\limits_{\N(\uF) \backslash \N(\bA)} f(ng) \d \dot n = 0.\]
  \end{defn}
 \begin{remark}
The group $\GL_2(\bA)$ acts via right translations
\[ g \in \GL_2(\bA) : f(x) \mapsto f(xg)\]
on  $\bm \mL^2(\chi)$ and $\bm \mL^2_0(\chi)$. We obtain two unitary representations, $( \lambda, \bm \mL^2(\chi))$ and $(\lambda _0, \bm \mL_0^2(\chi))$. 
\end{remark}
\begin{defn}
The $*$-algebra  $\Ccinf(\GL_2(\bA))$ is defined as the finite linear combinations of products
 \[ \prod\limits_v \phi_v,\]
 where each function $\phi_v : \GL_2(\F) \rightarrow \bC$ is smooth, compactly supported,
 and at almost all non-archimedean places the function $\phi_v$ is the characteristic function of the set $\GL_2(\o_v)$.

The product is given as the convolution product
\[ \phi_1 \ast \phi_2(x) \coloneqq \int\limits_{\GL_2(\bA)}  \phi_1(xg^{-1}) \phi_2(g) \d g, \]
and the involution as the usual 
\[ \phi^*(x) = \overline{\phi(x^{-1})}.\]
 \end{defn}
\subsection{The definition of the cuspidal distribution}
\begin{theorem}[\cite{GelbartJacquet}*{Corollary 2.4, page 218}]
The representation $( \lambda_0 , \bm \mL^2_0(\chi))$ is admissible and the operator 
\begin{align*} \underline{\lambda}_0(\phi):&  \bm \mL^2_0(\chi) \rightarrow \bm \mL^2_0(\chi) , \qquad \underline{\lambda}_0(\phi) f (x) \coloneqq  \int\limits_{\GL_2(\bA)} \phi(g) f(xg) \d g
\end{align*}
is trace class for all elements $\phi \in \Ccinf(\GL_2(\bA))$
\end{theorem}
\begin{corollary}[\cite{DeEc}*{Lemma 9.2.7, page 178}]
The representation $\lambda_0$ discretely decomposes into irreducible, unitary subrepresentations. Each of these unitary subrepresentations $(\pi, V_\pi)$ is trace class, that is, the operator
\[ \underline{\pi}(\phi) : V_\pi \rightarrow V_\pi, \qquad \pi(\phi)v = \int\limits_{\GL_2(\bA)} \phi(g) \pi(g) v  \d g \]
is a trace class operator for all $\phi \in \Ccinf(\GL_2(\bA)).$ We obtain
$$ \tr \underline{\lambda}_0(\phi) = \sum\limits_{\substack{ \pi \textup{ irreducible} \\ \textup{subrepresentation of }\bm \mL^2_0(\chi)}} \tr \underline{\pi}(\phi).$$
\end{corollary}
\begin{remark}The analogous statements for $( \lambda, \bm \mL^2(\chi))$ are false. \end{remark}
\begin{defn}
Introduce the $^*$-algebra
\begin{align*} \Ccinf(\GL_2(\bA), \overline{\chi})  \end{align*}
as the space of finite linear combination of products $\prod\limits_v \phi_v$, where $\phi_v : \GL_2(\F) \rightarrow \bC$ is a smooth function, compactly supported modulo the center, and satisfies
\[ \phi_v\left( \sma z_v& 0 \\ 0 & z_v \smz g_v\right) = \overline{\chi_v(z_v)} \phi(g), \qquad z_v \in \F^\times,\]
with the convolution product 
\[ \phi_1 \ast \phi_2(x) = \int\limits_{\Z(\bA) \backslash\GL_2(\bZ)} \phi_1(xg^{-1}) \phi_2(g) \d \dot g, \]
and the involution
\[    \phi^*(x) \coloneqq \overline{\phi(x^{-1})}.\]
\end{defn}
\begin{corollary}
The $*$-algebra representation of $\Ccinf(\GL_2(\bA))$, associated to $\underline{\lambda_0}$ or any unitary representation $\pi$ with central character $\chi$, factors through the surjection
\[     I_{\overline{\chi}}:  \Ccinf(\GL_2(\bA)) \twoheadrightarrow \Ccinf(\GL_2(\bA), \overline{\chi}), \qquad I_{\overline{\chi}} (\phi) (x) \coloneqq \int\limits_{\Z(\bA)} \chi(z) \phi(zx) \d z,\]
that is, the $*$-algebra representation
 \[ \pi (\phi) : \bm \mL_0^2(\chi) \rightarrow \bm \mL_0^2(\chi) , \qquad         \pi(\phi) f(x) = \int\limits_{\Z(\bA)\backslash \GL_2(\bA)} \pi(g) f(xg) \d \dot g\]
satisfies 
\[ \lambda_{0} \circ  I_{\overline{\chi}}  = \underline{\lambda}_0, \qquad \pi \circ I_{\overline{\chi}} =\underline{\pi}.\] 
\end{corollary}
\begin{proof}
That the map is surjective is a basic fact \cite{DeEc}*{Lemma 1.5.1, page 21}. Let $\pi$ be a unitary, trace class representation with central character $\chi$ of a locally compact group. The quotient integral formula \cite{DeEc}*{Theorem 1.5.2, page} gives a unique Haar measure $\textup{d} \dot{g}$ on $\Z(\bA) \backslash \GL_2(\bA)$ with
\begin{align*} \pi(\phi) v & \coloneqq \int\limits_{\GL_2(\bA)} \phi(g)  \pi(g) v \d g \\
                                  & = \int\limits_{\Z(\bA) \backslash \GL_2(\bA)} \int\limits_{\Z(\bA)}\phi(zg)  \pi(zg) v \d z \d \dot g   \eqqcolon     \lambda_{0, \chi} \circ  I_{\overline{\chi}} (\phi).   \qedhere
\end{align*}
\end{proof}
We will only work with $*$-algebra representations of $\Ccinf(\GL_2(\bA), \overline{\chi})$, since every irreducible, unitary representation has a unitary central character.

\begin{defn}[The cuspidal distribution]\label{defn:globalcuspidal}
We define the ($\chi$-)cuspidal distribution as the distribution
\[    J_{\textup{cusp}} : \Ccinf(\GL_2(\bA), \overline{\chi}) \rightarrow \bC , \qquad J_{\textup{cusp}}(\phi) = \tr \lambda_0(\phi).\]
\end{defn}

\begin{corollary}
$$ \tr \lambda_0(\phi) = \sum\limits_{\substack{ \pi \textup{ irreducible} \\ \textup{subrepresentation of }\bm \mL^2_0(\chi)}} \tr\pi(\phi).$$
\end{corollary}

\subsection{Factorization of the cuspidal automorphic representations}
 It is a natural goal to construct functions $\phi   \in \Ccinf(\GL_2(\bA), \overline{\chi})$, such that $\tr \pi(\phi)$ vanishes for many, but not all, of the representations. 
\begin{defn}
An irreducible subrepresentation of $\bm \mL^2_0(\chi)$ is called a cuspidal automorphic representation.
\end{defn}
                                                                                                                              
We want to describe a partition of $\bm \mL^2_0(\chi)$ into smaller subspaces.
\begin{theorem}[\cite{Flath}]\label{thm:flath} 
 Let $\pi$ be an irreducible, cuspidal automorphic representation of $\GL_2(\uF)$ with algebraic central character. It admits a factorization
\[ \pi = \bigotimes_v \pi_v\]                                                                                                                                                                                       
into irreducible, unitary, infinite-dimensional representations $\pi_v$ of $\GL_2(\F)$, such that for all but finitely many $v \in \mathcal{S}_f$ the representation $\pi_v$ is isomorphic to an unramified parabolic induction
\[ \pi_v \cong \mJ(1,1,s_v) .\]
\end{theorem}
\begin{defn}
If $\pi$ is an irreducible, cuspidal automorphic form with $\pi =\bigotimes_v \pi_v$, then we say $\pi_v$ is a factor of $\pi$.
\end{defn}

\begin{theorem}[\cite{Dixmier:Cstar-algebres}]
Every irreducible, unitary representation of $\GL_2(\F)$ is trace class.
\end{theorem}

\begin{corollary}   
Let $\phi = \bigotimes_v \phi_v$ be a tensor product of functions $\phi_v$ on $\GL_2(\F)$. We have a factorization
$$ J_{\textup{cusp}}(\phi) = \sum\limits_{\substack{ \pi \subset \lambda_0 \textup{ irreducible} \\ \pi =\bigotimes_v \pi_v }} \prod\limits_{v} \tr \pi_v(\phi_v).$$  
\end{corollary}
\subsection{The local factors of automorphic representations}
Now, we see that it is sufficient to construct functions $\phi_v   \in \Ccinf(\GL_2(\F), \overline{\chi_v})$, such that $\tr \pi_v(\phi)$ vanishes for many, but not all, of the representations $\pi_v$.  Let us examine the unitary representation theory of $\GL_2(\F)$.
\begin{defn}
 Let $(\pi, V_\pi)$ be an irreducible, unitary representation of $\GL_2(\F)$.
\begin{itemize}
   \item The representation $\pi$ is supercuspidal if for each vector $\vec{v}_1, \vec{v}_2 \in V_\pi$ the function
  \[g \mapsto \langle \vec{v}_1, \pi(g) \vec{v}_2 \rangle_{V_\pi}\]
is compactly supported modulo $\Z(\F)$.
 \item The representation $\pi$ is square-integrable if for each vector $\vec{v}\in V_\pi$ the following integral converges:
  \[ \int\limits_{\Z(\F) \backslash \GL_2(\F)} \left| \langle \vec{v}_1, \pi(g) \vec{v}_2 \rangle_{V_\pi} \right|^2 \d  g < \infty.\]
\item The representation $\pi$ is tempered if for each vector $\vec{v}\in V_\pi$ the following integral converges for all $\epsilon >0$:
  \[ \int\limits_{\Z(\F) \backslash \GL_2(\F)} \left| \langle \vec{v}_1, \pi(g) \vec{v}_2 \rangle_{V_\pi} \right|^{2+\epsilon} \d  g < \infty.\]
\end{itemize}
\end{defn}
 
\begin{theorem}[Coarse classification of irreducible, unitary representations of $\GL_2(\F)$, \cite{Bump:Auto}, \cite{BushnellHenniart:GL2}, \cite{JacquetLanglands}]
An irreducible, unitary representation with algebraic central character $\chi$ is isomorphic to one and only one of the following:
\begin{enumerate}[font=\normalfont]
\item a one-dimensional representation $\omega \circ \det$ for a unitary one-dimensional representation $\omega$ of $\F^\times$ with $\omega^2 =\chi$;
\item a principal series representation $\mJ_v(\mu_1, \mu_2,s) =(\pi, V_\pi)$\index{$\mJ_v(\mu_1, \mu_2,s)$} associated to two algebraic characters $\mu_j$ of $F^\times$ with $\mu_1\mu_2 =\chi$, and a complex number $s\in \bC$, i.e., the right regular representation
                                                    \[  \pi(g)  f(x) \coloneqq f(xg) \]
on the space of functions
        \begin{align*}V_\pi \coloneqq \Big\{ f: \GL_2(\F) \rightarrow \bC & \textup{ with }\int\limits_{ K_v} \left| f(k) \right|^2 \d k < \infty \textup{ and}\\ 
                                                                               &    f\left( \sma a & x \\ 0 & b \smz g \right) = \mu_1(a) \mu_2(b) \left| \frac{a}{b}\right|_v^{s+1/2}  f(g)   \Big\}  ;   \end{align*}
     only the complex values $\Re s = 0$ or, if $\mu_1= \mu_2$, additionally $-1/2 < s < 1/2$ can occur;
\item if $v$ is non-archimedean, to a supercuspidal representation, which can be realized as the compact induction from an maximal open subgroup which is compact modulo $\Z(\F)$;
\item if $v$ is non-archimedean, to the Steinberg/special representation, which is denoted by $\St(\omega)$ and defined as a subquotient $\mJ( \omega, \omega, 1/2) / \omega \circ \det$ for $\omega^2 =\chi$;
\item if $v$ is real, to a discrete series representation (denoted by $D_k(\mu_1, \mu_2)$), one isomorphism class for each even/odd integer $k \geq 2$  in the case $\chi =1$ / $\chi=\sign$.
\end{enumerate}
We have an isomorphism
\[     \mJ(\mu_1, \mu_2,s) \cong \mJ(\mu_2, \mu_1,-s).\]
\end{theorem}
For more precise statements refer to Sections~\ref{section:realclass}, ~\ref{section:complexclass}, and ~\ref{section:padicclass}. The one-dimensional representations cannot occur as local factors of cuspidal automorphic forms.
The limits of discrete series representations for $\GL_2(\bR)$ are included in case (2) as $\mJ_v(1, \sign, 0)$ and $\mJ_v(\sign, 1, 0)$, because they are not proper subquotients in the case $\GL_2(\bR)$ opposed to the situation in $\SL_2(\bR)$.
\newpage
\begin{thm}[\cite{Bump:Auto}, \cite{BushnellHenniart:GL2}, \cite{Wallach1}] \mbox{}
 \begin{itemize}
\item The supercuspidal, the Steinberg, and the discrete series representations are square-integrable. 
\item The principal series representations with $\Re s =0$ are tempered.
\item The principal series representations for $-1/2 < s <1/2$ with $s \neq 0$ are not tempered.
\item The one-dimensional representations are not tempered. 
\end{itemize}
\end{thm}

\begin{theorem}[Kazhdan, Clozel, Delorme \cite{ClozelDelorme}, \cite{Kazhdan:Pseudo}]
 Every square-integrable, irreducible representation $\pi_v$ of $\GL_2(\F)$ with central character $\chi_v$ has a pseudo (matrix)coefficient $\phi_{\pi_v} \in \Ccinf(\GL_2(\F), \overline{\chi_v})$, i.e.,
for every unitary, infinite-dimensional representation $\pi_0$ of $\GL_2(\F)$, we have
\[ \tr \pi_0(f_{\pi_v}) \coloneqq \begin{cases} 1, & \pi \cong \pi_0 \\ 0 , & \textup{otherwise}. \end{cases}\]   
\end{theorem}
I will explicitly construct the pseudo coefficient for all square-integrable representations of $\GL_2(\F)$ (Corollary~\ref{cor:DSpseudo}, Theorems~\ref{defn:testsc} and ~\ref{defn:testst}) and also prove a general result for locally profinite groups (Corollary~\ref{cor:superpseudo}). 
Among other things, this allows us to extend the computations of Knightly and Li to weight two modular forms (see the remark on page 214 in \cite{KnightlyLi}).
This construction has been previously discussed in different terms by Hejhal \cite{Hejhal1}*{page 459}.

\begin{corollary}[First spectral refinement]
Fix: \begin{itemize}
     \item a finite subset $\mS^{sq} \subset \mS$ of places
     \item  one irreducible, square-integrable representation $\tau_v$ with central character $\chi_v$ for every place $v \in \mS^{sq}$ 
    \end{itemize}

For each element $\phi =\bigotimes_v \phi_v \in \Ccinf(\GL_2(\bA), \overline{\chi})$ with $\phi_v = f_{\tau_v}$ for all $v \in \mS^{sq}$, we have that
$$ J_{\textup{cusp}}(\phi) = \sum\limits_{\substack{ \pi \subset \lambda_0 \textup{ irreducible} \\ \pi =\bigotimes_v \pi_v  \\ \pi_v \cong \tau_v \textup{ for each } v \in \mS^{sq}}} \prod\limits_{v \in \mS-\mS^{sq}} \tr \pi_v(\phi_v).$$  
\end{corollary}

\begin{remark}[No pseudo coefficient for principal series]
An irreducible principal series representation of $\GL_2(\F)$ does not have pseudo coefficients. In fact, if there exists a complex value $s \in \bC$ and an element $\phi \in \Ccinf(\GL_2(\F), \chi_v)$ with
\[ \tr \mJ_v( \mu_1,\mu_2, s) (\phi_v)    \neq 0,\]
then the set
\[ \{ s_0 \in \bC :  \tr \mJ_v( \mu_1,\mu_2, s_0) (\phi_v)   = 0 \}\]
is discrete.
\end{remark}
\begin{example}
Let $\F$ be a non-archimedean field and $\mathds{1}_{K_v}$ be the characteristic function of $\Z(\F)\GL_2(\o_v)$. For every irreducible, unitary representation $\pi_0$ of $\GL_2(\F)$, we compute
\begin{align*}
\tr \pi_0 (\mathds{1}_{K_v}) =\begin{cases} 1, & \pi_0 = \mJ(1,1,s),\\
                            1, & \pi_0 \textup{ trivial}, \\
                            0, & \textup{otherwise}.
                        \end{cases}
\end{align*}
\end{example}
The next theorem is a consequence of the Propositions~\ref{prop:realjacquet}, ~\ref{prop:complexjacquet}, Theorems~\ref{defn:testprinc}, ~\ref{defn:uHecke}, and ~\ref{defn:rHecke}.
\begin{thm}
Let $\mu_1$ and $ \mu_2$ be two algebraic one-dimensional representations of $\F^\times$. 
There exists a commutative subalgebra $A_{\mu_1, \mu_2}$ of $\Ccinf(\GL_2(\F), \overline{\mu_1\mu_2})$, such that 
\begin{enumerate}[font=\normalfont]
 \item for all unitary, irreducible representations $\pi_0$ of $\GL_2(\F)$ either $\tr \pi_0(f) = 0$ for all elements of $f \in A_{\mu_1, \mu_2}$ or 
 \( \pi_0\) is isomorphic to a unitary principal series \( \mJ(\mu_1, \mu_2,s) \), and 
\item  for any two non-isomorphic irreducible principal series representations $\mJ(\mu_1, \mu_2,s)$ and $\mJ(\mu_1, \mu_2, s_0)$, there exists an element $\phi' \in A_{\mu_1, \mu_2}$ with
 \[       \tr \mJ(\mu_1, \mu_2,s) (\phi') \neq \tr \mJ(\mu_1, \mu_2, s_0)   (\phi').\] 
\end{enumerate}
\end{thm}

\begin{corollary}[Second spectral refinement]
Fix two disjoint finite subsets of $\mS^{sq}$ and $\mS^{d}$ of $ \mS$ with  $\mS_\infty \subset \mS^{sq}\cup \mS^{d}$. Choose for every element $v \in \mS^{sq}$ one irreducible, square-integrable representation $\tau_v$ with central character $\chi_v$, and for every $v \in \mS^{d}$ two (possibly trivial) algebraic characters $\mu_{1,v}, \mu_{2,v}$ of $\F^\times$.

Choose $\phi_v = \phi_{\tau_v}$ for $v \in \mS^{sq}$, $\phi_v \in A_{\mu_{1,v}, \mu_{2,v}}$ for $v \in \mS^d$, and $\phi_v = \mathds{1}_{K_v}$ for all remaining (non-archimedean) places. Define $\phi = \bigotimes_v \phi_v$. We obtain
$$ J_{\textup{cusp}}(\phi) = \sum\limits_{\substack{ \pi \subset \lambda_0 \textup{ irreducible} \textup{ with } \pi =\bigotimes_v \pi_v  \\ \pi_v \cong \tau_v \textup{ at } v \in \mS^{sq} \\ \pi_v \cong \mJ_v( \mu_{1,v}, \mu_{2,v}, s_v) \textup{ for some }s_v\in \bC \textup{ at } v \in S^d \\ 
\pi_v \cong \mJ_v( 1, 1, s) \textup{ for some }s\in \bC \textup{ at } v \notin  S^d \cup S^{sq} }} \prod\limits_{v \in \mS^d} \tr \pi_v(\phi_v).$$  
\end{corollary}

\begin{remark}
In the special case where all archimedean places $\mS_\infty$ are contained in $\mS^{sq}$, the sum is finite.  If we assume that the factor at one archimedean place is not square-integrable, we will observe that the sum is infinite. See \cite{Reznikov:Scatt} and \cite{Harish:Auto} for similar results in this direction.
\end{remark}

\subsection{Classes of cuspidal automorphic representations}
  The last corollary indicates the need for a more specialized notation.
\begin{defn}[The similarity class of automorphic representations]
Two irreducible, cuspidal automorphic representations $\pi_1 = \bigotimes_v (\pi_1)_v$ and $\pi_2 = \bigotimes_v (\pi_2)_v$ with the same algebraic central character are similar if for all $v \in \mathcal{S}$
\begin{itemize}
 \item either the representations $(\pi_1)_v$ and $(\pi_2)_v$ are isomorphic square-integrable representations of $\GL_2(\F)$, or
 \item the representations $(\pi_1)_v$ and $(\pi_2)_v$ satisfy $(\pi_j)_v = \mJ_v((\mu_1)_v,  (\mu_2)_v, (s_j)_v)$ for possibly distinct complex values $(s_j)_v$ and the same algebraic character $\mu$ of $\M(\F)$.
\end{itemize} 
In this case, we write $ \pi_1 \equiv \pi_2$ for the similarity relation, and
  \begin{align*} \fX_\pi \coloneqq \Big\{  \pi_0  \textup{ irreducible, cuspidal automorphic rep. with } \pi     \equiv \pi_0 \Big\}
\end{align*}
for the equivalence class of $\pi$. 
\end{defn}

\begin{defn}
Let $\fX$ be a similarity class of automorphic representations. Fix $\pi \in \fX$. Define the sets
\begin{align*}
\mS^{ps}(\fX) &= \{ v \in \mS: \pi_v \textup{ is a principal series} \}  \\
\mS^{sq}(\fX) &= \{ v \in \mS: \pi_v \textup{ is square-integrable} \}
\end{align*}
\end{defn}
\begin{defn}[The character $\mu_\fX$]
Define an infinite family of algebraic characters
\[ \mu_{\fX} = (\mu_{1,v}, \mu_{2,v})_{v  \in  \mS^{ps}(\fX)}.\]
\end{defn}
 \begin{defn}[The spectral parameter of $\pi$]
The spectral parameter is a family of complex parameters
\[ s_\pi = (s_v)_{v  \in  \mS^{ps}(\fX_\pi)}, \]
with $\pi_v \cong \mJ_v(\mu_{1,v}, \mu_{2,v}, s_v)$ for all $v \in S^{ps}(\fX_\pi)$.
\end{defn}

\begin{remark}
The set $\mS^{ps}$ and the collection of characters $\mu_{\fX}$ depends only on the equivalence class of a cuspidal automorphic representation, whereas the family of complex numbers $s_\pi = (s_v)_{v  \in  \mS^{ps}(\fX)}$ determines $\pi \in \fX$ up to isomorphism. 
\end{remark}

\subsection{Twists by one-dimensional representations}
Twisting by a one-dimensional representation is an automorphism of the cuspidal automorphic spectrum.
 \begin{lemma}
 Let $\omega$ and $\chi$ be algebraic Hecke characters. The following statements are equivalent for any unitary representation of $\pi$:
\begin{enumerate}[font=\normalfont]
 \item the representation $\pi$ is a cuspidal automorphic representation with central character $\chi$
 \item the representation $\pi \otimes \omega \circ\det$ is a cuspidal automorphic representation with central character $\chi\cdot  \omega^2$
\end{enumerate}
\end{lemma}
\begin{proof}
 This is a well-known property of the Mackey induction functor
 $$\bm \mL^2( \chi) \otimes \omega \circ \det \cong \bm \mL^2(\chi \otimes (\omega \circ \det|_{\Z(\bA)})).$$ 
Additionally, the cuspidality is preserved by tensoring with one-dimensional representations. 
\end{proof}
\begin{defn}[Minimal automorphic representations]
An irreducible, cuspidal automorphic representation $\pi = \bigotimes_v (\pi)_v$ with central algebraic character is called \textbf{minimal automorphic representation} if $\pi_v$ is  isomorphic to either
\begin{itemize}
 \item $\mJ_v(\mu_v, 1, s_v)$,
 \item a minimal supercuspidal representation,\footnote{see ~\ref{eq:minimal} and ~\ref{eq:minimal2} for a definition} or
 \item an irreducible proper subquotient of $\mJ_v(\mu_v, 1, s_v)$, i.e., to either the Steinberg representation $\St_v = \St_v(1_v)$ or a discrete series representation $D_k( \mu,1)$.
\end{itemize}
\end{defn}
We need only concern ourselves with minimal automorphic representations.
\begin{lemma}\label{lemma:twistauto}
Every irreducible, cuspidal automorphic representation $\pi$ is isomorphic to the tensor product of a minimal automorphic representation and $\omega \circ \det$ for an algebraic Hecke character $\omega$.
\end{lemma}
\begin{proof}
Let $\pi$ be an irreducible, cuspidal automorphic representation $\pi$, then Theorem~\ref{thm:flath} applies. Let $\mS^u$ be the set of valuations $v$, such that $(\pi_v)_v$ is not an unramified principal series representation. If $\pi_v$ is a supercuspidal representation, then there exists by definition an algebraic character $\omega_v$ of $\F^\times$, such that $\pi_v \otimes \omega_v \circ \det$ is minimal. The remaining $\pi_v$'s are isomorphic to a principal series representation / to an irreducible subquotient $\mJ_v(\mu_v, \omega_v^{-1} , s)$. Note that \( \pi_v \otimes \omega_v \circ \det \)      is isomorphic to a principal series representation / to an irreducible subquotient $\mJ_v(\mu_v \omega_v, 1 , s)$. 
Define the algebraic Hecke character $\omega = \otimes_{v \in \mS^u} \omega_v$ via strong approximation 
\[ \uF^\times \times \prod\limits_{v \in \mS^u \cup \mS_\infty} \F \times \prod\limits_v \o_v^\times  \twoheadrightarrow \bA^\times.\] 
The automorphic cuspidal representation 
\[ \pi \otimes \omega \circ\det = \bigotimes_v  \pi_v \otimes \omega_v \circ\det   \]
is minimal.
\end{proof}

\subsection{Classification of supercuspidal representations}
Let $\F$ be non-archimedean. Thus, there exist two conjugacy classes of open, maximal subgroups in $\GL_2(\F)$, both of which are compact modulo the center. The group $\Z(\F) \GL_2(\o_v)$ and the normalizer of the Iwahori subgroup $\Gamma_0(\p_v)$ are representatives of these conjugacy classes.
\begin{thm}[\cite{Kutzko:SuperGL2}]
Every irreducible, unitary, supercuspidal representation is isomorphic to a compact induction of a finite-dimensional, irreducible, unitary representation from either $\Z(\F) \GL_2(\o_v)$ or the normalizer of the Iwahori subgroup.
\end{thm}

\begin{defn}\index{(un)ramified supercuspidal representations}
We say that an irreducible, unitary, supercuspidal representation is unramified [ramified] if it compactly induced from $\Z(\F) \GL_2(\o_v)$ [from the normalizer of the Iwahori subgroup].  We write $\rho_\pi$\index{$\rho_\pi$ for $\pi$ supercuspidal} for the corresponding irreducible, unitary representation of the group $\Z(\F) \GL_2(\o_v)$ [of the normalizer of the Iwahori subgroup].
\end{defn}

The use of the term ``(un)ramified'' originates from a correspondence with local characters of (un)ramified quadratic extensions, and has a very different meaning in the context of principal series representation.

\section{The final spectral refinement}

\subsection{The result} We fix for the construction of the similarity class
\begin{itemize}
 \item a finite subset $\mS^{rps} \subset \mS$,
\item for each $v \in \mS^{rps}$, we fix one algebraic non-trivial character $\mu_v$ of $\GL_2(\F)$, 
 \item a finite subset $\mS^{\St} \subset \mS_f$, which is disjoint from $\mS^{rps}$,
 \item a finite subset $\mS^{sc} \subset \mS_f$, which is disjoint from all previous subsets, 
\item  for every $v \in \mS^{sc}$, one irreducible, minimal supercuspidal representation $\tau_v$ of $\GL_2(\F)$, 
 \item a finite subset $\mS^{sq}_\bR \subset \mS_\bR$,  
 \item and for each $v\in \mS^{sq}_\bR$ an integer $k_v \geq 2$.
\end{itemize}

We define $\fX$ as the set of minimal cuspidal automorphic representations \[ \pi=\bigotimes_v\pi_v,\] such that
\begin{itemize}  
\item $\pi_v \cong \mJ(1, 1,s)$ for each $v \in \mS  - \left(  \mS^{rps} \cup \mS^{\St} \cup \mS^{sc} \cup \mS^{sq}_\bR \right)$,
\item $\pi_v \cong \mJ(\mu_v, 1,s)$ for each $v \in \mS^{rps}$,
\item $\pi_v \cong \St_v$ for each $v \in \mS^{\St}$,
\item $\pi_v \cong \tau_v$ for each $v \in \mS^{sc}$,
\item $\pi_v \cong D_{k_v}(\sign^{k_v}, 1)$ for each $v \in \mS^{sq}_\bR$ for $k_v \geq 2$.
\end{itemize}

For the construction of the test function, we choose 
\begin{itemize}
 \item a finite subset $\mS^{\bm H} \subset \mS_f^{ps}$ of valuations (possibly intersecting with $\mS^{rps}$), at those places where we intend to analyze the trace of the Hecke operator of elements, and
 \item for every place $v\in \mS_\infty^{ps}$ a smooth, compactly supported function $g_v: \bR \rightarrow \bC$. 
\item  for $v \in \mS_\infty^{sq} = \mS_\bR^{sq}$, a smooth, compactly supported function $g_v: \bR \rightarrow \bC$ with
\[  \int\limits_\bR  g_v(x) \e^{-\frac{k_v-1}{2} x}  \d x = 1.\]
 \end{itemize} 

Set $g \coloneqq (g_v)_{v\in \mS_\infty}$, and define the Fourier transform
 \[ h_v(\xi) =\int\limits_{\bR} g_v(x)  \e^{\im x \xi}\d x.\]
Functions $g_v$ for $v \in \mS_\bR^{sq}$ satisfy $h_v( \im (k-1)/2) = 1$ and are eliminated from the final expressions.

 We define $\phi_{\fX, (g_v)_v, \mS^{\bm H}} \in \Ccinf(\GL_2(\F), \overline{\chi})$, such that
\begin{align*}\mbox{}\\
&J_{cusp} (\phi_{\fX, (g_v)_v, \mS^{\bm H}} ) \\
\mbox{}\\
&= \sum\limits_{\pi \in \fX} \prod\limits_{v \in \mS_\infty^{ps}} h_v(\im s_v(\pi)) \prod\limits_{v \in \mS^{\bm H}} \left( q^{s_v(\pi)} +  q^{-s_v(\pi)} \right) \prod\limits_{\substack{v \in \mS_\bR^{sq}\\ \pi_v\cong D_k(\mu_1, \mu_2)}}h_v( \im(k-1)/2).\end{align*}
We decide to normalize $h_v( \im(k-1)/2) = 1$ for all  $v \in \mS_\bR^{sq}$.

 \begin{remark}[Some special cases] \mbox{}
\begin{enumerate}[font=\normalfont]
 \item Dimension formula: if both $S^{\bm H}$ and $S_\infty^{ps}(\fX)$ are empty sets, the value $J_{cusp}(\phi_g)$ is a non-negative integer.
 \item Trace formula for Hecke operators: if $S^{\bm H}$ is non-empty, and $\mS_\bR^{sq} = \mS_{\infty}$, the analysis boils down to a derivation of an Eichler-Selberg type trace formula \cite{Selbergtrace}, \cite{Eichler}, \cite{Oesterle:Thesis} from the non-invariant Arthur trace formula \cite{JacquetLanglands}, \cite{GelbartJacquet}, \cite{Arthur:Rank1}. This is obtained in \cite{KnightlyLi} over $\mathbb{Q}$ for $k \geq 3$. The approach differs from the alternative method of Arthur \cite{Arthur:Hecke} via his invariant trace formula \cite{Arthur:Invariant}. The approach of Arthur does not require the explicit construction of test functions.  
\item Weyl laws: if $S_\infty^{sq}$ is empty, we obtain an analogue of the Selberg trace formula for the eigenvalues of the Laplace-Beltrami operator from the non-invariant Arthur trace formula. 
\end{enumerate}
 \end{remark}

\subsection{The definition of the test functions and its spectral properties}\label{section:globaltest}
\emph{From this point forward, we fix a unique similarity class $\fX$ of minimal automorphic representations. We choose as well a finite (possibly empty) subset $\mS^{\bm H} \subset \mS^{ps}(\fX)$.}

The function 
\[ \phi \coloneqq \phi_{(g_v)_v, \fX, \mS^{\bm H}} \in \Ccinf(\GL_2(\bA), \overline{\chi})\] 
is defined in dependence on the finite (possible empty) family of functions $(g_v)_{v \in \mS^{ps}_\infty}$, the equivalence class $\fX$, and the finite (possibly empty) set $\mS^{\bm H} \subset \mS^{ps}_f$ 
as a tensor product 
 \[ \phi \coloneqq  \bigotimes\limits_{v \in \mS} \phi_v \]
of elements $\phi_v \in \Ccinf(\GL_2(\F), \overline{\chi_v})$. 

In what follows, we will carefully define $\phi_v$ in all possible cases and explain its spectral properties. 
\subsubsection{Outline of the construction}
We will refer to later chapters for the proofs, but let us briefly outline the method.
\begin{enumerate}
\item The Factorization Theorem ~\ref{thm:flath} reduces the global problem to a local one. The one-dimensional representations do not occur as factors of cuspidal $\GL(2)$ automorphic representation, and they can be ignored in this context.
\item All similarity classes can be determined (up to a twist by $\sign \circ \det$ at the real places) by the restriction of their elements to all maximal subgroups $\underline{K}$, which are compact modulo the center.
\item Twisting by algebraic Hecke character $\chi$ with $\chi^2 = 1$ is an automorphism of $\bm \mL_0^2(\chi)$ (see Lemma~\ref{lemma:twistauto}). The imprecision at the real places is not annoying, but will be exploited to our advantage.
\item The algebra $\Ccinf(\GL_2(\F), \overline{\chi}_v)$ decomposes as $\underline{K}$ bi-module, and only the sub-algebras associated to an irreducible representation $\rho_v$ of $\underline{K}$ are important
\begin{align*}\mH(\GL_2(\F), \rho_v) \coloneqq \Big\{ \phi_v &\in \Ccinf(\GL_2(\F): \phi(g) =  \\ &\int\limits_{\underline{K}/\Z(\F)}\int\limits_{\underline{K}/\Z(\F)}   \phi_v(k_v^{-1} g  \tilde{k}_v) \tr \rho_v(k_v \tilde{k}_v^{-1}) \d k_v \d \tilde{k}_v \Big\}\end{align*}
(see Proposition~\ref{prop:Kexp}).
\item Decompose $\Res_{\underline{K}} \pi_v$ for all unitary, infinite-dimensional representations of $\GL_2(\bR)$, $\GL_2(\bC)$ and $\GL_2(\F)$ (see Theorems~\ref{thm:realKtype}, ~\ref{thm:complexKtype}, and  Propositions~\ref{prop:scKtype} and ~\ref{prop:KtypeQp}).
The notion of similarity class makes also sense for local representations. Two cases occur.
\begin{itemize}
 \item The similarity class is determined by one irreducible representation $\rho$ of a maximal compact-mod-center subgroup of $\GL_2(\F)$ (at the real place possibly only modulo twisting by $\sign \circ \det\!$): 
\begin{itemize}
         \item complex unramified principal series representations
         \item real principal series representations
        \item supercuspidal representations
        \item non-archimedean principal series representations
\end{itemize}
A representation $\pi$ is in the similarity class (modulo twists) if $\rho$ lies in the restriction of $\pi$. 
 \item The similarity class is determined by two irreducible representations $\rho_1$ and $\rho_2$ of a maximal compact-mod-center subgroup of $\GL_2(\F)$ (at the real place possibly only modulo twisting by $\sign \circ \det$):
\begin{itemize}
         \item complex ramified principal series representations
         \item discrete series representations
        \item Steinberg representations
\end{itemize} 
 A representation $\pi$ is in the similarity class (modulo twists) if $\rho_1$ lies in the restriction of $\pi$, but $\rho_2$ does not.  
\end{itemize} 
\item The distribution $\phi_v \mapsto \tr \pi_v \phi_v$ vanishes on $\mH(\GL_2(\F), \rho_v)$ if the contragedient $\check{\rho}_v$ is not contained in $\Res_{\underline{K}} \pi_v$ (see Corollary~\ref{cor:charvanish}). Linear combinations of elements from $\mH(\GL_2(\F), \rho_v)$ for different $\rho_v$ serve for our purposes.
 \end{enumerate}

It should also be noted that an additional feature is implicitly imposed. Each element $\phi_v$ of  $\mH(\GL_2(\F))$ satisfies for all $k \in \underline{K}_v$
\[ \phi_v(k^{-1}xk) = \phi_v(x), \]
The non-invariant Arthur trace formula is equivariant with respect to $\bm K$. 
The truncation operator is a $\bm K$-intertwiner by definition (see \cite{GelbartJacquet}*{page 229}).

\subsubsection{Case: $v$ non-archimedean, $\pi_v$ supercuspidal} 
If $v$ is non-archimedean, and $\pi_v$ is supercuspidal, the representation $\pi_v$ is isomorphic to the compact induction of an irreducible, finite-dimensional representation of an open subgroup $K_0$ of $\GL_2(\F)$, which is compact modulo the center of $\GL_2(\F)$ and contains the Iwahori subgroup $\Gamma_0(\p)$ \cite{Kutzko:SuperGL2}. Set $C=1$ if $K_0 =\Z(\F) \GL_2(\o_v)$, and $C= 2 /(q+1)$ if $K_0$ is the normalizer of the Iwahori subgroup. We define
\[ \phi_v (x) \coloneqq \begin{cases} C \tr \rho(x) , & x \in K, \\ 0,& x \notin K . \end{cases}  \]
 We have for all irreducible, smooth, admissible representations $\pi_0$ of $\GL_2(\F)$ the following formula (see Defn-Thm.~\ref{defn:testsc}):
\[ \tr \pi_0(\phi_v) = \begin{cases} 1, & \pi_0 \cong \pi_v,\\ 0, &\textup{otherwise}. \end{cases}\]
\subsubsection{Case: $v$ non-archimedean, $\pi_v$ Steinberg} If $v$ is non-archimedean, and $\pi_v$ is a Steinberg representation $\St_v$,  the function $\phi_v \coloneqq \phi_{\St_v}$ is defined as pseudo coefficient of $\St_v$ (see Definition-Theorem~\ref{defn:testst}).
  We have for all irreducible, smooth, admissible representations $\pi_0$ of $\GL_2(\F)$ the following formula
\[ \tr \pi_0(\phi_v) = \begin{cases} 1, & \pi_0 \cong \St_v(\omega),\\ -1 , & \pi_0 = \omega \circ \det\!, \\ 0, &\textup{otherwise}. \end{cases}\]
\subsubsection{Case: $v$ non-archimedean, $\pi_v$ principal series \underline{without} interest in the Hecke eigenvalue} If $v \notin \mS^{\bm H}$ is non-archimedean, and $\pi_v$ is a principal series representation $\mJ(\mu_v,1,s)$ where $\mu_v$ is an algebraic character of $\F^\times$, we fix the function $f_\mu \in \Ccinf(\GL_2(\F), \overline{\mu})$                      
            \[ f_\mu \left(  x\right) = \begin{cases} \mu(a)   , &x =        \sma a & b \\ c & d \smz  \in \Z(\F) \Gamma_0(\p_v^N) , \\ 0, & \textup{else} . \end{cases}\]
         Here, we set $N=0$ if $\mu_v = 1$ and $N =\min \{n \geq 1: \mu_v\big|_{1+\p^n} = 1 \}$ otherwise. The group $\Gamma_0(\p_v^N)$ is $\GL_2(\o_v)$ for $N=0$ and 
$$\Gamma_0(\p^N) \coloneqq \left\{ \sma a & b \\ c & d \smz \in  \GL_2(\o_v) : a,d \in \o_v^\times, c \in \p_v^N\right\}$$
 for $N>0$. Next, define
            \[ f_\mu^{K_v} (x) \coloneqq \int\limits_{\GL_2(\o_v)} f_\mu(k^{-1}xk) \d k\]
and compute
\[  C \coloneqq  \int\limits_{\N(\F)} f_\mu^{K_v} (n) \d n.\]
We set 
\[ \phi_v(x) \coloneqq   f_\mu^{K_v} (x)/C.\]
We have for all irreducible, smooth, admissible representations $\pi_0$ of $\GL_2(\bF)$ the following formula (Defn.-Thm.~\ref{defn:testprinc}):
 \[ \tr \pi_0 (\phi_v) = \begin{cases} 1, & \pi_v  \cong \mJ_v( \mu_v,1,s), \\ 1, & \pi_0 \textup{ and } \mu_v \textup{ both  trivial}, \\0 , & \textup{otherwise}.\end{cases}\]
\subsubsection{Case: $v$ non-archimedean, $\pi_v$ principal series \underline{with} interest in the Hecke eigenvalue} If $v \in \mS^{\bm H}$ is non-archimedean, and $\pi_v$ is a principal series representation $\mJ(\mu_v,1,s)$ with $N$ as above, we fix the function $f_\mu \in \Ccinf(\GL_2(\F), \overline{\mu})$  supported on 
             \[ \Gamma_0(\p^N) \sma \w & 0 \\ 0 & 1 \smz\Z(\F) \Gamma_0(\p^N) \]
              with
            \[ f_\mu \left(  \sma a & b \\ c & d \smz \sma \w_v & 0 \\ 0 & 1 \smz \sma a' & b' \\  c' & d' \smz \right) = \mu(aa').\]
Next, define
            \[ f_\mu^{K_v} (x) \coloneqq \int\limits_{\GL_2(\o_v)} f_\mu(k^{-1}xk) \d k\]
and compute
\[  C \coloneqq q^{1/2} \int\limits_{\N(\F)} f_\mu^{K_v} \left( \sma \w & 0 \\ 0 & 1 \smz n\right) \d n.\]
We set 
\[ \phi_v(x) \coloneqq   f_\mu^{K_v} (x)/C.\]
We have for all irreducible, smooth, admissible representations $\pi_0$ of $\GL_2(\bF)$ the following formula (see Defn.-Thm.~\ref{defn:uHecke} /~\ref{defn:rHecke}):
 \[ \tr \pi_0 (\phi_v) = \begin{cases} q_v^{-s} + q_v^{s}, & \pi_v  \cong \mJ_v( \mu_v,1,s), \\ q_v^{1/2} + q_v^{-1/2}, & \pi_0 \textup{ and } \mu_v \textup{ both  trivial}, \\ 0 , & \textup{otherwise}.\end{cases}\]
\subsubsection{Case: $v$ real, $\pi_v$ prinicpal series representation} If $v$ is a real place, and $\pi_v = \mJ(\mu_1, \mu_2, s)$ a principal series representation, then  
         \begin{enumerate}
          \item          if $\mu_1\mu_2 = 1$, set 
\[ \rho_0 \coloneqq \Ind_{\SO(2)}^{\O(2)} 1\]
 and choose 
                              \[ \phi_v\left(k_1 \sma x & 0 \\ 0 & x^{-1} \smz k_2\right) 
															= \frac{\tr \rho_0(k_1 k_2)}{\dim(\rho_0)} \phi_v\left( \sma x & 0 \\ 0 & x^{-1} \smz\right)  \qquad k_1,k_2 \in \Z(\bR)\O(2), x >0  \]
such that 
  \[ g_v( t^{2} - t^{-2} - 2) = \frac{t}{2} \int\limits_{\N(\bR)} \phi_v \left( \sma t & 0\\ 0 &  t^{-1} \smz n\right) \d n, \qquad t > 0.\]
          We have for all irreducible, smooth, admissible representations $\pi_0$ of $\GL_2(\bR)$ the following formula (see Prop.~\ref{prop:realjacquet} and~\ref{prop:realoned}):
       \[ \tr \pi_0 (\phi_v) = \begin{cases} \frac{h_v(\im s)}{2}, & \pi_0  \cong \mJ( 1,1,s), \\ \frac{h_v(\im s)}{2}, & \pi_0 \cong \mJ(\sign, \sign,s), \\
			\frac{h_v(\im /2 )}{2} , & \pi_0 \textup{ trivial or } \pi_0 =\sign\circ\det, \\ 
			0, &  \textup{else}.\end{cases}\]
         \item if $\mu_1 \mu_2 = \sign$, set 
                     \[ \rho_1 \coloneqq \Ind_{\SO(2)}^{\O(2)} \left( \sma \cos \theta  & \sin \theta \\ - \sin \theta  & \cos \theta \smz   \mapsto \e^{\im \theta} \right)\]
                     and choose 
                              \[ \phi_v\left(k_1 \sma x & 0 \\ 0 & x^{-1} \smz k_2\right) =2 \frac{\tr \rho_1(k_1 k_2)}{\dim(\rho_1)} \phi_v\left( \sma x & 0 \\ 0 & x^{-1} \smz \right)  \qquad k_1,k_2 \in \Z(\bR) \O(2), x>0, \]
                    such that 
  \[ g_v( t^{2} - t^{-2} - 2) =  t \int\limits_{\N(\bR)}\phi_v \left( \sma t & 0\\ 0 &  t^{-1} \smz n \right) \d n \qquad t > 0.\]
We have for all irreducible, smooth, admissible representations $\pi_0$ of $\GL_2(\bR)$ the following formula (see Prop.~\ref{prop:realjacquet}):
 \[ \tr \pi_0 (\phi_v) = \begin{cases} h_v(\im s), & \pi_v  \cong \mJ( \sign,1,s), \\
                        0 , & \textup{else}.\end{cases}\] 
\end{enumerate}  
\emph{The normalization $1/2$ is used because we combine $\mJ( \dots)$ and its twist by $\sign \circ \det$. It has computational advantages not to separate them.}

\subsubsection{Case: $v$ real, $\pi_v$ discrete series representation} If $v$ is a real place, and $ \pi_v$ is a discrete series representation $D_n$ of weight $n \geq 2$, the function $\phi_v \coloneqq \phi_{D,n}$ is defined as in Corollary~\ref{cor:DSpseudo}. Let us briefly explain the construction here:
        \begin{enumerate}
         \item  Define $\epsilon_j: \SO(2) \rightarrow \bC^1$ by $\epsilon_j \left(\sma \cos \theta & \sin \theta \\ -\sin \theta & \cos \theta \smz \right) = \e^{\im j \theta}$. 
         \item Define $\rho_j = \Ind_{\SO(2)}^{\O(2)} \epsilon_j$.
         \item Choose for $j=n$ and $j =n-2$ a function $\phi_v^j \in \Ccinf(\GL_2(\bR), \sign^n)$, with
                 \[ \phi_v^j\left(k_1 \sma x & 0 \\ 0 & x^{-1} \smz k_2\right) = \phi^j_v\left( \sma x & 0 \\ 0 & x^{-1} \smz\right) \frac{\tr \rho_j(k_1k_2)}{\dim(\rho_j)},  \qquad k_1,k_2 \in \Z(\bR)\O(2), x>0 \]
                and   
                \[ g_v( t - t^{-1} - 2) =\frac{t}{2}  \int\limits_{\N(\bR)} \phi_v^j \left( \sma t & 0\\ 0 &  t^{-1} \smz n\right) \d n, \qquad t > 0.\]
           \item Define for $n>2$
             \[ \phi_v  \coloneqq  \phi_v^n - \phi_v^{n-2}.\]
         \end{enumerate}
  We have for all irreducible, smooth, admissible representations $\pi_0$ of $\GL_2(\bR)$ the following formula (see Corollary~\ref{cor:DSpseudo}):
       \[ \tr \pi_0 (\phi_v) = \begin{cases} 1/2, &\pi_0 \cong D_n(\mu_1, \mu_2), \\ -1/2, & \pi_0  \textup{ one-dimensional and }n=2\\ 0  , &\textup{else}.\end{cases}\]
 \emph{The normalization $1/2$ is used because we combine $D_n( \dots)$ and its twist by $\sign \circ \det$. This has computational advantages.}
\subsubsection{Case: $v$ complex} If $v$ is a complex valuation, then $\pi_v$ is isomorphic to a principal series representation $\mJ( \epsilon_n, 1,s)$ with $\epsilon_n : \bC^\times \mapsto (z/|z|)^n$. We may require $n\geq 0$, and $s$ is uniquely determined.
        \begin{enumerate}[font=\normalfont] 
         \item If $n = 0$, then we choose $\phi_v \in \Ccinf(\GL_2(\bC) //\Z(\bC) \U(2))$, such that 
                           \[ g_v(t+ t^{-1}  -2) =2 \uppi t^2 \int\limits_{\N(\bC)} \phi_v\left( \sma t & 0 \\ 0 & t^{-1} \smz n \right) \d n, \qquad t>0.\]
        \item If $n = 1$, let $\rho_1$ denote the representation of $\U(2)$ on $\bC^2$. We choose $\phi_v \in \Ccinf(\GL_2(\bC), \epsilon_{-1})$, such that
                            \[ 
                                            \phi_v\left(k_1 \sma x & 0 \\ 0 & x^{-1} \smz k_2\right) = \frac{\tr \left( \rho_1 \otimes \epsilon_{-1} \right) (k_1k_2)}{\dim(\rho_1)} \phi_v\left( \sma x & 0 \\ 0 & x^{-1} \smz \right), \qquad k_1,k_2 \in \Z(\bC)\U(2), x>0 
                            \]
                          and
                           \[ g_v(t+ t^{-1}  -2) = 2 \uppi t^{2} \int\limits_{\N(\bC)} \phi_v\left( \sma t & 0 \\ 0 & t^{-1} \smz n \right) \d n, \qquad t > 0.\]
        \item If $n >2$, define the irreducible representation $\rho_j = \textup{Sym}^j(\bC^2)$ of $\U(2)$ for $j=n$ and $j=n-2$ as the symmetric tensor of the natural representation $\U(2)$ on $\bC^2$.
                We can paste $\rho_j$ and the trivial representation of $\Z(\bC)$ to a representation of $\Z(\bC) \U(2)$, which we also call $\rho_j$.
             Choose functions $\phi_v^{j} \in \Ccinf( \GL_2(\bC), \overline{\epsilon_n})$ with
                \[    \phi_v^j\left( k_1 \sma x & 0 
							\\ 0 &x^{-1} \smz k_2\right) = \phi_v^j\left(\sma x & 0 
							\\ 0 &x^{-1} \smz \right)  \frac{\tr \rho_j \otimes \epsilon_{-n} (k_1k_2)}{\dim(\rho_j)} , \qquad k_1,k_2 \in \Z(\bC)\U(2), x >0 \]
                and
                  \[ g_v( t + t^{-1} -2) = 2 \uppi t^{2} \int\limits_{\N(\bC)} \phi_v^k\left( \sma t & 0 \\ 0 & 1 \smz n \right) \d n, \qquad t>0.\]
               Define
             \[ \phi_v \coloneqq \phi_v^n - \phi_v^{n-2}.\]
            \end{enumerate} 
We have for all irreducible, smooth, admissible representations $\pi_0$ of $\GL_2(\bC)$ the following formula (see Prop.~\ref{prop:complexjacquet} and~\ref{prop:complexoned}):
       \[ \tr \pi_0 (\phi_v) = \begin{cases} h_v(\im s), &\pi_0 \cong \mJ_v(\epsilon_n,1,s), \\ h_v(\im/2), & \pi_0  \textup{ trivial and }n=0 , \\ 0  , &\textup{else}.\end{cases}\]

\subsubsection{Conclusion}
No finite-dimensional representation can occur as a factor of an automorphic representation. Note that at the real places $v$, one actually counts the similarity class $\fX$ and $\fX \otimes \sign \circ \det\!$, and therefore we have normalized by $1/2$.

  Because $\phi$ is defined as a tensor product of local functions, the character distribution of every cuspidal automorphic representation $\pi  =\bigotimes_v \pi_v$ factors according to Theorem~\ref{thm:flath}
\[ \tr \pi \left( \phi \right) = \prod\limits_{v} \tr \pi_v(\phi_v).\]    
We conclude:
\newcommand{\pX}{ \phi_{ \fX, \underline{g}, \mS^{\bm H}}}
\begin{thm}[The spectral refinement]\label{thm:globaltest}
The function \[ 
                 \pX\in \Ccinf(\GL_2(\bA),\overline{\chi})
             \]
satisfies 
\begin{align}\label{eqs} J_{cusp}(\pX) &= \sum\limits_{\pi \subset \fX} \left( \prod\limits_{v \in \mS^{\bm H}} q_v^{-s_v(\pi)}  + q^{+s_v(\pi)}  \right)   \left( \prod\limits_{v \in \mS_\infty - \mS^{sq}_\bR} h_v(s_v(\pi)) \right).  \end{align}  
\end{thm}
 \begin{remark}\mbox{}
  \begin{itemize}
 \item The explicit construction is important, since we will appeal to the non-invariant Arthur trace formula. 
\item  If the set $\mS_\infty - \mS^{sq}_\bR$ is empty, for example, if $\uF$ is a global function field, Equation~\ref{eqs} should be read as
\begin{align} J_{cusp}(\pX) &= \sum\limits_{\pi \subset \fX} \left( \prod\limits_{v \in \mS^{\bm H}} q_v^{-s_v(\pi)}  + q^{+s_v(\pi)}  \right) . \end{align} 
 \item Let $R$ be the number of real places of $\uF$. We actually count all $\fX$-equivalence classes of cuspidal representations, which are related by twisting via an algebraic character of $\GL_2(\bR)^{R}$. This has the particular advantage that all hyperbolic distributions associated to all $\alpha \in \uF\times -\{1\}$, which are negative for at least one real embedding, vanish. Twisting with $\chi \circ \det$ with $\chi^2 = 1$ for an algebraic Hecke character has a nontrivial effect on the trace formula. 
\end{itemize}
 \end{remark}

\section{The coarse Arthur trace formula}
The coarse (non-invariant) Arthur trace will give us an alternative expression for \(  J_{\textup{cusp}}(\phi) \) in terms of (weighted) orbital integrals and the remaining irreducible subrepresentations in the orthogonal difference 
$\bm \mL^2(\chi) \ominus \bm \mL_0^2(\chi)$. In this section, we provide it merely as an abstract identity of distributions. We subdivide the distribution occurring in the the Arthur trace formula in subfamilies. Each subfamily will require a somewhat different approach.   

We introduce an equivalence relation on $\GL_2(\uF)$, whose equivalence classes are in one-to-one correspondence with conjugacy classes of $\PGL_2(\uF)$.
\begin{defn} 
Two elements $\gamma, \gamma'$ of $\GL_2(\uF)$ are equivalent if  there exists $z \in \Z(\uF)$ and $g \in \GL_2(\uF)$ with
\[ z\gamma  = g^{-1}\gamma' g .\]
For $\gamma \in \GL_2(\uF)$, we denote the equivalence class of $\gamma$ by $\{ \gamma \}$.
 \end{defn}
The Arthur trace formula now provides us with an alternative formula for $J_{\textup{cusp}}$ in terms of distributions associated to equivalence classes $\{ \gamma \} \subset \GL_2(\uF)$, Eisenstein series, residues of Eisenstein series, and one-dimensional representations.
The coarse Arthur trace formula takes the following form: 
\begin{empheq}[box=\mygreybox]{align}     J_{\textup{cusp}} (\phi)  & = \sum\limits_{ \{ \gamma \} \subset \GL_2(\uF)} J_{\{ \gamma \}} (\phi) + \sum\limits_{\substack{ \mu : \M(\uF) \backslash \M(\bA) \rightarrow \bC \textup{ alg.} \\ \mu_1\mu_2 =\chi}}  J_{\mu}^{\textup{Eis}}(\phi) \nonumber\\ 
                 &         +  \sum\limits_{\substack{\omega :\uF^\times \backslash \bA^\times \rightarrow \bC \textup{ alg.} \\ \omega^2 =\chi}}  J_{\omega}^{\textup{res}}(\phi)    +         \sum\limits_{\substack{ \omega: \uF^\times \backslash \bA^\times \rightarrow \bC \textup{ alg.} \\ \omega^2 = \chi }}  J_{\omega}^{\textup{one}}(\phi) .
\end{empheq}
In the subsequent sections, we will define and study the distributions on the right-hand side. In particular, we will try to evaluate them for the function $\phi_{ \fX, (g_v)_v, \mS^{\bm H}}$ and express the value in terms of $g_v$ and $h_v$.

The distributions associated to equivalence classes in $\GL_2(\uF)$ need to be subdivided in to several classes. To this end, let us classify the equivalence classes in $\GL_2(\uF)$. The theory of the Rational Canonical Form implies that the conjugacy class of an element in $\GL_2(\uF)$ is uniquely determined by its characteristic and minimal polynomial together. An element $\gamma \in \GL_2(\uF)$ has characteristic polynomial 
\[ \det(X- \gamma) = X^2 - \tr( \gamma) X - \det(\gamma).\]
It is conjugated to
\begin{itemize}
 \item  a scalar multiple of the identity, i.e., \( \sma  \alpha & 0 \\ 0 & \alpha \smz \) for $\alpha \in \uF^\times$ if the minimal polynomial has degree one: $(X-\alpha)$, 
 \item  a hyperbolic element, i.e.,   \( \sma  \alpha & 0 \\ 0 & \beta \smz \) for $\alpha, \beta \in \uF^\times$ with $\beta \neq \alpha$ if the minimal polynomial factors into two distinct linear factors $(X-\alpha)(X-\beta)$,
 \item  a parabolic element, i.e.,\( \sma \alpha & \alpha \\ 0 & \alpha \smz \) for $\alpha \in \uF^\times$ if the minimal polynomial factors into $(X-\alpha)^2$,
 \item   an elliptic element, i.e., \( \sma 0 & 1 \\  -\det(\gamma) & \tr(\gamma) \smz\) if the characteristic polynomial is irreducible. 
\end{itemize}
This results in a partition of the discrete group $\GL_2(\uF)$ into equivalence classes: 
 \[  \GL_2(\uF)  = \left\{ \sma 1 & 0 \\ 0 & 1 \smz \right\} \amalg   \left\{ \sma 1 & 1 \\ 0 & 1 \smz \right\}   \amalg \coprod\limits_{\alpha \in \uF^\times-\{1\}}  \left\{ \sma \alpha  & 0 \\ 0 & 1 \smz \right\}   \amalg \coprod\limits_{\gamma \textup{ ell.} \bmod \Z(\uF)} \left\{ \gamma \right\}\]
This partition in turn yields a partition of the distributions:
   \begin{empheq}[box=\mygreybox]{align*}    \sum\limits_{ \{ \gamma \} \subset \GL_2(\uF)} J_{\{ \gamma \}} (\phi) =  &J_{1} (\phi) \qquad \textup{(identity distr.)}\\ 
&  + \quad J_{par}(\phi) \qquad \textup{(parabolic distr.)}\\
& + \sum\limits_{\alpha \in \uF^\times -\{1\}} J_\alpha(\phi)  \qquad \textup{(hyperbolic distr.)} \\
 & + \sum\limits_{\gamma  \textup{ ell.} \bmod \Z(\uF)} J_\gamma(\phi) \qquad \textup{(elliptic distr.)}. 
\end{empheq}
In coming sections, we rigorously define these distributions, explain their factorizations into local distributions, and compute them for the test function constructed in Section~\ref{section:globaltest}.
\section{The identity distribution}\label{section:globalidentity}
The identity distribution is the distribution associated to the equivalence class of the identity in $\GL_2(\uF)$ (see \cite{GelbartJacquet}*{page 235(6.13)}, \cite{JacquetLanglands}*{page 271(i)}).

Let $\mS_*^{ups} / \mS_*^{rps}$ denote the subset of places in $\mS_*$, at which the factor of $\pi \in \fX$ is a un-/ramified principal series representations.
\begin{defn}[The global identity distribution]
The identity distribution on $\phi \in \Ccinf(\GL_2(\bA), \overline{\chi})$ is given as 
\[ J_{1} (\phi) = \vol( \GL_2(\uF) \Z(\bA) \backslash \GL_2(\bA)) \phi(1) .\]
\end{defn}
\begin{thm}[The identity distribution at $\phi_{\fX}$]\label{thm:globalidentity}   \mbox{}
\begin{itemize}
 \item    If $\mathcal{S}^{\bm H}(\pi) \neq \emptyset$ is non-empty, the identity distribution vanishes 
\[ J_{1} (\pX) = 0.\]
\item     If $\mathcal{S}^{\bm H}(\pi) = \emptyset$  is empty and $\uF$ is a global number field, the identity distribution evaluates to
\begin{align*}J_{ 1 } (\pX)  = &   \frac{\vol( \GL_2(\uF) \Z(\bA) \backslash \GL_2(\bA))  }{(4 \uppi)^{|\mathcal{S}_\bR|} (8 \uppi^2)^{|\mathcal{S}_\bC|} } \times   V_\fX      \\ 
                       & \qquad \times  \prod\limits_{\substack{v \in  \mathcal{S}_\bR^{sq} \\ \pi_v \cong D_n(\mu_1, \mu_2)}} \frac{k_v-1}{2} \\ 
                                         & \qquad \times    \prod\limits_{v \in \mathcal{S}_\bR^{ups}}  \int\limits_{\bR}    h_{\phi_v} (r) r \tanh(\uppi r) \d r \\
                                         & \qquad \times       \prod\limits_{v \in \mathcal{S}_\bR^{rps}}  \int\limits_{\bR}    h_{\phi_v} (r) r \coth( \uppi r)  \d r \\
                                     & \qquad \times    \prod\limits_{v \in \mathcal{S}_\bC^{ups}}  \int\limits_{\bR}    h_{\phi_v} (r) r^2  \d r.  \end{align*}
\item If $\mathcal{S}^{\bm H}(\pi) = \emptyset$  is empty and $\uF$ is a global function field,  the identity distribution evaluates to \index{$V_\fX$}
\begin{align*}J_{ 1 } (\phi_\fX)  = \vol( \GL_2(\uF) \Z(\bA) \backslash \GL_2(\bA)) V_\fX  .  \end{align*}
\end{itemize}
The constant $V_\fX$ is given as a product
\[ V_\fX  =\prod\limits_{v \in \mS_f} C_v,\]
with all but finitely many factors being one. The local factors are computed for $v \in \mS_{f}$
\[ C_v \coloneqq \begin{cases} 1, & \pi_v \textup{ unramified principal series}, \\ 
                                             \frac{q_v^N - q_v^{N-1}}{q^{\lfloor N/2 \rfloor} - 1}   , & \pi_v \textup{ ramified p.s.} \,  \pi_v = \mJ(\mu_v,1,s_v),\, \cond(\mu_v)  =\p_v^N, \\
                                               q-1, & \pi_v \textup{ Steinberg}, \\
                                               \dim(\rho_{\pi_v}), &       \pi_v \textup{ unramified supercuspidal}, \\
                                                \frac{q_v+1}{2} \dim(\rho_{\pi_v}), &      \pi_v \textup{ ramified supercuspidal}.
 
                 \end{cases}\]
  \end{thm}
\begin{defn}  \index{$C_\fX$}
   \[ C_\fX \coloneqq \frac{\vol( \GL_2(\uF) \Z(\bA) \backslash \GL_2(\bA))  }{(4 \uppi)^{|\mathcal{S}_\infty|}} \times   V_\fX           \prod\limits_{\substack{v \in  \mathcal{S}_\bR^{sq} \\ \pi_v \cong D_n(\mu_1, \mu_2)}} \frac{n-1}{2}.\]
\end{defn}

\begin{proof}
By definition, we have a factorization 
 \[ \phi_\fX(1) = \prod\limits_{v \in \mS} \phi_v(1_v).\]
The local distributions $\phi_v(1_v)$ are explicitly computed in the subsequent chapters. The computation at the non-archimedean places follows directly from of the definitions (see Proposition~\ref{prop:padicone}). The computations at the archimedean places are more subtle (see Corollary~\ref{cor:realDSone} for $v \in \mathcal{S}_\infty^{sq}$, Proposition~\ref{prop:realone} for $\mS_\bR^{ps}$, and Proposition~\ref{prop:complexone} for $\mS_\bC^{ps}$).                
\end{proof}

\section{The distributions of the one-dimensional representations}\label{section:globaloned}
Restriction: If not specified otherwise, we will assume that the factors at complex places are unramified, i.e., isomorphic to $\mJ(1,1,s_v)$.

\begin{defn}[Global one-dimensional character distributions]
 Let $\omega$ be a global algebraic Hecke character. The global one-dimensional distribution is the character distribution of $\omega \circ \det$
\[ J_\omega^{one} (\phi) = \int\limits_{ \Z(\bA) \backslash \GL_2(\bA)}   \phi(g) \omega(\det(g))\d g.\]
\end{defn}

\begin{thm}[The one-dimensional distribution at $\phi_{\fX}$] \label{thm:globaloned}   \mbox{}
\begin{itemize}
 \item The global one-dimensional distribution vanishes if one of the following factors occurs for $\pi \in \fX$ at one place $v$:
        \begin{enumerate}[font=\normalfont]
         \item a supercuspidal representation 
         \item a discrete series representation of weight $n \geq 3$
          \item a ramified principal series representation
        \end{enumerate}
 \item If the global one-dimensional distribution does not vanish, then $\chi = 1$ and it evaluates in the number field case  to
         \begin{align*}\sum\limits_{\omega^2 =1} J_\omega(\phi) & =  -(-1)^{\# \mS^\St+ \# \mS^{sq}_\infty}   \prod\limits_{v \in \mS^{\bm H}} \left(  q_v^{-1/2} + q_v^{1/2} \right)  \prod\limits_{v \in \mS_\infty^{ups}} h_v( \im /2 ) , \end{align*}
 \item and in the function field case to
         \[             \sum\limits_{\omega^2 =1} J_\omega(\phi) = - (-1)^{\# \mS^\St} \prod\limits_{v \in \mS^{\bm H}} \left( q_v^{-1/2} + q_v^{1/2} \right) .\]
\end{itemize} 
\end{thm}
\begin{proof}
 Appeal to the integral identity~\ref{eqintad}
\[ \int\limits_{ \Z(\bA) \backslash \GL_2(\bA)}   \phi(g) \omega(\det(g))\d g =\prod\limits_v   \int\limits_{ \Z(\bA) \backslash \GL_2(\F)}   \phi_v(g) \omega_v(\det(g_v))\d g_v.\]
The relevant local computations are Propositions~\ref{prop:realoned}, ~\ref{prop:complexoned} and ~\ref{prop:padiconed}.
\end{proof}

\section{The parabolic distributions}\label{section:globalpara}
This distribution can be found in \cite{JacquetLanglands}*{page 271(v)} and \cite{GelbartJacquet}*{page 235(6.13), page 244(7.14+7.17)}.
 
\begin{defn}[The global parabolic distribution]
The parabolic distribution
\[ J_{par} (\phi)  = \lim\limits_{s \rightarrow 1} \left( \zeta_{\GL_2(\uF)} ( s, \phi) - \frac{\lim\limits_{z \rightarrow 1} (s-1) \zeta_{\GL_2(\uF)} ( z, \phi)}{(s-1)} \right) \]
 is the finite part at $s=1$ of the zeta-type integral
\[ \zeta_{\GL_2(\uF)} ( s, \phi) \coloneqq \int\limits_{\bA^\times} F_\phi(x) \left\| x \right\|^{s} \d^\times_{\bA} x,\]
where $F_\phi$ depends upon $\phi$ in the following way:
\[ F_\phi(x) = \int\limits_{\bm K} \phi\left( k^{-1} \sma 1 & x \\ 0 & 1 \smz k\right) \d k.\]
Recall that the compact group $\bm K$ was defined as a product
 \[\bm K \coloneqq \prod\limits_{v \in \mathcal{S}_f} \GL_2(\o_v)  \times \prod\limits_{v \in \mathcal{S}_\bR} \O(2) \times \prod\limits_{v \in \mathcal{S}_\bC} \U(2).\]\index{$\bm K$} 
\end{defn}
 \begin{lemma}[Factorization]\label{para:fact}
Let $\zeta_{\uF}(s)$ be the global zeta function of $\uF$. Consider the Laurent expansion at $s=1$ with
\[ \zeta_{\uF}(s) = \frac{\lambda_{-1}}{s-1} + \lambda_0 + \dots,\]
and assume that $\phi =\bigotimes_v \phi_v$. We then obtain a factorization into local distributions
\begin{align*}
  J_{par} (\phi)  =& \lambda_0 \prod\limits_{v \in \mathcal{S}} \frac{\zeta_{\GL_2(\F)}( 1, \phi_v)}{\zeta_v(1)} \\ &  + \lambda_{-1} \sum\limits_{w \in S} \left( \prod\limits_{v \in \mathcal{S}-\{w \}}\frac{\zeta_{\GL_2(\F)}( s, \phi_v)}{\zeta_v(1)} \right) \cdot \frac{\partial}{\partial s} \Big|_{s=1} \frac{\zeta_{\GL_2(\mathds{F}_w)}( 1, \phi_w)}{\zeta_w(s)}  
\end{align*}
for 
\[\zeta_{\GL_2(\F)}( s, \phi_v) \coloneqq      \int\limits_{\F} \int\limits_{K_v} \phi_v\left( k_v^{-1} \sma 1 & x_v \\ 0 & 1 \smz  k_v \right) \left| x_v \right|_v^{s}  \d k_v \d^\times_v x_v,\]
with  $K_v$  being the compact group $\GL_2(\o_v)$, $\U(2)$,  or $\O(2)$ if $v$ is non-archimedean, complex, or real.
\end{lemma}
\begin{proof}
See \cite{Gelbart}*{Proposition 1.2, page 47}. We use l'H{\^o}pital's rule and compute the finite part at $s= 1$ of 
\[    \zeta_{\GL_2(\uF)} ( s, \phi) =   \frac{\zeta_{\GL_2(\uF)} ( s, \phi)}{\zeta_{\uF}(s)} \cdot \zeta_{\uF}(s) \] 
as the value
\[ \frac{ \lambda_{0}   \frac{\zeta_{\GL_2(\uF)} ( s, \phi)}{\zeta_{\uF}(s)}     +  \lambda_{-1}   \frac{\partial}{\partial s} \Big|_{s=1} \frac{\zeta_{\GL_2(\uF)} ( s, \phi)}{\zeta_{\uF}(s)}    }{s-1}  .\]
The sum and products are finite, since we have for the characteristic function \[ \mathds{1}_v\coloneqq\mathds{1}_{\GL_2(\o_v) \Z(\F)}\]
 of the set $\GL_2(\o_v) \Z(\F)$ that
\[    \zeta_{\GL_2(\F)}( s,\mathds{1}_{v}) = \zeta_v(s).\]
We are able to factor
\begin{align*}     \frac{\zeta_{\GL_2(\uF)} ( 1, \phi)}{\zeta_{\uF}(1)} & = \prod\limits_{v : \phi_v \neq \mathds{1}_{v}}    \frac{\zeta_{\GL_2(\F)} ( 1, \phi_v)}{\zeta_{\F}(1)},    \\
                    \frac{\partial}{\partial s} \Big|_{s=1} \frac{\zeta_{\GL_2(\uF)} ( s, \phi)}{\zeta_{\uF}(s)}   &    = \sum\limits_{u : \phi_u \neq \mathds{1}_{u}} 
\left(  \frac{\partial}{\partial s} \Big|_{s=1} \frac{\zeta_{\GL_2(\mathds{F}_u)} ( s, \phi_u)}{\zeta_{\mathds{F}_u}(s)}   \right) \\  
 & \qquad \times  \prod\limits_{v \neq u: \phi_v \neq \mathds{1}_{v}}   \frac{\zeta_{\GL_2(\F)} ( 1, \phi_v)}{\zeta_{\F}(1)}.                                                                                                                                                                           \qedhere
\end{align*}
                                                                                             
\end{proof} 
The next theorem has numerous special cases. This is inconvenient but inevitable.

\begin{thm}[The parabolic distribution at $\phi_{\fX}$]\label{thm:globalpara}                                  \mbox{}
   \begin{enumerate}[font=\normalfont]   
 \item Let $\uF$ be either a global function field or an algebraic number field. If the set $\mathcal{S}^{\bm H} \neq \emptyset$ is non-empty, or $\pi \in \fX$ has at least two square-integrable factors, the parabolic distribution vanishes
  \[ J_{par}(\phi_\fX) =0.\]
\item Let $\uF$ be an algebraic number field. If $\fX$ contains no square-integrable places and $\mathcal{S}^{\bm H} = \emptyset$, the parabolic distribution evaluates to\footnote{See remark below for more concrete expressions of $\frac{\partial}{\partial s}\Big|_{s=1}  \frac{\zeta_{\GL_2(\mathds{F}_u)} (\phi_{u}, s) }{\zeta_{u}(s) }$.}
\begin{align*}
&    J_{par} (\phi_\fX) \\   &= \left(  \lambda_0 +   \lambda_{-1}   \sum\limits_{\substack{ u \in \mathcal{S}^{ups}_f \\ \cond(\mu_u) = \p_u^{N_u}}}  \log(q_u) \left( (N_u+1/2) +\frac{-N_u+3/2}{q_u}  + \frac{1}{2q_u^2} \right)  \right) \prod\limits_{v \in \mathcal{S}_\infty} g_{\phi_v}( 0 )  \\ 
               & \qquad +  \lambda_{-1}    \sum\limits_{u \in \mathcal{S}_\infty} \left( \prod\limits_{v \in \mathcal{S}_\infty - \{u\} } g_{\phi_{v}}(0) \right)   \frac{\partial}{\partial s}\Big|_{s=1}  \frac{\zeta_{\GL_2(\mathds{F}_u)} (\phi_{u}, s) }{\zeta_{u}(s) } .
 \end{align*}
\item Let $\uF$ be a global function field. If $\fX$ contains no square-integrable places and $\mathcal{S}^{\bm H} = \emptyset$, the parabolic distribution evaluates to
\begin{align*}
&    J_{par} (\phi_\fX)   =\left(  \lambda_0 +   \lambda_{-1}   \sum\limits_{\substack{ u \in \mathcal{S}^{ups}_f \\ \cond(\mu_u) = \p_u^{N_u}}}  \log(q) \left( (N+1/2) +\frac{-N+3/2}{q}  + \frac{1}{2q^2} \right) \right) . \end{align*}
\item Let $\uF$ be an algebraic number field and $\mathcal{S}^{\bm H} = \emptyset$. If exactly one factor $\pi \in \fX$ is a discrete series representation (say at $u$), and all remaining factors are principal series representations, the parabolic distribution evaluates to
\begin{align*}
    J_{par} (\phi_\fX)  =    \lambda_{-1} \prod\limits_{v \in \mathcal{S}_{\infty}^{ps}} g_{\phi_v}(0).  
\end{align*}
 \item Let $\uF$ be an algebraic number field and $\mathcal{S}^{\bm H} = \emptyset$. If exactly one factor $\pi \in \fX$ is a Steinberg representation (say at $u$), and all remaining factors are principal series representations, the parabolic distribution evaluates to
\begin{align*}
    J_{par} (\phi_\fX)  =     \lambda_{-1} \frac{\log(q_u)}{1- q_u^{-1}}  \prod\limits_{v \in \mathcal{S}_{\infty}^{ps}} h_{\phi_v}(0).  
\end{align*}
 \item Let $\uF$ be an algebraic number field and $\mathcal{S}^{\bm H} = \emptyset$. If exactly one factor $\pi \in \fX$ is a supercuspidal representation (say at $u$), and all remaining factors are principal series representations, the parabolic distribution evaluates to
\begin{align*}
    J_{par} (\phi_\fX)  &=       \lambda_{-1}  C_{\pi_u} \log(q_u)(1-q_u^{-1}) \sum\limits_{k=0}^\infty  q_u^{-k} \dim \Hom_{\N(\p_u^k)}( \rho_{\pi_u}, 1) \times   \prod\limits_{v \in \mathcal{S}_{\infty}^{ps}} h_{\phi_v}(0), \\ 
   &C_{\pi_u} =\begin{cases} 1, & \pi_u \textup{ unramified s.c.}\\ \frac{2}{q+1} , &  \pi_u \textup{ ramified s.c.}. \end{cases}  
\end{align*}
\item   Let $\uF$ be a global function field and $\mathcal{S}^{\bm H} = \emptyset$. If exactly one factor $\pi \in \fX$ is a Steinberg representation (say at $u$), and all remaining factors are principal series representations, the parabolic distribution evaluates to
\begin{align*}
    J_{par} (\phi_\fX)  =    \lambda_{-1} \frac{\log(q_u)}{1- q_u^{-1}}.  
\end{align*}
\item   Let $\uF$ be a global function field and $\mathcal{S}^{\bm H} = \emptyset$. If exactly one factor $\pi \in \fX$ is a supercuspidal representation (say at $u$), and all remaining factors are principal series representations, the parabolic distribution evaluates to
\begin{align*}
    J_{par} (\phi_\fX)  =     \lambda_{-1}  C_{\pi_u} \log(q_u)(1-q_u^{-1}) \sum\limits_{k=0}^\infty  q_u^{-k} \dim \Hom_{\N(\p_u^k)}( \rho_{\pi_u}, 1),
\end{align*}      where the constant $C_{\pi_u}$ was introduced above.
\end{enumerate}
\end{thm}
\begin{proof}
 This follows immediately from the factorization in Lemma~\ref{para:fact} and the local computations in the real/complex/non-archimedean case by Proposition~\ref{prop:realpara} and Corollary~\ref{cor:realDSpara}/Proposition~\ref{prop:complexpara}/Proposition~\ref{prop:padicpara}.
\end{proof}
\begin{remark}\label{rem:globalpara}
Particularly in the instances indicated by point (2) in Theorem~\ref{thm:globalpara}, it is often convenient to have more explicit formulas for $u \in \mathcal{S}_\infty^{ps}$ as given by Propositions~\ref{prop:complexpara} and ~\ref{prop:complexpara}.
\begin{itemize}
 \item  If $u$ is a real place with an \emph{unramified} principal series representation,  the derivative is
\begin{align*} \frac{\partial}{\partial s}\Big|_{s=1}  \frac{\zeta_{\GL_2(\mathds{F}_{u})} (\phi_{u}, s) }{\zeta_{u}(s) }  =&  \frac{h_{\phi_{u}}(0)}{4} + \left(  \log(\uppi) - \gamma_0  \right) \frac{g_{\phi_{u}}(0)}{2}     \\ 
                                          &  - \frac{1}{2 \uppi} \int\limits_{\bR} h_{\phi_{u}}(r) \frac{\Upgamma'}{\Upgamma} (1+\im r)\d r. \end{align*}
 \item  If $u$ is a real place with a \emph{ramified} principal series representation,  the derivative is   
\begin{align*} \frac{\partial}{\partial s}\Big|_{s=1}  \frac{\zeta_{\GL_2(\mathds{F}_{u})} (\phi_{u}, s) }{\zeta_{u}(s) }  =& \frac{h_{\phi_{u}}(0)}{4} +  \left( \log(\uppi)- \gamma_0   \right) \frac{g_{\phi_{u}}(0)}{2}   \\
                                                        &   - \frac{1}{2 \uppi} \int\limits_{\bR} h_{\phi_{u}}(r) \frac{\Upgamma'}{\Upgamma} (1+\im r)\d r \\
                                   & + \int\limits_{0}^\infty  \frac{g_\phi(r)}{\e^{r/2}-\e^{-r/2}} \left( 1 - \cosh\left(\frac{r}{2}\right) \right) \d r.
\end{align*}
   \item  If $u$ is a complex place place with a \emph{ramified} principal series representation, the derivative is
\begin{align*} \frac{\partial}{\partial s}\Big|_{s=1}  \frac{\zeta_{\GL_2(\mathds{F}_{u})} (\phi_{u}, s) }{\zeta_{u}(s) }  =&                                                     \frac{h_\phi(0)}{4}   + \left(   \log(2 \uppi)- \gamma_0 \right) g_{\phi_{u}}(0)    \\&    -      \frac{1}{ \uppi}    \int\limits_0^{\infty}   h_{\phi_{u}}(r)   \frac{\Upgamma'}{\Upgamma}(1-2 \im r)  \dr r .     \end{align*}
 \item    If $u$ is a complex place place with an \emph{unramified} principal series representation, I have not provided a formula.
\item We know for supercuspidal $\pi$ as $\dim (\rho_\pi) \rightarrow \infty$:
\[ \sqrt{\dim(\rho_{\pi_u})}^{1-\epsilon}  \ll    \sum\limits_{k=0}^\infty  q_u^{-k} \dim \Hom_{\N(\p_u^k)}( \rho_{\pi_u}, 1) \ll  \sqrt{\dim(\rho_{\pi_u})}.\] 
\end{itemize}
\end{remark}

\section{The hyperbolic distributions}\label{section:globalhyper}
References are  \cite{JacquetLanglands}*{page 272(iv)}, \cite{GelbartJacquet}*{page 244(7.15+7.19)}, \cite{Gelbart}*{Prop.1.1, page 46}.
We choose the compact groups $\bm K \coloneqq \prod\limits_{v \in \mathcal{S}_f} \GL_2(\o_v)  \times \prod\limits_{v \in \mathcal{S}_\bR} \O(2) \times \prod\limits_{v \in \mathcal{S}_\bC} \U(2)$.
By definition of $\GL_2(\bA)$ as a restricted product of the groups $\GL_2(\F)$, the Iwasawa decomposition carries over 
\[ \B(\bA) \bm K = \GL_2(\bA) .\]
Let $\Delta_{\B(\bA)}$ be the modular character of $\B(\bA)$. We define
\[ H : \GL_2(\bA) \rightarrow (0, \infty) , \qquad H(bk) = \Delta_{\B(\bA)} (b)  \quad( b \in \B(\bA), k \in \bm K).\]

\begin{defn}[A global hyperbolic distribution]
Let $\alpha \neq 1$. The global hyperbolic distribution of $\sma \alpha & 0 \\ 0 & 1 \smz$ is defined as the integral
\[ J_{\alpha} ( \phi) \coloneqq \int\limits_{\N(\bA)} \phi(n^{-1} \sma  \alpha & 0 \\ 0 & 1 \smz n ) \log H\left( \sma 0 & -1 \\ 1 &0 \smz n\right) \d n.\] 
\end{defn}
\begin{lemma}[Factorization]\label{hyperfac}
Let $\phi = \bigotimes_v \phi_v$ be a tensor product, then  the distribution factors into local integrals
\begin{align*}J_{\alpha} ( \phi) = &  \frac{- \lambda_{-1}}{2} \sum\limits_{u \in S}  \left(  \prod\limits_{v \in \mathcal{S} -\{ u \}} \int\limits_{\N(\F)} \phi_v(n_v^{-1} \sma  \alpha & 0 \\ 0 & 1 \smz n_v ) \d n_v \right) \\
                                                                                                                             & \cdot                 \int\limits_{\N(\mathds{F}_u)} \phi_u(n_u^{-1} \sma  \alpha & 0 \\ 0 & 1 \smz n_u )  \log H_v\left( \sma 0 & -1 \\ 1 & 0\smz n_u \right) \d n_u.                                                                                    
\end{align*}
\end{lemma}
\begin{proof}
 Note that
\[    \Delta_{\B(\bA)} \left(  \sma a & * \\ 0 & b \smz      \right) = \left| a/b \right|_{\bA}  = \prod\limits_{u \in S} \left| a_u/ b_u \right|_u.\]
Consequently, we obtain the decomposition
\[         \log  H\left( g \right)  =\sum\limits_{u \in S} \log H_u (g_u).\qedhere\]
\end{proof}

 \begin{thm}
 If the set $\mathcal{S}^{sq}_f \neq \emptyset$ is non-empty, or $\mathcal{S}^{sq}$ contains at least two places, all hyperbolic distributions vanish
\[ \sum\limits_{\alpha \in \uF^\times} J_\alpha( \phi_\fX ) =0  .\] 
 \end{thm}
\begin{proof}
This follows from the factorization in Lemma~\ref{hyperfac} and Proposition~\ref{prop:padichyper}. 
\end{proof}

\begin{defnthm}     \mbox{}
Assume $\mathcal{S}_f^{ps} =\mathcal{S}_f$. Define 
\begin{align*} \fX_{\textup{hyper}} = \Bigg\{ x \in \uF^\times :  v(x) = & \begin{cases} \pm 1, & v\in \mathcal{S}^{\bm H} \\ 0, & v \in \mathcal{S}_f - \mathcal{S}^{\bm H} \end{cases} \textup{ and } \\
                                                              &  \iota(x)  > 0 \textup{ for all real embeddings } \iota : \uF \hookrightarrow \bR \Bigg\}. \end{align*} 

We have that
\[ \sum\limits_{\alpha \in \uF^\times} J_\alpha( \phi_\fX) =   \sum\limits_{\alpha \in \fX_{\textup{hyper}}} J_\alpha( \phi_\fX).\]
\end{defnthm}
\begin{proof}
 The restriction follows from both Proposition~\ref{prop:padichyper} and the fact that $\phi_v$ at real places is supported on elements with positive determinant only. 
\end{proof}

 Now let us prove a particularly useful theorem for spectral estimates.
\begin{defnthm}\label{defn:thmszg}    \mbox{}
Let $\uF$ be an algebraic number field and $\mS^{\bm H} = \emptyset$.
   \begin{itemize}
\item
Consider an archimedean place $v$, and define the constant
\[ \epsilon_v \coloneqq \min \left\{  1, \log |\alpha|_v : \alpha \in \o_{\uF}^\times -\{1\}, \, \alpha >0   \textup{ for all real embeddings} \right\} > 0.\]
\item
The function $\phi_\fX$ has \textup{small hyperbolic support}\index{small hyperbolic support} at $v$ if $g_v$ is supported in \[ [ -\epsilon_v ,\epsilon_v].\] 
\item
If $\phi_\fX$ has small hyperbolic support at one real place, all hyperbolic distributions vanish
\[ \sum\limits_{\alpha} J_\alpha ( \phi_\fX)=0.\]
\item
If $\uF$ has no real places, and $\phi_\fX$ has small hyperbolic support at all complex places, all but a finite number of hyperbolic distributions vanish:
\[ \sum\limits_{\alpha \in \uF^\times-1} J_\alpha ( \phi_\fX)=\sum\limits_{\beta \neq 1 \textup{ root of unity}} J_\beta ( \phi_\fX).\]
\end{itemize} 
\end{defnthm}
\begin{proof}
 These support considerations follow directly from Proposition~\ref{prop:realhyper} and Proposition~\ref{prop:complexhyper}. Note that $\left| x \right|_v = 1$ for all places $v$ forces $x$ to be a root of unity. An algebraic number field with real places has only $\pm 1$ as roots of unity. Both the local hyperbolic and weighted hyperbolic integrals vanish for $\alpha= -1$ at a real place $v$ because $\phi_v$ is supported on matrices of positive determinant.
\end{proof}

\begin{thm}
The sum of hyperbolic distributions vanishes if the elements in $\fX$ have at least one square-integrable factor.
\end{thm}
\begin{proof}
For discrete series representations, it will shortly be seen that $J_*(\phi_{\fX})$ does not depend upon $g_v$ for $v \in \mS_\bR^{sq}$. Additionally, we can choose the support sufficiently small as suggested in Definition-Theorem~\ref{defn:thmszg}, so that the sum of hyperbolic distributions must vanish for this particular choice, hence for all choices $g_v$. 
\end{proof}
\begin{remark}
The choices by Knightly and Li \cite{KnightlyLi}*{Section 24.5, page 286} differ from mine and so do their computations. Knightly and Li choose the factors at the archimedean primes to be a matrix coefficient whereas I choose them to be pseudo matrix coefficients.
\end{remark}

\begin{thm}
If $\uF$ is a global function field and all factors are principal series representations, the hyperbolic distribution evaluates to 
    \begin{align*} J_\alpha  ( \phi_\fX) = & \sqrt{ \prod\limits_{v \in \mathcal{S}^{\bm H}}  q_v } \left( \sum\limits_{u \in \mathcal{S}^{\bm H}} \frac{2 \log(q_u)}{\sinh(\log(q_u) /2)} \right) .  \end{align*}
\end{thm}

\section{The elliptic distributions}\label{section:globalelliptic}
References are \cite{JacquetLanglands}*{page 272(ii+iii)}, \cite{GelbartJacquet}*{page 244(*)}.

We say that an elliptic element in $\GL_2(\uF)$ is \textbf{separable/inseparable} if the characteristic polynomial is separable/inseparable. Algebraic number fields do not have inseparable elliptic elements, since every field of characteristic zero is a perfect field.

We say that an elliptic element is \textbf{split/elliptic at a place $v$} if the characteristic polynomial is reducible/irreducible over $\F$.
If $v$ is finite, we say that an elliptic element is \textbf{unramified/ramified elliptic at non-archimedean $v$} if  $\gamma$ is elliptic at $v$ and $v( \det \gamma)$ is even/odd.

\begin{defn}
Let $\gamma $ be an elliptic element in $\GL_2(\uF)$, then the global elliptic distribution associated to $\gamma$ is defined as
\[ J_\gamma (\phi) = \int\limits_{\GL_2(\uF) \Z(\bA) \backslash \GL_2(\bA)} \phi(g^{-1} \gamma g) \d g .\]
\end{defn}

\begin{lemma}\label{lemma:globalelliptic}
 If the function $\phi = \bigotimes_v \phi_v$  factors, so does the elliptic distribution: 
\begin{align*} J_\gamma(\phi) =  &\; \vol(\GL_2(\uF) \Z(\bA) \backslash \GL_2(\bA)_\gamma) \\ 
                        &\times \quad  \prod\limits_{v : \phi_v \neq \mathds{1}_{\GL_2(\o) \Z(\F)}} \qquad \int\limits_{\GL_2(\F)_\gamma \backslash \GL_2(\F)} \phi_v(g_v^{-1} \gamma g_v) \d g_v. \end{align*}
 The local factor at $v$ is 
\begin{itemize}
 \item an elliptic orbital integral if $\gamma$ is elliptic at $v$,
 \item a hyperbolic orbital integrals if $\gamma$ is separable and splits at $v$,
\item a parabolic orbital integral if $\gamma$ is inseparable and splits at $v$ (this only happens if $\uF$ is a global function field).
\end{itemize}
\end{lemma}

\begin{remark}
The elliptic global integral factors in elliptic, hyperbolic, or parabolic local integrals. I have expressed these distributions explicitly via the results in ~\ref{prop:realhyper}, ~\ref{prop:realelliptic}, ~\ref{cor:realDSelliptic}, ~\ref{prop:complexhyper}, ~\ref{prop:complexhyper}, ~\ref{prop:padichyper}, ~\ref{prop:padicpara}, and ~\ref{prop:padicelliptic}. I confine myself with weaker explicit estimates because I cannot possibly organize explicit statements in an efficient manner. 
\end{remark}

 \begin{thm}\label{thm:hyperellvanish}
Let $\uF$ be an algebraic number field. If $\mS^{\bm H} $ is empty, then there exists a constant $C>0$ such that if all functions $g_v $ are supported in $[C, C]$, we have that 
\[ J_{hyper}( \phi_\fX) +   J_{ell}( \phi_\fX)  = \sum\limits_{\alpha \neq 1 \textup{ root of unity}} J_\alpha(\phi_\fX) +   \sum\limits_{\substack{ \gamma \textup{ elliptic} \\ |\gamma|_v = 1 } } J_\gamma(\phi_\fX).\]
There exists a uniform upper bound, such that all functions $g_v$ with support in $[ -C, C]$ satisfy
\[  J_{hyper}( \phi_\fX) +   J_{ell}( \phi_\fX)  \ll_{\uF}  C_\fX \prod\limits_{v \in \mS_\infty^{ps}} \max_{x \in \bR} |g_v(x)|, \]
where the implicit constant depends only upon $\uF$.
 \end{thm}
\begin{proof}
 Assume first that no factor is unramified supercuspidal. Then $\gamma$ has to lie in $\Z(\F) \GL_2(\o_v)$ for all non-archimedean places, and thus is an element of $\Z(\uF) \GL_2(\o_{\uF})$ (see Propositions~\ref{prop:padichyper} and ~\ref{prop:padicelliptic}). Let $x^\dagger$ be the complex conjugate of a matrix in $\GL_2(\bC)$ or $\GL_2(\bR)$.  Pick representatives from each equivalence class in $\GL_2(\o_{\uF})$. Certainly 
\[ C_v= \min \{  \log \left| \tr\left(\gamma \gamma^\dagger\right)  / \det (\gamma) \right|    : \gamma \in \GL_2(\o_{\uF})  \textup{ with }   \left| \tr\left( \gamma \gamma^\dagger\right)  / \det (\gamma) \right|_v \neq 1 \} \]
 is a positive constant for all archimedean places $v$ by the discreteness of $\GL_2(\o_{\uF})  \subset \GL_2(\bR)^{n_1} \times \GL_2(\bC)^{n_2}$. If $g_v$ is supported in $[-C_v, C_v]$ for all archimedean places, all elliptic and hyperbolic distribution vanish, except for those who are elliptic at real places, or whose eigenvalues are roots of unity at complex places. Again, by discreteness there are finitely many of these elements, since representatives have to lie discretely in
\[ \SO(2)^{n_1} \times \U(2)^{n_2}.\]
The number of elements 
\[ \GL_2(\o_{\uF}) \cap   \SO(2)^{n_1} \times \U(2)^{n_2} \]
is therefore finite. The bounds are easily obtained by bounding each factor
\[ J_\gamma( \phi_\fX)        \ll  \prod\limits_{v : \pi_v \textup{ unr.superc.}} \tr \rho_{\pi_v}(\gamma) \prod\limits_{v \in \mS_\infty^{ps}} \max_{x \in \bR} |g_v(x)|. \]
We estimate trivially
\[  \tr \rho_{\pi_v}(\gamma) \leq \dim \rho \]
and obtain accordingly an admittedly coarse estimate
\[ \prod\limits_{v : \pi_v \textup{ unr.superc.}} \tr \rho_{\pi_v}(\gamma) \leq V_\fX.\]

Let $\n$ be the square-free ideal of $\o_{\uF}$, such that for each prime ideal $\p$ diving $\n$, the factors of elements in $\fX$ are ramified supercuspidal representations at the corresponding $v$. 
By the construction of $\phi_\fX$ and Proposition~\ref{prop:padicelliptic}, the element $\gamma$ must be a ramified elliptic element at those $v$'s.
Indeed, $\gamma $ is an element of the normalizer $N\Gamma_0(\n)$ of 
\[ \Gamma_0(\n) = \left\{ \sma \alpha & \beta \\ \gamma & \delta \smz \in \GL_2(\o_{\uF}) : \gamma \in \n \right\} \]
in $\GL_2(\uF)$.  This is a discrete subgroup of $ \GL_2(\bR)^{n_1} \times \GL_2(\bC)^{n_2}$ as well, so  we can similarly define 
\[ C_v= \min \{  \log \left| \tr \left(\gamma \gamma^\dagger\right)  / \det (\gamma) \right|    : \gamma \in N\Gamma_0(\n)  \textup{ with }   \left| \tr\left(\gamma \gamma^\dagger\right)  / \det (\gamma) \right|_v \neq 1 \}. \] 
The constant is positive for all archimedean places $v$. Choose all $g_v$ supported in $[-C_v, C_v]$, then all elliptic and hyperbolic distributions vanish, except for those that are elliptic at real places or whose eigenvalues are roots of unity at complex places. Now appeal to the rational canonical form, and the invariance of $J_\gamma$ in $\gamma$, i.e., $\gamma \in \GL_2(\uF)$ is conjugated to
\[ \sma 0 & 1 \\ - \det \gamma & \tr \gamma \smz.\]
Modulo $\Z(\uF)$, we can assume that
\[ \prod\limits_{ v\in \mS_f} \left| \det \gamma  \right|_v = \prod\limits_{ v\in \mS_f} \left| \n  \right|_v ,\]
and $\tr \gamma \in  \o_{\uF}$.
We apply the product formula~\ref{thm:productformula}. It implies that an unramified elliptic element of the above form splits at least at one archimedean place. For the integral not to vanish, the place at which it splits needs to be complex and the value a root of unity. There are at most finitely many roots of unity in an algebraic number field. We have obtained the estimate
 \begin{align}\label{eq:esthep} J_{hyper}( \phi_\fX) +   J_{ell}( \phi_\fX)  \leq \# \{ \textup{ roots of unity } \}   \dim(\rho) \ll_{\uF} V_\fX .\qedhere\end{align}
\end{proof}

\section{The Eisenstein distributions}\label{section:globaleisen}
References are \cite{JacquetLanglands}*{page 272(vii+viii)}, \cite{GelbartJacquet}*{page 243(7.13), page 244(7.19)}, \cite{Gelbart}*{Prop.1.3, page 48}

Let $\mu_1, \mu_2$ be algebraic Hecke characters, and $s$ a complex number. 
\begin{defn}[The global parabolic induction and the intertwiner]
 Define the global parabolic induction as the right regular representation of $\GL_2(\bA)$ on
\begin{align*} \mJ( \mu_1, \mu_2, s) = \Big\{ f: \GL_2(\bA) & \rightarrow  \bC \textup{ smooth} \\
                                                                                         & f\left( \sma a & * \\ 0 & b \smz k\right) = \mu_1(a) \mu_2(b) \left\| a / b \right\|_{\bA}^{s+1/2} f(k) ,\; k \in \bm K \Big\}.
\end{align*}
We write $\phi \mapsto \mJ(\mu_1, \mu_2, s, \phi)$ for the integrated representation. We define the intertwiner for $\Re s >1/2$
\begin{align*} \mM(\mu_1, \mu_2, s)  :        \mJ( \mu_1, \mu_2, s) &  \rightarrow \mJ( \mu_2, \mu_1, - s)     \\
                                                      f    &  \mapsto \mM  f :\left(  g \mapsto \int\limits_{\N(\uF) \backslash \N(\bA)}  f\left( \sma 0& -1 \\1&0 \smz n g\right) \d n \right).\end{align*}
\end{defn}
We say that $h : \bC \times \GL_2(\bA)  \mapsto  \bC$ is a section of $(\mu_1, \mu_2)$ if the function satisfies $h(s,k) = h(s',k)$ for all $s, s' \in \bC$ and $k \in \bm K$. Then $s \mapsto \mM(\mu_1, \mu_2, s) h(s,g)$ is holomorphic as a function for $\Re s>1/2$ for all $g \in \GL_2(\bA).$

 Factorize
\[ \mu_1 = \bigotimes_v \mu_{1,v}, \qquad \mu_{2} = \bigotimes_v \mu_{2,v}.\]

\begin{defn}[The local parabolic induction and the intertwiner]
 Define the local parabolic induction as the right regular representation of $\GL_2(\F)$ on
\begin{align*} \mJ_v( \mu_{1,v}, \mu_{2,v}, s) = \Big\{ f: & \; \GL_2(\F)  \rightarrow  \bC \textup{ smooth} \\
                                                                                         & f\left( \sma a & * \\ 0 & b \smz k_v\right) = \mu_{1,v}(a) \mu_{2,v}(b) \left| a/ b \right|_{v}^{s+1/2} f(k_v) ,\; k_v \in  K_v \Big\}.
\end{align*}
We write $\phi \mapsto \mJ(\mu_1, \mu_2, s, \phi)$ for the integrated representation. We define the intertwiner for $\Re s >1/2$
\begin{align*} \mM(\mu_1, \mu_2, s)  :        \mJ( \mu_1, \mu_2, s) &  \rightarrow \mJ( \mu_2, \mu_1, - s)     \\
                                                      f    &  \mapsto \mM  f :\left(  g \mapsto \int\limits_{\N(\uF) \backslash \N(\bA)}  f\left( \sma 0& -1 \\1&0 \smz n g\right) \d n \right).\end{align*} 
\end{defn}

\begin{thm}[Meromorphic continuation of the intertwiner \cite{GelbartJacquet}*{Theorem 4.19, page 227}]\label{thm:scattfact}
The operator-valued function  $s \mapsto \mM(\mu_1, \mu_2,s)$ admits a meromorphic continuation to all $s \in \bC$ with the functional equation
\[         \mM(\mu_2, \mu_1, -s)         \mM(\mu_1, \mu_2, s)   =  \textup{Identity}.\]
It admits a factorization for $\Re s > 1/2$ into local intertwiners
\[   \mM(\mu_1, \mu_2, s)  = \bigotimes_{v \in \mathcal{S}} \mM((\mu_1)_v, (\mu_2)_v, s).\]  
\end{thm}

This factorization is not applicable to $\Re s = 0$. This is unfortunate beause the Eisenstein distribution is defined as the following integral \cite{JacquetLanglands}*{page 271(vi)}, \cite{GelbartJacquet}*{page 240(6.37)}, \cite{Gelbart}*{page 48(1.4)}:
\begin{defn}[Eisenstein distributions]
If $\uF$ is an algebraic number field, the Eisenstein distribution associated to $(\mu_1, \mu_2)$ is defined as
\[ J_{\mu_1, \mu_2}(\phi)    = \frac{1}{4 \uppi}\int\limits_{\Re s = 0}  \tr \left(  \mM(\mu_2,\mu_1, -s) \mM'(\mu_1,\mu_2, s) \mJ(\mu_1,\mu_2, s, \phi)\right)  \d s,\]
if $\uF$ is a global function field whose field of constants has cardinality $q_F$, the Eisenstein distribution is defined as
\[ J_{\mu_1, \mu_2}(\phi)    = \frac{\log(q_F)}{4 \uppi}\int\limits_{0}^{2 \uppi / \log(q_F)}  \tr \left(  \mM(\mu_2,\mu_1, -s) \mM'(\mu_1,\mu_2, s) \mJ(\mu_1,\mu_2, s, \phi)\right)  \d s.\] 
\end{defn}

Let $\mS_*^{ups} / \mS_*^{rps}$ denote the subset of places in $\mS_*$, at which the factor of $\pi \in \fX$ is a un-/ramified principal series representations.
\begin{thm}[Explicit computation of the Eisenstein distribution]\label{thm:globalEis}                     \mbox{}
 \begin{enumerate}[font=\normalfont]
 \item  If $\fX$ contains a supercuspidal integrable representation, then all Eisenstein distributions vanish.
  \item  If $\fX$ contains at least two square-integrable representations, then all Eisenstein distributions vanish.
 \item If $\fX$ contains exactly one Steinberg representation, the Eisenstein distribution evaluates to
\begin{align*} J_{\textup{Eis}}(\phi)    =&\frac{2 \log(q_u)}{4 \uppi}  \prod\limits_{v \in \mS^{\bm H}} \ \left( q_v^{-\frac{1}{2}}+q_v^{\frac{1}{2}} \right)  \prod\limits_{v \in \mS_\infty^{ps}} h_v\left(\im /2\right)
\end{align*}
in the number field case, and to
\begin{align*} J_{\textup{Eis}}(\phi)    =&\frac{2 \log(q_{\uF}) \log(q_u)}{4 \uppi}  \prod\limits_{v \in \mS^{\bm H}} \ \left( q_v^{-\frac{1}{2}}+q_v^{\frac{1}{2}} \right)  \prod\limits_{v \in \mS_\infty^{ps}} h_v\left(\im /2\right)
\end{align*}
in the function field case. 
 \item If $\fX$ contains exactly one discrete series representation of weight $n\geq 2$, the Eisenstein distribution evaluates to
\begin{align*} J_{\textup{Eis}}(\phi)    =&\frac{1}{4 \uppi} \prod\limits_{v \in \mS^{\bm H}}  \left( q_v^{-\frac{n-1}{2}}+q_v^{\frac{n-1}{2}} \right) \prod\limits_{v \in \mS_\infty^{ps}} h_v\left(\im \frac{n-1}{2}\right).
\end{align*}
  \item If $\fX$ contains only principal series representations, define for $\Re s>0$
\[ \frac{\Lambda_{v}^{\fX} (2s)}{\Lambda_v^\fX(2s+1)}  =\begin{cases} \frac{\zeta_v(2s)}{\zeta_v(2s+1)} = \frac{1-q_v^{-2s-1}}{1-q_v^{-2s}} , & v \in \mS_f^{ups},
\\ 1, & v \in \mS_f^{rps}, 
\\ \frac{\zeta_\bR(2s)}{\zeta_\bR(2s+1)}  =    \frac{\sqrt{\uppi} \Upgamma(s)}{\Upgamma(s+1/2)},  & v \in \mS_\bR^{ups}\\
 \im \frac{L_\bR(2s, \sign)}{L_\bR(1+2s, \sign)}  =    \im \frac{\sqrt{\uppi} \Upgamma(s+1/2)}{\Upgamma(s+1)}, & v \in \mS_\bR^{rps}, \\
 2 \uppi \frac{\zeta_\bC(2s)}{\zeta_{\bC}(2s+1)} = \frac{1}{2s}, & v \in \mS_\bC^{ups}.\end{cases} \]
We define for $\Re s >1/2$ the function
\[ \Lambda^\fX_{\uF} (s) = \prod\limits_v   \Lambda_{v}^{\fX} (s),\]
which admits a meromorphic continuation to the whole complex plane. 
In the number field case, the Eisenstein distribution evaluates to
\begin{align*}
 J_{\textup{Eis}}(\phi) &= \frac{1}{4 \uppi}\int\limits_{\Re s = 0}  \prod\limits_{v \in \mS^{\bm H}} (q_v^{s} + q_v^{-s}) \prod\limits_{v \in \mS_\infty} h_v(\im s)  \frac{\partial \log}{\partial s} \frac{\Lambda_{\uF}^{\fX}(2s)}{\Lambda_{\uF}^\fX(2s+1)} \d s. 
\end{align*}
In the function field case, the Eisenstein distribution evaluates to
 \begin{align*}
 J_{\textup{Eis}}(\phi) &= \frac{\log(q_F)}{4 \uppi} \int\limits_{0}^{2 \uppi / \log(q_F)}   \prod\limits_{v \in \mS^{\bm H}} (q_v^{s} + q_v^{-s})  \frac{\partial \log}{\partial s} \frac{\Lambda_{\uF}^{\fX}(2s)}{\Lambda_{\uF}^\fX(2s+1)} \d s. 
\end{align*}
\end{enumerate}
\end{thm}   
We will require the following lemma:
\begin{lemma}[Factorization and $\bm K$-invariance]\label{scattfactor}      \mbox{}
 \begin{enumerate}[font=\normalfont]
\item
If $\phi =\bigotimes_v \phi_v$ factors as a tensor, then the character distribution of the Eisenstein series factors as well
\[ \tr \mJ( \mu_1, \mu_2, s, \phi)      = \prod\limits_{v \in \mathcal{S}}  \tr  \mJ_v( (\mu_1)_v, (\mu_2)_v, s, \phi_v).\]
 \item
The intertwiner $\mM(\mu_2,\mu_1, -s) \mM'(\mu_1,\mu_2, s)$ acts on every irreducible $\bm K$-isotype $\rho$ of $\mJ(\mu_1, \mu_2, s)$ by a scalar $\frac{ \lambda(\mu_1, \mu_2, \rho, s)'}{\lambda(\mu_1, \mu_2,\rho, s)}$. The function 
\[s \mapsto \lambda(\mu_1, \mu_2, \rho, s)\] 
is a meromorphic function with a functional equation
 \[ \lambda(\mu_2, \mu_1, \rho, -s) \lambda(\mu_1, \mu_2, \rho, s) = 1.\]
For $\Re s>1/2$, factorizing $\mu_j = \bigotimes_v \mu_{j,v}$ for $j=1,2$ and $\rho = \bigotimes_v \rho_v$, we obtain a factorization of the scalar 
\[       \lambda(\mu_1, \mu_2, \rho, s) = \prod\limits_{v}  \lambda_v( \mu_{1,v}, \mu_{2,v}, \rho_v, s).\]
\item
Assume furthermore that for all places $v$, the element $\phi_v$ lies in $\mH(\GL_2(\F), \rho_v)$ for an irreducible representation of $\rho_v$ of $K_v$. We have
\begin{align*}  &   \tr \left(  \mM(\mu_2,\mu_1, -s) \mM'(\mu_1,\mu_2, s) \mJ(\mu_1,\mu_2, s, \phi)\right) \\ 
 &\qquad  =     \left(  \frac{\partial \log}{\partial s}\lambda(\mu_1, \mu_2, \rho, s) \right) \times \prod\limits_{v \in \mathcal{S}}  \tr  \mJ_v( (\mu_1)_v, (\mu_2)_v, s, \phi_v) \d s.\end{align*}
 \end{enumerate}
\end{lemma}
\begin{proof}[Proof of the lemma] \mbox{}
\begin{enumerate}
 \item All except a finite number of the factors are one (see Theorem~\ref{defn:testprinc}).
 \item The operator $\mM'(\mu_1, \mu_2, s)$ is not a $\GL_2(\bA)$-intertwiner, but a $\bm K$-intertwiner. Since $\mM(\mu_1, \mu_2,s)$ and $\mJ(\mu_1, \mu_2, s)$ are irreducible for $\Re s=0$ as an exterior tensor product of irreducible representations $\bigotimes_v \mJ_v(\mu_{1,v}, \mu_{2,v},s)$, the operator acts by scalar $\lambda(\mu_1, \mu_2, \rho ,s)$ on every $\rho$-isotype. Since all $\rho$-isotypes have single multiplicity, the scalar depends only on the isomorphism class of $\rho$. The functional equation in Theorem~\ref{thm:scattfact}
\[ \mM(\mu_1, \mu_2, s) \mM(\mu_2, \mu_1, -s) = 1 \]
yields that
\[            \lambda(\mu_1, \mu_2, \rho ,s) = \lambda(\mu_2, \mu_1, \rho ,-s)^{-1} \neq 0,\]
and finally
\[    \mM(\mu_2,\mu_1, -s) \mM'(\mu_1,\mu_2, s) |_\rho  = \frac{ \lambda(\mu_1, \mu_2, \rho ,s)'}{ \lambda(\mu_1, \mu_2, \rho ,s) }.\]
Additionally, the factorization follows from Theorem~\ref{thm:scattfact}.
 \item If $\phi = \bigotimes_v \phi_v$ for $\phi_v \in \mH(\GL_2(\F), \rho_v)$, then $\mJ(\mu_1, \mu_2, s, \phi)$ projects onto the $\bm K$-isotype of $\rho$ according to the last point of Theorem~\ref{thm:smoothrep}.\qedhere
 \end{enumerate}
\end{proof}

\begin{proof}[Proof of the theorem]
Before beginning, we observe that only $J(\mu,1,s) \cong \mJ(1, \mu,-s)$ has only certain isotype. For the remaining isotypes, we will argue for $\Re s >1/2$, and the results extend by uniqueness of analytic continuation.
\begin{enumerate}
 \item Because the local parabolic inductions do not contain a $K_v$-isotype $\rho_v$ such that the representation $\pi_v \cong \Ind_{C_v}^{\GL_2(\F)} \rho_v$ is supercuspidal, the operator $\mJ_v(\mu_v,1,s , \phi_v)$ is zero for $\phi_v\in \mH(\GL_2(\F), \rho)$. 
\item For a square-integrable representation $\pi_v$, which is not supercuspidal, we have constructed the pseudo coefficient from two irreducible representations $\rho_v^{(j)}$ for $j=0,1$ and functions
\[ \phi_v = \phi_v^{(0)}  - \phi_v^{(1)} , \qquad \phi_v^{(j)} \in \mH(\GL_2(\F), \rho_j).\]
Because $\phi_v$ is a pseudo coefficient of a square-integrable representation
\[ \tr \mJ(\mu_v, 1,s, \phi_v^{(0)} ) =\tr \mJ_v(\mu_v,1,s, \phi_v^{(1)})  .\]
The constructed test function is consequently of the form
\begin{align*}  \phi & \coloneqq \left( \bigotimes\limits_{v \notin \mS^{sq}} \phi_v  \right) \otimes \left( \bigotimes\limits_{v \notin \mS^{sq}} \phi_v^{(0)}  -  \phi_v^{(1)} \right)  \\ 
             &= \left( \bigotimes\limits_{v \notin \mS^{sq}} \phi_v  \right) \otimes \sum\limits_{v \notin \mS^{sq}:  j_v \in \{0,1\} } \bigotimes\limits_{v \notin \mS^{sq}} (-1)^{j_v} \phi_v^{(j_v)} \\
              &\underset{\textup{short-hand}} = \sum\limits_{\bm j \coloneqq (j_v)_v \in \{0,1\}^N }  (-1)^{| \bm j|} \phi_{\bm j}.\end{align*}
							Here, $N$ is the cardinality of $\mS^{sq}$. We obtain according to Lemma~\ref{scattfactor}
\begin{align*}      \tr \mM(\mu_2,\mu_1, -s) & \mM'(\mu_1,\mu_2, s) \mJ(\mu, 1, s, \phi) \\
& =          \sum\limits_{\bm j =(j_v)_v \in \{0,1\}^N}                  (-1)^{|\bm j|}       \tr \mM(1,\mu, -s) \mM'(\mu, 1, s) \mJ(\mu, 1, s, \phi_{\bm j})       \\
 & = \mJ(\mu, 1, s, \phi_{\bm 0})         \sum\limits_{\bm j =(j_v)_v \in \{0,1\}^N}      (-1)^{|\bm j|}   \frac{\partial \log}{\partial s}   \lambda(\mu_, 1, \rho_{\bm j}, s)  \\
& = \mJ(\mu, 1, s, \phi_{\bm 0})         \sum\limits_{\bm j =(j_v)_v \in \{0,1\}^N}      (-1)^{|\bm j|} 
  \sum\limits_{v \in \mS} \frac{\partial \log}{\partial s}   \lambda(\mu_v, 1, \rho_{v}^{(j_v)}, s)  \\
                                                                                   & = \mJ(\mu, 1, s, \phi_{\bm 0})         \sum\limits_{\bm j =(j_v)_{ v} \in \{0,1\}^N}      (-1)^{|\bm j|}   \sum\limits_{v \in \mS^{sq}} \frac{\partial \log}{\partial s}   \lambda(\mu_v, 1, \rho_{v}^{(j_v)}, s) . 
\end{align*}
If $N >1$, the sum cancels.
\item If $N=1$, and a Steinberg representation occurs as factor, then $\rho^1_v $ is $\Ind_{\Gamma_0(\p_v)}^{\GL_2(\o)} 1 \ominus$ the trivial representation and $\rho^2_v$ is the trivial representation
 \begin{align*}
& \left(   \frac{\partial \log}{\partial s}   \lambda(\mu_v, 1, \rho_n, s)     -      \frac{\partial \log}{\partial s}   \lambda(\mu_v, 1, \rho_{n-2}, s)  \right)   \\
&   = \frac{\partial \log}{\partial s} \left(   \lambda(\mu_v, 1, \rho_n, s)   /   \lambda(\mu_v, 1, \rho_{n-2}, s)  \right) \\
& \underset{\textup{Prop.~\ref{prop:padicinter}/~\ref{prop:intersteinberg}}}=  \frac{\partial \log}{\partial s}  \frac{\zeta_v(2s+1)}{\zeta_v(2s-1) } \\
& = - \frac{2 \log(q) q^{-2s+1}}{1-q^{-2s+1}} + \frac{2 \log(q) q^{-2s-1}}{1-q^{-2s-1}}\\
& =           2 \log(q) \frac{q^{-2s-1}( 1-q^{-2s+1}) -  q^{-2s+1}  (1-q^{-2s-1})}{( 1-q^{-2s+1}) (1-q^{-2s-1}) }\\
&  = 2 \log(q)  \frac{q^{-2s-1} -  q^{-2s+1} }{ 1-q^{-2s+1}  -q^{-2s-1} +q^{-4s} } \\
& =             2 \log(q)  \frac{q^{-1} -  q^{+1} }{ (q^{2s} + q^{-2s}) -( q^{+1} + q^{-1} )} 
 \end{align*}
The poles are at $s = \pm 1/2$. By residue calculus, we obtain in the number field case
\begin{align*}
&     \frac{1}{4 \uppi}\int\limits_{\Re s = 0}  \tr \left(  \mM(\mu_2,\mu_1, -s) \mM'(\mu_1,\mu_2, s) \mJ(\mu_1,\mu_2, s, \phi)\right)  \d s \\
 & =  \frac{1}{4 \uppi}\int\limits_{\Re s = 0}  \left( 2 \log(q_u)  \frac{q_u^{-1} -  q_u^{+1} }{ (q_u^{2s} + q^{-2s}) -( q_u^{+1} + q_u^{-1} )}  \right)  \prod\limits_{v \in \mS^{\bm H}} q_v^{s}+q_v^{-s} \prod\limits_{v \in \mS_\infty} h_v(\im s) \d s  \\
 & \underset{\substack{ \Re s \rightarrow +\infty\\ \textup{residue at }s=1/2 }}= 2 \log(q_u)  \prod\limits_{v \in \mS^{\bm H}} \left( q_v^{-\frac{1}{2}}+q_v^{\frac{1}{2}} \right)  \prod\limits_{v \in \mS_\infty^{ps}} h_v(\im /2).
\end{align*}
The proof in the function field case is essentially the same. We only have to move the finite contour. Here the integral along the line $\Im s  = 2 \uppi / \log(q_F)$ cancels with that along $\Im s = 0$, and the same argument works. 
\item If $N=1$, and a discrete series representation of weight $n \geq 2$ occurs as factor, then $\rho^{(j)}_v \coloneqq \rho_{n -2j} \coloneqq \Ind_{\SO(2)}^{\O(2)} \epsilon_{n-2j}$. In this case, we have
\begin{align*}
& \left(   \frac{\partial \log}{\partial s}   \lambda(\mu_v, 1, \rho_n, s)     -      \frac{\partial \log}{\partial s}   \lambda(\mu_v, 1, \rho_{n-2}, s)  \right)   \\
&   = \frac{\partial \log}{\partial s} \left(   \lambda(\mu_v, 1, \rho_n, s)   /   \lambda(\mu_v, 1, \rho_{n-2}, s)  \right) \\
& \underset{\textup{Prop.~\ref{prop:realinter}}}=  \frac{\partial \log}{\partial s}  (-1) \cdot \frac{\Upgamma(s+1/2 +n/2-1)\Upgamma(s+1/2 -n/2+1)}{\Upgamma(s+1/2 +n/2)\Upgamma(s+1/2 -n/2) } \\
& \underset{\Upgamma(x+1) =x \Upgamma(x)}=    \frac{\partial \log}{\partial s} \frac{n+1-2s}{n+1+2s}   \\
& =            \frac{4(n-1)}{4s^2 -(n-1)^2 }.
 \end{align*}
The poles are at $s = \pm (n-1)/2$. Residue calculus yields
\begin{align*}
&     \frac{1}{4 \uppi}\int\limits_{\Re s = 0}  \tr \left(  \mM(\mu_2,\mu_1, -s) \mM'(\mu_1,\mu_2, s) \mJ(\mu_1,\mu_2, s, \phi)\right)  \d s \\
 & =  \frac{1}{4 \uppi}\int\limits_{\Re s = 0}  \left(  \frac{1}{s -\frac{(n-1)}{2} }  - \frac{1}{s +\frac{(n-1)}{2} }  \right)  \prod\limits_{v \in \mS^{\bm H}} \left( q_v^{s}+q_v^{-s} \right) \prod\limits_{v \in \mS_\infty} h_v(\im s) \d s  \\
 & \underset{\substack{ \Re s \rightarrow \infty\\ \textup{residue at }s=(n-1)/2}}=   \prod\limits_{v \in \mS^{\bm H}} \left( q_v^{-\frac{n-1}{2}}+q_v^{\frac{n-1}{2}} \right) \prod\limits_{v \in \mS_\infty^{ps}} h_v\left(\im \frac{n-1}{2}\right).
\end{align*}
\item The last statement follows immediately from Lemma~\ref{scattfactor} and the computations in Propositions~\ref{prop:realinter},~\ref{prop:complexinter}, and ~\ref{prop:padicinter}.\qedhere
\end{enumerate}
\end{proof}

\begin{remark} \mbox{}
\begin{itemize}
\item If $\fX$ contains a square-integrable representation, then the trace formula is of the same form as the trace formula of a compact locally symmetric space, i.e., the Casimir operators have only discrete spectrum. This is related to the Jacquet-Langlands functoriality \cite{JacquetLanglands}. The vanishing of the Eisenstein distributions has also been verified in the classical setting by Risager \cite{Risager}, building mainly on work of Huxley \cite{Huxley:Scatt} and Str\"ombergsson \cite{Stroembergsson:JL}. 
\item If $\fX$ contains at least two square-integrable representations, we have for $\phi_\fX = \bigotimes_v \phi_v$ that
 \[                                                                            \int\limits_{\N(\F)} \phi_v(mn) \d n =0 \]
for at least two places. This yields a simple trace formula \cite{Gelbart}*{Section V.2}.
 \end{itemize}
\end{remark}

\section{The residual distributions}\label{section:globalresidual}
References are  \cite{JacquetLanglands}*{page 271(vi)}, \cite{GelbartJacquet}*{page 240(6.37)}, \cite{Gelbart}*{page 48(1.4)}.

\begin{defn}[The residual distribution]
If $\uF$ is an algebraic number field, the residual distribution is defined for each algebraic Hecke character $\mu$ with $\mu^2 = \chi$ as
\[        J^{res}_{\mu} (\phi) =- \frac{1}{4} \tr \mM(\mu,\mu,0) \mJ(\mu,\mu, 0 )(\phi).\]
If $\uF$ is a global function field whose field of constants has cardinality $q_F$, the residual distribution is defined for each algebraic Hecke character $\mu$ with $\mu^2 = \chi$ as
\begin{align*}J^{res}_{\mu} (\phi) =  &- \frac{\log(q_{\uF})}{4} \tr \mM(\mu,\mu,0) \mJ(\mu,\mu, 0)(\phi) \\ 
                                                        & - \tr \frac{\log(q_F)}{4} \mM\left(\mu,\mu, \frac{\im \uppi}{\log q_{\uF}}) \right) \mJ\left(\mu,\mu, \frac{\im \uppi}{\log(q_{\uF})}\right) (\phi),
\end{align*}
where the sum runs through the Hecke characters of $\uF$.
\end{defn}

\begin{thm}[Explicit residual distributions]\label{thm:globalres}
 The residual distribution vanishes if at least the factor at one place is either
\begin{itemize}
 \item a discrete series representation of odd weight,
 \item an unramified principal series representation,
 \item a supercuspidal representation.
\end{itemize}
In all other cases, we define for $\Re s >1/2$, the function
\[ \Omega_{\uF}^\fX (s) = \prod\limits_{v}  \Omega_v^\fX(s), \]
where each factor is defined
\[     \Omega_v^\fX(s)      =\begin{cases} \frac{\zeta_v(2s)}{\zeta_v(2s+1)} , & v \in \mS_f^{ups},
\\ 
 \frac{\zeta_v(2s)}{\zeta_v(2s-1)}   -    \frac{ \zeta_v(2s)}{\zeta_v(2s+1)}, & v \in \mS^{\St}, \\
\frac{\zeta_\bR(2s)}{\zeta_\bR(2s+1)}  ,  & v \in \mS_\bR^{ups}\\ 
  \frac{ \sqrt{\uppi} \Upgamma(s) \Upgamma(s+1/2)}{ \Upgamma(s+1/2+n/2) \Upgamma(  s +1/2-n/2)} -   \frac{ \sqrt{\uppi} \Upgamma(s) \Upgamma(s+1/2)}{ \Upgamma(s-1/2+n/2) \Upgamma(  s -1/2-n/2)},    & v \in \mS_\bR^{sq}, \\
               \qquad \pi_v \textup{ discrete series of weight }n \geq 2 (n \textup{ even}), \\
 2 \uppi \frac{\zeta_\bC(2s)}{\zeta_{\bC}(2s+1)} = \frac{1}{2s}, & v \in \mS_\bC^{ups}.   
\end{cases} \] 
This function admits a meromorphic continuation. 
We obtain for a number field
\[  \sum\limits_{\mu^2=\chi} J^{res}_{\mu} (\phi^\fX)   =- \frac{1}{4}   2^{\# \mS^{\bm H}}  \Omega_{\uF}^\fX(0) \prod\limits_{v \in \mS^{ps}_\infty} h_v(0),\]
and for a function field
\[   \sum\limits_{\mu^2=\chi}  J^{res}_{\mu} (\phi^\fX)  =- \frac{\log(q_{\uF})}{4} \left( \Omega_{\uF}^\fX(0) +  \Omega_{\uF}^\fX\left(\frac{\im \uppi}{\log q_{\uF}}\right) \right) .\]    
\end{thm}
\begin{proof}
All claims follow more or less directly from the $K$-type classification in the last three chapters, Corollary~\ref{cor:charvanish}, the construction of the test functions and the computations of the local intertwiner in Propositions~\ref{prop:realinter}, ~\ref{prop:complexinter} and~\ref{prop:padicinter}.

Every test function factors as tensor product $\phi = \bigotimes_v \phi_v$, and either $\phi_v$ is in $\mH(\GL_2(\F), \rho_v)$, or 
\[ \phi_v = \phi_{v,1} -\phi_{v,2} \]
for $\phi_{v,j} \in  \mH(\GL_2(\F), \rho_{j,v})$. Now the corresponding local operators $ \mJ_v(\mu_v,\mu_v, 0 )(\phi_v)$ project  onto isotypes, where we have computed the action of the local intertwiners in Propositions~\ref{prop:realinter}, ~\ref{prop:complexinter} and~\ref{prop:padicinter}. Unfortunately at $s=0$, we have no factorization. The operator
\[ \mM(\mu, \mu,s) \]
is an intertwiner, meromorphic in $s$. It acts by a scalar, after restriction to $\prod\limits_v K_v$ (Schur's lemma).  Its value is computed for $\Re s> 1/2$ as an Euler product. The conclusion follows by uniqueness of analytic continuation.
\end{proof}

\chapter{The relation to classical trace formulas}

Although I advocate the use of the similarity classes, some readers may prefer to see the formalism in classical language.  

\section{Maass wave forms of weight zero and full level}
Let us consider the situation
\[ \mL^2( \SL_2(\bZ) \backslash \bH).\]
By strong approximation \cite{Bump:Auto}*{Proposition 3.3.1, page 294}, we obtain
\[ \mL^2( \SL_2(\bZ) \backslash \bH) \cong \mL^2( \GL_2(\bQ) \Z(\bA) \backslash  \GL_2(\bA) / \O(2) \times \prod\limits_p \GL_2(\bZ_p)).\]
The only similarity class which admits an $\O(2) \times \prod\limits_p \GL_2(\bZ_p)$-invariant vector is the one where all factors are unramified principal series representations. 

Let $g$ be a compactly supported, smooth, even function on $\bR$, and let
\[ h(r) =\int\limits_\bR g(x) \e^{\im r x} \d x \]
be its Fourier transform. 

Let us focus on the case $\mS^{\bm H} = \emptyset$. The cuspidal distribution (Theorem~\ref{thm:globaltest})
\begin{align}
&\sum\limits_{\substack{ \pi \in \fX \\ \pi = \bigotimes_v \mJ(1,1,s_v)}} h( s_\infty)\end{align}
equals the sum of the identity distribution (Theorem~\ref{thm:globalidentity})
\begin{align}
 \frac{\vol( \Z(\bA) \GL_2(\bQ) \backslash \GL_2(\bA))}{4 \uppi} \int\limits_{\bR} h(r) r \tanh( \uppi r) \d r,
\end{align}
of the parabolic distribution (Theorem~\ref{thm:globalpara})
\begin{align}
& \lambda_0 g(0)   +\\
 &   \lambda_{-1} \left(  \frac{\log(\uppi) - \gamma_0}{2} g(0) + \frac{h(0)}{4}  - \frac{1}{2 \uppi}  \int\limits_\bR  h(r)  \frac{\Upgamma'}{\Upgamma} ( 1+ \im r) \d r \right),
\end{align}
of no hyperbolic terms (because $\bZ^\times =\{ \pm 1 \}$), of the elliptic equivalence classes split at the real place (Lemma~\ref{lemma:globalelliptic} and Proposition~\ref{prop:realhyper}) )
\begin{align}
  \sum\limits_{\substack{ \{ \gamma \} \in \GL_2(\bQ) \textup{ elliptic }\\ \textup{ split at } \infty \textup{ with }\det \gamma > 0\\  \textup{ and conjugated to } z \sma \alpha & 0 \\ 0 & 1\smz \textup{ in } \GL_2(\bR) }}  &  \vol(\GL_2(\bQ) \Z(\bA) \backslash \GL_2(\bA)_\gamma) \cdot \frac{g\left( \log(\alpha)/2 \right)}{\cosh( \log(\alpha) /2)}  ,
\end{align}    
the elliptic equivalence classes elliptic at the real place (Lemma~\ref{lemma:globalelliptic}  and Proposition~\ref{prop:realelliptic})
\begin{align}
      \sum\limits_{\substack{ \{ \gamma \} \in \GL_2(\bQ)  \textup{ elliptic } \\ \textup{ elliptic at } \infty \textup{ and conjugated to }\\ z \sma \cos \theta & \sin \theta \\ -\sin \theta & \cos \theta \smz  \textup{ in } \GL_2(\bR) } }  &\vol(\GL_2(\bQ) \Z(\bA) \backslash \GL_2(\bA)_\gamma)  \times  \frac{1}{2 |\sin \theta|}  \nonumber   \\
&   \times \int\limits_{-\infty}^\infty h(r) \frac{\cosh( 2r(\uppi- \theta)) + \e^{\im n \uppi} \cosh(2r \theta)}{ \cosh(2 \uppi r) + 1} \dr r,
\end{align}
the Eisenstein distribution  ($\upzeta=$ Riemann zeta function, Theorem~\ref{thm:globalEis}) 
\begin{align} \frac{1}{4 \uppi}\int\limits_{\Re s = 0}  h( \im s)  \frac{\partial \log}{\partial s}  \sqrt{\uppi} \frac{\Upgamma(s) \upzeta(2s) }{\Upgamma(s+1/2) \upzeta(2s+1)} \d s,
\end{align}
the one-dimensional distribution (Theorem~\ref{thm:globaloned})\footnote{The reader might be surprised to see this distribution. It corresponds to the Laplace eigenvalue zero for the constant function, which is usually ``counted'' on the spectral side.}
\begin{align}
 - h( \im /2),
\end{align}
and the residual distribution (Theorem~\ref{thm:globalres})  
\begin{align}
-\frac{ h(0)}{4} \lim\limits_{s \rightarrow 0}  \sqrt{\uppi} \frac{\Upgamma(s) \upzeta(2s) }{\Upgamma(s+1/2) \upzeta(2s+1)} .
\end{align}

Some fudge factors are explicitly expressed as special values of the Riemann zeta function
\begin{align*}
 \lambda_{-1} &= 1, \\
\lambda_0 &= \frac{1}{2}\left( \gamma_0 - \log(4 \uppi) \right),\\
\lim\limits_{s \rightarrow 0}  \sqrt{\uppi} \frac{\Upgamma(s) \upzeta(2s) }{\Upgamma(s+1/2) \upzeta(2s+1)} & = -1.
\end{align*}
The above formula is precisely what Hejhal \cite{Hejhal2}*{page 209}  and Iwaniec \cite{Iwaniec:Spectral}*{Theorem 10.2, page 167} obtained via classical methods. The classical formulas hold for a more general set of test functions. The extension works here as well.

\section{The general situation}
If one wants to work in a situation close to the classical setting, but in the generality of global functions, I suggest  working with isotypes of 
\[ \mL^2_0(\GL_2(\uF) \Z(\bA) \backslash \GL_2(\bA)). \]
Given a finite set $Q$ of non-archimedean places, and an open (possibly non-proper) subgroup $U_v \subset \GL_2( \o_v)$ for each $v \in Q$, and an irreducible, finite-dimensional, unitary representation $\sigma_v$ of $U_v$, we denote
\[ U_Q = \prod\limits_{v \in Q} U_v \times \prod\limits_{v \notin Q} K_v , \qquad \sigma_Q = \bigotimes\limits_{v \in Q} \sigma_v \otimes \bigotimes\limits_{v \notin Q} 1 .\]
What follows is a recipe to obtain a trace formula for the $\sigma_Q$-isotype of  \( \mL^2_0(\GL_2(\uF) \Z(\bA) \backslash \GL_2(\bA))\).

We can define an integer for each similarity class of $\GL(2)$-automorphic representations over $\uF$
\[ m_\fX( \sigma_Q) = \dim_\bC \Hom_{U_Q} [ \Res_{U_Q}\pi , \sigma_Q ], \qquad (\pi \in \fX).\] 
The definition does not directly depend upon $\pi$, but rather upon $\fX$. If $\sigma_Q$ is the trivial representation of $U_Q$, this corresponds to the set of invariant vectors.  For at most finitely many similarity classes, the integer $m_\fX(\sigma_Q)$ can be nonzero.

For all $v \in \mS_\infty$, let $g_v$ be a compactly supported, smooth, even function on $\bR$, and let
\[ h_v(r) =\int\limits_\bR g_v(x) \e^{\im r x} \d x \]
be its Fourier transform. 

The cuspidal distribution in our setting then gives
 \[  \sum\limits_{\fX} m_\fX(\sigma_Q) \sum\limits_{\pi \in \fX}  \prod\limits_{v \in \mS^{\bm H}} \left( q_v^{-s_v(\pi)} + q_v^{+s_v(\pi)}  \right) \prod\limits_{v \in \mS_\infty} h_v(\im s_v(\pi)).\]
Similarly, the remaining distributions have to be summed up with multiplicity $m_\fX( \sigma_Q)$. My formulas hold only for minimal similarity classes, so a twist by one-dimensional characters will generally be necessary.
\begin{example}
For the choice
\[ Q= \{ v \}, \qquad U_v =\Gamma_0(\p_v), \qquad \sigma_v = 1,\]
  then $m_\fX( \sigma_Q ) =0$ for all but the similarity classes with only unramified principal series representations at all places $u \neq v$ and at $v$ there can be either
\begin{itemize}
 \item an unramified principal series representation,
 \item a Steinberg representation, or 
 \item their twists by $\chi \circ \det$ for a one-dimensional representation $\chi: \F^\times \rightarrow \bC^1$, with $\chi^2 =1$.
\end{itemize}
For these representations $m_\fX( \sigma_Q ) =1$. Note that the number $\# \{ \chi:  \chi^2 =1 \}$ is the index of $[ \F^\times : (\F^\times)^2]$, and may vary with $\F$, yet another reason to prefer similarity classes.
\end{example}

All the above corresponds to the case where all archimedean representations are unramified principal series and the central character is trivial. 
One can alternatively fix a discrete series representation at some of the real places, change the central character, allow weight one at a real place, and so on.

However, it is not true that the representation theory of $\prod\limits_{v \in \mS} K_v$ and its subgroups can separate all similarity classes, i.e., it is false that there exists, for every similarity class $\fX_0$, a pair $(U_Q, \sigma_Q)$, such that 
\[ m_{\fX}(\sigma_Q)   =\begin{cases} 1, & \fX = \fX_0, \\ 0, & \fX \neq \fX_0. \end{cases}\]
This is why we need to argue in terms of the normalizer of the Iwahori subgroup in some cases. This is only necessary for similarity classes with unramified supercuspidal representations.

\chapter{Counting automorphic representations}

\section{Dimension formulas}
\begin{thm}
If $\mS^{ps}_\infty =\emptyset$, then 
\[ \# \{ \pi \in \fX\} < \infty \]
is a finite number. The number can become arbitrarily large if $\fX$ varies.
\end{thm}
\begin{proof}
This is an immediate corollary of the prior computations with the choices $\mS^{ps}_\infty$ and $\mS^{\bm H} = \emptyset$. The cuspidal distributions then gives
\[ \sum\limits_{\pi \in \fX} 1  = \# \{ \pi \in \fX\} .\]
Given a global field $\uF$, we choose a finite place and a sequence of similarity classes $(\fX_n)_{n \in \bN}$ such that at least one factor of elements $\pi = \bigotimes_v \pi_v \in \fX_n$ is supercuspidal, all remaining non-archimedean ones are unramified principal series representations, and $\pi_v \cong \mJ(\mu_v, 1, s_v(\pi))$ for $\mu_v$ with conductor $\p_v^n$.
Recall the definition of $\pX$. We have by Theorem~\ref{thm:globalidentity} that
\[ J_1(\pX)  \equiv  \frac{q_v^N - q_v^{N-1}}{q^{\lfloor N/2 \rfloor} - 1}.\]
The elliptic terms associated to $\gamma \in \PGL_2(\uF)$ give zero for $n>1$, if $\gamma$ is elliptic at $v$ (Proposition~\ref{prop:padicelliptic}), or
are bounded absolutely if $\gamma$ is hyperbolic at $v$ (Proposition~\ref{prop:padichyper}), or if $\gamma$ is parabolic at $v$ (Proposition~\ref{prop:padicpara}).
All the remaining geometric distributions vanish.
\end{proof}

 \begin{remark}
The explicit trace formulas with the choices $\mS^{sq}_\infty =\emptyset$ and $\mS^{\bm H} =\emptyset$ give an explicit correspondence between elliptic elements and automorphic representations.
 \end{remark}

\section{Weyl laws}
In this section the asymptotic formulas for the eigenvalues of Casimir operators are stated.

\begin{thm}\label{thm:weyl}         \mbox{}
\begin{itemize}
 \item If $\uF = \bQ$ and all factors in $\fX$ are prinicpal series representations, the following asymptotic holds  for $T\geq 1$:
         \begin{align*} \mbox{} \\ &  \# \{ \pi \in \fX: s_\bR(\pi) \leq T \}   \\
\mbox{}\\
                                     & \qquad =C_\fX T^2 -  \frac{2}{\uppi} T \log T  + \mO( C_\fX T ).
         \end{align*}
 \item If $\uF$ is an  algebraic number field, and every element in $\fX$ has a square integrable local factor,  the following asymptotic holds  for $T_v \geq 1$:
        \begin{align*} \mbox{}  \\ \# &\{ \pi \in \fX: s_v(\pi) \leq T_v \textup{ for all } v\in S_\infty^{ps}(\fX) \} \\
\mbox{} \\
 &=C_\fX  \prod\limits_{v \in \mS_\bR^{ps}}T_v^2  \prod\limits_{v \in \mS_\bC^{ps}}T_v^3  +  \mO_{\uF}\left( C_\fX \sum_{u \in S_\infty^{ps}} \frac{1}{T_u} \prod\limits_{v \in \mS_\bR^{ps}}T_v^2  \prod\limits_{v \in \mS_\bC^{ps}}T_v^3 \right).\end{align*}
 \item If $\uF$ is an algebraic number field, the following asymptotic holds for $T_v \geq 1$:
 \begin{align*} \mbox{}  \\ \# &\{ \pi \in \fX: s_v(\pi) \leq T_v \textup{ for all } v\in S_\infty^{ps}(\fX) \} \\ 
   \mbox{}\\
                        &=C_\fX  \prod\limits_{v \in \mS_\bR^{ps}}T_v^2  \prod\limits_{v \in \mS_\bC^{ps}}T_v^3  \\   
\mbox{}\\
                         &\qquad   +  \mO_{\uF}\left( C_\fX \sum_{u \in S_\infty^{ps}} \frac{1}{T_u} \prod\limits_{v \in \mS_\bR^{ps}}T_v^2  \prod\limits_{v \in \mS_\bC^{ps}}T_v^3 + \sum_{w \in S_\bR^{ps}}   \log T_w \prod\limits_{v \in \mS_\infty} T_v\right).
 \end{align*}
\end{itemize}
The constant 
\[ C_\fX \coloneqq \frac{\vol( \Z(\bA) \GL_2(\uF)\backslash \GL_2(\bA))}{(4 \uppi)^{\# S_\infty}} \prod\limits_{v \in \mS_f \cup \mS^{sq}_\infty(\fX)}  C_v\]
is defined as a finite product (most factors are one)
\[ C_v \coloneqq \begin{cases} 1, & \pi_v \textup{ unramified principal series}, \\ 
                                             \frac{q_v^N - q_v^{N-1}}{q^{\lfloor N/2 \rfloor} - 1}   , & \pi_v \textup{ ramified p.s.}  \pi_v = \mJ(\mu_v,1,s_v), \cond(\mu_v)  =\p_v^N, \\
                                               q-1, & \pi_v \textup{ Steinberg}, \\
                                               \dim(\rho_{\pi_v}), &       \pi_v \textup{ unramified supercuspidal}, \\
                                                \frac{q_v+1}{2} \dim(\rho_{\pi_v}), &      \pi_v \textup{ ramified supercuspidal},\\
                                                      \frac{n-1}{2}, & \pi_v \textup{ discrete series of weight n}. 
                 \end{cases}\]
\end{thm}
The proof of this theorem is the content of the last two sections in this chapter. We start with a discussion of this theorem. To translate the statement into classical language, we require a definition. 
\begin{defn}
Let $\Gamma$ be a congruence subgroup of $\GL_2(\bZ)$, and $U \subset \prod\limits \GL_2(\bZ_p)$ the unique open subgroup with $\GL_2(\bQ) \cap U = \Gamma$ (as constructed in the introduction for $\SL(2)$), then we define for a similarity class of $\GL(2)$-automorphic representations over $\bQ$, the integer
 \[ m_\fX(\Gamma) \coloneqq \dim_\bC( \pi_f^U) \]
as the dimension of the $U$-invariant vectors of $\pi_f $, where $\pi =\pi_\infty \otimes \pi_f \in \fX$.
\end{defn}
The independence from the element $\pi \in \fX$ chosen follows from the classification of $K$-types. For all but a finite number of similarity classes, the number is zero.

We show now how to translate this back in a congruence setting:
\begin{corollary}\label{cor:weyler}
 Let $\Gamma$ be a congruence subgroup of $\GL_2(\bZ)$,  which contains the element $\sma -1 & 0\\ 0 & 1 \smz$, and $X(\Gamma)$ a orthonormal basis of Hecke Maass cusp forms of level $\Gamma$ and weight zero or one. Then we obtain the asymptotic formula
\begin{align*} &  \# \{ \lambda_\infty \leq T^2+1/4 \}  = \\
 &  \frac{[\GL_2(\bZ) : \Gamma]}{12} T^2 -  \sum\limits_{\substack{ \fX \textup{ with}\\ \textup{all factors princ.s.}}} m_\fX ( \Gamma) \frac{2}{\uppi} T \log T + \mO(  [\GL_2(\bZ) : \Gamma]   T ).
\end{align*}
\end{corollary}
\begin{proof}
 Use strong approximation
\[ \Gamma \backslash \bH  \cong \GL_2(\bQ) \Z(\bA) \backslash \GL_2(\bA) / U \times \O(2).\]
We add all Weyl laws with multiplicity $m_\fX ( \Gamma)$ accordingly. The main term is well-known, so by matching them we obtain
\[ \sum\limits_{\fX} m_\fX ( \Gamma) C_\fX    = \frac{[\GL_2(\bZ) : \Gamma]}{12} .\qedhere\]
\end{proof}

For the standard congruence subgroups some information is provided in chapter nine to compute $m_\fX(\Gamma)$ for all $\fX$ whose elements do not have square integrable factors.

No results uniform in the level aspect seem to be known, although they follow fairly straightforwardly from Section 2 of Lapid and M\"uller's article \cite{LapidMueller:SLN} in conjunction with Huxley's computation of the scattering matrix \cite{Huxley:Scatt}. Risager \cite{Risager}, relying on work of Str\"ombergsson \cite{Stroembergsson:JL} and Huxley \cite{Huxley:Scatt}, have shown already that the intermediate terms vanish if the Maass forms of weight zero to be counted come from a Jacquet-Langlands lift, i.e., the Weyl law is that of a compact surface in this case. The main term was proven in the above generality by Reznikov \cite{Reznikov:Scatt} only in the case where all archimedean factors are principal series representation. Even in the special case of imaginary quadratic number fields, nothing seems to be known beyond the main term \cite{Elstrodt}. The above theorem gives the first instance of a Weyl law where one allows principal series and discrete series to occur at the archimedean primes. I believe that some mild saving can be obtained in the $t_v$-aspect. See for example Randol \cite{Randol} or Venkov \cite{Venkov}, which appeal to the Selberg zeta function for a saving in the $t_v$-aspect. 

The proof method for the Weyl law is due to H\"ormander, and we will follow closely \cite{LapidMueller:SLN}. With only this method, no better bounds are achievable in the $t_v$- and $\fX$-aspects.
\begin{corollary}
\( \# \left\{ \pi \in \fX : s_v(\pi) = 0 \forall v \in \mS_\infty^{ps} \right\} \ll_{\uF} C_\fX. \)
\end{corollary}
Duke \cite{Duke:Dimension} and Michel/Venkatesh \cite{MichelVenkatesh:Weight1} have gone beyond the trivial bound in this special case. Michel and Venkatesh rely on the Kuznetsov-Bruggeman formula as provided by Bruggeman and Miatello \cite{BruggemanMiatello}. The $\GL(2)$ trace formula should in principle be able to yield the same, local bound, but the hyperbolic terms must be analyzed more carefully than in Theorem~\ref{thm:hyperellvanish}.

\section{A local bound}
 Since our proof is an exact rephrasing of Section 2 in \cite{LapidMueller:SLN}, I advise the reader to read this proof first. The strategy is to prove a local bound first, and derive the asymptotic law therefrom.

\begin{defn}
For every tuple $\bm{x}=(x_v)_{v \in S^{ps}_\infty}$ of positive number $x_v  \geq 0$, define the integers
\[ R_{\fX}( \bm x) \coloneqq \# \left\{ \pi \in \fX: | \Im s_v(\pi)  -x_v  | \leq  1/2 \textup{ for all } v \in S_\infty^{ps}  \right\} .\]
 \end{defn}

\begin{proposition}[Local bound]\label{prop:localbound}
We have a local bound for $x_v \geq 1/2$
\[ R_{\fX} ( \bm x) \ll_F C_{\fX}  \prod\limits_{v\in S_\bR} x_v \prod\limits_{v\in S_\bC} x_v^2 .\]
\end{proposition}
The proof of this proposition will take up the rest of this section. Compare with \cite{Duistermaat:Weyl}*{Lemma 2.3} and \cite{Mueller:Weyllaw}*{Lemma 2.2}. 
\begin{lemma}[Existence of test functions]
For every positive constant $C_v >0$, there exists a smooth, even function $g_v : \bR \rightarrow [0, \infty)$, such that
\begin{itemize}
 \item the function $g_v$ is supported on $[-C_v, C_v]$ and
 \item its Fourier transform 
\[ h_v(\xi) \coloneqq \int\limits_{\bR} g(x) \e^{\im x \xi} \d x\]
is a real-valued, non-negative function on $\bR$, with positive values on  $[-1/2,1/2]$ and $\{ \im y : y \leq 1/2 \}$, and with $h_v(0) =1$. 
\end{itemize} 
\end{lemma}  
Choose $C_v$ as suggested in Theorem~\ref{thm:hyperellvanish}, but also such that $C_v \leq \frac{\log q_v}{2 S_\infty}$ for all residue characteristics $q_v$ of non-archimedean completions of $\uF$.

\begin{proof}[Proof of the lemma]
We omit the subscript $v$ in the proof. Let $g_0$ be a smooth, even, non-negative function on $\bR$ supported and positive on $[ -C/2, C/2]$, which is strictly decaying for $x \rightarrow \pm \infty$. Then $g_1 \coloneqq g_0 \ast \overline{g_0}$ is supported on $[-C,C]$ and non-negative. The Fourier transform $h_1$ of $g_1$ is the square of the Fourier transforms of $g_0$, hence non-negative on $\bR$ and even.
We have that 
\[ h_1(0) = \int\limits_{\bR} |g_0(x)^2| \d x >0.\]
 Since $h_1$ is an entire function, we have that $h_1 \geq 1/2$ on $[-\delta/2 , \delta/2] \cup [-\im\delta/2 ,\im  \delta/2 ]$ for $\delta \leq 1$. Define
\[ h_{2} (x) \coloneqq h_1( x / \delta).\]
The Fourier transform of $h_2$ is supported on $[ -C \delta, C \delta] \subset [ -C, C]$. The function $h(x) \coloneqq h_2(x) / h_2(0)$ is a suitable candidate.
\end{proof}
For every $v \in S_\infty^{ps}$ choose $g_v$ as proposed by the above lemma.
$$R_{\fX}( \bm x)   \leq \frac{1}{\min \{ \prod\limits_v h(x_v): |x_v| \leq 1 \} }  \sum\limits_{\pi \in \fX} \prod\limits_{v \in S_\infty}  \frac{h(\im s_v(\pi) + t_v) + h(\im s_v(\pi) -  t_v)}{2} $$
Define
\[ h^{\bm x} ( ( z_v )_v)  \coloneqq \bigotimes\limits_{v \in S_\infty^{ps}} \frac{h_v(x_v +z_v) + h_v(x_v - z_v)}{2}. \] 
Define $\phi_\fX^{\bm x}$ such that $\phi_v$ is the function with $(h^{\bm x})_v = h_{\phi_v}$ for all $v \in S_\infty^{ps}$, and the rest chosen according to Section~\ref{section:globaltest}.

We will use the Arthur trace formula 
\begin{align*}                   \sum\limits_{\pi \in \fX} \prod\limits_{v \in S_\infty} &  \frac{h(\im s_v(\pi) + x_v) + h(\im s_v(\pi) - x_v)}{2} \\  &= \sum\limits_{ \{ \gamma \} \subset \GL_2(\uF)} J_{\{ \gamma \}} (\phi_\fX^{\bm x}) + \sum\limits_{\substack{ \mu : \M(\uF) \backslash \M(\bA) \rightarrow \bC \textup{alg.} \\ \mu_1\mu_2 =\chi}}  J_{\mu}^{\textup{Eis}}(\phi_\fX^{\bm x}) \nonumber\\ 
                 &         +  \sum\limits_{\substack{\omega :\uF^\times \backslash \bA^\times \rightarrow \bC \textup{alg.} \\ \omega^2 =\chi}}  J_{\omega}^{\textup{res}}(\phi_\fX^{\bm x})    +         \sum\limits_{\substack{ \omega: \uF^\times \backslash \bA^\times \rightarrow \bC \textup{alg.} \\ \omega^2 = \chi }}  J_{\omega}^{\textup{one}}(\phi_\fX^{\bm x}) 
\end{align*}
 and estimate the values on the right hand side separately. We will use the explicit formulas from the first chapter. 

As an immediate consequence of Theorem~\ref{thm:hyperellvanish}, we obtain:
\begin{corollary}[Bound for elliptic and hyperbolic estimates]
 \[  \sum\limits_{\gamma \in \GL_2(\uF) \atop \textup{ elliptic} \bmod\Z(\F)} J_{\gamma} (\phi^{\bm x}) + \sum\limits_{\alpha \in \uF \neq 0,1} J_{\alpha} (\phi^{\bm x})   \ll C_\fX.\] 
\end{corollary}

\begin{lemma}[Estimate for the identity distribution]
We have an estimate for the global parabolic distribution
\begin{align*} J_1 (\phi^{\bm x}) \ll_F C_\fX \prod\limits_{ v \in S_\infty^{ps} } (|x_v|_v+1),  \end{align*}
 which depends upon $F$.
\end{lemma}
\begin{proof}
Appeal to Theorem~\ref{thm:globalpara}.
Let $v \in S_\bC$, then we estimate
\begin{align*}  \int\limits_{\bR} h_{v}^{x_v} (r) r^2 \d r =    \int\limits_{\bR}  \frac{ h_{v} (r-x_v) + h_v(r+x_v) }{2} r^2 \d r \\ 
                                                                          = \int\limits_{\bR}          h_v(r) r^2  \d r + x_v^2 \int\limits_{\bR} h_v(r) \d r \\
                                                                          = -\frac{g_v''(0)}{2 \uppi} + x_v^2 g_v(0).
\end{align*}
Let $v \in S_\bR^{ups}$, then we find the upper bound
\begin{align*}  \int\limits_{\bR} h_{v}^{x_v} (r) r \tanh( \uppi r) \d r& \ll  \int\limits_{\bR} h_{v}^{x_v} (r) |r| \d r \\
                                                                                           & =      \int\limits_{\bR} h_{v} (r) \frac{|r-x_v| + |r+x_v|}{2} \d r \\
                                                                                           & \ll      \int\limits_{\bR} h_{v} (r) |r| \d r + x_v \int\limits_{\bR} h_{v} (r)  \d r . 
\end{align*}
The estimate at $v \in S_\bR^{rps}$ is similar. 
\end{proof}
For the next estimate, we need to recall Stirling's formula
\begin{thm}[Stirling's formula \cite{IwaniecKowalski}*{Section 5.A.4, page 151}]\label{thm:stirling}
 For $\Re z > 0$, we have the following equality for the Euler gamma function: 
\begin{align}
\ln \Upgamma (z) = \left(z-\frac12\right)\ln z -z + \frac{\ln {2 \pi}}{2} + 2 \int_0^\infty \frac{\arctan (t/z)}{\exp(2 \pi t)-1}\,{\rm d}t.
\end{align}
 \end{thm}
Differentiating gives the next corollary.
 \begin{corollary}\label{cor:stirling}
 For $\Re z \geq 1$, we have the following estimate for the digamma function:
\begin{align}
\frac{ \Upgamma'}{\Upgamma} (z) & = \log (z) + \mO(1).
\end{align}  
 \end{corollary}

 \begin{proposition}
\[ J_{par}( \phi^{\bm x}) \ll C_\fX \sum\limits_v \log (2+x_v)  . \]
\end{proposition}
\begin{proof}
 The fudge factors of the parabolic distribution are always bounded by $C_\fX$. This can be checked by a case-by-case comparison between Theorem~\ref{thm:globalidentity} and Theorem~\ref{thm:globalpara}.
   Moreover, by Corollary~\ref{cor:stirling} and according to Remark~\ref{rem:globalpara}, the bound
\begin{align*}\int\limits_{\bR}                    h( t-x) \frac{ \Upgamma'}{\Upgamma} (1+\im t) \d t  & \ll        \int\limits_{\bR}                    h(t) |\log( 1 + \im t + \im x )|\d t   + \mO(1) \\ 
                                                                                                                                      & \ll       \int\limits_{\bR}                    h(t) \log \Re ( t+x ) \d t   + \mO(1)    \\
                                                                                                                                       & \ll       \int\limits_{\bR}                    h(t) \log \Re ( t) \d t \log(x)   + \mO(1) \end{align*}
 yields the estimate.
\end{proof}

\begin{lemma}[Bound for the one-dimensional and the residual distribution]
  \[  \sum_{\omega^2=1} J_{\omega \circ \det} (\phi^{\bm x}) +  J_{\omega}^{res} (\phi^{\bm x})   \ll C_\fX.\] 
\end{lemma}
\begin{proof}
This is straightforward in that it only does not vanish when only unramified principal series, Steinberg, or weight-2-discrete series representations occur.
\end{proof}

\begin{lemma}[Bound for the continuous spectrum]\label{lemma:Eislocal}
\[ J_{\textup{Eis}} (\phi^{\bm x}) \ll      C_\fX \sum_{v \in S^{ps}}  \log(x_v+2)   .      \]
\end{lemma}
\begin{proof}
If $\fX$ contains at least one local square-integrable representation as factor, then we have immediately $J_{\textup{Eis}} (\phi^{\bm x}) \ll    C_\fX$. 
If $\fX$ contains only principal series representations, then the function $\Lambda_{\uF}^{\fX}$ differs from the function
\[ \Lambda_{\uF}(s)  \coloneqq \frac{\upzeta_{\uF}(2s)}{\upzeta_{\uF}(2s+1)}, \]
only by a finite number of factors, where $\upzeta_{\uF}$ is the completed Dedekind zeta function of $\uF$. We will compute the contribution with $\Lambda_{\uF}$ first, and then show that the finite number of factors contribute zero because of the support considerations.

 For $-1/2  \leq  \Re s  \leq 2$, we have the standard estimate \cite{IwaniecKowalski}*{Proposition 5.7, page 102}  on $\Res = 0$
\[  \frac{\zeta_{\uF}'}{\zeta_{\uF}} (2s) -   \frac{\zeta_{\uF}'}{\zeta_{\uF}} (2s+1)   \ll \log(|s| +10) .\]
 Select any place $u \in S_\infty$ and approximate trivially:
\begin{align*} \int\limits_{\Re s = 0}   & \left(  \frac{\zeta_{\uF}'}{\zeta_{\uF}} (2s) -   \frac{\zeta_{\uF}'}{\zeta_{\uF}} (2s+1)  \right) \prod\limits_v    \frac{h_v(\im s - x_v) + h_v(\im s - x_v)}{2}  \d s  \\ 
                                                      & \ll                     \int\limits_{\Re s = 0}    \log(|s|+10)   \frac{h_u(\im s - x_u) + h_u(\im s - x_u)}{2}  \d s  \\ &\qquad \ll  \log(x_u). 
\end{align*}
Note that the case with real ramified principal series essentially yields the same bound.
By substracting a finite number of factors, we obtain the bound.
\begin{align*} \int\limits_{\Re s = 0}   & \left(  \frac{\log(q) q^{-2s}}{1-q^{-2s}} -  \frac{\log(q) q^{-1-2s}}{1-q^{-1-2s}}  \right) \prod\limits_v    \frac{h_v(\im s - x_v) + h_v(\im s - x_v)}{2}  \d s  \\ 
                                                      & =    \log(q)   \sum\limits_{k>0}              \int\limits_{\Re s = 0}    \left(  q^{-2ks} - q^{-(2k+1)s} \right) \prod\limits_v    \frac{h_v(\im s - x_v) + h_v(\im s - x_v)}{2}  \d s \\ 
                                                      &  =  \sum\limits_{k>0}   C_{q,k}  \mF\left( t \mapsto  \prod\limits_v    \frac{h_v(t - x_v) + h_v(t - x_v)}{2} \right)   ( k \log(q)) =0,
\end{align*}
since the Fourier transforms of a product of the functions $h_v$ can be expressed in terms of convolution product of the functions $g_v$, and then vanishes by support considerations $\left( C < \frac{\log q_v}{ 2 \# \mS_\infty}\right)$.
\end{proof}    
All estimates for the local bound (Proposition~\ref{prop:localbound}) have now been established.

\section{The proof of the asymptotic formula}

We use short-hand $K_{\bm t} = \{ (y_v)_{v \in S_\infty^{ps}} : | y_v | \leq t_v \}$. We exploit that
\[ \int\limits_{\bR} h_v(x+ t) \d x = g_v(0) = 1.\]
Consequently, we have
\[    R_{\fX}( \bm x)   = \underbrace{\int\limits_{\bR} \dots  \int\limits_{\bR}}_{|S_\infty^{ps}| - \textup{times}} \sum\limits_{\pi : |s_v(\pi)| \leq x_v} \prod\limits_{v} h_v( \im s_v(\pi) + t_v)  \d \bm t. \]
Set $d = \# S_\infty^{ps}$ and $a_v =2$ [$a_v = 3 $] if $v$ is real [complex]. As a consequence of the local bound, we arrive at
\begin{lemma}
\begin{align*}
& \int\limits_{\bR^d}        \sum\limits_{\pi : |s_v(\pi)| \leq x_v} \prod\limits_{v} h_v( \im s_v(\pi) + t_v)   \d \bm t  \\ & -     \int\limits_{K_{\bm x}}        \sum\limits_{\pi } \prod\limits_{v} h_v( \im s_v(\pi) + t_v)   \d \bm t  \\ & \qquad \leq C_\fX \sum\limits_{u}  \frac{1}{t_u}\prod\limits_v t_v^{\alpha_v}.
\end{align*}
\end{lemma}
\begin{proof}
 Compare with \cite{Mueller:Weyllaw}*{Lemma 2.3}. 
We write
\begin{align*}
  &   \int\limits_{\bR^d}        \sum\limits_{\pi : |s_v(\pi)| \leq x_v} \prod\limits_{v} h_v( \im s_v + t_v)   \d \bm t -     \int\limits_{K_{\bm x}}        \sum\limits_{\pi} \prod\limits_{v} h_v( \im s_v(\pi + t_v)   \d \bm t = \\ 
  &    \int\limits_{\bR^d-K_{\bm x}}        \sum\limits_{\pi : |s_v(\pi)| \leq x_v} \prod\limits_{v} h_v( \im s_v + t_v)   \d \bm t -      \int\limits_{K_{\bm x}}        \sum\limits_{\pi : |s_v(\pi)| \geq x_v} \prod\limits_{v} h_v( \im s_v(\pi) + t_v)   \d \bm t,
\end{align*}
 and prove the estimate for each summand separately. We partition
 \[ \bR^d-K_{\bm x} = \coprod\limits_{P \subset S_\infty^{ps} \atop P \neq \emptyset } \{ (y_v)_v : |y_v| \geq x_v  \textup{ for all } v\in P; |y_v| \leq x_v \textup{ for all } v\notin P \}.\]
Every function $h_v$ is a Schwartz function, and therefore $$h_v(t) \ll \max\{ (1+|t|)^{-100},1\}.$$ 
Consider for example the set
\[ [-x_r,x_r] \times \cdots  \times [-x_w, x_w] \times \left(  \bR - [-x_u, x_u ]  \right) \times \cdots \times \left( \bR - [-x_p, x_p ] \right) \]
and estimate
\begin{align*}
       &    \int\limits_{-x_r}^{x_r}  \dots  \int\limits_{-x_{w}}^{x_{w}}  \int\limits_{\bR - [-x_u, x_u ] }  \dots \int\limits_{\bR - [-x_p, x_p ]} \sum\limits_{\pi: |s_v(\pi)| \leq x_v} h_r( \im s_r + t_r) \cdots  h_p( \im s_p + t_p) \d \bm t\\
       & \ll \sum\limits_{\pi: |s_v(\pi)| \leq x_v} |x_r|  \cdots |x_w|  (1 +s_u + x_u)^{-50} \cdots (1 +s_p + x_p)^{-50 } \\
   & \underset{\textup{local bound}} \ll   |x_r|  \cdots |x_w|  \sum\limits_{(n_v) \in \bZ^d : n_v \leq x_v}  C_\fX \prod\limits_{v} |x_v|_v (1 +s_u +n_u)^{-50} \cdots (1 +s_p + n_p)^{-50 }.  
\end{align*}
For the remaining sets, we obtain  similar bounds along the same lines. We only partition the set $\{ \pi : |s_v(\pi)| \geq x_v \}$ and interchange the role played by the integral and the series.
\end{proof}
This is within the error term. We must find an asymptotic formula for  
\[                   \int\limits_{K_{\bm x}}        \sum\limits_{\pi } \prod\limits_{v} h_v( \im s_v(\pi) + t_v)   \d \bm t.\]
Again, we appeal to the coarse Arthur trace formula, and the above expression is equal to the sum of the integrals
\begin{align*}
 \int\limits_{K_{\bm x}}  J_{*} (\phi_\fX^{\bm t}) \d \bm t .
\end{align*}

\begin{lemma}[\cite{Mueller:Weyllaw}*{Equation (2.10)}]\label{lemma:aux}
 Let $p:\bR \rightarrow \bC$ be an even function with 
\[ p(x) \ll 1+ |x|,\]
then we have an asymptotic formula
\[ \int\limits_{-x}^{x} \int\limits_{\bR}  h(t- r) p(r) \d r \d t= \int\limits_{-x}^{x} p(t) \d t + \mO(|x|+1)\]
for every Schwartz function $h: \bR \rightarrow [0, \infty)$ even with $\mL^1$-norm
\[ \int\limits_\bR h(x) \d x  = 1. \]
\end{lemma}

 We rely heavily on the formulas provided in the first chapter. The elliptic, hyperbolic, residual and one-dimensional distribution remain within the error term $\mO(C_\fX \prod\limits_{v} x_v)$, in fact, the bounds from the previous section are sufficient in these cases. 

 The main contribution comes from the identity (Theorem~\ref{thm:globalidentity})
\begin{align*}
J_{ 1 } (\phi_\fX^{\bm x})  = &   \frac{\vol( \GL_2(\uF) \Z(\bA) \backslash \GL_2(\bA))  }{(4 \uppi)^{|\mathcal{S}_\infty|}} \times   V_\fX      \\ 
                       & \qquad \times  \prod\limits_{\substack{v \in  \mathcal{S}_\bR^{sq} \\ \pi_v \cong D_n(\mu_1, \mu_2)}} \frac{n-1}{2} \\ 
                                         & \qquad \times    \prod\limits_{v \in \mathcal{S}_\bR^{ups}}  \int\limits_{\bR}    \frac{h_{v} (r-x) + h_{v} (r+x) }{2} r \tanh(\uppi r) \d r \\
                                         & \qquad \times       \prod\limits_{v \in \mathcal{S}_\bR^{rps}}  \int\limits_{\bR}    \frac{h_{v} (r-x) + h_{v} (r+x) }{2} r \coth( \uppi r)  \d r \\
                                     & \qquad \times    \prod\limits_{v \in \mathcal{S}_\bC^{ups}}  \int\limits_{\bR}   \frac{h_{v} (r-x) + h_{v} (r+x) }{2} r^2  \d r. 
\end{align*}
The estimate obtained from Lemma~\ref{lemma:aux} is
\begin{align*}
\int\limits_{K_{\bm x}}  J_{ 1 } (\phi_\fX^{\bm t}) \d \bm t = &   C_\fX  \prod\limits_{v \in \mathcal{S}_\bR} x_v^2   \prod\limits_{v \in \mathcal{S}_\bC^{ups}} x_v^3 \\
                       & +          \mO\left( C_\fX  \left( \sum\limits_{u \in \mS_\infty} \frac{1}{x_u} \prod\limits_{v \in \mS_\bR} x_v^2  \prod\limits_{v \in \mS_\bC} x_v^3 \right)    + C_\fX\right).
\end{align*}
 Note that Lemma~\ref{lemma:aux} applies for all but the complex places, where the integral kernel is very easy to handle.

The Eisenstein distribution and the parabolic distribution contribute also to the error term $\mO(C_\fX \prod\limits_{v} x_v)$ if every element from $\fX$ has one square integral representations (see Theorems~\ref{thm:globalpara} and
~\ref{thm:globalEis}). This proves the second Weyl law.

If all factors of elements in $\fX$ are principal series representations, we refer for the parabolic contributions to Theorem~\ref{thm:globalpara},  the consecutive remark, and Lemma~\ref{lemma:aux}. The next estimate is a consequence
\begin{align}
  \int\limits_{K_{\bm x}}  J_{par} (\phi_\fX^{\bm t})   \d \bm t & \ll  \sum\limits_{u \in \mS_\infty} \int\limits_{-x_u}^{x_u} \left| \frac{\Upgamma'}{\Upgamma}(1+ \im t) \right| \d t   + C_\fX \prod\limits_{v} x_v      \\
& \ll    \sum\limits_{u \in \mS_\infty} x_u \log(x_u)          + C_\fX \prod\limits_{v} x_v.  
\end{align}
We have used without mentioning explicitly that the fudge factors in Theorem~\ref{thm:globalpara} are far smaller than the factor $C_\fX$. If $\uF = \bQ$, we can obtain an asymptotic formula via Lemma~\ref{lemma:aux}:
\begin{align*}
  \int\limits_{K_{\bm x}}  J_{par} (\phi_\fX^{\bm t})   \d \bm t &  = -\frac{\lambda_{-1}}{\uppi} \int\limits_{x}^{-x}  \frac{\Upgamma'}{\Upgamma} ( 1 +\im r) \d r + \mO(C_\fX x).\\
  \int\limits_{x}^{-x}  \frac{\Upgamma'}{\Upgamma} ( 1 +\im r) \d r  & = 2 \arg \Upgamma(1+\im x) = x \log x  - x
\end{align*} 
In the particular case $\uF =\bQ$,  the residue of the Riemann zeta function is $\lambda_{-1} = 1$. 
So the parabolic contributions are in this case
\[           \frac{-1}{\uppi} ( x \log(x) -x) .\]

If all factors of elements in $\fX$ are principal series representations, we refer for the Eisenstein distribution to Theorem~\ref{thm:globalEis}. If there is more than one archimedean place, we obtain
\[  J_{\textup{Eis}}(\phi^{\bm t}) = \frac{1}{4 \uppi}\int\limits_{\Re s = 0}  \frac{\partial_s \Lambda_{\uF}^{\fX}}{\Lambda_{\uF}^\fX}(s)  \prod\limits_{v \in \mS_\infty} \frac{ h_v(\im s + t_v)  +  h_v(\im s + t_v) }{2} \d s.  \]
Fix $u \in \mS_\infty$. For a general number field, we estimate trivially
\begin{align*}
  \int\limits_{K_{\bm x}}   J_{\textup{Eis}}(\phi^{\bm t})   \d \bm t& \ll      \prod\limits_{v \neq u}  x_v \max|h_v| \frac{1}{4 \uppi}\int\limits_{\Re s = 0}   \frac{ h_u(\im s + t_u)  +  h_v(\im s + t_u) }{2} \frac{\partial \log}{\partial s} \frac{\Lambda_{\uF}^{\fX}(2s)}{\Lambda_{\uF}^\fX(2s+1)} \d s \d s   \d t_u 
\\ & \ll      \prod\limits_{v \neq u}  x_v \max|h_v| \frac{1}{4 \uppi}\int\limits_{\Re s = 0} \left( \frac{\partial_s \Lambda_{\uF}^{\fX}}{\Lambda_{\uF}^\fX}(s)  \right)  \frac{ h_u(\im s + t_u)  +  h_v(\im s + t_u) }{2} \d s   \d t_u.
\end{align*}
With Lemma~\ref{lemma:aux}, we obtain with the local estimate~\ref{lemma:Eislocal} 
\[    \int\limits_{K_{\bm x}}   J_{\textup{Eis}}(\phi^{\bm t})   \d \bm t \ll   \log(x_u)   \prod\limits_{v}   x_v.\]
This proves the last Weyl law. Note that $x_v \log(x_v)$-error terms can be absorbed in the $\mO(x_v^2)$-error term for complex places.

For $\bQ$, we can be more precise. Because of the support conditions, we can remove without harm local factors (see the end of the proof of Lemma~\ref{lemma:Eislocal}). Also whether we use the formula for weight zero or weight one does not matter. Let $\Lambda_\bQ$ be the completed Riemann zeta function 
\begin{align*}
  \int\limits_{K_{\bm x}}   J_{\textup{Eis}}(\phi^{\bm t})   \d \bm t& =   \frac{1}{4 \uppi} \int\limits_{-x}^x \int\limits_{\Re s = 0}  \frac{ h_u(\im s + t_u)  +  h_v(\im s + t_u) }{2} \frac{\partial \log}{\partial s} \frac{\Lambda_{\bQ}(2s)}{\Lambda_{\bQ}(2s+1)} \d s    \d t_u \\
 & =      \frac{1}{4 \uppi \im}\int\limits_{-x}^x \frac{\partial \log}{\partial s} \frac{\Lambda_{\bQ}(2\im t)}{\Lambda_{\bQ}(2 \im t+1)} \d t + \mO(x)\\
 & =  \frac{1}{ 8 \uppi \im}\int\limits_{-2x}^{2x} \frac{\partial \log}{\partial s} \frac{\Lambda_{\bQ}(\im t)}{\Lambda_{\bQ}(2 \im t+1)} \d t + \mO(x).
\end{align*} 
This expression counts the number of zeros up to height $x$ with a negative sign for the orientation of the contour:
\[ -   \frac{x}{ \uppi} ( \log(x) - \log( \uppi)) + \mO( \log x ).\]
See \cite{IwaniecKowalski}*{Theorem 5.8, page 104} for the argument. This provides us also with an argument for the first sharper asymptotic law, because the parabolic terms and the Eisenstein terms give together
\[ \frac{-2}{\uppi} x\log x        + \mO(x).   \]

\chapter*{Part II --- Abstract harmonic analysis on groups}
Part two of this thesis contains abstract results about locally compact groups, and can be read independently of the rest of this thesis. Variants of some of the statements and more can be found in \cite{Meyer:Smooth}, but we require results which are equivariant with respect to a compact subgroup.
I have elected to include a general discussion of  harmonic analysis on a locally compact group in order to justify and conceptualize later considerations and choices.  
The reader who is only interested in the trace formula, its related computations and constructions for $\GL(2)$, is recommended to skip Part two upon the first reading. 

The representation theory of a reductive group over a local field depends largely upon the representation of the maximal compact subgroups, or more accurately, upon the representation theory of the maximal subgroups which are compact modulo the center. This concept goes by various names. In the case of Lie groups, the concept is related to the theory of weights. In the locally profinite case, this concept is the theory of types.
In the theory of automorphic forms, modular forms, and Maass wave forms, this concept is directly related to the notion of level and weight.

Let $F$ be a local field. The smooth, admissible representations of $\GL_n(F)$ can be distinguished into two categories:
 \begin{enumerate}[font=\normalfont]
\item The supercuspidal representations occur only if $F$ is not isomorphic to $\bR$ or $\bC$. These can be exhausted as compact induction from maximal, open subgroups which are compact modulo the center.
\item The parabolic inductions, their subrepresentation, and their subquotients can be realized as subspaces/subquotients of the normalized induction from a parabolic subgroup. The Abel transform plays a central role in this analysis.
\end{enumerate}

Although I focus only on $\GL(2)$ in this thesis, I have chosen to work here in the context of locally compact groups and treat them locally as projective limits of Lie groups. 
In doing so, we avoid giving separate proofs for Lie and totally disconnected groups. I found it more efficient
 to reduce theorems about locally compact groups to their counterparts in Lie theory. In the special case of totally disconnected groups, most proof are moderate exercises only, since we can work locally with projective limits of finite groups instead of projective limits of general Lie groups.

In the fourth chapter, we explain some of the structure results for locally compact groups, in particular the solution to Hilbert's fifth problem. 
We state in the main result how to approximate a locally compact group locally by projective systems of Lie groups.

In the fifth chapter, we will translate classical results for both the Hecke algebra of a Lie group and a totally disconnected group to the setting of a locally compact group. We will work equivariantly with respect to a compact subgroup. We conclude by briefly introducing the Abel transform, which can be defined in the presence of an Iwasawa-type decomposition of the group. The Abel transform will play a central role in the representation theory of reductive algebraic groups over local fields. 

In the sixth chapter, we relate the Abel transform to the analysis of the parabolic inductions and their subquotient. We briefly review the main results of the representation theory of a locally compact group. We relate the representation theory to the invariant harmonic analysis of a locally compact group, and discuss traces of representations, Gelfand pairs, Plancherel theorems, orbital integrals, and parabolic inductions. 

We emphasize that the  main features of this part are:
\begin{itemize}
 \item Let $G$ be a locally compact group, and let $K$ be a compact subgroup. We introduce $\Ccinf(G)$ and decompose it as $K$-bi-module.
 \item We give a trace computation for the compact induction from subgroups, which are compact modulo the center.
\item We introduce a fairly general notion of a parabolic induction including all the standard cases, and compute its trace as a functional via the Abel transform.
\end{itemize}

These statements become trivial in the case of a discrete group. The toolbox introduced here has more value if the irreducible, unitary or the smooth, admissible representations of the group in question are classified, 
and if the decomposition of their restrictions to a compact subgroup are known a priori.

\chapter{Structure theory of locally compact groups}

\section{Topological and locally compact groups}
\begin{defn}
A group is a topological group if it is a topological space and the multiplication and the inversion maps are continuous. 
A topological group is a locally compact group if it is a Hausdorff space with a relative open, compact set.
\end{defn}
The Hausdorffness condition can be removed from the definition of a locally compact group. Analysis on Hausdorff spaces factors then through a quotient group, which is Hausdorff.
\begin{proposition}[\cite{DeEc}*{Proposition 1.1.6}]
Let $G$ be a topological group. A continuous map from $G$ to a $T_1$-space factors through the quotient group $G/ \textup{cl}\{1\}$, where $\textup{cl}\{1\}$ is the closure of the set, which contains only the identity element and is a closed, normal subgroup.
The quotient group $G/ \textup{cl}\{1\}$ is a Hausdorff topological group.
\end{proposition} 
The assumption that $G$ is locally compact is crucial. This property is essentially equivalent to the existence of quasi-invariant Radon measures on a group.
A Radon measure is a functional on the space $C_c(G)$ of continuous, compactly supported functions $G \rightarrow \bC$. 
In a case where the group does not admit a compact, relatively open set, the space $C_c(G)$ contains no non-zero element.   

\section{The identity component}
\begin{defnthm}[\cite{DeEc}*{Proposition 4.1.2}]
The path-connected component of the unit element in a topological group is a closed, normal subgroup. If $G$ is a locally compact group, we denote this component by $G_0$.\index{$G_0$}  
\end{defnthm}

Since $G_0$ is normal, we can consider the group extension
\[ 1 \rightarrow  G_0  \rightarrow G \rightarrow G/G_0 \rightarrow 1.\]
This extension provides us with a good starting point for understanding the local structure of a locally compact group. It also allows us to introduce the following standard definitions:
\begin{defn}
Let $G$ be a locally compact group.
 \begin{itemize}
\item  The group $G$ is (path-)connected if $G=G_0$.
\item  The group $G$ is almost connected if the quotient group $G/G_0$ is compact.
 \item The group $G$ is locally pro-finite if the group $G_0 = \{1 \}$.
 \end{itemize}
\end{defn}
It follows that for each topological group $G$, the quotient group $G/G_0$ is a totally disconnected group. 
We want to have a sufficiently flexible definition of a Lie group, to the extent that discrete groups are Lie group. We also want to remove the condition of being para-compact. 
We consider a singleton, i.e., a point with the discrete topology, as the zero-dimensional Euclidean space.
\begin{defn}
A smooth manifold is a topological space which is locally homeomorphic to the Euclidean space, possibly of dimension zero, and whose transition maps are smooth. 
A Lie group is a smooth manifold whose group operations are smooth.
\end{defn}
We will see shortly that almost connected groups are projective limits of such Lie groups.

\section{Van Dantzig's Theorem and its consequences}
\begin{theorem}[Van Dantzig's Theorem]\label{thm:dantzig} \mbox{}
 \begin{itemize}
  \item  Van Dantzig's Theorem: Every locally profinite group contains a neighborhood base at the identity of open, compact subgroups. 
 \item  In a locally pro-finite group, every compact subgroup is contained in an open, compact subgroup.
 \end{itemize}
\end{theorem}
\begin{proof}
Van Dantzig's Theorem is  well-known. See e.g. \cite{MontgomeryZippin}*{Theorem 2.3, page 54} or \cite{DeEc}*{Theorem 4.1.6, page 95} for a proof. The second point is a conclusion of van Dantzig's Theorem; let $G$ be a totally disconnected group, and let $K$ be a compact subgroup. Pick an open compact subgroup $O$ of $G$. Consider
\[ O' =  \bigcap\limits_{k \in K} k^{-1} O k.\]
There exists a finite set $F \subset K$ such that
\[ O' = \bigcap\limits_{k \in F} k^{-1} O k,\]
hence $O'$ is a compact, open subgroup as a closed subset of $K \cdot O \cdot K$ and commutes with any element of $K$. The group generated by $O'$ and $K$ equals $O' \cdot K$ and is an open compact group, which contains $K$.
\end{proof}

\begin{corollary}[Every compact subgroup is contained in an almost connected, open subgroup]\mbox{}
 Let $G$ be a locally compact group with a compact subgroup $K$, then it contains a closed subgroup $H$
 \begin{enumerate}[font=\normalfont]
 \item which is open,
 \item almost connected, 
 \item and which contains $K$.
\end{enumerate}
\end{corollary}
\begin{proof}
Consider the surjection  $q: G \twoheadrightarrow G/G_0$. Consider an open, compact group $O \subset   G/G_0$ with $O \supset q(K)$, which exists by Theorem~\ref{thm:dantzig}. 
The pullback $q^{-1}(O)$ is an open, closed and almost connected subgroup, which contains $K$.
\end{proof}

 \section{Approximation by Lie groups}
The following results are often referred to as the solution of Hilbert's fifth problem. 
\begin{theorem}[Gleason-Yamabe Theorem] \mbox{}
Let $G$ be a locally compact, almost connected group, then it admits a net\footnote{It is sufficient to consider sequences if and only if $G$ is metrizable.} of normal compact subgroups $\mN = \{ N \}$, which is partially ordered by inclusion, and satisfies 
 \begin{enumerate}[font=\normalfont]
 \item  $G = \lim\limits_{\leftarrow} G/N$, or equivalently $\bigcap_{N \in \mN} N = \{ 1 \}$,
 \item $G/N$ is isomorphic to a Lie group.
\end{enumerate} 
\end{theorem}
Proofs can be found in \cite{Yamabe}*{Theorem 5, page 364} and \cite{MontgomeryZippin}*{Theorem 4.6, page 175}. We have introduced Lie groups in a way that they include discrete groups. We give two special cases as examples:
\begin{example}\mbox{}
\begin{itemize}
 \item Let $G$ be a compact group, thus every irreducible representation of $G$ is unitarizable and finite-dimensional by the Peter-Weyl Theorem. Let $\fS$ be the set of finite subsets of the irreducible representations, which is partially ordered by inclusion, and for $S \in \fS$ define $\rho_S = \oplus_{\rho \in S} \rho$. It is true that $G$ is the projective limit of the images of $\rho_S$:
\[ G \cong \lim\limits_{\leftarrow \atop S \subset \fS \textup{ finite}}  \textup{im} \rho_S.\]
The image of every $\rho_S$ is certainly closed as $G$ is compact and is a subset of the compact Lie group $\U(\dim(\rho_S))$. Every closed subgroup of a Lie group is a Lie group. So every compact group is homeomorphic to a projective limit of compact Lie groups. In particular, every compact, totally disconnected group is a projective limit of finite groups, i.e., a pro-finite group.
 \item Let $A$ be a locally compact abelian group. Every neighborhood of the identity contains a compact open normal subgroup $K$ such that $A/K \cong \mathbb{R}^n \times \mathbb{T}^m \times D$ for some discrete abelian group $D$, so it is a Lie group by definition. See \cite{HofmannMorris:Compact}*{Corollary 7.54}.
\end{itemize}
\end{example}

\begin{corollary}
Let $G$ be a locally compact group, let $K$ be a compact subgroup. Then there exists an open, closed subgroup $\underline{G} \subset G$ and a net of compact subgroups $\mN$, which is partially ordered by inclusion, and satisfies 
 \begin{enumerate}[font=\normalfont]
 \item  $\underline{G}$ contains $K$,
 \item $\underline{G} = \lim\limits_{\leftarrow} \underline{G}/N$, or equivalently $\bigcap_{N \in \mN} N = \{ 1 \}$,
 \item  $\underline{G}/N$ is isomorphic to a Lie group.
\end{enumerate}  
\end{corollary}

\begin{example}[Lie group and locally pro-finite groups] \mbox{}
\begin{itemize}
 \item If $G$ is a Lie group, then we can pick $\underline{G} = G$ and the family of normal, compact subgroups $\mN = \{ \{1 \}\}$. 
 \item If $G$ is a locally pro-finite group, then every almost connected subgroup~$\underline{G}$ is compact, i.e., pro-finite. Let $\mN$ be the family of normal, finite-index subgroups, which becomes a net under the partial ordering of inclusion.
\end{itemize}
\end{example}
\chapter{Hecke algebras}
\section{Smooth functions on locally compact groups}
I would like to introduce a fairly general definition of a smooth function on a locally compact group, which is due to Bruhat \cite{Bruhat:Distributions}. This is our motivation behind the introduction of projective systems of Lie groups.
\begin{defn}[Smooth functions on a locally compact group]    \mbox{}
Let $G$ be a locally compact group. A function $\phi : G \rightarrow \bC$ is smooth if for all $g \in G$ and every almost connected, closed, and open subgroup $G'$, there exists a normal, compact subgroup $N = N_{G',g}$ of $G'$ such that
\begin{enumerate}
 \item the group $G'/N$ is a Lie group, 
 \item the function $x \in G' \mapsto \phi(x g^{-1})$ is bi-$N$-invariant, 
 \item and yields a smooth function $G'/N \rightarrow \bC$ between smooth manifolds.
\end{enumerate}
\end{defn}

\begin{example}Let us compare the notion of smooth functions on a general locally compact group with that on Lie groups and locally pro-finite groups.
\begin{itemize}
 \item If $G$ is a Lie group, then the notion of a smooth function is equivalent to the usual one.
 \item If $G$ is a totally disconnected group, then a function is smooth if and only if it is locally constant.
 \end{itemize}
\end{example}

Let $G$ be a locally compact group always considered to be endowed with a \textbf{right} invariant Haar measure $\mu_G$. Fix an open, relative compact subset $O$ of $G$.
Define the modular character 
\[ \Delta_G: G \rightarrow (0, \infty), \qquad \Delta_G(g) = \frac{\mu_G(g^{-1}O)}{\mu(O)}.\]
The modular character is independent of the Haar measure and the open compact set $O$. We usually write $\d \mu_G(g) = \d g$. For any $\mu_G$-integrable function $f$ on $G$ and any element $x \in G$, we provide the following integral identity
\begin{align*} 
\int\limits_G f(g) \d g & = \int\limits_{G} f(gx) \d g,\\
\int\limits_G f(g) \d g & = \int\limits_{G} \Delta_G(x) f(xg) \d g,\\
\int\limits_G f(g) \d g &= \int\limits_G f(g^{-1}) \Delta_G(g^{-1}) \d g.
\end{align*}
The last identity asserts that $\Delta_G(g) \d g$ is a left invariant Haar measure.
\begin{defn}
The space $\Ccinf(G)$\index{$\Ccinf(G)$} of smooth, compactly supported functions on $G$ is given a $*$-algebra structure via the following operations
\[ \phi_1 \ast \phi_2 (x) = \int\limits_G \phi_1(xg) \phi_2(g^{-1}) \d g, \qquad \phi^*(g) = \Delta_G(g^{-1}) \overline{\phi(g^{-1})}.\]
\end{defn}
\begin{lemma}\label{lemma:transl}
 Let $G$ be a locally compact group. For any almost connected, closed, and open subgroup $G'$ of $G$ and every net $\mN$ of compact normal subgroups 
\[ \Ccinf(G) = \bigoplus\limits_{\gamma \in G/G'} \lambda_\gamma \left(  \lim\limits_{\rightarrow} \Ccinf( G'/N)  \right), \]
where $\lambda_{\gamma} \phi(x)=\phi(x\gamma^{-1})$.
\end{lemma}
\begin{proof}
Since $G'$ is open, the quotient space $G/G'$ is discrete, and we can write $f \in \Ccinf(G)$ in a unique fashion as a sum
\[ f(x) = \sum\limits_{\gamma} f_\gamma(x),\]
 where $f_\gamma$ is supported on $G' \gamma$. Since any $\phi$ is compactly supported, only a finite number is non-zero and the sum is a finite sum. Now the rest follows by the definition of smoothness.
\end{proof}
If $G$ is Lie group, then $\Ccinf(G)$ carries the locally uniform topology. If $G$ is an almost connected group, seen as the projective limit of a net Lie groups $(G_j)$, then $\Ccinf(G)$ carries the inductive limit topology of the $\Ccinf(G_j)$. A locally compact group is as topological space the disjoint union of translates of almost connected groups. It is sufficient to define the topology on $\Ccinf(G)$ as the inductive limit topology.

The definition of a restricted product has been already used in the previous chapters, when using ad\`elic groups. Let us confirm that the above definitions coincide in this case with the usual notions as well.
\begin{example}[Restricted products]\label{ex:restricted}
Let $(G_{i}, K_{i})_{i \in I}$ be a family (indexed by a set $I$) of locally compact groups $G_i$ and some distinct compact, open subgroups.  
We can define the restricted product
       \begin{align*} G &=\resprod\limits_{i \in I} \;  \left(G_i, K_i\right) \\
             &= \left\{ (g_i)_{i \in I} \in \prod\limits G_i : g_i \in K_i \textup{ for all but finitely many } i \in I \right\} \\
            & =  \lim\limits_{\substack{\rightarrow \\ S \subset I \; \textup{finite}}}  \prod\limits_{i \in S} G_i \times \prod\limits_{s \notin S} K_i,\end{align*}
         where the topology is given on $\prod\limits_{i} K_i \subset G$ by the usual product topology, and the topology on $G$ is the projective limit topology: a nonempty, open subset is given precisely by a finite subset $S \subset I$, and finitely many $O_i \subset G_i$ open, non-empty:
          \[ O = \prod\limits_{i \in S}  O_i \times \prod\limits_{i \notin S} K_i.\]
         The group $G$ is metrizable if and only if $I$ is countable and every group $G_i$ is metrizable.  Furthermore, the algebra
           \begin{align*} \Ccinf(G) &= \resotimes_{i \notin I} \left( \Ccinf(G_i), \mathds{1}_{K_i} \right),\end{align*}
          can be described as the span of tensors  $\bigotimes_i \phi_{i}$ of functions $\phi_i \in \Ccinf(G_i)$, such that almost all $\phi_i$ are the characteristic functions of the open group $K_i$.  This discussion addresses, for example, the situation, where the group in question is the group of the finite ad\`elic points of a group scheme defined over a global function field. For the number field case, one can add a finite number of copies of real reductive Lie groups.
\end{example}

\section{Dixmier-Malliavin factorization}
Every smooth function is the convolution product of smooth functions, briefly denoted by
\[ \Ccinf(G) \ast \Ccinf(G) = \Ccinf(G).\]
The algebra $\Ccinf(G)$ will act by endomorphisms on vector spaces, and the above decomposition allows us to work with positive elements only.

Again, the statement is a moderate exercise when $G$ is discrete or even if $G$ is a locally pro-finite group \cite{DeEc}*{Section 9.4}.
\begin{theorem}[Dixmier-Malliavin Theorem]\label{thm:dixma}
Every smooth function $\phi \in \Ccinf(G)$ on a locally compact group $G$ can be written as a finite sum of convolution products
\[ \phi = \sum\limits_{j=1}^N \phi_{1,j} \ast \phi_{2,j}\]
for a finite collection of elements $\phi_{i,j} \in \Ccinf(G), i=1,2, j=1, \dots n,$.
\end{theorem}
\begin{proof}
This follows immediately from the results of Lie theory. This is given as Theorem 3.1 in \cite{DixMa}.
\end{proof}

\section{Existence of a Dirac net}
If $G$ is not discrete, then the algebra $\Ccinf(G)$ has no unit element. 
\begin{defn}[Dirac net]\index{Dirac net}
Let $G$ be a locally compact group. A net $(f_{j})_{j \in J}$ of elements in $\Ccinf(G)$ is a Dirac net if 
             \begin{enumerate}[font=\normalfont]
              \item every element $f_j$ is non-negative, i.e., $f_j(g) \geq 0$,
              \item every element $f_j$ is normalized, i.e., $\int\limits_{G} f_j(g) \d g = 1$,
              \item and the net is concentrated near the identity, i.e., for any neighborhood $U$, there exists an index $\alpha_U \in J$ such that the support of $f_j$ for $j \geq \alpha_U$ is contained in $U$.
             \end{enumerate}
\end{defn}
\begin{lemma}[Dirac nets exist]
Let $G$ be a locally compact group, then $\Ccinf(G)$ admits a Dirac net. 
\end{lemma}
\begin{proof}
Dirac nets certainly exist in Euclidean space, e.g. consider the sequence $f_n = g_n /|| g_n ||_{\mL^1}$ for
\[ g_n : \bR^d \rightarrow [0, \infty), \qquad g_n(\vec{x}) = \begin{cases} 
                                                               \textup{exp}( - ( \left\| \vec{x} \right\|^2 -1/n)^{-2}), &\left\| \vec{x} \right\| < 1/n, \\ 
                                                                                     0 , &     \left\| \vec{x} \right\| \geq 1/n.
                                                              \end{cases}  \]
By definition of a smooth manifold via smooth atlases, this remains true for smooth manifolds as well. In the case of an almost connected group $G'$, consider a net of normal compact subgroups $(N_\alpha)_{\alpha \in A}$ such that $G' / N_\alpha$ is a Lie group and $G' = \lim\limits_\leftarrow G'/N_\alpha$. Consider a Dirac net $(f_{n, \alpha})_{n \in \bN}$, then $(f_{j})_{j \in J}$ will do,  where we define  \( J =\bN \times A, \) with the partial ordering $(n, \alpha) \leq (n', \alpha')$ if and only if $n \leq n'$ and $\alpha \leq \alpha'$. Let $G$ be a general locally compact group, then there exists an open, closed, almost connected subgroup $G'$. The Dirac net of $\Ccinf(G')$ is also a Dirac net of $\Ccinf(G)$.
\end{proof}
 \begin{defn}
 A net $(f_j)_j$ of elements in a topological $*$-algebra $A$ is an approximate identity if for all $h \in A$, we have that
 \[ f_j \ast h \rightarrow h, \qquad h \ast f_j \rightarrow h. \]
 \end{defn}
We can topologize $\Ccinf(G)$ via Lemma~\ref{lemma:transl} as an inductive limit of projective limits.

\begin{lemma}[Dirac nets are approximate identities]\label{lemma:diracapprox}
Let $G$ be a locally compact group, and let $(f_j)_{j \in J}$ be a Dirac net. The Dirac net is an approximate identity for $\Ccinf(G)$, i.e., for all $h \in \Ccinf(G)$, we have
\[ f_j \ast h \rightarrow h, \qquad h \ast f_j \rightarrow h. \]
\end{lemma}
\begin{proof}
 This holds for Lie groups \cite{Harish:DiscreteII}*{Lemma 3, page 7}. Hence, it holds for almost connected, locally compact groups. To obtain the general statement, we rely on the fact that every locally compact group admits an open, closed, almost connected subgroup. Since translates of functions supported by such a subgroup generate the space, the result follows for arbitrary functions. The right (left) convolution commutes with left (right) translation.
\end{proof}
The above lemma allows us to realize the right translation as convolutions with a Dirac net $(f_j)_{j \in J}$
\[ \phi( \blank x)  := \lim\limits_{j \in J} \phi \ast f_i(x \blank).\]
The left translation can be realized similarly:
\[ \phi(x\blank)  := \lim\limits_{j \in J} f_i(\blank x) \ast \phi .\]
\begin{lemma}[$K$-invariant Dirac nets]
Let $G$ be a locally compact group, and let $K$ be a compact subgroup, then there exists a Dirac net $(f_j)_{j \in J}$ with $f_j(k^{-1}g k) = f_j(g)$ for all $k \in K$ and $j \in J$. The Dirac net $(f_j)_j$ is said to be $K$-invariant. 
\end{lemma}
\begin{proof}
Pick any Dirac net $(h_j)_{j \in J}$, fix the probability Haar measure $\d k$ on $K$ and define
\[ f_j(x) := \int\limits_K h_j(k^{-1}g k) \d k.\]
Certainly $f_j$ remains non-negative, normalized and satisfies $f_j(k^{-1}g k) = f_j(g)$  for all $k \in K$ and $j \in J$. It remains to be shown that there exists a $K$-conjugation-invariant base of neighborhoods. This is the content of the next lemma.
\end{proof}
\begin{lemma}
Let $G$ be a locally compact and $K$ be a compact subgroup, for every open, relatively compact neighborhood $O$ of the identity, there exists an open neighborhood $O'$ of the identity which is contained in $O$, and satisfies $k^{-1}O'k \subset O'$ for all elements $k \in K$.
\end{lemma}
\begin{proof}
Set $C$ as the closure of $O$. The map
\[ \alpha: K \times C \rightarrow G, \qquad \alpha:(k,o) \mapsto k^{-1}ok\]
is continuous, has compact image, and the image contains $O$. So the set
\[ \bigcap\limits_{k \in K} k^{-1}Ck \]
is compact and there exists a finite subset $F \subset O$ such that
\[ \bigcap\limits_{k \in K} k^{-1}Ck = \bigcap\limits_{k \in F} k^{-1}Ck.\]
Since the open kernel of $k^{-1}C k$ is $k^{-1}Ok$, we have that
\[ O'\coloneqq\bigcap\limits_{k \in K} k^{-1}Ok = \bigcap\limits_{k \in F} k^{-1}Ok.\]
The set $O'$ is open, relatively compact and contains $O$. It is invariant under conjugation by $K$.
\end{proof}

 \section{The decomposition into Hecke algebras}
\subsection{Useful conventions on Haar measure and representations of compact groups}
The Haar measure on a compact group is finite. We will hence assume that it is normalized to a probability measure, that is to say, the group has unit measure.

The Peter-Weyl Theorem asserts that every continuous, irreducible representation of a compact group is finite-dimensional and admits an invariant sesqui-linear product.\footnote{Our convention is that a sesqui-linear product $\langle \cdotp, \cdotp \rangle$ is linear in the second argument.} Most statements in this section rely on the Schur orthogonality relations \cite{Knapp:Semi}*{Corollary 1.10, pg.15}.
\begin{theorem}[The Schur orthogonality relations]\label{thm:schurortho}
Let $K$ be a compact group, let $(\rho_1, V_1)$ and $(\rho_2,V_2)$ be two unitary, finite-dimensional, irreducible representations.

For all $v_1, v_1' \in V_1$ and $v_2, v_2' \in V_2$, the following identities hold
\begin{align*}
\int\limits_{K} \langle \rho_1(k)v_1,  v_1' \rangle \overline{ \langle \rho_2(k) v_2, v_2' \rangle} \d k & =\begin{cases} 0 , & \rho_1\not\cong \rho_2, \\ \frac{\langle v_1,  v_2\rangle \overline{ \langle v_1',  v_2' \rangle}}{\dim(V_1)}, & \rho_1 \cong \rho_2\end{cases}  
\end{align*}
and
\begin{align*}
\int\limits_{K} \tr_{V_1}(\rho_1(k)) \overline{ \tr_{V_2}( \rho_2(k))} \d k & =\begin{cases} 0 , & \rho_1\not\cong \rho_2, \\ 1, & \rho_1 \cong \rho_2.\end{cases}   
\end{align*}
\end{theorem}

\subsection{Projections on $\Ccinf(G)$ and their relations}
\begin{defn}[$K$-expansion]\label{defn:Kexp}
 Let $G$ be a locally compact group, let $K$ be a closed subgroup. Let $(\rho_1, V_1)$ and $(\rho_2, V_2)$ be finite-dimensional, unitary representations of $K$. We define the projections 
\begin{align*}
&\mP_{\rho_1, \rho_2} : \Ccinf(G) \rightarrow \Ccinf(G) \otimes \Endo_\bC(V_1) \otimes \Endo_\bC(V_2), \\
 &\mP_{\rho_1, \rho_2} \phi(g) = {}_{\rho_1}\!\phi\!_{\rho_2}(g)  = \dim(\rho_1) \dim(\rho_2) \int\limits_{K}\int\limits_{K} \rho_1(k_1) \phi(k_1^{-1} g k_2)\rho_2(k_2^{-1}) \d k_1 \d k_2.   
\end{align*}
and
\begin{align*}
&\mP^K : \Ccinf(G) \rightarrow \Ccinf(G), \\
 &\mP^K \phi(g) =  \phi^K (g)  = \int\limits_{K}  \phi(k^{-1} g k) \d k.   
\end{align*}
We define furthermore for each finite-dimensional, unitary representation $(\rho, V)$ of $K$:
\begin{align*} \mP^\rho : \Ccinf(G) \rightarrow \Ccinf(G) \otimes \Endo_\bC(V_1)  \\ 
         \mP^\rho \phi(g) =      \phi^\rho(g)  = \dim(\rho)     \int\limits_{K} \int\limits_{K} \phi(k_1^{-1} g k_2) \rho(k_1k_2^{-1}) \d k_1 \d k_2.
\end{align*}
\end{defn}
Let us collect some obvious invariance properties:
\begin{lemma}
In the notation of Definition~\ref{defn:Kexp}, we have for all elements $k, k' \in K$ the following relations
\begin{align*}_{\rho_1}\! \phi\! _{\rho_2}(kgk') &= \rho_1(k) _{\rho_1}\! \phi\! _{\rho_2}(g) \rho_2(k') ,\\
                     \phi^K(kgk^{-1}) &= \phi^K(g), \\    \phi^\rho (kgk') &= \rho(k) \phi^\rho(g) \rho(k')  .
\end{align*}
\end{lemma}
\begin{proof}
 A compact group is unimodular. By the invariance of the Haar measure, we observe the invariance:
 \begin{align*} 
  \int\limits_{K}\int\limits_{K}& \rho_1(k_1) \phi(k_1^{-1} k g k' k_2)\rho_2(k_2^{-1}) \d k_1 \d k_2 \\
&  =\int\limits_{K}\int\limits_{K} \rho_1(kk_1) \phi(k_1^{-1} g k_2)\rho_2(k_2^{-1}k') \d k_1 \d k_2, \\
   \int\limits_{K} \int\limits_{K} & \phi(k_1^{-1} k g k' k_2) \rho(k_1k_2^{-1}) \d k_1 \d k_2           \\         &=       \int\limits_{K} \int\limits_{K} \phi(k_1^{-1} g k_2) \rho(k k_1 k_2^{-1} k') \d k_1 \d k_2.      \qedhere    
 \end{align*}
\end{proof}

\begin{lemma}\label{lemma:Kid}
In the notation of Definition~\ref{defn:Kexp}, we have that
\begin{align*}
 \tr_{V_1 \otimes V_2} \left( _{\rho_1}\! \left( \phi^K \right)\! _{\rho_2}(g) \right) = \begin{cases} 0, & \rho_1 \neq \rho_2,  \\ \tr_{V_1} \phi^{\rho_1}, & \rho_1 \cong \rho_2.\end{cases}
\end{align*}

\end{lemma}
\begin{proof}
For two irreducible representation $\rho_1,\rho_2$ of $K$, we decompose the integral
\begin{align*}
&\frac{\tr \left( _{\rho_1}\! \phi^K\! _{\rho_2}(g) \right)}{\dim(\rho_1) \dim(\rho_2)} =  \int\limits_{K}\int\limits_{K}\int\limits_K \tr \left( \rho_1(k_1) \right)  \phi(k^{-1} k_1^{-1}g k_2  k) \tr \left( \rho_2(k_2^{-1}) \right) \d k_1 \d k_2  \d k      \\
&\qquad = \int\limits_{K}\int\limits_{K} \phi(k_1^{-1} g k_2) \int\limits_K \tr \left( \rho_1(k_1k^{-1}) \right)  \tr \left( \rho_2(kk_2^{-1}) \right) \d k \d k_1 \d k_2.
\end{align*}
The Schur orthogonality relations for matrix coefficients~\ref{thm:schurortho} yield for an orthonormal basis $(\vec{v}_{i,j})_{j=1}^{\dim\rho_i}$ of $V_i$: 
\begin{align*} 
& \int\limits_K \tr \left( \rho_1(k_1k^{-1}) \right)   \tr \left( \rho_2(kk_2^{-1}) \right) \d k\\  
& = \sum\limits_{i,j}   \int\limits_{K} \langle \vec{v}_{1,i}, \rho_1( k_1k^{-1}) \vec{v}_{1,i} \rangle_1 \cdot   \langle \vec{v}_{2,j}, \rho_2(kk_2^{-1}) \vec{v}_{2,j} \rangle_2  \d k     \\
 & =    \sum\limits_{i,j}   \int\limits_{K} \langle  \rho_1(k_1^{-1}) \vec{v}_{1,i}, \rho_1(k^{-1}) \vec{v}_{1,i} \rangle_1 \cdot   \langle\rho_2(k^{-1}) \vec{v}_{2,j}, \rho_2(k_2^{-1}) \vec{v}_{2,j} \rangle_2  \d k     \\
  & =    \sum\limits_{i,j}   \int\limits_{K} \langle  \rho_1(k_1^{-1}) \vec{v}_{1,i}, \rho_1(k^{-1}) \vec{v}_{1,i} \rangle_1 \cdot  \overline{ \langle\rho_2(k_2^{-1}) \vec{v}_{2,j}, \rho_2(k^{-1}) \vec{v}_{2,j} \rangle_2  \d k}     \\
&\qquad  = \begin{cases} 0     , & \rho_1 \neq \rho_2, \\ \frac{1}{\dim(\rho_1)}\sum\limits_{j,i} \langle \vec{v}_{1,i} ,\vec{v}_{2,j}  \rangle\langle \rho(k_1^{-1}) \vec{v}_{1,i},  \rho_1(k_2^{-1}) \vec{v}_{2,j} \rangle  & \rho_1 = \rho_2.\end{cases} 
\end{align*} 
If $\rho_1 \cong \rho_2$, we may safely assume that $\vec{v}_{1,i} = \vec{v}_{2,j}.$ Rewriting yields the result
\[ \sum\limits_{j} \langle \rho(k_1^{-1}) \vec{v}_{1,j},  \rho_1(k_2^{-1}) \vec{v}_{2,j} \rangle = \tr_{V_1} \rho_1(k_1 k_2^{-1}).\qedhere\]
\end{proof}

\begin{lemma}[Relations]
In the notation of definition~\ref{defn:Kexp}, the operators $\Ccinf(G) \rightarrow \Ccinf(G)$ given by 
\[ p_{\rho_1, \rho_2}\phi \mapsto \tr_{V_1 \otimes V_2} (_{\rho_1}\! \phi\! _{\rho_2}(g) ), \qquad  p^K: \phi \mapsto \phi^{K}, \qquad p^\rho: \phi \mapsto \tr_{V} \phi^\rho\]
are $*$-algebra homomorphisms and projections.  Additionally, all of the above operators commute and the following relations are given\footnote{The Kronecker delta function $\delta_{\{X=Y\}}$ is zero (one) if $X \neq Y$ ($X=Y$).}
\begin{align*}
p^K      \circ p^\rho &= p^\rho \\ 
p^K      \circ p_{\rho_1, \rho_2} &= \delta_{\{\rho_1 = \rho_2\}} \cdot p^{\rho_1} \\ 
p_{\rho_1, \rho_2} \circ p_{\rho_3, \rho_4}& = \delta_{\{\rho_1 =\rho_3\}} \delta_{\{\rho_2 =\rho_4\}} \cdot   p_{\rho_1, \rho_2} \\ 
p^{\rho'} \circ p^\rho   &= \delta_{\{\rho'=\rho\}} \cdot    p^\rho     \\ 
p_{\rho_1, \rho_2} \circ p^\rho  &= \delta_{\{\rho_1 = \rho\}} \delta_{\{ \rho_2=\rho\}}  \cdot   p^\rho
\end{align*}
Also for all $\phi, \phi' \in \Ccinf(G)$, we have following formula for the convolution product:
\[    \left(   p_{\rho_1, \rho_2} \phi \right) \ast \left( p_{\rho_3, \rho_4} \phi' \right) = \delta_{\{ \rho_2 =\rho_3\}}    p_{\rho_1, \rho_4}  \left(\phi \ast \phi' \right) .\]
\end{lemma}
\begin{proof}
This follows from the Schur orthogonality relations~\ref{thm:schurortho} for characters of compact groups. The relation $ p_{\rho_1, \rho_2}\circ p^K      = \delta_{\{\rho_1 = \rho_2\}} \cdot p^{\rho_1}$ has been verified in Lemma~\ref{lemma:Kid}. 
The other relations follow with similar computations.
\end{proof}
\begin{defn}
We denote the image of $p_{\rho_1, \rho_2}$ by $\mH(G, \rho_1,\rho_2)$, of $\mP^K$ as $\Ccinf(G)^K$ and of $p^\rho$ as $\mH(G, \rho)$.
\end{defn}
The above definition only depends on the isomorphism classes of the irreducible representation. The Dixmier-Malliavin Theorem reveals
\[ \mH(G,\rho_1, \rho_2) \ast \mH(G,\rho_3, \rho_4)  =\begin{cases} \mH(G, \rho_1, \rho_4), & \rho_2 \cong \rho_3, \\ 
                                                                \{ 0 \}, & \rho_2 \not\cong \rho_3.\end{cases}.\]
The convolution product on $\mH(G, \rho_1, \rho_2)$ is zero if $\rho_1$ and $\rho_2$ are not isomorphic.

From the above lemma, we get a decomposition of the Hilbert space $\mL^2(G)$ as unitary $K \times K$-bi-module. The following proposition provides the same decomposition on the $\Ccinf(G)$-level, which is more difficult to achieve. We deduce this decomposition from the Lie group case.
\begin{proposition}\label{prop:Kexp}
Let $G$ be a locally compact group and let $K$ be a compact subgroup. 
Every smooth function $\phi \in \Ccinf(G)$ or $\phi \in \mathrm{C}^\infty(G)$  satisfies the following identities:
\begin{align*}
\phi &= \sum\limits_{\rho_1, \rho_2} p_{\rho_1,\rho_2} \phi,  \\ 
\phi^K &  = \sum\limits_{\rho}  p^\rho \phi, 
\end{align*}
 where the sums run through all irreducible, unitary representations of $K$, and converges absolutely.  
\end{proposition}
\begin{proof}
The second equality follows from the first by the Lemma~\ref{lemma:Kid}. 

The first equality is proven in \cite{Harish:DiscreteII} for a unimodular Lie group. The statement for a general locally compact group follows. 

As we have seen, a locally compact group $G$ has an open, closed, almost connected subgroup, which contains $K$. It is sufficient to prove the result for smooth functions which are supported on $G'$, since $G/G'$ is a discrete space.

Let $\phi$ be a smooth function on $G'$, then there exists a normal, compact subgroup $N$ of $G'$ such that $\phi$ is bi-$N$-invariant and such that $G'/N$ is a Lie group. We consider $\phi$ as a smooth function on the Lie group $G_1 = G'/N$. Since $N$ is normal in $G'$, the group $N$ commutes with all elements $k \in K$. Thus $K$ acts from the right and from the left on $\Ccinf(G_1)$, and the action factors through $K_1 := K / K \cap N \cong KN/N$. The isomorphism between $ K / K \cap N $ and $KN /K$ goes by the term second isomorphism theorem. 

Now if $G_1$ is not unimodular\footnote{I am not aware of any interesting compact subgroup in non-unimodular groups, but to be thorough, I include the argument for the general case.}, the modular character of $G_1$ gives rise to a group extension
\[  \textup{ker} \Delta_{G_1} \rightarrow G_1 \xrightarrow{\Delta_{G_1}} (0, \infty).\]
We have an isomorphism of $C^\infty$-manifolds
\[ G_1 = \textup{ker} \Delta_{G_1}  \times      \textup{image} \Delta_{G_1},\]
and this map is $K_1$-invariant, since $\Delta_{G_1}(kgk') =   \Delta_{G_1}(g) $ for all $k, k' \in K_1$, in particular $K_1 \subset G_2 = \textup{ker} \Delta_{G_1}$.
We identify $\Ccinf(G_1)  = \Ccinf( G_2) \otimes \Ccinf(   \textup{image} \Delta_{G_1})$, and assume that $\phi = \phi_2 \otimes \phi_3$ for $\phi_2 \in \Ccinf(G_2)$ and $\phi_3 \in  \Ccinf(   \textup{image} \Delta_{G_1})$.

The first identity holds now for $\phi_2 \in \Ccinf(G_2)$ by lemma 9 in \cite{Harish:DiscreteII}*{pg.14} respectively \cite{Harish-Chandra:Collected_III}*{pg.551}.
\end{proof}

\begin{lemma}
Let $G$ be a locally compact group with a compact subgroup $N$. Let $\phi \in \Ccinf(G)$ be bi-$N$-invariant.

The following identities hold:
\begin{align*} \phi &= \sum\limits_{\rho_1, \rho_2 \atop \Res_{K \cap N} \rho_j = 1}   p_{\rho_1, \rho_2}    \left( _{\rho_1}\!\phi\!_{\rho_2} \right),                               \\
                   \phi^K& = \sum\limits_{\rho \atop \Res_{N \cap K} \rho = 1}  p^\rho \phi
\end{align*}
In particular, if $G$ is a locally pro-finite group, then the above sums are finite.
\end{lemma}
\begin{proof}
 This follows from the Schur orthogonality relations for matrix coefficients~\ref{thm:schurortho}:
\begin{align*}
 &   \int\limits_{K}\int\limits_{K} \tr( \rho_1(k_1)) \phi(k_1^{-1} g k_2)\tr( \rho_(k_2^{-1})) \d k_1 \d k_2   \\
 &  \int\limits_{K/K\cap N}\int\limits_{K/K\cap N} \phi(k_1^{-1} g k_2) \int\limits_N \tr(\rho_1(k_1n_1)) \d n_1 \int\limits_N \tr(\rho_2(n_2^{-1}k_2^{-1})) \d n_2  \d \dot{ k_1} \d \dot{k_2},
\end{align*}
since they imply $\int\limits_N \tr(\rho(kn)) \d n = 0$ if $\Res_{N \cap K}\rho \neq 1$. This implies the first identity and the second follows by lemma~\ref{lemma:Kid}. 
\end{proof}

 \subsection{The decomposition of $\Ccinf(G)$ into Hecke algebras}
\begin{defn}[Hecke algebras]\label{defn:heckealgebra}                                                                              \index{$\mH(G, \rho)$} \index{$ \Ccinf(G)^K $}
Let $G$ be a locally compact group and $K$ be a compact subgroup. Let $(\rho,V), (\rho_1,V_1), (\rho_2,V_2)$ be unitary, irreducible representations of $K$. We define the vector spaces
\begin{align*}
 \mmH( G , \rho_1, \rho_2) &:= \left\{ f \in \Ccinf(G) \otimes \Hom_\bC(V_2, V_1) : f(k_1 g k _2)  = \rho_1(k_1) f(g) \rho_2(k_2) \right\} ,  \\
 \mmH(G, \rho)  &:= \mmH(G, \rho, \rho).
\end{align*}
\end{defn}
The space $\mmH(G, \rho_1, \rho_2)$ is closely related to the algebra $\mH(G,\rho_1, \rho_2)$, and the space $\mmH(G, \rho)$ is closely related to $\mH(G,\rho)$. We have vector space isomorphisms
\[ \mH(G, \rho_1, \rho_2) \cong \mmH(G, \rho_1, \rho_2) \otimes \Endo_\bC(V_{\rho_1}, V_{\rho_2}), \qquad \mH(G, \rho) = \mmH(G,\rho) \otimes \Endo_\bC(V_\rho).\]
We define a $*$-algebra structure on $\mH(G, \rho)$ only. Let $\phi, \phi_1$ and $\phi_2$ be elements in $\mH(G, \rho)$.
The convolution product is defined as
\[ \phi_1 \ast \phi_2 (x)  = \int\limits_{G} \phi_1(g) \phi_2(g^{-1}x) \d g,\]
and the * involution is given by
\[ \phi^*(x) = \Delta_G(x^{-1}) \phi(x^{-1})^{\dagger},\]
where $\phi(x^{-1})^\dagger$ is the adjoint of the element $\phi(x^{-1}) \in \Endo_\bC(V_\rho)$. 

\begin{proposition}[$\Ccinf(G)$ and $\Ccinf(G)^K$ as $K-K$ bi-module]\label{prop:module}
Let $G$ be a locally compact group and $K$ a compact subgroup. Let $(\rho,V), (\rho_1,V_1), (\rho_2,V_2)$ be unitary, irreducible representations of $K$. For elements $v,w \in V$ and $v_j \in V_j$ for $j=1,2$, we have projections
\[  \phi \mapsto   \left( _{\rho_1}\!\phi\!_{\rho_2} \right)_{v_1,v_2} := \langle v_1,   _{\rho_1}\!\phi\!_{\rho_2}  v_2 \rangle_{V_1}, \qquad   \phi \mapsto \phi^\rho_{v,w} \langle v,  \phi^\rho  w \rangle_{V}.\]
There exists a dense embedding of vector spaces \index{$\Upsilon$}
\begin{align*} \Upsilon:  \bigoplus\limits_{\rho_1, \rho_2} &\left( \mmH( G , \rho_1, \rho_2)  \otimes V_1^* \otimes V_2 \right) \xrightarrow{\cong} \Ccinf(G), \\ 
                               &        f_{\rho_1, \rho_2} \otimes \vec{v}_1 \otimes \vec{v}_2 \mapsto \left( g \mapsto \langle  \vec{v}_1,  f_{\rho_1, \rho_2}(g) \vec{v}_2 \rangle_1 \right),
\end{align*}
 where the sum runs through one orthonormal system.
The isomorphism $\Upsilon$ restricts to an isomorphism
\begin{align*} \Upsilon:  \bigoplus_{\rho} \mmH( G , \rho)  \otimes \Endo_\bC(V_\rho) \xrightarrow{\cong} \Ccinf(G)^K. \end{align*} 
\end{proposition}
A proof for a locally compact, totally disconnected group can be found in \cite{BushnellKutzko:GLNopen}*{Proposition 4.2.4, pg.148}. 
\begin{proof}
We decompose by Proposition~\ref{prop:Kexp}:
\begin{align*}
 \phi = \sum\limits_{\rho_1, \rho_2} \tr_{\rho_1 \otimes \rho_2} \left( _{\rho_1}\!\phi_{\rho_2} \right)
\end{align*}
with
\[ _{\rho_1}\!\phi_{\rho_2} \colon G \rightarrow \Endo_{\bC}( V_1) \otimes \Endo_{\bC}(V_2) \cong V_1 \otimes V_1^* \otimes V_2 \otimes V_2^* \]
satisfies
\[     _{\rho_1}\!\phi_{\rho_2} (k_1 g k_2)  = \rho_1(k_1) \phi(g)  \rho_2(k_2). \]
The representation $\rho_1$ only acts on $V_1$ and the representation $\rho_2$ only acts on $V_2^*$. This proves that $\Upsilon$ is an isomorphism of vector spaces. 
\end{proof}

\begin{corollary}
The algebra $\Ccinf(G)^K$ is Morita equivalent to the topological closure of the algebra $\bigoplus_{\rho}  \mmH(G,\rho)$.
\end{corollary}
\begin{proof}
This follows directly from Proposition~\ref{prop:module}.
\end{proof}
 \begin{corollary}
 The $*$-algebra $\mH(G, \rho)$ admits an approximate identity. In particular $\bigoplus_{\rho} \mH(G, \rho)$ contains an approximate identity who is a $K$-invariant Dirac net in $\Ccinf(G)$. 
 \end{corollary}
\begin{proof}
 Pick any Dirac net $(f_j)_j$ in $\Ccinf(G)$, and consider $(f_j)_j^\rho$. This is an approximate identity for the $*$-algebra $\mH(G, \rho)$ according to Proposition~\ref{prop:module} and Lemma~\ref{lemma:diracapprox}. 
\end{proof}
\begin{example}[Approximate identity in the totally disconnected case]
Let $G$ be a locally pro-finite group. Then $\mH(G, \rho)$ is unital. The function
\[ x \mapsto \begin{cases} \frac{\tr \rho(x)}{\dim(\rho)}, & x \in K, \\ 0,  & x \notin K,
             \end{cases}    \]
is the unit element of $\mH(G, \rho)$.
\end{example}

\section{The Abel transform}\label{section:Abel}
We generalize the definition of the Abel transform as given in \cite{Lang:SL2}. The Abel transform appears under different names, such as the Harish-transform \cite{Lang:SL2}, or  the constant term \cite{Laumon1} in literature.
This transform allows to translate the representation theory of a locally compact group to the representation theory of its subgroups.
\subsection{The Iwasawa decomposition}
We impose a topological datum on a topological group. This datum is most natural in the context of $\GL_n(F)$, 
where $F$ is an arbitrary field, and we use it soon only in the context where $n=2$ and $F$ is a local field.
\begin{defn}[Iwasawa datum]
Let $G$ be a locally compact group. A triple $(N, M, K)$ of closed subgroups is called an Iwasawa datum\index{Iwasawa datum $(N,M,K)$} if the following conditions are satisfied:
\begin{itemize}
\item the group $M$ normalizes the group $N$, i.e., $m^{-1} n m \in  N$ for all $n\in N$ and $m\in M$,
\item the group $M$ and $N$ have trivial intersection, i.e., $M \cap N = \{1\}$, 
\item the group $K$ is compact,
\item we have a surjection 
   \[  N\times M  \times K \twoheadrightarrow G, \qquad (n,m,k) \mapsto nmk.\]  
\end{itemize}
The Iwasawa datum is strict if \( nmk=n'm'k' \) implies $m=m'$. The Iwasawa datum is unimodular if $M$ and $N$ are unimodular. 
\end{defn}
\begin{example}[Some examples]\label{ex:iwasawa}
 The trivial Iwasawa data are $(G, \{1 \}, \{1\})$ and $(\{1\}, G,\{1\})$. If we apply the theory of this section to these triples, we obtain trivial statements.

 The following triples are examples which the reader should keep in mind. Consider $G = \GL_2(\bR)$, $\GL_2(\mathbb{Q}_p)$ or $\GL_2(\bC)$, with
\[ N =\sma 1 &  *  \\
                 0 & 1\smz           \qquad M  =\sma  * & 0 \\  0 & * \smz, \qquad  K = \begin{cases} \U(2), & G=\GL_2(\bC), \\ 
                                                                                                                                                                           \O(2), & G=\GL_2(\bR), \\
                                                                                                                                                                            \GL_2(\mathbb{Z}_p), & G=\GL_2(\mathbb{Q}_p),
                                                                                                                                      \end{cases}
\]
and $G = \GL_4(\bR)$, $\GL_4(\mathbb{Q}_p)$ or $\GL_4(\bC)$ with
\[ N =\sma 1 &  0& * & * \\
                 0 & 1 & * & * \\
                 0 & 0 & 1 & * \\
                 0  &  0 &  0 & 1 \smz, \qquad M  =\sma  * & * & 0 & 0 \\ * & * & 0 & 0 \\ 0 & 0 & *& 0 \\ 0 & 0 &  0 & * \smz, \qquad  K = \begin{cases} \U(4), & G=\GL_4(\bC), \\ 
                                                                                                                                                                           \O(4), & G=\GL_4(\bR), \\
                                                                                                                                                                            \GL_4(\mathbb{Z}_p), & G=\GL_4(\mathbb{Q}_p).
                                                                                                                                      \end{cases}
\]
Then $(N,M,K)$ is an Iwasawa datum. It is not strict, since $K \cap M \neq \{ 1 \}$. When we replace $M$  by a co-compact subgroup, i.e., for $\GL_2(\bR)$ and $\GL_2(\bC)$ set 
 \[        M'  =  \left\{ \sma  xa  & 0 \\ 0 &x/a  \smz  : x, a >0 \right\},\]
and  for $\GL_2(\mathbb{Q}_p)$
\[        M'  =  \left\{ \sma  p^n  & 0 \\  0 &p^m \smz  : n, m\in \mathbb{Z} \right\},\]
then the triple $(N,M',K)$ is a strict Iwasawa datum.
If we replace $M$ by
\[     M'' = \left\{ \sma  za & zb & 0 & 0 \\ zc & zd & 0 & 0 \\ 0 & 0 & *& 0 \\ 0 & 0 &  0 & * \smz  : ad-bc =1, z \in (0, \infty) \right\},\]
 for $\GL_4(\bC)$ or $\GL_4(\bR)$, or
 \[     M''  = \left\{ \sma  p^m a & p^m b & 0 & 0 \\ p^m c &  p^m d & 0 & 0 \\ 0 & 0 & *& 0 \\ 0 & 0 &  0 & * \smz  : ad-bc =1, m \in \mathbb{Z} \right\},\]
for $G = \GL_4(\mathbb{Q}_p)$,  then $(N, M'', K)$ is not a strict Iwasawa datum either. 
\end{example}  
The group decomposition $G =MNK$ results in a decomposition of the Haar measures.
\begin{lemma}[Measure decomposition of an Iwasawa datum]\label{lemma:iwasawa}
Let $G$ be a locally compact group, and let $(N, M, K)$ be an Iwasawa datum. Define $B$ as the semi-direct product $N \rtimes M$. There exist left Haar measures $\d g$, $\d n$, $\d m$ and $\d k$ on $G$, $M$, $N$ and $K$ such that for all $f \in \mL^1(G, \d g)$, we have that
\begin{align*} \int\limits_{G} f(g) \d g &= \int\limits_{M} \int\limits_{N} \int\limits_K f(mnk) \d m \d n \d k \\ 
                                                     & = \int\limits_{M} \int\limits_{N} \int\limits_K \frac{\Delta_N(n) \Delta_M(m)}{\Delta_B(nm)} f(nmk) \d m \d n \d k .
\end{align*}
Furthermore, we can choose three of the Haar measures arbitrarily, and the last is uniquely determined by the others.
\end{lemma}
\begin{proof}
Since $M$ normalizes $N$ and both groups have trivial intersection, the semi-direct product is well-defined and equal to the group generated by $M$ and $N$. The left Haar measure $\d b$ of a semi-direct product $B= N \rtimes M$ is computed according to the quotient integral formula, see \cite{DeEc}*{Theorem 1.5.2}:
\[ \int\limits_{B} f(b) \d b =   \int\limits_{M} \int\limits_N f(mn) \d m \d n .\]
Two Haar measures are chosen arbitrarily, and the last one is determined uniquely. Furthermore, the inversion formula yields
\[   \int\limits_{B} f(b) \d b =  \int\limits_{B} f(b^{-1}) \Delta_B(b)^{-1} \d b =   \int\limits_{M} \int\limits_N f(n^{-1}m^{-1}) \Delta_B(mn)^{-1}  \d m \d n.\]
Since $K$ is compact, we can appeal to \cite{DeEc}*{Proposition 1.5.5} 
\[ \int\limits_{G} f(g) \d g =  \int\limits_{B} \int\limits_K f(bk) \d b \d k = \int\limits_{M} \int\limits_{N} \int\limits_K f(mnk) \d m \d n \d k. \]
Two of the three Haar measures $\d g$, $\d b$ and $\d k$ are chosen arbitrarily, and the remaining one is then determined uniquely. 
\end{proof}

\subsection{The Abel transform of an Iwasawa datum}
From this point forward, I will restrict my attention on unimodular groups and unimodular Iwasawa data. This is mostly for convenience, but I am also not aware of any interesting non-unimodular Iwasawa datum. We define the Abel transform in this context.
\begin{defn}[The Abel transform]\label{defn:abel}
Let $G$ be a unimodular, locally compact group with an Iwasawa datum $(N, M, K)$. Let $(\rho,V_\rho)$ be a unitary, finite-dimensional representation of $K$, and let $\Delta_B$ the modular character of $B = N \rtimes M$. 

The following operator is referred to as the Abel transform\index{Abel transform, $\mA_\rho$} 
\begin{align*}  \mA_\rho: \mmH(G, \rho) \rightarrow \mmH(M, \Res_{M \cap K} \rho), \qquad \mA_\rho \phi(m) =\Delta_B(m)^{1/2} \int\limits_N   \phi(mn) \d n.\end{align*}
\end{defn}
\begin{proposition}
The Abel-transform is a $*$-algebra-homomorphism. 
\end{proposition}
\begin{proof}
The operator is certainly linear. Let us now verify that it is a $*$-linear homomorphism, i.e.,
\[ \mA_\rho( \phi^*) = \mA_\rho(\phi)^*.\]
Since $G$ is unimodular and $B$ is the semi-direct product $N \rtimes M$, we have that the operator $\mA =\mA_\rho$ respects the $*$ operation. Let $\phi \in \mmH(G, \rho)$, then
\begin{align*}
 \mA(\phi^*) (m) &=  \Delta_B(m)^{1/2} \int\limits_N \phi^* ( m n)  \d n \\
     &=   \Delta_B(m)^{1/2}  \int\limits_N  \phi (n^{-1} m^{-1} )^\dagger \d n  \\
&=   \Delta_B(m)^{1/2}  \int\limits_N  \phi (n m^{-1} )^\dagger \d n \\
 &=  \Delta_B(m)^{1/2}  \Delta_B(m)^{-1}  \int\limits_N  \phi ( m^{-1} n)^\dagger  \d n ,
\end{align*}
On the other hand, the definition gives 
\begin{align*}   \mA(\phi)^* (m) =   \mA(\phi) (m^{-1})^\dagger =  \Delta_B(m^{-1})^{1/2} \int\limits_N  \phi(m^{-1} n )^\dagger  \d n ,\end{align*}
which verifies the $*$-property.

Now let us verify that the Abel transform is an algebra homomorphism. Consider $\phi_1, \phi_2 \in \mmH(G, \rho)$:
\begin{align*}
        \Delta_B^{1/2}(m)   \mA ( \phi_1 \ast \phi_2) (m) &    = \int\limits_N \phi_1 \ast \phi_2( mn)  \d n  \\ 
&=    \int\limits_N \int\limits_{G} \phi_1(mny) \phi_2(y^{-1}) \d y \d n   \\
&=     \int\limits_N \int\limits_{G} \phi_1(mx) \phi_2(x^{-1}n) \d x \d n.
 \end{align*}
The Iwasawa decomposition and the property of $K$-invariance yield now:
 \begin{align*}
& \Delta_B^{1/2}(m) \mA ( \phi_1 \ast \phi_2)(m)\\
 & =    \int\limits_N \int\limits_{M \times N \times K}  \phi_1(m m_0  n_0 k) \phi_2(k^{-1} n_0^{-1}m_0^{-1}n)  \d k \d n_0 \d m_0 \d n  \\
& =      \int\limits_{M \times N}  \int\limits_N \phi_1(m m_0^{-1}  n_0) \phi_2(n_0^{-1}m_0^{-1}n)   \d n \d n_0 \d m_0          \\
& =   \Delta_B(m_0)^{-1}    \int\limits_{M \times N}  \int\limits_N\phi_1(m m_0^{-1}  n_0) \phi_2(n_0^{-1}n m_0^{-1})  \d n \d n_0 \d m_0  \\
& =  \Delta_B(m_0)^{-1}     \int\limits_{M \times N}  \int\limits_N\phi_1(m m_0^{-1}  n_0) \phi_2( n m_0^{-1})   \d n_0  \d n \d n_0 \d m_0    \\   
& =    \int\limits_N \int\limits_{M \times N}  \phi_1(m m_0^{-1}  n_0) \phi_2(m_0^{-1} n )  \d n \d n_0 \d m_0 .    
\end{align*}                                                                                                      
We have shown that
\begin{align*}
 \mA ( \phi_1 \ast \phi_2) (m) =&\\
   \int\limits_{M} \int\limits_{N} &\Delta_B^{-1/2}(mm_0^{-1})   \int\limits_N \phi_1(m m_0^{-1}  n_0) \Delta_B^{-1/2}(m_0)  \phi_2(m_0^{-1} n ) \d n \d n_0 \d m_0\\ 
                           & {}\qquad \qquad = \mA \phi_1 \ast \mA \phi_2 (m).\qedhere
\end{align*}
\end{proof}

\subsection{The dual of the Abel transform}
The following definition is inspired by \cite{Sarkar:Abel}*{Section 3, page 258}, but it has no significance for the rest of the thesis. I included it for the sake of completeness.
\begin{defn}[The dual Abel transform --- Case $\Res_{m^{-1}Nm \cap K}\rho =1$]\label{defn:dualabel}
 Let $G$ be a unimodular group, and let $(N,M, K)$ be a strict, unimodular Iwasawa datum. Assume that
\[ H_M: G \twoheadrightarrow M, \qquad H(nmk)=m,\]
is smooth, and that there exists a smooth map
\[ H_{K} : G \twoheadrightarrow K\]
such that $H_{K}(g) = k$. Then there exists a (unique) element $n \in N$ such that $g = nmk$.
Let $(\rho, V)$ be an irreducible representation of $K$ such that for all $m \in M$ the restriction $\Res_{m^{-1}Nm \cap K}\rho =1$ is trivial.

The dual Abel transform is the operator
 \begin{align*}  \mA_\rho^\dagger: \; &\mmH(M, \Res_{M \cap} \rho) \rightarrow \mmH(G, \rho), \\ 
\mA^\dagger_\rho f(x)&\coloneqq \\ 
 \int\limits_K \int\limits_K &\rho( k_0^{-1}) \Delta_B(H_M(k_0xk^{-1}))^{1/2} f(H_M(k_0xk^{-1}))  \rho(H_K(k_0xk^{-1}) k)\d k_0 \d k.\end{align*}
\end{defn}

\begin{lemma}
The operator $\mA_\rho^\dagger$ is well-defined, $K$-bilinear and independent of the section $H_{K} : G \twoheadrightarrow K$ considered.
\end{lemma}
\begin{proof}
The integration along compact subgroups preserves both smoothness, and the property of being compactly supported. 
Note that $nmk = n' m k'$ implies that $k' k^{-1}  = m^{-1} n'^{-1}n m$, i.e., $k$ is defined up to right translation by $N \cap K$. Since $\rho$ is trivial on $N \cap K$, the dual Abel transform is independent of $H_K$
Furthermore, by the invariance of the Haar measure, we observe that
\[ \mA^\dagger_\rho f(k'xk'') = \rho(k')      \mA^\dagger_\rho f(x) \rho(k'').\qedhere\]
\end{proof}

\begin{proposition}[$\mA_\rho^\dagger = (\mA_\rho)^\dagger$]
   The dual Abel transform $\mA_\rho^\dagger$ is the adjoint of the Abel transform $\mA_\rho$, that is,
\begin{align*}   \int\limits_{G} & \tr_V \mA_\rho^\dagger f(x)  \phi^\dagger(x) \d x    = \int\limits_{M}  \tr_V f(m) \mA_\rho { \phi(m)}^\dagger   \d m.\end{align*}
\end{proposition}
\begin{proof}
The adjoint formula follows from a computation
\begin{align*}   \int\limits_{G} &\mA_\rho^\dagger f(x)  \phi^\dagger(x) \d x\\
&    =  \int\limits_G \int\limits_K \rho(k_0)^{-1}   \int\limits_K  \Delta_B(H_M(k_0xk^{-1}))^{1/2} f(H_M(k_0xk^{-1})) \rho(H_K(k_0xk^{-1}) k)  \phi^\dagger(x) \d k \d k_0 \d x \\ 
& =   \int\limits_K \rho(k_0)^{-1}   \int\limits_K \int\limits_G \Delta_B(H_M(x))^{1/2} f(H_M(x))  \rho(H_K(x) k)  \phi(k_0^{-1} xk)^\dagger\d k \d x \d k_0   \\
 & =  \int\limits_K \rho(k_0)^{-1}  \int\limits_G \Delta_B(H_M(x))^{1/2}f(H_M(x))   \rho( H_K(x)) \\
 & {} \qquad \qquad \qquad \cdot \int\limits_K \rho(k) \rho(k)^\dagger \phi( x)^\dagger  \d k   \d x \rho(k_0) \d k_0.
\end{align*}
The matrix $\rho(k) \rho(k)^\dagger  = \rho(k) \rho(k)^{-1}$ is the identity. Now, we appeal to the measure decomposition in Lemma~\ref{lemma:iwasawa} for the inner integral 
\begin{align*}  &   \int\limits_G \Delta_B(H_M(x))^{1/2}   f(H_M(x))   \rho( H_K(x))  \phi(x)^\dagger  \d x    \\
             & =       \int\limits_{N} \int\limits_{M} \int\limits_K  \Delta_B(H_M(nmk))^{1/2}   f(H_M(nmk)) \rho( H_K(nmk))  \phi(mnk  )^\dagger  \d m \d n \d k \\ 
             & =        \int\limits_{M}      f(m)  \Delta_B(m)^{1/2} \int\limits_{N} \int\limits_K \rho(k)  \phi( mnk )^\dagger   \d k \d n \d m \\ 
             & =     \int\limits_{M}   f(m)  \Delta_B(m)^{1/2} \int\limits_{N}   \phi(mn )^\dagger   \d n \d m  \\ 
              & =      \int\limits_{M}   f(m)  \left( \mA_\rho \phi(m ) \right)    \d m .
\end{align*}
The duality follows, since
\begin{align*}   \tr_V \int\limits_K \rho(k_0)^{-1}   \int\limits_{M}   f(m)  \left( \mA_\rho \phi(m ) \right)^\dagger  &  \d m  \rho(k_0) \d k_0 \\ 
                                                                                                                                             =  &     \tr_V  \int\limits_{M}   f(m)  \left( \mA_\rho \phi(m ) \right)^\dagger  \d m.\qedhere\end{align*}
\end{proof}

\chapter{Invariant harmonic analysis and representation theory}
\section{Preliminaries in representation theory}
\begin{defn}[Standard terminology]\mbox{}
\begin{enumerate}
 \item A (topological) representation $(\pi, V_\pi)$ of $G$ on a topological vector space is a weakly continuous group homomorphism into the invertible operators of $V_\pi$.
 \item We say that a vector $\vec{v}$ of $V_\pi$ for a representation $(\pi, V_\pi)$ of $G$ is smooth if for all functionals $l:V_\pi \rightarrow \bC$ the matrix coefficient
      \[ G \rightarrow \bC, \qquad g \mapsto l ( \pi(g) \vec{v})\]
is smooth. If all vectors are smooth, then we say that the representation $(\pi, V_\pi)$ is smooth.    
 \item  Let $K$ be a \textbf{large} subgroup of $G$, that is, a closed subgroup which contains the unique (up to conjugation)\footnote{Regarding the existence and uniqueness of a maximal compact subgroup in a connected group, the reader may consult \cite{MontgomeryZippin}{Theorem 4.13}.} maximal compact subgroup of the identity component $G_0$ and is open under the surjection $G \twoheadrightarrow G/G_0$. A representation $(\pi, V_\pi)$ of $G$ is said to be admissible if the restriction to $K$ decomposes with finite multiplicity.\footnote{The definition is independent of the large subgroup chosen. If it holds for one large compact subgroup, it holds for all large compact subgroups, since any two open compact subgroups in a totally disconnected group are commensurable up to conjugation, i.e., their common intersection has finite index in either one of the groups.}\index{large compact subgroup}
 \item A representation $(\pi, V_\pi)$ of $G$ is said to be unitarizable if there exists a positive sesquilinear product 
\[ \langle -,- \rangle :  V_\pi \times V_\pi \rightarrow \bC,\]
such that
\[ \langle \pi(g) \vec{v}, \pi(g) \vec{w} \rangle =     \langle  \vec{v}, \vec{w} \rangle.\]
\item  A representation $(\pi, V_\pi)$ of $G$ on a Hilbert space is said to be trace class if the operator associated to $\phi \in \Ccinf(G)$
\[   \pi(\phi) : V_\pi \mapsto V_\pi, \qquad \pi(\phi)v = \int\limits_{G} \phi(g) \pi(g) \d g \]
is trace class, i.e., for each orthonormal basis $X$ and each $\phi \in \Ccinf(G)$, the following series converges absolutely:
\[ \sum\limits_{v \in X} \langle v, \pi(\phi) v \rangle < \infty. \]  
\end{enumerate}
\end{defn}
\begin{example}
Let $\bA$ be the ring of ad\`eles of a global field. Then
\[ \bm K      = \prod\limits_{v \textup{ valuation}} K_v , \qquad K_v =\begin{cases} \O(2), & v \textup{ real},\\
                                                                         \U(2), & v \textup{ complex}, \\
                                                                            \GL_2(\o_v) , & v \textup{ non-archimedean}
                                                                       \end{cases}\]
is a large subgroup of $\GL_2(\bA)$.
\end{example}

I will list a number of standard facts which allow for an algebraic classification of smooth, admissible representations.
 \begin{theorem}[Smooth and algebraic representations]\label{thm:smoothrep}
Let $(\pi, V_\pi)$ be a topological representation of a locally compact group.
\begin{enumerate}
 \item The set of smooth vectors $V_\pi^\infty$ is a dense, invariant subspace of $V_\pi$. We denote the restriction to the smooth vectors by $(\pi_\infty, V_\pi^\infty)$. The space $V_\pi^\infty$ is endowed with the uniform topology, where a net $(v_j)_j$ converges if and only if the net of functions $\left( g \mapsto \pi(g) v_j\right)_j$ converges inside $\Cinf(G, V_\pi)$. The representation on $(\pi_\infty, V_\pi^\infty)$ is topological.
 \item Let $K$ be a large subgroup of $G$, then every smooth vector has a $K$-expansion. That is to say, given a finite-dimensional unitary representation $\rho$ of $K$, the projection
 \[ P_\rho: v \rightarrow v^\rho :=\dim(\rho) \int\limits_K \overline{\tr \rho(k)} \pi(k)v \d k \]
 surjects onto the $\rho$-isotype of $V_\pi$, i.e., the largest $K$-invariant subspace such that $\pi|_K$ is equivalent to possibly several copies of $\rho$. A vector $v \in V_\pi^\rho$ is automatically smooth, and for every smooth vector $v \in V_\pi^\infty$, the series
 \[ \sum\limits_{ \rho } v^\rho\]
 converges absolutely towards $v$ in the uniform topology, where the sum runs through a set of representatives $\rho$ of isomorphism classes of unitary, irreducible representations of $K$. In particular, the algebraic sum 
\[ V_{\pi,\textup{alg}}  = V_{\textup{alg},K} := \bigoplus_{\rho} V_\pi^\rho \]
of the $K$-isotypes is a dense subspace in $(\pi^\infty, V_\pi^\infty)$.
 \item Let $K$ be a large subgroup, and let $\rho$ be a representation of $K$ and $\check{\rho}$ the contragredient of $\rho$. Let $G$ carry a left-invariant Haar measure. 
Let $\phi$ be an element of the algebra $\Ccinf(G)$ of smooth functions. Then the operator
\[ \pi(\phi) : v  \mapsto \int\limits_{G}  \phi(g) \pi(g) v \d g\]
maps $V_\pi$ into the smooth vectors $V_\pi^\infty$; in fact, $\pi(g_0) \pi(\phi) = \pi( \phi( g_0^{-1}\blank ).$
Let $\phi$ be an element of the Hecke algebra $\mH(G, \check{\rho})$ (see Definition~\ref{defn:heckealgebra}), and then the operator
\( \pi( \phi) \)
maps $V_\pi$ into the $\rho$-isotype $V_\pi^\rho$; in fact, $\pi(\phi)v = \pi(\phi) v^\rho$.
\end{enumerate} 
 \end{theorem}
\begin{proof}
The first two statements are true for all unimodular Lie groups \cite{Harish:DiscreteII}*{Corollary 1, Lemma 4, Lemma 9, page 7-13}. They follow for a general locally compact group by considering it locally as a projective limit of Lie groups. This argument was demonstrated on numerous occasions in the preceding chapter, and repeating it is unnecessary. The third statement is proved by two short computations:
for any vector $v \in V_\pi$, any element $\phi \in \Ccinf(G)$ and $g_0 \in G$, we compute 
\begin{align*} \pi(g_0) \pi( \phi ) v &= \int\limits_G \phi(g) \pi(g_0g) v \d g  \\ 
                                  & = \int\limits_G  \phi(g_0^{-1}g) \pi(g) v \d g = \pi(\phi(g_0^{-1} \blank) v;
\end{align*} 
and for any vector $v \in V_\pi$ and $\phi \in \mH(G, \check{\rho})$, we compute
\begin{align*} \pi( \phi ) v &= \int\limits_G  \phi(g) \pi(g) v \d g \\
& =  \int\limits_G \left( \int\limits_K \phi(gk) \tr \rho(k)v \d k \right)  \pi(g) \d g \\ 
                                  & = \int\limits_G \phi(g) \pi(g) \left( \int\limits_K \tr \rho(k) \pi(k^{-1}) v \d k\right) \d g \\
                                  & = \int\limits_G \phi(g) \pi(g) \left( \int\limits_K \overline{\tr \rho(k) }\pi(k) v \d k\right) \d g = \pi(\phi) v^\rho.          \qedhere
\end{align*}
\end{proof}
\begin{corollary}
Let $(\pi, V_\pi)$ be a topological representation of a locally compact group, let $K$ be a large subgroup of $G$, and let $\rho$ be a finite-dimensional representation of $K$.
The following statements are equivalent:
\begin{itemize}
\item the $\rho$-isotype has finite multiplicity in the restriction of $\pi$, and
\item the operator $\pi(\phi)$ has finite rank for all elements $\phi \in \mH(G, \check{\rho})$.
\end{itemize}
\end{corollary}

\begin{defn}[The character distribution]\label{defn:chardistr}
Let $(\pi, V_\pi)$ be an admissible, unitarizable representation of a locally compact group $G$, and let $K$ be a large subgroup of $G$. For all irreducible unitary representations $\rho$ of $K$, we define the character distribution
\[  \mH(G, \rho) \rightarrow \bC, \qquad \phi \mapsto \tr \pi(\phi).\]
\end{defn}
In principle, the definition applies to $\mH(G, \rho_1, \rho_2)$ as well, but the trace is always zero if $\rho_1 \not\cong \rho_2$. The character distributions for representations on non-Hilbert spaces can be defined more generally, but this would require the introduction of an extensive amount of material.  

The content of the following observation directly follows from the above Definition and Theorem~\ref{thm:smoothrep}.
\begin{corollary}[Vanishing character distribution]\label{cor:charvanish}
Let $(\pi, V_\pi)$ be an admissible, unitary representation of a locally compact group $G$, let $K$ be a large subgroup of $G$, and $\rho$ be a finite-dimensional, unitary representation of $K$. 

The character distribution \( \Theta_{\pi} \) vanishes on $\mH(G, \check{\rho})$ if $\rho$ is not contained in $\Res_K V_\pi$.
\end{corollary}
\newpage 
\begin{examples}\mbox{}
\begin{itemize}
 \item Each irreducible, unitarizable representation of a locally compact abelian group is one-dimensional, smooth, admissible, and trace class. The irreducible representations thus form a group, denoted by $\widehat{A}$, which is locally compact in the compact-open topology. For a one-dimensional representation $\chi$, the character distribution of $\chi$ is the Fourier transform at $\chi$, i.e., for $\phi \in \Ccinf(A)$
 \[ \tr \chi(\phi) = \int\limits_{A}      \chi(a) \phi(a) \d a .\]
 \item Each irreducible representation of a compact group is finite-dimensional, smooth, unitarizable, and trace class. This is the Peter-Weyl Theorem. 
 \item Every smooth, admissible representation of an algebraic reductive Lie group over a local field is trace class. It is understood in many cases that the distribution $\phi \in \Ccinf(G) \rightarrow \tr \pi(\phi)$ determines the representation $\pi$ up to Naimark equivalence \cite{BernsteinZelevinskii:GLN}*{Corollary 2.20, page 20},  \cite{Knapp:Semi}*{Theorem 10.6, page 336}.
 \item Let $G$ be a locally compact group with a closed, co-compact subgroup $H$, i.e., $H \backslash G$ is a compact space. Then $G$ is unimodular as well, and there exists a right invariant measure $\d \nu$ on $H \backslash G$, which is unique up to a constant. The right regular representation on $\mL^2(H \backslash G, \d \nu)$ is a unitary, trace class representation \cite{DeEc}*{Theorem 9.2.2, page 175}.
 \item An example of a representation, which is not trace class, is the right regular or left regular representation on every non-compact group, which is unitary, but not trace class.
\item Another counterexample --- more in the spirit of the main focus of this thesis --- is the right regular representation of $\SL_2(\bR)$ or $\SL_2(\bA_\mathbb{Q})$ on $\mL^2(\SL_2(\bZ) \backslash\SL_2(\bR))$ or $\mL^2(\SL_2(\mathbb{Q}) \backslash \SL_2(\bA_\mathbb{Q}))$ respectively. 
\end{itemize}
\end{examples}

\section{Hilbert space and unitary representation}

\begin{defn}
Let $G$ be a second-countable, locally compact group with a large compact subgroup $K$. 
\begin{itemize}
\item A \textbf{$K$-unitary representation} of $G$ is a strongly continuous representation of $G$ on a separable Hilbert space, such that the restriction to $K$ is unitary. 
\item A \textbf{subquotient} of a $K$-unitary representation is the image of a $G$-equivariant projection onto a $K$-isotype. 
\item Two $K$-unitary representations $(\pi_1,V_1)$ and $(\pi_2,V_2)$ are \textbf{(Naimark) equivalent} if there exists a closed (possibly unbounded), injective operator $T:V_1 \rightarrow V_2$ with closed range and dense image, and $\pi_2(g) T = T \pi_1(g)$ for all elements $g\in G$.
\end{itemize}
\end{defn}

\begin{thm}
Let $G$ be a second-countable, locally compact group with a large compact subgroup $K$. Consider two $K$-unitary, irreducible representations $\pi_1$ and $\pi_2$. The following are equivalent:
\begin{enumerate}
 \item the representations $\pi_1$ and $\pi_2$ are equivalent
 \item the representations $\pi_1$ and $\pi_2$ have a common matrix coefficient
\item if both $\pi_1$ and $\pi_2$ have a common $\rho$-isotype with multiplicity one, the character distributions of $\pi_1$ and $\pi_2$ coincide on $\mH(G, \rho)$ 
\end{enumerate}
\end{thm}
\begin{proof}
Look at \cite{Koornwinder:SL2C}. The equivalence of the first two statements is given as Lemma 1.4 on page 409. The equivalence between point one and three is Theorem 1.5 on page 410. 
\end{proof}

\begin{thm}[\cite{Rao}]
Let $\F$ be a local field. Let $G =\GL_2(\F)$ and $K$ be $\GL_2(\o)$ / $\U(2)$ / $\O(2)$ for $\F$ non-archimedean / complex / real. 
Every smooth, admissible representation can be embedded densely into a $K$-unitary representation, and every $K$-isotype has at most multiplicity one.
 \end{thm}
For what follows, we summarize some of classical results about unitary representations of type $I$ groups. Definitions, proofs and references can be found in \cite{Barut}*{Chapter 5} and \cite{DeEc}*{Chapter 8}.
\begin{example}
Let $\F$ be a local field. The group $\GL_2(\F)$ is a unimodular, locally compact, separable group of type $I$. \end{example} 
\begin{theorem}
Let $G$ be a separable, locally compact group. Every unitary representation $(\pi, V)$ on a separable Hilbert space admits a direct integral decomposition
\[      (\pi,V) = \int_X^\oplus (\pi_x,V_x) \d \mu(x) \]
into irreducible representation $(\pi_x, V_x)_{x \in X}$ for some measure space $(X, \mu)$. The decomposition is unique if the representation is type $I$.
\end{theorem}
\begin{theorem}[Abstract Plancherel formula]
Let $G$ be a unimodular, locally compact, separable group of type $I$. Let $\widehat{G}$ denote the set of irreducible, unitary representations which are contained in the right regular representation of $G$ on $\mL^2(G)$.
With the Fell topology, the set $\widehat{G}$ becomes a local Hausdorff space.
Then there exists a unique positive Radon measure $\d_{\textup{Pl}} \pi$ on $\widehat{G}$, also known as the Plancherel measure, such that for all $\phi \in \Ccinf(G)$
\[ \phi(1) = \int\limits_{\widehat{G}} \tr \pi(\phi)   \d_{\textup{Pl}} \pi.\]
\end{theorem}
\begin{proof}
The theorem is proven in \cite{Dixmier:Cstar-algebres}*{18.8.1} in terms of the Hilbert-Schmidt norm
\[ || \phi ||_{HS}  = \int\limits_{\widehat{G}}  || \pi(\phi) ||_{HS}  \d_{\textup{Pl}} \pi.\]
The Dixmier-Malliavin Theorem and the polar decomposition generalize this to the trace. 
\end{proof}
\begin{example}
Let $\F$ be a local field. The unitary one-dimensional and the complementary principal series representations of $\GL_2(\F)$ are unitarizable, but not contained in the right regular representation. Explicit computations are given in Sections~\ref{section:realuni} and~\ref{section:complexuni}.
\end{example}

\section{Abstract parabolic inductions}\label{section:jacquet}

\begin{defn}[The parabolic induction]
Let $G$ be a locally compact group with a unimodular Iwasawa datum $(M,N, K)$, let $B = N \rtimes M$. Let $(\pi, V_\pi)$ be a smooth representation of $M$.

Define the Jacquet-module$\mJ_B^G(\pi)$ of $\pi$ as the right regular representation on the space of smooth functions
 \[  f: G \rightarrow V_\pi  ,\]
such that  for all $m \in M, n \in N, k \in K$
\[  f(mnk) = \pi(m) \sqrt{\Delta_B(m)} f(k).\]
 The functions are necessarily bounded, since $K$ is compact.
\end{defn}
\begin{lemma}
If $( \pi , V_\pi)$ is a unitarizable representation of $M$, then $\mJ_B^G(\pi)$ is a unitarizable representation of $G$. 
\end{lemma}
However, there do exist unitarizable parabolic inductions which do not come from a unitarizable representation of $M$, such as the complementary series representations of $\GL_2(\bR)$ or $\GL_2(\mathbb{Q}_p)$.
 The importance of the parabolic induction in the representation theory of  real reductive Lie groups can be understood from the following theorem:
\begin{theorem}[Casselman submodule theorem \cite{Casselman-Osborne:Restriction_admissible}, \cite{Harish:SS2}]
Let $G$ be an algebraic linear reductive group over $\bR$. Then every irreducible representation is a subquotient of a Jacquet-module associated to a parabolic subgroup.
\end{theorem}
For non-archimedean fields, the above statement does not hold. There are also the super-cuspidal representations, which were first discovered by Mautner \cite{Mautner:Spherical2}.

Now let us emphasize the importance of the Abel transform in the context of the parabolic induction: 
\begin{thm}\label{thm:jacquettrace}
Under the assumptions and in the notation of Definition\textup{~\ref{defn:abel}}, let $\pi$ be an irreducible, $K$-unitarizable, admissible representation. 

We have for $\phi \in \mH(G, \rho)$ the following formula:
\[ \tr \mJ_B^G \pi ( \phi) =  \tr \pi( \mA_\rho  \phi).\]
\end{thm}
The proof strategy is the same as for the Frobenius character formula..
 \begin{lemma}
For each orthonormal basis $X$ of $V_\pi$, we have a formula
\[ \tr\mJ_B^G \pi( \phi)    = \int\limits_K \sum\limits_{v, w\in X} \int\limits_B \langle v, \phi(k^{-1}bk)  \pi\Delta^{1/2}(b) w \rangle \d b = \tr \mJ_B^G \pi(\phi^K).\] 
 \end{lemma}
\begin{proof} 
Endow $K$ with a normalized Haar measure, and $B = MN$ with the unique left invariant Haar measure $d_l b$, such that for each $f \in \Ccinf(G)$, see \cite{DeEc}*{Proposition 1.5.5}:
\[ \int\limits_{G} f(g) \d g =  \int\limits_{B} \int\limits_{K}  f(bk) \d_l b \d k.\]
Let $\phi \in \Ccinf(G)$.  Consider
\begin{align*}
\mJ_B^G \pi(\phi) f(x) &= \int\limits_G \phi(g) \mJ_B^G \pi(g) f( x) \d g \\
 &\underset{g=bk}= \int\limits_K \int\limits_B  \phi(x^{-1}bk) f( bk) \d k \d_l b \\ 
  &=  \int\limits_K  \int\limits_B \phi(x^{-1}bk)  \pi\Delta^{1/2}(b) \d_l b f( k) \d k.
\end{align*}
The operator $\pi(\phi)$ is thus a kernel transformation $\mL^2(K) \underset{B \cap K}\otimes V_\pi \rightarrow \mL^2(K)  \underset{B \cap K}\otimes V_\pi$, with kernel
\[ \mK_\phi(x,k) =   \int\limits_B \phi(x^{-1}bk)  \pi\Delta^{1/2}(b) \d_r b.\]
Since $\pi$ is a trace class representation, we have that for each orthonormal basis $X$ of $V_\pi$ the sum
\[ \tr \pi( \phi(x^{-1} \blank k)|_B = \sum\limits_{v \in X} \int\limits_B \langle v, \phi(x^{-1}bk)  \pi\Delta^{1/2}(b) v \rangle \d_r b < \infty \]
converges absolutely. Every trace-class operator on $V_\pi$ is a Hilbert-Schmidt operator, and the series is
\[  \sum\limits_{v \in X, w\in X} \int\limits_B \langle v, \phi(x^{-1}bk)  \pi\Delta^{1/2}(b) w \rangle < \infty  \]
absolute convergent.

The restriction of $V_\pi$ to $K$ is a unitary representation of $K$.  
Choose a $K$-equivariant orthonormal basis, in the sense that any vector sits exactly in one $\rho$-isotypic component. 
Since $K$ is compact and the basis is $K$-invariant, the function
\[ \phi_{S}(x,k) = \sum\limits_{(v, w) \in S} \int\limits_B \langle v, \phi(x^{-1}bk)  \pi\Delta^{1/2}(b) w \rangle \d b \]
indexed by finite subset $S \subset X \times X$ converges uniformly to 
\[ \sum\limits_{v, w\in X} \int\limits_B \langle v, \phi(x^{-1}bk)  \pi\Delta^{1/2}(b) w \rangle \d b.\] 
The function
\[ (x, k) \mapsto    \sum\limits_{v \in X, w\in X} \int\limits_B \langle v, \phi(x^{-1}bk)  \pi(b) w \rangle \d b \]
is thus continuous, and in particular, square integrable over $K \times K$. The operator $\pi(\phi)$ is therefore a Hilbert-Schmidt operator, since the kernel is $\mL^2(K \times K)$ integrable with the Hilbert-Schmidt norm being
\[ || \pi(\phi)||_{HS} =\int\limits_K \int\limits_K \sum\limits_{v \in X, w \in X} \int\limits_B \langle v, \phi(x^{-1}bk)  \pi\Delta^{1/2}(b) v \rangle \d x \d k.\]
Note that $\pi$ is a $*$-algebra homomorphism
\[ \pi(\phi_1) \circ\pi(\phi_2) = \pi(\phi_1 \ast \phi_2), \qquad \pi( \phi^*) = \pi(\phi)^\dagger.\]
By the Dixmier-Malliavin Theorem, we know that
\[ \phi =  \sum\limits_{j} \phi_{1,j} \ast \phi_{2,j}, \qquad  \pi( \phi)    = \sum\limits_{j} \pi(\phi_{1,j}) \circ  \pi(\phi_{2,j})\]
is a finite sum of products of Hilbert-Schmidt operators, hence a trace class operator. By the polarization identity, we can write $\pi(\phi_{1,j}) \circ  \pi(\phi_{2,j})$ as a sum of four positive, trace class operators
\begin{align*}  4 \left( g \ast h^* \right)  = &   ( g + h) \ast ( g + h)^* -  ( g - h) \ast ( g - h)^*   \\& - \im  ( g -\im h) \ast ( g - \im h)^* + \im  ( g + \im  h) \ast ( g + \im h)^* .\end{align*}
This means we only have to verify the formula for an element of the form $\phi \ast \phi^*$ because the trace is a linear functional. We thus have that
 \[ \tr \pi(\phi \ast \phi^*) = || \pi(\phi)  ||_{HS}.\]
The following rules apply to the kernel functions:
\begin{align}\label{eq:kernelrelations}    \mK_{\phi^*} (x,k) = \overline{\mK_{\phi} (k,x)} , \qquad          \mK_{\phi \ast \phi_0} (x,k) = \int\limits_K \mK_{\phi} (x,y)  \mK_{\phi} (y,k) \d y.\end{align}
We thus have that
\begin{align*} \tr \pi(\phi \ast \phi^*) & =  \left\| \pi(\phi)  \right\|_{HS, V_{\mJ \pi}} \\
& =  \int\limits_K  \int\limits_K \left\| \mK_{\phi} (u,k) \right\|_{HS, V_\pi}^2 \d u \d k \\ 
& =\sum   \int\limits_K  \int\limits_K  \langle v, \overline{\mK_{\phi} (u,k)} \mK_{\phi} (u,k) v \rangle  \d u \d k. 
\end{align*}
Removing and adding the outer brackets yields the proof via the relation~\ref{eq:kernelrelations}:
\begin{align*}\int\limits_K  \int\limits_K   \overline{\mK_{\phi} (u,k)} \mK_{\phi} (u,k)  \d u \d k& =           \int\limits_K  \int\limits_K  \mK_{\phi^*} (k,u) \mK_{\phi} (u,k)  \d u \d k \\ 
                                                                                                                                    & =       \int\limits_K   \mK_{\phi^* \ast \phi} (k,k)   \d k.    \qedhere
\end{align*}
\end{proof}

\begin{proof}[Proof of the theorem]
By the quotient integral formula and the decomposition $B=MN$, we have for $f \in \mL^1(B)$
\[ \int\limits_{B} f(b) \d  b = \int\limits_M \int\limits_N f(mn) \d m \d n.\]
Since $\pi \Delta^{1/2}$ is trivial on $N$, we obtain
\begin{align*}\int\limits_K \sum\limits_{v, w\in X} \int\limits_B \langle v, \phi(k^{-1}bk)  \pi\Delta^{1/2}(b) w \rangle \d b &=  \sum\limits_{v, w\in X} \int\limits_B \langle v, \mA \phi^K(k^{-1}mk)  \pi(m) w \rangle \d b \\&\underset{def.}= \tr \pi( \mA(\phi)). \qedhere\end{align*}
\end{proof}

\section{The compact induction}
\begin{defn} \index{$\Ccinf(G, \chi)$}
Let $G$ be a locally compact group with center $Z$. Let $\chi$ be a one-dimensional representation of $Z$. Define $\Ccinf(G, \chi)$ as the space of smooth functions $G \rightarrow \bC$, which are compactly supported modulo the center with
\[   f(zg) = \chi(z) f(g), \qquad \textup{for all }g\in G, z \in Z.\]
\end{defn}
\begin{lemma}
The right and the left regular representation of $G$ on $\Ccinf(G, \chi)$
\begin{align*} l(g) : \phi \longmapsto \left( x \mapsto \Delta_G(g^{-1}) \phi(g^{-1}x) \right),
               r(g) : \phi \longmapsto \left( x \mapsto \phi(xg) \right)
\end{align*}
are unitarizable if and only if $\chi$ is unitary.
\end{lemma}

\begin{defn}[Compact-mod-center]\index{compact-mod-center subgroup}
 A closed subgroup $K$ in a locally compact group $G$ is called a compact-mod-center subgroup 
\begin{enumerate}
 \item if it contains the center of $G$, and
 \item if it is compact modulo the center of $G$.
\end{enumerate}
\end{defn}

\begin{defn}[Compact induction]\index{$\ind_{\underline{K}}^G \rho$}
Let $G$ be a locally compact group, and let $\underline{K}$ be a compact-mod-center subgroup. Let $\rho$ be a finite-dimensional representation of $\underline{K}$. The representation $\ind_{\underline{K}}^G \rho$ denotes then the right regular representation on the smooth functions $G \rightarrow V_\rho$, which are compactly supported modulo the center and satisfy
\[ f(kg) = \rho(k) f(g) \qquad \textup{ for all }k\in \underline{K}, g \in G. \]
\end{defn}
A smooth, admissible representation is called supercuspidal if its matrix coefficients are compactly supported modulo the center.

\begin{example}[\cite{BernsteinZelevinskii:GLN}, \cite{BushnellHenniart:GL2}]
 Let $\F$ be a non-archimedean field. Then every irreducible supercuspidal representation of $\GL_n(\F)$ is isomorphic to the compact induction of an irreducible representation from a compact-mod-center subgroup. All irreducible unitary representations of $\GL_n(\F)$, which are not isomorphic to subquotients of parabolic inductions associated to parabolic subgroups, are supercuspidal representations. There are precisely $n$ conjugacy classes of maximal compact-mod-center subgroups in $\GL_n(\F)$.
\end{example}

\begin{theorem}[\cite{Bushnell:Induced}*{Theorem 1, page 107}]
Let $G$ be a locally profinite, unimodular group and let $K$ be a compact-mod-center subgroup of $G$. Let $\rho$ be an irreducible representation of $K$.  The following assertions are equivalent:
\begin{enumerate}
 \item the representation $\Ind_K^G \rho$ is admissibile;
 \item the representation $\ind_K^G \rho$ is admissibile;
\item there is an isomorphism of representations $\ind_K^G \rho  \cong \Ind_K^G \rho$;
 \item the representation $\Ind_K^G \rho$  decomposes into a finite sum of irreducible supercuspidal representations.
\end{enumerate}
\end{theorem}

\begin{lemma}
The representation $\ind_{\underline{K}}^G \rho$ is unitarizable if and only if the central character of $\rho$ is unitarizable.
\end{lemma}
We compute the character distribution of the compact induction along the same lines as the Frobenius character formula. The first formula can be found in \cite{Sally:CharacterSL2}, \cite{Kutzko:CharacterGLL}*{page 201}.
We will call it the Iwasawa-Frobenius character formula, since it relies on the structure of the Iwasawa decomposition.
\begin{theorem}[The Iwasawa-Frobenius character formula]\label{thm:frobiwa}
Let $G$ be a unimodular locally compact group, $B$ a closed subgroup, $Z$ the center of $G$, and $K$ a compact-mod-center subgroup with $B \cdot K = G$. Normalize the Haar measures in such a fashion that for all $f \in \Ccinf(G)$:
 \[ \int\limits_{G} f(g) \d g = \int\limits_{(Z \cap  B) \backslash B} \int\limits_{K} f(kb) \d k \d_r b.\]
Let $\rho$ be a unitary irreducible representation of $K$ such that $\pi =  \Ind_{K}^G \rho$ is an admissible, trace class representation.

Its character distribution satisfies  for all $\phi \in \Ccinf(G)$:
\[  \tr \pi ( \phi )  = \int\limits_{(Z \cap B) \backslash B} \int\limits_K \phi(b^{-1} k b) \tr \rho( k) \d k \d_r b .\]
\end{theorem}
\begin{proof}
The normalization of the Haar measure is given by \cite{DeEc}*{Proposition 1.5.5, page 25}. 
As an induced representation of a unitary representation, the representation $\pi$ is unitarizable. 
Let us consider an arbitrary element $f \in \Ind_{K}^{G} \rho$ with
\[ f(kg) = \rho(k) f(g).\]
 Set $\tilde{B} = (Z \cap B) \backslash B$. For all $\phi \in \Ccinf(G)$, we compute
\begin{align*}
\pi(\phi) f(x)& = \int\limits_{G} \phi(g)f(xg) \d g \\
                & =  \int\limits_{G} \phi(x^{-1}g)f(g) \d g \\
                 & = \int\limits_{K} \int\limits_{\tilde{B}}  \phi(x^{-1} k b) f(kb)   \d k  \d_r b\\ 
                 &   = \int\limits_{K} \int\limits_{\tilde{B}}  \phi(x^{-1} k b) \rho(k) \d k f(b)     \d_r b.
\end{align*}
Therefore, the operator $\pi(\phi)$ is a kernel transformation $\mL^2(\tilde{B}) \otimes V_\rho \rightarrow \mL^2(\tilde{B}) \otimes V_\rho$ with kernel
\[ K_\phi(x, b) = \int\limits_{K}  \phi(x^{-1} k b) \rho(k) \d k.\]
The claim of the above theorem follows if we can argue that
\begin{align}\label{eq:traceid} \tr \pi(\phi)    =  \int\limits_{\tilde{B}} \tr_{\rho} K_\phi(x, x) \d x.\end{align}
We write according to the Dixmier-Malliavin Theorem
\[ \phi =\sum\limits_{j=1}^R \phi_{1,j}  \ast \phi_{2,j}.\]
 We know that 
 \[ K_\phi(x, b) = \int\limits_{K}  \phi(x^{-1} k b) \rho(k) \d k.\]
 The operator $\pi(\phi)$ is furthermore a linear composition of self adjoint, positive trace class operators via the polarization identity
 \begin{align}\phi_1 \ast \phi_2  =\frac{1}{4} \left(   ( \phi_1 +\phi_2)^{*2}   - ( \phi_1 - \phi_2)^{*2}  - \im (\phi_1 -\im \phi_2)^{*2}+ \im (\phi_1 + \im \phi_2)^{*2} \right),\end{align} 
 where we use the notation $f^{*2} = f \ast f$. The trace for such an operator is computed simply by an integral over the diagonal of $\tilde{B} \times \tilde{B}$ \cite{DeEc}*{Lemma 9.3.1, page 180}, and 
Equation~\ref{eq:traceid} is justified.
\end{proof}
We state and prove the next theorem for locally profinite groups only, but an analogue holds also in full generality. I am only aware of interesting applications of the above theorem in the context of reductive groups over non-archimedean fields, since the compact inductions from compact-mod-center subgroups in linear reductive Lie groups are never admissible.

The next theorem includes two distinct compact-mod-center subgroups. Situations may arise, where it may be necessary to switch the compact-mod-center subgroup. For example, the group $\GL_2(\mathbb{Q}_p)$ contains two distinct conjugation classes of compact-mod-center subgroups. 
\begin{theorem}\label{thm:frobdc}
Let $G$ be a locally profinite group, $B$ a closed subgroup, and let $K$ and $K'$ be open compact-mod-center subgroupsa with $B \cdot K =G$, $B \cdot K' =G,$ and $K \cap K'$ open in $K'$. Let $Z$ be the center of $G$, and let $Z$ be contained in $B$. We fix the unique right Haar measure, in such a fashion that $B/Z$, $K$ and $G/Z$ carry right Haar measures with
\begin{align}\label{eq:measbzk} \int\limits_{G}  f(g) \d g =\int\limits_{B/Z}  \int\limits_{K} f(kb)\d k \d_r b =  \int\limits_{G/Z} \int\limits_{Z} f(z\dot{g}) \d z \d \dot{g}.\end{align} 
Additionally, there is a unique constant $C_{K | K'} >0$, such that
\[ \int\limits_{G}  f(g) \d g =C_{K | K'} \int\limits_{B/Z}  \int\limits_{K'} f(kb)\d k \d_r b \]
 Let $\rho$ be a unitary finite-dimensional representation of $K$, such that 
\[ \pi =  \Ind_{K}^{G} \rho \] 
is an admissible, trace class representation.
For all $\phi \in \Ccinf(G)$, the following formula is valid:
\[   \tr \pi ( \phi )  = \sum\limits_{x \in (K \cap K') G/K'} \delta_{K|K'}(x)\int\limits_{K}  \phi^{K'} (x^{-1} k x) \tr \rho(k) \d k,\]
where we have defined
\[   \phi^{K'} (x) = \int\limits_{K'/Z} \phi((k')^{-1} x k') \d k',\]
and $H_B$ is the extension of the modular character $\Delta_B$ of $B$ to $G$ by
\[ H_B(bk) = \Delta_B(b), \qquad \textup{ for all } k\in K', \; b \in B\]
and 
\[ \delta_{K' | K }(x) = C_{K | K'} \vol_{G/Z}((K \cap K') x K) \int\limits_{K \cap K' /Z } H_B(k_2x)  \d k_2. \]
\end{theorem}
Before proving the theorem, we need a short lemma:
 \begin{lemma}[Measures on double cosets]\label{lemma:measdc}
 Let $G$ be a locally profinite group, and let $I_1, I_2$ be two open compact  subgroups with normalized Haar measure, then we have
\[ \int\limits_{G} \phi(g) \d g = \sum\limits_{w \in I_1 \backslash G / I_2} \vol_G(I_1x I_2)  \int\limits_{I_1} \int\limits_{I_2} \phi(i_1 x i_2) \d i_1 \d i_2. \]
\end{lemma}
An analogous statement holds in greater generality \cite{Liu}.
\begin{proof}
The space $ I_1 \backslash G / I_2$ is discrete, and the coset $I_1xI_2$ is open. The identity is true for every characteristic function of a coset. We also have for all $\phi \in \Ccinf(G)$ that
\[ \int\limits_{G} \phi(g) \d g =  \int\limits_{G}    \int\limits_{I_1}  \int\limits_{I_2} \phi(i_1 g i_2) \d i_1 \d i_2 \d g.\]
The function
\[ g \mapsto  \int\limits_{I_1}  \int\limits_{I_2} \phi(i_1 g i_2) \d i_1 \d i_2 \]
is compactly supported, constant on $I_1xI_2$-cosets, and can be written as a finite linear combination of characteristic functions of cosets.
\end{proof}
\begin{proof}[Proof of the theorem]
The function $H_B$ is well-defined since $B\cap K$ is open compact-mod-center and therefore $\Delta_B |_{B \cap K} = 1$. We can rely on Theorem~\ref{thm:frobiwa} and write
\[ \tr \pi ( \phi )  = \int\limits_{Z \backslash B} \int\limits_K \phi(b^{-1} k b) \tr \rho( k) \d k \d_r b. \] 
 Since $\tr \pi( \phi^{K'}) =\tr \pi(\phi)$, we may write
\[ \tr \pi ( \phi )  = \int\limits_{Z\backslash B} \int\limits_K \int\limits_{K'/Z} \phi (k_0^{-1} b^{-1} k b k_0) \d k_0 \tr \rho( k) \d k \d_r b . \]
$B\cap K/Z$ and $B\cap K'/Z$ are open compact in $B/Z$. The suggested normalization (~\ref{eq:measbzk}) yields that $B\cap K/Z$ has measure one in $B/Z$:  
  \[  \int\limits_{G/Z} f(g) H_B(g) \d g =  \int\limits_{B/Z} \int\limits_{K'/Z} f(kb) \d k \d_r' b.\]
 The Iwasawa decomposition yields that there exists a unique left Haar measure on $B$, such that
\[ \int\limits_{G/Z} f(g) H_B(g) \d g =  \int\limits_{B/Z} \int\limits_{K'/Z} f(bk) \Delta_B(b)  \d k \d_l' b =  \int\limits_{B/Z} \int\limits_{K'/Z} f(bk) \d k \d_r' b.\]
Now let us compare the measures $\d_r' b$ and $\d_r b$. By the uniqueness of right Haar measures, we have a constant $C_{K|K'}$, such that 
\[ \int\limits_{B} f(kb) \d_r' b =  C_{K|K'} \cdot \int\limits_{B} f(b) \d_r b.\]
So we have that 
\[ \tr \pi ( \phi )  = \int\limits_{Z\backslash G} \int\limits_K \phi (g^{-1} k g)   \tr \rho( k) \d k H_B(g) \d g. \]
Now, we appeal to Lemma~\ref{lemma:measdc} for  $C_x =    \vol_{G/Z}((K \cap K') x K) $:
\begin{align*}&  \tr \pi ( \phi )  \\
&=  \sum\limits_{x \in (K \cap K'/Z)\backslash(G/Z)/(K'/Z)} C_{K|K'} C_x \cdot  \\
& \qquad \qquad \int\limits_{Z\backslash K'} \int\limits_{Z\backslash K' \cap K}  \int\limits_K  \phi(k_1^{-1} x^{-1}k_2^{-1} k k_2 x k_1)   \tr \rho( k) \d k H_B(k_2 x k_1)  \d k_1 \d k_2  \\ 
                                        &= \sum\limits_{x \in (K \cap K')\backslash G/K'} \int\limits_K  \phi^{K'}(x^{-1}kx)   \tr \rho( k) \d k \cdot C_{K|K'} C_x  \int\limits_{K \cap K' /Z } H_B(k_2x)  \d k_2.  \qedhere
\end{align*}
\end{proof}
\begin{example}[$F$-reductive groups]
In the case of a reductive group $G$ of a non-archimedean field, we have a formula for the measure associated to $G//I$, where $I$ is the Iwahori subgroup possibly enlarged by the center.
 The coset space $G//I$ carries the structure of a Coxeter group. Relevant computations for the measures can be found in \cite{IwahoriMatsumoto:Bruhat}*{Proposition 3.2, page 44}.
The Iwahori subgroup is always contained as a co-finite subgroup modulo the center in at least one representative of each conjugacy class of maximal compact-mod-center subgroups.
 \end{example}
We give an example, how to exploit this formula for a computation.
\begin{corollary}[Existence of pseudo-matrix coefficients]\label{cor:superpseudo}
In the notation and under the assumptions of Theorem\textup{~\ref{thm:frobdc}}, pick any compact group $K_0$ and set $K = K_0 Z = K'$. Assume furthermore that $\pi = \Ind_{KZ}^{G} \rho$ is irreducible for an irreducible representation $\rho$ of $KZ$.
Consider
\[ \phi(x) = \begin{cases}  \tr \rho(x^{-1}) , & x \in K_0, \\ 0, & \textup{otherwise}.\end{cases}\]
For every irreducible, trace-class representation $\pi_0$ of $G$, we have that \( \pi_0(\phi)=0 ,\) unless $\pi_0 \cong \ind_K^G \rho \cdot \chi$ for some unitary character $ \chi : Z \rightarrow \bC^1$ with $\Res_{Z \cap K} \rho = \chi$, and in the latter case
\[ \pi_0(\phi) = 1.\]  
\end{corollary}
\begin{proof}
By admissibility, we have that $\ind_K^G \rho \cong \Ind_K^G \rho$ , see \cite{Bushnell:Induced}*{Theorem 1, page 107}. Schur's Lemma implies that
\[ \Hom_K ( \rho, \Res_K \pi_0) =0,\]
unless $\pi$ is contained in $\Ind_{K}^G \rho$. Induction by steps in the form
\[ \Ind_{K}^G \rho \cong \Ind_{Z K}^G \Ind_K^{Z K} \rho,\]
and the pasting lemma in the form
\[ \Ind_K^{Z K} \rho = \int\limits_{\widehat{Z/(Z \cap K)}}^\oplus \rho \otimes \chi \d \chi,\] 
 yield that this is only the case if $\pi_0 \cong \Ind_{K}^G \rho \otimes \chi$. Frobenius reciprocity implies 
\begin{align*}\bC &\cong \Endo_G(  \Ind_{Z K}^G \rho \otimes \chi) \\
  & \cong \Hom_{ZK}  ( \rho \otimes \chi,  \Res_{K Z} \Ind_{ZK}^G \rho) 
\cong \Hom_{K}  ( \rho ,  \Res_{K} \Ind_{ZK}^G \rho) .
\end{align*}
Mackey's restriction-induction formula yields
\[  \Res_K \Ind_K^G \rho \cong \bigoplus_{x \in G//K} \Ind_{K \cap K^x}^K \rho.\]
Applying Frobenius reciprocity again results in
\begin{align*}  \Hom_{K}  ( \rho,  \Res_K \Ind_{ZK}^G \rho \otimes \chi) 
&\cong \bigoplus_{x \in G//K}  \Hom_{K}  ( \rho,   \rho^x) \\
&  \cong      \bigoplus_{x \in G//K}  \Hom_{K \cap K^x}  ( \rho,  \rho^x).\end{align*}
Since the dimension of this vector space is zero for $x \notin K$, only the component corresponding to the coset of the unit does not vanish. Let us turn our attention to the character distribution. The integral is given as
\begin{align*}  \int\limits_K  \overline{\tr \rho(x^{-1}kx)}   \tr \rho( k) \d k
& =   \frac{1}{C_x} \int\limits_{ K \cap K^x} \overline{ \tr \rho^x(k)}   \tr \rho( k) \d k  \\ &   =\frac{1}{C_x}   \dim    \Hom_{K \cap K^x}  ( \rho,  \rho^x). \end{align*}
Moreover $C_{K|K}=1$ by definition plus the integral
\[ \int\limits_{K \cap K' /Z } H_B(k_2x)  \d k_2 =1\]
for $x \in K$.
\end{proof}

\section{The theory of Gelfand pairs}\label{section:gelfand}
\begin{theorem}[Gelfand principle]
The following statements are equivalent:
\begin{itemize}
\item the algebra $\mmH(G, \rho)$ is commutative;
\item the representation $\rho$ occurs in the restriction of every irreducible, unitary, admissible representation of $G$ at most with single multiplicity;
\item either of these facts are true for the contragedient $\check{\rho}$.
\end{itemize}
\end{theorem}
\begin{proof}
The algebra $\mmH(G, \rho)$ is commutative if and only if $\mmH(G, \check{\rho})$ is, since we have an anti-algebra homomorphism
\[        \mmH(G, \rho) \xrightarrow{\cong}  \mmH(G,\check{ \rho}), \qquad \phi \longmapsto \left( x \mapsto \phi(x^{-1}) \Delta_G(x)\right).\]
Similar to the Schur isomorphism, we now have a dense embedding
\[ \mH(G, \rho) \hookrightarrow \Endo_K( \Ccinf(G, \chi)^\rho ), \qquad \phi \mapsto \pi(\phi).\]
The algebra \( \Endo_K( \Ccinf(G, \chi)^\rho ) \) is commutative if and only if $\rho$ has single multiplicity in every irreducible representation in $\Ccinf(G)$. Otherwise, there would exist two non-commuting projections.
\end{proof}

\begin{defn}[Gelfand pairs] \mbox{}
\begin{itemize}
\item We say that $(G, K)$ is a Gelfand pair if $\Ccinf(G//K)$ is abelian.
\item We say that $(G,K, \rho)$ is a Gelfand triple if $\mmH(G, \rho)$ is abelian.
\item We say that $(G,K)$ is a strong Gelfand pair if $(G,K, \rho)$ is a Gelfand triple for all irreducible representations $\rho$ of $K$.
\end{itemize}
\end{defn}
\begin{example}  \mbox{}
 The pairs $(\GL_2(\bC), \U(2))$, $(\SL_2(\bC), \SU(2))$, $( \GL_2(\bR), \O(2))$, and \\
$(\SL_2(\bR), \SO(2))$ are strong Gelfand pairs. Let $\F$ be a non-archimedean field with ring of integers $\o$, then $(\GL_2(\F), \GL_2(\o))$ is a strong Gelfand pair \cite{Silberger:SphericalPGL2}.
%\item The pair $(\GL_n(\bC), \U(n))$ is a Gelfand pair, but not a strong Gelfand pair for $n\geq 4$.
%\item Every irreducible, unitary representation of $\GL_n(\bR)$, $\GL_n(\bC)$, and $\GL_n(\F)$ contains a $\rho$-isotype of the unique maximal compact group, such that $\mH(G, \rho)$ is abelian \cite{BushnellKutzko:GLNopen}, \cite{}. Consequently, every automorphic representation of $\GL(n)$ over a global field contains a $\rho$-isotype .
\end{example}

The next statement is taken from \cite{Godement:Spherical}, \cite{Harish-Chandra:Harmonic}*{Part II, page 12}, \cite{Warner1}. We say that an associative algebra is $n$-commutative if
\[ \sum\limits_{x \in S_n} \sign(x) a_{x(1)}  a_{x(2)}  \cdots a_{x(n)} =0\]
for every $n$-tuple of elements.
\begin{theorem}[Godement prinicple]
The following statements are equivalent:
\begin{itemize}
\item the algebra $\mH(G, \rho)$ is $n$-commutative;
\item the representation $\rho$ occurs in the restriction of every smooth, admissible representation of $G$ at most with multiplicity $n$;
\item either of these facts are true for the contragedient $\check{\rho}$.
\end{itemize}
\end{theorem}

The commutativity is verified often by the ``Gelfand-Kahzdan trick.'' The following theorem is inspired by, and certainly implies \cite{Lang:SL2}*{Theorem 1, page 21}.
\begin{proposition}[Gelfand-Kahzdan trick]\label{theorem:gelfandcomm} \mbox{}
Let $G$ be a unimodular, locally compact group, and let $K$ be a compact subgroup, such that there are two homeomorphisms
\[ \sigma, \tau : G \mapsto G, \]
  such that
\begin{enumerate}
 \item $\tau^n$ and $\sigma^m$ are the identity automorphism for some $n, m \in \bN$,
 \item $\tau$ is an anti-automorphism, i.e., $\tau(g_1 g_2) = \tau(g_2) \tau(g_1)$,
 \item $\sigma$ is an automorphism, i.e., $\sigma(g_1 g_2) = \sigma(g_1) \sigma(g_2)$, and
 \item for any element $g \in G$, there exists $k_g \in K$ such that $k_g \sigma(g) k_g^{-1} = \tau(g)$.
\end{enumerate}
The algebra $\Ccinf(G)^K$ of $K$-conjugation invariant functions is then commutative.
\end{proposition}
\begin{proof}
 Define the linear operators
\[ \Ccinf(G) \rightarrow \Ccinf(G), \qquad { f \mapsto f^\sigma, \atop f \mapsto f^\tau, } \qquad { f^\sigma(x) = f (\sigma(x)), \atop f^\tau(x) = f(\tau(x)).}\] 
We immediately note that (4) implies $f^\sigma = f^\tau$ for $f(k^{-1}xk) =f(x)$ for all $k\in K$ and $x \in G$.
It is sufficient to prove the following claim, since $f \mapsto f^\sigma$ is a $*$-isomorphism.
 \end{proof}
\begin{claim}
\( (f_1 \ast f_2)^\sigma = f_1^\sigma \ast f_2^\sigma, \qquad (f_1 \ast f_2)^\tau = f_2^\tau \ast f_1^\tau. \)
\end{claim}
\begin{proof}[Proof of claim]
At most, the (anti-)automorphism scales the Haar measure by a positive constant $C>0$ by the uniqueness of invariant measure, since the group is unimodular. 
Because the (anti-)automorphismS have finite order according to point (1), the constant satisfies $C^{nm} = 1$, i.e., $C=1$. The first equality is straightforward:
\begin{align*}
  f_1^\sigma \ast f_2^\sigma (x) &= \int\limits_G f_1(\sigma(xy^{-1})) f_2(\sigma(y)) \d y  \\
                                             & \underset{(3)}= \int\limits_G f_1(\sigma(x) y^{-1}) f_2(y) \d y =   f_1 \ast f_2( \sigma(x)).
\end{align*}
The second equality is more complicated:
\begin{align*}
  f_1^\tau \ast f_2^\tau (x) &= \int\limits_G f_1(\tau(xy^{-1})) f_2(\tau(y)) \d y      \\
                                             & \underset{(2)}= \int\limits_G f_1(y^{-1}\tau(x) ) f_2(y) \d y  \\
                                             & \underset{y \mapsto y^{-1}}= \int\limits_G  f_2(y^{-1}) f_1(y )\d y  \\
                                              & \underset{y \mapsto y \tau(x)^{-1}}= \int\limits_G  f_2(\tau(x)y^{-1}) f_1(y )\d y = f_2 \ast f_2( \tau(x)).\qedhere
\end{align*}
 \end{proof}

\section*{Remarks}

\begin{itemize}
 \item It would be worthwhile to reprove the decomposition of $\Ccinf(G)$ as $K$-bi-module for a closed subgroup $K$, which is compact only modulo the center $Z$. For such subgroups, an analogue of the Peter-Weyl Theorem holds if there exists a compact subgroups $K_0 \subset K$ with $ZK_0$ has finite index in $K$.
 \item The approach to the computation of the character distribution is somewhat different than what is usually done. The distribution of a smooth, admissible representation of a reductive group over local fields is a locally integrable function $\Theta_\pi: G \rightarrow \bC$ such that
\[ \Theta_\pi(\phi) = \int\limits_G \phi(g) \Theta_\pi(g) \d g.\]
Frequently, only the function $\Theta_\pi$ is computed, but it is more useful in our context to efficiently calculate the scalar $\Theta_\pi(\phi)$ for fixed $\phi \in \Ccinf(G)$ or $\phi \in \mH(G, \rho)$.
\end{itemize}

\chapter*{Part III --- Local harmonic analysis on $\GL(2)$}
In this part, we partially specialize the harmonic analysis developed in Part two to the locally compact group $\GL_2(\F)$ for a local field $\F$ (short: local field), i.e., either \cite{Weil:Basic}
\begin{itemize}
 \item the field $\bR$ of real numbers,
 \item the field $\bC$ of complex numbers, 
\item a non-archimedean field $\F$, either
\begin{itemize}
 \item a finite extension of the field $\mathbb{Q}_p$ of the $p$-adic rationals,  or
 \item a finite extension of the Laurent series in one-variable over a finite field. 
\end{itemize} 
\end{itemize}
 The treatment is to a large extent kept independent from the previous parts, despite the overlap.

We stressed the common ground for the theory in Part I. We will not stress the similarities between the harmonic analysis of the different groups $\GL_2(\bR)$, $\GL_2(\bC)$ and $\GL_2(\F)$ from this point forward, but it should be self-explanatory from the organization and enumeration of the material.

For $\GL_2(\bC)$, I have omitted a full treatment. I focus instead on bi-invariant harmonic analysis. This restricts the spectral analysis to automorphic forms with complex constituents which are unramified principal series representations. There are no classical references for integral identities. The difficulties to be overcome have a purely special function theoretic origin. Although this harmonic analysis seems less interesting, because of the absence of discrete series representations, its omission is a remaining gap in a complete treatment of the $\GL(2)$ trace formula from a computational point of view.  

\textbf{The main goal of this part is to show the computation of all local distributions and constants, which appear in the Arthur trace formula as given in \cite{JacquetLanglands}*{page 271ff.} for a global field on a special set of test functions.}

We will not describe the distributions explicitly here, but rather address them by their names only. The distributions occuring in the Arthur trace formula are 
\begin{enumerate}[font=\normalfont]
 \item the spectral distributions: 
the character distribution of all irreducible unitary representations of $\GL_2(\F)$ \cite{GelbartJacquet}*{page 244(7.16)}. These are 
 \begin{itemize} 
\item                                                                                            
the one-dimensional representations,
\item  the continuous series representations, 
\item the complementary series representations, 
\item the discrete series representations for $\F = \bR$,  
\item the supercuspidal representations and Steinberg representations for $\F$ non-archimedean,
                                                                                           \end{itemize} 
 \item the Eisenstein spectrum (Section~\ref{section:globaleisen}),\footnote{Some of the values have to be computed globally.}
  \item the Eisenstein residues (Section~\ref{section:globalresidual}),\footnote{These have to be computed globally and are only treated in a global context.}
  \item the identity distribution: 
this is a specialization of the Plancherel formula (Section~\ref{section:globalidentity}), 
  \item the parabolic distribution:
 the value of the local zeta integral and its derivative at $s=1$ (Section~\ref{section:globalpara}),
 \item the hyperbolic distribution:
 the orbital integral and the weighted orbital integral of a hyperbolic element (Section~\ref{section:globalhyper}),
 \item the elliptic distribution:
 the orbital integral of an elliptic element if $\F \neq \bC$ (Section~\ref{section:globalelliptic}).
\end{enumerate}
Note that only \cite{JacquetLanglands} treats the function field case, whereas the exposition in \cite{GelbartJacquet} and \cite{Gelbart} works in the algebraic number field setting only. The differences between the function field setting and the number field setting are minor. At the non-archimedean places, the local harmonic analysis only depends mildly upon the residue characteristic.

By a special set of test functions, we mean a subset of $\Ccinf(G)$, which has the following two properties:
\begin{itemize}
 \item only ``very few'' character distributions do not vanish
 \item the test functions are able to determine/distinguish automorphic representations up to their factorization into local factors
\end{itemize}
Let us be precise. Consider the maximal compact subgroup $K$ of $\GL_2(\F)$, i.e., either $\U(2)$, $\O(2)$ or $\GL_2(\o_v)$ depending on whether $\F$ is complex, real, or non-archimedean.   Define $\overline{K} $ as the product of $K$ and the center $\Z(\F)$ of $\GL(\F)$. For a central unitary one-dimensional representation $\chi : \Z(\F) \rightarrow \bC^1$, define
\begin{align*} \Ccinf(\GL_2(\F), \chi) = \Big\{ \phi : \GL_2(\F) \rightarrow&\, \bC \textup{ smooth, compactly supported modulo } \Z(\bF) \\ 
                                                                                                      & \phi(zg) = \overline{\chi(z)} \phi(g) \textup{ for all }z \in \Z(\F) \Big\} . \nonumber
\end{align*}
 Note that the irreducible representations of $\overline{K}$ are unitarizable and finite-dimensional.
\begin{defnu}[Distinction and separation] \mbox{}
 \begin{itemize}
  \item  (Parametrization) We say that two infinite-dimensional, irreducible, unitary representations of $\GL_2(\F)$ are $\overline{K}$-equivalent if they are isomorphic as $\overline{K}$ representations.
  \item  (Distinction property) We say that a non-zero function $\phi \in \Ccinf(\GL_2(\F))$ is a pseudo-coefficient of a $K$-equivalence class $\{ \pi \}$ if the character distribution vanishes on $\phi$ for all unitary infinite-dimensional representations\footnote{We have to exclude the one-dimensional representations here.} but those from the $\overline{K}$-equivalence class $\{ \pi \}$. We say that $\phi$ is $\{ \pi \}$-distinguishing element.\footnote{Note that the Plancherel formula implies that for non-zero elements in $\Ccinf(G)$, not all character distributions can vanish simultaneously.}
 \item  (Separation property) A $\{ \pi \}$-distinguishing element $\phi$ separates two isomorphism classes of representations $\pi_1, \pi_2 \in \{ \pi \}$ if the character distributions differ, i.e.,
\[ \pi_1(\phi) \neq \pi_2(  \phi).\]
We say that a subset $X$ of the $\{ \pi \}$-distinguishing elements in $\Ccinf(G)$ separates $\{ \pi \}$ if for any two representation $\pi_1, \pi_2 \in \{\pi \}$ there exists an element $f \in X$ which separates $\pi_1$ and $\pi_2$. 
 \end{itemize} 
\end{defnu}
The parametrization via irreducible representations of $\overline{K}$ is necessary and crucial for the classification of all unitary representations.
 The distinction property allows to specialize the trace formula to a more explicit form, such as the classical Selberg trace formula for Maass forms 
on $\Gamma \backslash \mathbb{H}$, and the Eichler-Selberg trace formula for the Hecke eigenvalues of modular forms. Actually, our specialization will be an improvement over these classical formulas, in the sense that our trace formula analyzes only one fixed $\overline{K}$-equivalence class at every place. The separation property guarantees that the full power of the Arthur trace formula is exhausted and no information is given away. 

There are three classes of $\overline{K}$-equivalence classes:
\begin{itemize}
 \item The irreducible parabolic inductions include the continuous series representations and the complementary series representations. The construction is fairly easy, since a representation $\rho$ of $\overline{K}$ will always exist which does not occur in any other $\overline{K}$ equivalence class. Unfortunately, these $\overline{K}$-equivalence classes are large and have uncountably many elements. If we generalize the definition of a Hecke algebra $\mH(G, \rho)$ to open subgroups $K$ which are only compact modulo the center, then elements of $\mH(G, \rho)$ separate and distinguish this $\overline{K}$-equivalence class up to twist by one-dimensional characters.
 \item The irreducible, infinite-dimensional subquotients of the parabolic inductions include the discrete series representations, and the Steinberg representations. The $\overline{K}$-equivalence class contains only two $G$-isomorphism classes of representations, so the separation axiom is easy to satisfy. For the construction, one has to carefully study the Abel transform and arrange suitable linear combinations of elements $\phi_\rho \in \mH(G, \rho)$ for different irreducible representations $\rho$ of $\overline{K}$.
\item For the supercuspidal representations, we have to appeal to their classification. There will always exist a representation $\rho$ of $\overline{K}$ which does not occur in any other $\overline{K}$-equivalence class of unitary representations. Additionally, the $K$-equivalence class contains either one or two elements, depending on whether the supercuspidal representation is associated to a ramified or unramified quadratic extension. In the first case, everything works as intended and the function $x \mapsto \tr \rho(x)$ is essentially the only option for the distinguishing function. In the second case, we have to switch to another open subgroup, which is compact modulo the center. This subgroup is the normalizer of the Iwahori subgroup $\Gamma_0(\p)$ inside $\GL_2(\F)$. The method is the same -- only in terms of this latter subgroup.
\end{itemize}
The above $\overline{K}$-equivalence relation is weaker than the ``local analogue'' of the similarity classes. However, the only distinction is that ramified supercucpidal
representation contain not one but two isomorphism classes in a $\overline{K}$-equivalence class. The Seperation property ensures that we can seperate the two isomorphism classes, and construct pseudo coefficients in this way as well.

The results, to be quoted in this part, rely more or less directly on methods involving the use of the Lie algebra, root systems, or the Bruhat-Tits building. These concepts are modern and powerful but require much notation. I have not introduced these concepts here, because for the group $\GL(2)$, they seemed like overkill.

There is one technical difference between this part and the harmonic analysis in the preceding part. Here, we will work with $\Ccinf(\GL_2(\F),\chi)$ instead of $\Ccinf(\GL_2(\F))$, i.e., modulo the center. This is a minor technical modification, and it is rather easy to translate results between the two settings. If the one-dimensional representation $\chi$ of $\Z(\F)$ and the irreducible representations of $K$ coincide on $K \cap \Z(\F)$, then everything can be generalized directly via the surjective algebra homorphism
\[ \Ccinf(\GL_2(\F)) \twoheadrightarrow \Ccinf(\GL_2(\F), \chi) , \qquad \phi(g) = \int\limits_{\Z(\F)} \chi(z) f(z) \d z.\]
Observe that as a consequence of Schur's Lemma, the restriction of a unitary irreducible representation to its center is a one-dimensional representation. Its character distribution on $\Ccinf(G)$ factors through $\Ccinf(\GL_2(\F), \chi)$. There is no loss of generality here.

An alternative route suggests working with $\GL_2(\F)^1 = \{ g  \in \GL_2(\F): \left|\det g \right|_v= 1 \}$ instead of $\GL_2(\F)$, since the former group has a compact center. The differences in the representation theory and the harmonic analysis of $\GL_2(\F)$ and $\GL_2(\F)^1$ are not crucial, and one can translate between the settings rather easily. This route is preferred in the papers of Arthur, cf. \cite{Arthur:Invariant}. 

Note that naively switching to the group $\PGL_2(\F)$ instead of $\GL_2(\F)$ is not always possible. The representation theory of $\PGL_2(\F)$ is not as rich as that of $\GL_2(\F)$. Also understanding the representations of $\GL_2(\F)$ from those of $\SL_2(\F)$ is in general non-trivial, in particular, it is hard for the supercuspidal representations if the residue characteristic of $\F$ is two. I have avoided every approach which requires a case-by-case analysis depending on the characteristic or residue characteristic.

\chapter{Harmonic analysis on $\GL(2,\bR)$}

\section{Haar measure}
The computations depend in a fairly obvious manner upon the Haar measures. We will fix the Haar measures to avoid confusion.

Let $\d r$ be the Lebesgue measure on the real line, which assigns unit measure to the interval $[0,1]$. For all elements $f \in \Ccinf(\bR)$ and $h \in \Ccinf(\bR^\times)$, we have 
\begin{align*} \int\limits_{\bR} f(z) \d_\bR z&=    \int\limits_{\bR^\times} f(z) \left|z \right|_\bR \d^\times_\bR z  , \qquad    \left|z \right|_\bR= |z|, \\
                                                     &=   \sum\limits_{\sigma \in\{ \pm 1\}}  \int\limits_{0}^{\infty}   f(\sigma r )      \d r,  \\ 
  \int\limits_{\bR^\times} h(z)  \d^\times_\bR z  & =   \sum\limits_{\sigma \in\{ \pm 1\}}  \int\limits_{0}^{\infty}  h(\sigma r)     \frac{\d r}{r}.
\end{align*}
We consider the locally compact group $\GL_2(\bR)$ with its closed subgroups
\begin{align*}
 \N(\bR) & \coloneqq \left\{ \sma 1 & x \\ 0 & 1 \smz : x \in \bR \right\}, \\
 \M(\bR) & \coloneqq  \left\{ \sma \alpha & 0 \\ 0 & \beta  \smz : \alpha, \beta \in \bR^\times \right\}, \\
   \Z(\bR) & \coloneqq  \left\{ \sma z & 0 \\ 0 &z  \smz : z \in \bR^\times \right\}, \\
\B(\bR) & \coloneqq  \left\{ \sma \alpha &x \\ 0 & \beta  \smz : \alpha, \beta \in \bR^\times, x \in \bR \right\},
\end{align*}
and its closed compact subgroups $\O(2)$ and $\SO(2)$. Only the group $\B(\bR)$ is not unimodular.

The compact groups are endowed with the unit Haar measures. The group $\SO(2)$ has index two in $\O(2)$, hence for all elements $f \in \mL^1(\O(2))$, the following integral identity holds
\[ 2 \int\limits_{\O(2)}    f(k) \d_{\O(2)} k =    \int\limits_{\SO(2)}    f(k) \d_{\SO(2)} k     + \int\limits_{\SO(2)}    f\left( \sma -1 & 0 \\ 0& 1 \smz k\right) \d_{\SO(2)} k.\]
We endow $\N(\bR)$ with the Haar measure of $\bR^+$ by identifying
\[ \bR^+ \xrightarrow\cong \N(\bR), \qquad x \mapsto \sma 1 & x \\ 0 & 1 \smz. \]
We endow $\Z(\bR)$ with the Haar measure of $\bR^\times$ via
\[ \bR^\times \xrightarrow\cong \Z(\bR) , \qquad \alpha  \mapsto \sma \alpha & 0 \\ 0 & \alpha  \smz.\]
For the group $\M(\bR)$, we define the Haar measure by $\M(\bR) \cong \bR^\times \times \bR^\times$:
\begin{align*}
 \int\limits_{\M(\bR)} f(m) \d m = \int\limits_{\bR^\times}    \int\limits_{\bR^\times} f\left( \sma z \beta& 0 \\ 0 & z \smz \right) \dpr z\,\dpr \beta \qquad \textup{ for } f \in \mL^1(\M(\bR)).
\end{align*}
We have an integral formula
\begin{align*}   \int\limits_{\M(\bR)} f(m) \d m  \underset{z \mapsto z /\sqrt\beta}= &    \sum_{\sigma \in \pm}   \int\limits_{0}^\infty    \int\limits_{\bR^\times} f\left( \sma \sigma z \sqrt\beta& 0 \\ 0 & z \sqrt\beta^{-1} \smz \right) \dpr z\, \dpr \beta \\
                                                                  \underset{a = \sqrt\beta}= &    \sum_{\sigma \in \pm}  2 \int\limits_{0}^\infty    \int\limits_{\bR^\times} f\left( \sma \sigma z a& 0 \\ 0 & z a^{-1} \smz \right) \dpr z\, \dpr a\\
                                                                                                     = &    \sum_{\sigma \in \pm}   \int\limits_{\bR^\times}    \int\limits_{\bR^\times} f\left( \sma \sigma z a& 0 \\ 0 & z a^{-1} \smz \right) \dpr z\,\dpr a.
\end{align*}

We fix the unique left invariant Haar measure $\d_l b$ on $\B(\bR)$ with
\[            \int\limits_{\B(\bR)} f(b) \d_l b =  \int\limits_{\M(\bR)} \int\limits_{\N(\bR)}  f(mn) \d n \d m \qquad \textup{ for } f \in \mL^1(\B(\bR)), \]
and the unique right invariant Haar measure $\d_r b$ by
\[            \int\limits_{\B(\bR)} f(b) \d_r b =  \int\limits_{\N(\bR)}  \int\limits_{\M(\bR)} f(nm) \d m \d n \qquad \textup{ for } f \in \mL^1(\B(\bR)). \]
The modular character is then given as
\[ \Delta_{\B(\bR)} : \sma a & * \\ 0 & b \smz \mapsto |a/b|.\]
The Iwasawa decomposition $\GL_2(\bR) =\B(\bR) \SO(2)$ yields a unique Haar measure on $\GL_2(\bR)$ such that
 \begin{align*}
 \int\limits_{\GL_2(\bR)} f(g) \d g & =  2\int\limits_{\bR^\times} \int\limits_{0}^\infty \int\limits_{\bR} \int\limits_{\O(2)}  f\left( \sma z & 0 \\ 0 & z \smz \sma a & 0 \\ 0 & a^{-1} \smz \sma 1 & x \\ 0 & 1 \smz k\right) \d k \,\dr x \, \dpr a \, \dpr z\\
                                               & =\sum\limits_{\sigma \in \{ \pm 1 \}} \int\limits_{\bR^\times} \int\limits_{0}^\infty \int\limits_{\bR} \int\limits_{\SO(2)}  f\left( \sma \sigma z & 0 \\ 0 & z \smz \sma a & 0 \\ 0 & a^{-1} \smz \sma 1 & x \\ 0 & 1 \smz k\right) \d k\,\dr x\, \dpr a\, \dpr z\\
                                                & =\int\limits_{\M(\bR)} \int\limits_{\N(\bR)} \int\limits_{\SO(2)}  f\left( mn k\right) \d k \d n \d m. 
 \end{align*}
The Haar measures on $O(2)$ and $\SO(2)$ are both denoted by $\d k$, although they differ by a constant when restricted to $SO(2)$.                                                                                                                                      

\section{The compact subgroups $\SO(2)$ and $\O(2)$}
All irreducible representations of the abelian group $\SO(2)$ are one-dimensional. We can identify them with the group of integers
\[ \epsilon_n \colon \sma \cos \theta & - \sin \theta \\ \sin \theta & \cos \theta \smz  \mapsto \e^{\im n \theta}, \qquad n \in \bZ.\] 
\begin{theorem}[All irreducible representations of $\O(2)$]
All irreducible representations of $\O(2)$ are contained for some $n \in \mathbb{Z}$ in
           \[ \rho_n = \Ind_{\SO(2)}^{\O(2)} \epsilon_n ,\]
         where
\begin{itemize} 
 \item $\rho_n$ is irreducible if and only if $n \neq 0$,
 \item $\rho_n \cong \rho_{n_0}$ if and only if $n = \pm n_0$, 
 \item $\rho_0 = 1 \oplus \det$.
\end{itemize}
\end{theorem}
\begin{proof}
The group $\SO(2)$ is abelian and the group $\O(2)$ is a semi-direct product of two abelian groups $\O(2) = \SO(2) \rtimes  \{ \sma \pm 1 & 0 \\ 0 & 1\smz\}$. We determine the irreducible representations via the Mackey machine \cite{Mackey:Unitary}*{page 73}, which is designed for the analysis of irreducible representations of group extensions.  
We identify $\SO(2)$ with $\bC^1=\{ \e^{\im \theta}: \theta \in \bR \}$ via
\begin{align*}k_\theta =  \pma \cos \theta & -\sin \theta \\ \sin \theta & \cos \theta \pmz \mapsto \e^{\im \theta}.\end{align*}
We identify the Pontryagin dual of the abelian group $\widehat{\SO(2)}$ with the group $\bZ$ of integers via sending $n \in \bZ  \mapsto \epsilon_n$, where $\epsilon_n: \e^{\im \theta} \mapsto \e^{\im n \theta}$.  
Conjugation by $\sma -1 & 0 \\ 0 & 1\smz$ on $\SO(2)$ results on $\SO(2) \cong \bC^1$ in taking inverses, i.e., $\e^{\im \theta} \mapsto \e^{-\im \theta}$, and hence on the Pontryagin dual $ n \mapsto -n$ as well.
\begin{align*}  \textup{Orbit}(\epsilon_n) = \begin{cases} \{ \epsilon_n , \epsilon_{-n}  \}, & n \neq 0, \\ \{ 1 \}, & n=0 , \end{cases}           
\qquad \textup{Stab} ( \epsilon_n) = \begin{cases} \{ 1 \} , & n \neq 0, \\     \{ \sma \pm 1 & 0 \\ 0 & 1\smz\}, & n=0. \end{cases} \end{align*}
Hence for $n> 0$, we observe that the irreducible representation \( \Ind_{\SO(2)}^{\O(2)} \epsilon_n \) is irreducible, and for $n =0$, that \(\Ind_{\SO(2)}^{\O(2)} 1\) splits into two one-dimensional representations, i.e., the trivial representation and the determinant map $\textup{det}:\O(2)\rightarrow \{\pm1\}$.
\end{proof}
An alternative proof avoiding the Mackey Machine can be found in \cite{KnightlyLi}*{Section 11.2, page 155}.

\section{The representation theory of $\GL(2,\bR)$}\label{section:realclass}
We will now list all the unitary representations of $\GL_2(\bR)$. They are all given as subquotients or subrepresentations of Jacquet-modules by the Casselman submodule theorem. 
References include  \cite{Bump:Auto}, \cite{GoldfeldHundley1}, \cite{JacquetLanglands}, and \cite{KnightlyLi}.
 
Consider a one-dimensional representation $\mu : \B(\bR) \rightarrow \bC^\times$. It determines uniquely two one-dimensional representations $\mu_j : \bR^\times \rightarrow \bC^\times$ such that
\[ \mu_{12} \left( \sma a & * \\ 0 & b \smz \right) = \mu_1(a) \mu_2(b).\]
\begin{defn}
Let $s \in \bC$. The representation $\mJ(\mu,s) = \mJ(\mu_1, \mu_2,s)$\index{$\mJ(\mu,s) = \mJ(\mu_1, \mu_2,s)$ for $\GL_2(\bR)$} is the right regular representation of $\GL_2(\bR)$ on the space of smooth functions $$f: \GL_2(\bR) \rightarrow \bC,$$ which satisfy
\[ f\left( \sma a & * \\ 0 & b \smz g\right) = \mu\left( \sma a & 0 \\ 0 & b \smz\right)  \left| \frac{a}{b} \right|^{s+1/2}f(g).  \]
\end{defn}
Every one-dimensional representation $\chi : \bR^\times \rightarrow \bC^\times$ can be uniquely decomposed as
\[ \chi( \pm t ) = \chi_{alg} (\pm 1) t^{s_\chi}  , \qquad t \in(0,\infty), \]
for some unique $s_\chi \in \bC$ and $\chi_{alg}$ being either trivial or the sign character. We say that $\chi$ is \textbf{algebraic} if $s_\chi= 0$.\index{algebraic character case $\bR$} Similarly, we say that
\[ \mu \left( \sma a & * \\ 0 & b \smz \right) = \mu_1(a) \mu_2(b)\]
is algebraic if $\mu_1$ and $\mu_2$ are algebraic. It is sufficient to consider parabolic inductions with algebraic $\mu$, since we have an isomorphism
\[ \mJ(\mu_1, \mu_2, s) = \left| \det(\blank) \right|^{s_{\mu_1}/2 + s_{\mu_2}/2} \otimes \mJ( \mu_{1, alg}, \mu_{2, alg} , \frac{s_{\mu_1} -s_{\mu_2}}{2} ).\]
Generally, we have that
\[ \chi \circ \det \otimes \mJ(\mu, s) \cong \mJ(\mu \cdot  \chi\circ\det|_{\B(\bR)}, s) = \mJ(\mu_1 \chi, \mu_2 \chi, s).\]
The central character of $\mJ(\mu_1, \mu_2,s)$ is given by $\mu_1 \mu_2$. If $\mu$ is algebraic, it will be algebraic as well. \textbf{We assume from now on that all the one-dimensional representations denoted as $\mu, \mu_1, \dots$ are algebraic.}

The parabolic induction $\mJ(\mu_1, \mu_2, s)$ with algebraic characters $\mu_1, \mu_2$ is irreducible and unitarizable in the following two cases:
\begin{itemize}
 \item if $\Re s =0$, or
 \item if $-1/2 < s < 1/2$ and $\mu_1 = \mu_2$.
\end{itemize}
In general, it will be neither irreducible nor unitarizable. If it is reducible, it contains a unique irreducible invariant subspace and a unique irreducible invariant subquotient. Either the subquotient is finite-dimensional (then $\Re s > 1/2$) or the subspace is finite dimensional (then $\Re s<1/2$). The subspace (the subquotient) is a unitarizable representation if and only if it is either one-dimensional or infinite-dimensional.
These representations exhaust all unitary, irreducible representations. To be precise:
\begin{theorem}[Classification of the unitary dual]
Every irreducible, unitarizable representation of $\GL_2(\bR)$ with algebraic central character is isomorphic either to
\begin{enumerate}[font=\normalfont]
     \item a one-dimensional representation $\chi \circ \det$ for $\chi=1$ or $\chi = \sign$,
     \item a continuous series representation,
       \begin{enumerate}[font=\normalfont]
        \item  a principal series representation $\mJ(\mu_1, \mu_2, s)$ for $\Re s = 0$ and $\mu_j \in \{ 1, \sign \}$, 
        \item a complementary series representation $\mJ(\mu, \mu, s)$ for $-1/2 < \Re s < 1/2$ and $\mu \in \{ 1, \sign\}$,
       \end{enumerate}
     \item a discrete series representation,  for $\mu_1 \neq \mu_2$ and $\mu, \mu_1, \mu_2 \in \{ 1, \sign \}$
        \begin{enumerate}[font=\normalfont]
            \item  an even discrete series representation, i.e., the unique irreducible subrepresentation\index{$D_k(\mu), D_k(\mu_1, \mu_2)$ discrete series representation} 
                      \[ D_k(\mu) :=D_k(\mu,\mu)  \subset \mJ( \mu, \mu, \frac{k-1}{2} ) , \qquad k \geq 2 \textup{ even} , \dots, \]
           \item an odd discrete series representation, i.e., for $\mu_1 \neq \mu_2$ the unique irreducible subrepresentation
                       \[ D_k(\mu_1, \mu_2)  \subset \mJ( \mu_1, \mu_2, \frac{k-1}{2} ) , \qquad k \geq 3 \textup{ odd}, \dots.\]
        \end{enumerate}
\end{enumerate}
We have an isomorphism $\mJ(\mu_1, \mu_2, s) \cong \mJ(\mu_2, \mu_1, -s)$. All the other listed representations are non-equivalent.
\end{theorem}
There are various proofs of this in the literature \cite{KnightlyLi}*{Theorem 11.15, page 164, table.1, page 186 and the subsequent discussion},\cite{JacquetLanglands}*{Chapter 5}, \cite{Bump:Auto}, \cite{GoldfeldHundley1}.
The references also address growth properties. The discrete series representations have integrable matrix coefficients if $k> 2 $, and but only square integrable matrix coefficients if $k=2$. The limit of discrete series representations is among the principal series representations, which are tempered, i.e., their matrix coefficients are contained in $\mL^{2+\epsilon}(\GL_2(\bR) / \Z(\bR))$ for all $\epsilon >0$. The complementary series representations are not tempered. They do not occur in the right regular representation of $\GL_2(\bR)$. The Selberg eigenvalue conjecture\index{Selberg eigenvalue conjecture}  asserts that the complementary series representations do not occur as constituents of automorphic representations.   

As subquotients of parabolic inductions, all unitary representations of $\GL_2(\bR)$ are automatically admissible and have a character distribution. 
The computations of these distributions will depend on the $\O(2)$-type decomposition of the unitary representations.
\begin{theorem}[\cite{KnightlyLi}*{Proposition 11.12, page 161  and Theorem 11.15, page 164}]\label{thm:realKtype}
We observe the following $\O(2)$-type decomposition for the infinite-dimensional representations:
\begin{align*}
\Res_{\O(2)} \mJ(1, 1, s) &=  1  \oplus \bigoplus_{n \geq 2 \atop n \textup{ even}} \rho_n,\\
 \Res_{\O(2)} \mJ(\sign, \sign, s) &=  \det  \oplus \bigoplus_{n \geq 2 \atop n \textup{ even}} \rho_n,\\
 \Res_{\O(2)} \mJ(\mu_1, \mu_2, s) &=   \bigoplus_{n \geq 1 \atop n \textup{ odd}} \rho_n, \qquad \mu_1 \neq \mu_2,\\
 D_k(\mu, \mu)  & = \bigoplus_{n \geq k \atop n \textup{ even}} \rho_n,\\
D_k(\mu_1, \mu_2)& =     \bigoplus_{n \geq k \atop n \textup{ odd}} \rho_n, \qquad \mu_1 \neq \mu_2 .
\end{align*}
\end{theorem}

Hence, the irreducible representations $D_k(\sign) $  and $D_k(1)$ are in the same $\overline{K}$-equivalence class. DBut thee difference is superficial in the sense that
\[ D_k(1) = \sign \circ \det \otimes D_k(\sign).\]
Similar statements are true for $D_k(\mu_1, \mu_2)$ and $D_k(\mu_2, \mu_1)$. For computational and notational simplificity, we prefer an approach which does not distinguish between these representations.\footnote{We will shortly introduce a subspace of $\mH(G, \rho)$, which is denoted by $\mSH(G, \rho)$. The subspace $\mSH(G,\rho)$ cannot separate $D_k(1)$ and $D_k(\sign)$, whereas $\mH(G, \rho)$ can. The difference between $D_k(1)$ and $D_k(\sign)$ is trivial in the sense that the global spectral analysis of the related automorphic forms is stable under character twists. $\mSH(G, \rho)$ can be treated much more economically in terms of required notation.}

\section{The Abel inversion for $\GL(2, \bR)$}
As demonstrated in Section~\ref{section:Abel} and Section~\ref{section:jacquet}, the Abel transform plays a central role in the analysis of the character distributions of irreducible submodules of parabolic inductions. A good, explicit understanding of the Abel transform seems required for computational aspects in the harmonic analysis of $\GL(2)$. The Abel transform has an inverse, which is the main technical detail for an explicit trace formula.

Let $\rho$ be a unitary, finite-dimensional representation of $\Z(\bR) \O(2)$. As an outcome of the computations, it will become clear that we can restrict our attention to the subspace \( \mSH(\GL_2(\bR), \rho) \)\index{$\mSH(\GL_2(\bR), \rho)$} of $\mH(\GL_2(\bR), \rho)$ with little loss. The space $\mSH(\GL_2(\bR), \rho)$ has index two in $\mH(\GL_2(\bR), \rho)$ if $\dim(\rho) \neq 1$. It is defined as the space of smooth functions \( \phi : \GL_2(\bR) \rightarrow \bC \), which have compact support modulo the center, and satisfy 
\[ \phi\left(k_1 \sma t& 0 \\ 0 & t^{-1} \smz k_2\right)  = \phi\left( \sma t& 0 \\ 0 & t^{-1} \smz \right) \frac{\tr \rho(k_1k_2)}{\dim(\rho)} \qquad \textup{ for all }k_1, k_2 \in \Z(\bR) \O(2), t \geq 1.\]
We have to normalize by the dimension because $1 \in \Z(\bR) \O(2)$.

We will see that the space $\mSH(G, \rho)$ for $\rho$ irreducible and two-dimensional cannot separate \(\pi\) and its twist by a non-trivial one-dimensional representation of $\GL_2(\bR)$.
Because of this, it seems more pleasant for us to argue with      \index{$\rho_n$, $n \geq 0$ - repr. of $\O(2)$}
\[ \rho_n \coloneqq \Ind_{\SO(2)}^{\O(2)} \epsilon_n \]
for all $n\geq 0$, despite the fact that $\rho_0$ is not irreducible:
\[ \rho_0 = 1 \oplus \sign \circ \det. \]
\begin{lemma}
Elements from $\mSH(\GL_2(\bR), \rho_n)$ are supported on 
\[ \GL_2(\bR)^+ \coloneqq \left\{ g \in \GL_2(\bR) : \det g >0 \right\}.\]
\end{lemma}
\begin{proof}
We have that $\tr \rho_n \left( \sma -1 & 0 \\ 0 & 1 \smz \right) =0$ for $n \geq 1$ by the Frobenius character formula.
 \end{proof}

 \begin{remark}
 Although everything is phrased in terms of $\GL_2(\bR)$, the harmonic analysis is that of $\GL_2^+(\bR)$. I have decided to go this route after realizing that it has notational and computational advantages, and because it is not less general.
 \end{remark}

 Furthermore, we can identify
\[ \mSH(\GL_2(\bR), \rho) \cong_\rho \Ccinf[0, \infty)\]
by setting
\[ \phi\left( \sma \e^{x} & 0 \\ 0 & \e^{-x}\smz \right) = \Phi_\phi ( \e^{2x} + \e^{-2x}-2),\]
and in general
\begin{align}\label{eq:identification} \phi(g) = \frac{1}{2} \Phi_\phi\left( \frac{\tr g g^\dagger }{|\det (g)|}-2 \right) \tr \rho(k), \end{align}
if $gk^{-1}$ is self adjoint for $k \in \O(2) \Z(\bR)$. By the polar decomposition, there does always exist such an element $k$ which is unique up to conjugation.

The coordinates chosen in this section are to a great extent arbitrary. We have adopted the choices and the notation of the standard texts \cite{Hejhal1}, \cite{Hejhal2}, \cite{DeEc}, \cite{Iwaniec:Spectral} \cite{Kubota}, \cite{Venkov}.\footnote{Most of the sources prefer to use the notation $k_\phi$ (for kernel) instead of $\Phi_\phi$. This would be inconvenient, since we prefer that $k$ is always an element of a compact subgroup.} Hopefully, this simplifies the transfer and comparison for the reader familiar with at least one of these references.

The Chebyshev functions $T_\lambda$ are defined for each real number $\lambda$ on page 114 in \cite{Matsushita} as special values of the hypergeometric series. 
Their connection to the representation theory of $\O(2)$ is as follows: For nonzero integers $n\neq 0$ and elements $k \in \O(2)$, the trace of $\rho_n = \Ind_{\SO(2)}^{\O(2)} \epsilon_n$ is
\[ \tr \rho_n \left( k \right) = \begin{cases}  2 T_{|n|}( \cos(\theta)),  &    k =   \sma  \cos\theta & \sin\theta \\ -\sin\theta & \cos\theta \smz, \\
                                                                               0,          &   \det(k) =-1.
                                 \end{cases}\] 
The constant $\lambda$ is usually referred to as weight. We will here avoid a detailed definition, and only characterize them as a solution to
\begin{align}\label{eq:tn} T_\lambda( \cos(x)) = \cos( \lambda x), \qquad T_\lambda(  \cosh(t)) = \cosh(\lambda t), \end{align}
and more generally as solution to
\begin{align}\label{eq:tn2} T_\lambda( x) = \frac{(x+\sqrt{x^2-1})^\lambda + (x-\sqrt{x^2-1})^\lambda }{2}. \end{align}
The solutions to these equations extend to the domain of complex numbers, where the $\lambda$-th power is understood as a principal value. They specialize to the Chebyshev polynomials, if $\lambda$ is a non-negative integer. 
Everything in the current and following sections holds true for real $\lambda$. In later considerations integer values are sufficient and we will focus on this case. Arbitrary real, non-integral $\lambda$ correspond to the harmonic analysis on the universal cover $\widetilde{\SL_2(\bR)}$ of $\SL_2(\bR)$.

The Abel transform $\mA_{\rho_n}$ can then be expressed on $\Ccinf[0, \infty)$ rather than on $\mSH(\GL_2(\bR), \rho)$ via the Identification~\ref{eq:identification}.
\begin{defnthm}[Main inversion identity \cite{Matsushita}, \cite{Hejhal2}]\label{thm:realabelinversion}
For each non-negative real number $\lambda \geq 0$, define the operators 
\[A_\lambda, \hat{A}_\lambda: \Ccinf[0, \infty)  \rightarrow \Ccinf[ 0, \infty) \]
by setting
\begin{align*}
A_\lambda \Phi(x) &=   \int\limits_{0}^\infty \Phi\left(x + \xi^2 \right)  2T_\lambda \left( \frac{\sqrt{x+4}}{\sqrt{x +4+ \xi^2}} \right) \d \xi \\ 
\hat{A}_\lambda \Phi(x) & = - \frac{1}{\uppi}       \int\limits_{0}^\infty \Phi\left(x + \eta^2 \right)  2T_\lambda \left( \sqrt{\frac{x +4+ \eta^2}{x+4}} \right) \d \eta.
\end{align*}
Let $\Phi'$ denote the derivative of $\Phi$. Both of the operators $A_\lambda$ and $\hat{A}_\lambda$ are invertible, since the following inversion formula holds
\[      \hat{A}_\lambda  (A_\lambda \Phi)'(x)  = \Phi(x), \qquad \left( A_\lambda (\hat{A}_\lambda \Phi(x)) \right)'= \Phi(x).\]
This inversion formula holds more generally for every function $f:[0, \infty) \rightarrow \bR$ which is differentiable on $[0, \infty)$ and satisfies
\[ \left| f(x) \right| \ll (x)^{-\alpha},  \quad \left| f'(x) \right| \ll (x)^{-\alpha-1} \]
for some $\alpha > \max \{ \lambda/2, 1/2 \}$. 
We have the following alternative integral kernels for $\hat{A}_\lambda$ and $A_\lambda$:
\begin{align*}
 A_\lambda f(x)   & = \int\limits_{\bR} f\left(x+\xi^2\right)          \left( \frac{\sqrt{x+4} + \im \xi}{\sqrt{x+4}  - \im \xi} \right)^{\lambda/2} \d \xi, \\
 \hat{A}_\lambda f(x)  &  =- \frac{1}{2\uppi}  \int\limits_{\bR} f\left(x+\xi^2\right)          \left( \frac{\sqrt{x+4+\xi^2}-\xi}{\sqrt{x+4+\xi^2}  + \xi} \right)^{\lambda/2} \d \xi. \\ 
\end{align*}  
\end{defnthm}

\begin{proof}
Certainly the operators are well defined. They are bounded. They send smooth, compactly supported functions into the space of smooth, compactly supported functions by the Lebesgue's Dominated Convergence Theorem. The inversion formula is classical for $\lambda=0$, and can be found in many references \cite{DeEc}, \cite{Lang:SL2}, \cite{Iwaniec:Spectral}. I will only present a complete proof for the inversion formula in the case $\lambda=0$ and $T_\lambda=1$, since it is short and painless:
\begin{align*}
\hat{A}_0 \left( A_0 f \right)'(x) &= - \frac{4}{\uppi}  \int\limits_{0}^\infty \int\limits_{0}^\infty f'\left(x + \xi^2 + \eta^2 \right)       \d \xi \d \eta \\ 
                           &= - 2   \int\limits_{0}^\infty  \int\limits_{0}^{\pi/2} f'\left(x + R^2\sin(\theta)^2 + R^2 \cos(\theta)^2 \right)    \d  \theta \,  R \d R \\
                            &= -     \int\limits_{0}^\infty f'\left(x + R^2 \right)    2 R \d R \\ 
                            & = -    \int\limits_{0}^\infty f'\left(x + r \right)      \d r  = f(x) .
\end{align*}
We will not give the proof for arbitrary non-negative real weight $\lambda$, since the known proofs are long. I am aware of two essentially equivalent proofs, see \cite{Hejhal1}*{page 455ff.}, \cite{Hejhal2}*{eq. 6.5 and 6.5, page 386} partly due to Selberg and, see \cite{Matsushita}*{Theorem 2.3.1, page 114}, partly due to Shintani. The proofs go via the Mellin inversion formula and reduce the inversion formula to integral expressions for the Beta function. 

The proposition gives two alternative kernel transformations. In fact, Hejhal \cite{Hejhal1},\cite{Hejhal2} prefers the later presentation of $A_\lambda$ and $\hat{A}_\lambda$, whereas Matsushita \cite{Matsushita} prefers the former definition in a slightly different normalization. The statement that the inversion formula holds for a larger set of test functions can only be found in \cite{Hejhal1}*{page 455ff.} It remains for us to prove that both definitions coincide. This is the content of the next lemma. \end{proof}
 \begin{lemma}\label{lemma:hejhalalt}
  For $f \in \Ccinf([0, \infty))$, we have that           
\begin{align*}
   \int\limits_{0}^\infty& f\left(x + \xi^2 \right)  2T_\lambda \left( \frac{\sqrt{x+4}}{\sqrt{x +4+ \xi^2}} \right) \d \xi \\  
& = \int\limits_{\bR} f\left(x+ \xi^2\right)          \left( \frac{\sqrt{x+4} + \im \xi}{\sqrt{x+4}  - \im \xi} \right)^{\lambda/2} \d \xi, \\
 - \frac{1}{\uppi}       \int\limits_{0}^\infty & f\left(x + \eta^2 \right)  2T_\lambda \left( \sqrt{\frac{x + 4 + \eta^2}{x+4}} \right) \d \eta \\
  &  =- \frac{1}{\uppi}  \int\limits_{\bR} f\left(x+\eta^2\right)      \left( \frac{\sqrt{x+4+\eta^2}-\eta}{\sqrt{x+4+\eta^2}  + \eta} \right)^{\lambda/2} \d \eta.
\end{align*}
\end{lemma} 
\begin{proof}
We start with Equation~\ref{eq:tn2}. For $x \geq 0$, set $y =x+4$ and derive
\begin{align*}  T_\lambda \left( \frac{\sqrt{y}}{\sqrt{y + \xi^2}} \right)& =  \frac{(\sqrt{y} + \im |\xi|)^\lambda + (\sqrt{y} - \im |\xi|)^\lambda}{2 (y + \xi^2)^{\lambda/2}}  \\
                                        & = \frac{(\sqrt{y} + \im |\xi| )^\lambda + (\sqrt{y} - \im |\xi|)^\lambda}{2  (\sqrt{y} + \im |\xi|)^{\lambda/2} (\sqrt{y}-\im|\xi|)^{\lambda/2}} \\
                                         & =  \frac{(\sqrt{y} + \im |\xi| )^{\lambda/2}}{2 (\sqrt{y} - \im |\xi|)^{\lambda/2}} + \frac{(\sqrt{y} - \im |\xi|)^\lambda}{2  (\sqrt{y} + \im |\xi|)^{\lambda/2}},     
\end{align*}
 and therefore we obtain
\begin{align*}
&      \int\limits_{0}^\infty f\left(x +\xi^2 \right)  2T_\lambda \left( \frac{\sqrt{x+4}}{\sqrt{x+4 + \xi^2}} \right) \d \xi  \\ 
  & =   \int\limits_{0}^\infty f\left(x + \xi^2  \right)  \frac{(\sqrt{x+4} + \im \xi )^{\lambda/2}}{2 (\sqrt{x+4} - \im \xi)^{\lambda/2}} \d \xi\\
	& \qquad + \int\limits_{0}^\infty f\left(x + \xi^2  \right)  \frac{(\sqrt{x+4} - \im \xi )^{\lambda/2}}{2 (\sqrt{x+4} + \im \xi)^{\lambda/2}} \d \xi \\
                       & =   \int\limits_{-\infty}^\infty f\left(x + \xi^2  \right)  \frac{(\sqrt{x+4} + \im \xi )^{\lambda/2}}{2 (\sqrt{x+4} - \im \xi)^{\lambda/2}} \d \xi.
\end{align*}
Similarly, we encounter 
\begin{align*}  T_\lambda \left( \frac{\sqrt{y + \eta^2}}{\sqrt{y}} \right)& =  \frac{(\sqrt{y + \eta^2} + |\eta| )^\lambda + (\sqrt{y+\eta^2} -  |\eta|)^\lambda}{2y^{\lambda/2}}  \\
                                        & = \frac{(\sqrt{y + \eta^2} + |\eta| )^\lambda + (\sqrt{y+\eta^2} -  |\eta|)^\lambda}{ 2 (\sqrt{y+\eta^2} +  |\eta|)^{\lambda/2} (\sqrt{y+\eta^2} -  |\eta|)^{\lambda/2}} \\
                                         & =  \frac{(\sqrt{y+\eta^2} +  |\eta| )^{\lambda/2}}{ 2(\sqrt{y+ \eta^2} - |\eta|)^{\lambda/2}} + \frac{(\sqrt{2x+2+ \eta^2} -|\eta|)^{\lambda/2}}{ 2 (\sqrt{y+ \eta^2} + |\eta|)^{\lambda/2}},     
\end{align*}
and therefore we obtain
\begin{align*}
 &     - \frac{1}{\uppi}   \int\limits_{0}^\infty f\left(x +\eta^2 \right)  T_\lambda \left( \frac{\sqrt{x+4+\eta^2}}{\sqrt{x+4 }} \right) \d \eta \\
                                &= -  \frac{1}{\uppi}       \int\limits_{-\infty}^\infty f\left(x + \eta^2 \right)   \Re   \frac{(\sqrt{x+4+\eta^2} +  \eta )^{\lambda/2}}{ (\sqrt{x+4+ \eta^2} - \eta)^{\lambda/2}}  \d \eta .\qedhere
\end{align*}
\end{proof}

\begin{remark}[Common alternatives for $A_0$ and $\widehat{A}_0$]
Often one finds the following alternative definition for $A_0$: 
\begin{align*} A_0 f(x)  \underset{t= x + \xi^2} =  \int\limits_{x}^\infty \frac{f(t)}{\sqrt{t-x}}   \d t   \end{align*}
and
\[ \widehat{A}_0 f(t) \underset{w= t + \eta^2} = \frac{-1}{\uppi} \int\limits_{t}^\infty\frac{ f\left(w\right)}{\sqrt{w-t}}   \d w ,\]
see e.g. \cite{Hejhal1}*{Proposition 4.1, page 15}, \cite{Kubota}*{Theorem 5.3.1,page 56}, \cite{Iwaniec:Spectral}*{page 33}, etc.
\end{remark}

The operator $A_\lambda$ is precisely the Abel transform on the group level, which was introduced in Section~\ref{section:Abel}. 
\begin{lemma}[Relation to the Abel transform]\label{lemma:relabel}
Let $n\geq 0$. Define for $\phi \in \mSH(\GL_2(\bR), \rho_n)$ and $x \geq 0$, the unique compactly supported, smooth function $\Phi_\phi$ on $[0, \infty)$ with  
\[ \Phi_\phi \left( x^2 + x^{-2} -2 \right) :=    \phi\left( \sma x& 0 \\ 0 & x^{-1} \smz \right).\]
\index{$\Phi_\phi \left( x^2 + x^{-2} -2 \right) :=  \phi\left( \sma x& 0 \\ 0 & x^{-1} \smz \right)$}
Then, we obtain
\[ A_n \Phi_\phi\left(  x^2 + x^{-2}  -2 \right) :=  \mA_\rho \phi\left( \sma x & 0 \\ 0 &x^{-1} \smz \right).\]  
\end{lemma}
\begin{proof}
Consider the homeomorphism
\[ A^+ \xrightarrow\cong [2, \infty), \qquad g =\sma \e^{x} & 0 \\ 0 & \e^{-x} \smz \mapsto \tr g^\dagger g = \e^{2x} + \e^{-2x}.\]
An element $\phi$ in $\mSH(\GL_2(\bR), \rho_n)$ can be written by the Cartan decomposition as
\[    \phi \left( g  \right)  =   \Phi_\phi\left(  \frac{\tr g^\dagger g}{ |\det g|} - 2 \right) T_n( \cos(\theta_{x,t})) \cdot \tr \rho_n \left( \sma \sign(\det g) & 0 \\ 0 & 1 \smz \right)/2, \]
for some smooth compactly supported function $\Phi_\phi: [0, \infty) \rightarrow \bC$, and
where $ \sma \cos(\theta_{x,t}) & \sin(\theta_{x,t}) \\ -\sin(\theta_{x,t}) & \cos(\theta_{x,t}) \smz a =g$ for some symmetric matrix $a$. 
Recall that
\[\tr \Res_{\SO}  \rho_n ( \kappa(\theta_{x,t})) = T_n( \cos(\theta_{x,t})), \qquad \kappa(\theta_{x,t}) = \sma \cos(\theta_{x,t}) & \sin(\theta_{x,t}) \\ -\sin(\theta_{x,t}) & \cos(\theta_{x,t}) \smz.\]
For $x \geq 0$, we have 
\begin{align}\label{eq:diegleichung}   \phi \left( \sma x & t \\ 0 &x^{-1} \smz \right)  =  \Phi_\phi\left(  x^2+ x^{-2} +t^2 - 2\right) 2T_n( \cos(\theta_{x,t})). \end{align}
We obtain that
\[ \cot \theta_{x,t} = \frac{x + 1/x}{t}.\]
 We have accordingly
\[  T_n( \cos(\theta_{x,t})) = T_n\left( \left( 1+    \frac{t^2}{(x + 1/x)^2} \right)^{-1/2} \right) = T_n\left(  \sqrt{\frac{(x + 1/x)^2}{t^2 +(x + 1/x)^{2} }} \right) .  \]
So we have achieved
\begin{align*}
             \mA_\rho \phi\left( \sma x& 0 \\ 0 &x^{-1} \smz \right) & = \int\limits_{\bR} \phi \left( \sma x & t \\ 0 &x^{-1} \smz \right)  \d t \\ 
          & = \int\limits_{\bR}  \Phi( x^2+ x^{-2} +t^2-2) T_n\left(  \sqrt{\frac{(x + 1/x)^2}{t^2 +(x + 1/x)^{2} }} \right)  \d t  \\
           & =     \int\limits_{\bR}  f( x^2+ x^{-2} +t^2-2) T_n\left(  \sqrt{\frac{x^2 + 2 +  x^{-2}}{t^2 +x^2 + 2 +x^{-2} }} \right)  \d t \\
            &  =    2 \int\limits_{0}^\infty  f( x^2+ x^{-2}-2 +t^2 ) T_n\left(  \sqrt{\frac{x^2 + 2 +  x^{-2}}{t^2+x^2 + 2 +x^{-2} }} \right)  \d t \\
            &= A_n f\left( x^2+ x^{-2}-2 \right). 
\end{align*}
\end{proof}
We introduce new notation cf. \cite{Hejhal1}*{page 15-16} and \cite{Hejhal2}*{page 385-386}:
\begin{defn}[$Q_\phi$, $g_\phi$ and $h_\phi$]\label{defn:realfunctions}
Let $n \geq 0$. Define for $\phi \in \mSH(\GL_2(\bR), \rho_n)$
\begin{itemize}
 \item  the smooth, compactly supported function $\Phi_\phi \in \Ccinf([0, \infty))$
       \[  \Phi_\phi\left( \e^{2x} + \e^{-2x} - 2  \right) =  \phi\left( \sma \e^{x} & 0 \\ 0 &\e^{-x} \smz \right),\]
 \item  the smooth, compactly supported function $Q_\phi \in \Ccinf([0, \infty))$ via\index{$Q_\phi= A_n \Phi_\phi$} 
       \[ Q_\phi = A_n \Phi_\phi,\]
 \item  the smooth, even,  compactly supported function $g_\phi \in \Ccinf(\bR)^{\textup{even}}$ via\index{$g_\phi(x)  = Q_\phi\left( \e^{x} + \e^{-x} - 2 \right)$} 
             \[  g_\phi(x)  = Q_\phi\left( \e^{x} + \e^{-x} - 2 \right),\]
 \item  the even, entire function $h_\phi$, which is a Schwartz function on every line $\im y + \bR$, via \index{$h_\phi(x) = \int\limits_{\bR} g_\phi(u) \e^{\im r u} \d u$.}            
$$h_\phi(x) = \int\limits_{\bR} g_\phi(u) \e^{\im r u} \d u.$$
\end{itemize}
\end{defn}
For the remainder of this section, we will explicitly express all local real distributions occurring in the Arthur trace formula in terms of $h_\phi$ and $g_\phi$ only.  Note that since we are dealing with $\GL(2)$ instead of $\SL(2)$, some minor modifications appear.

\section{The character of the infinite-dimensional representations}
Let $\mu = (\mu_1, \mu_2)$ be a pair of algebraic one-dimensional representations, i.e., $\mu_j$ is either the trivial representation or the sign representation of $\{ \pm 1 \}$. 
Define
\[ \mu : \B (\bR) \rightarrow \bC, \qquad \mu\left(\sma a & * \\0 & b \smz \right) =  \mu_1(a) \mu_2(b),\]
and write
\[ (\mu,s): \B (\bR) \rightarrow \bC, \qquad (\mu,s) \left(\sma a & * \\0 & b \smz \right) = \left| a/b\right|^s \mu_1(a) \mu_2(b).\]
We obtain the following formula of Theorem~\ref{thm:jacquettrace} for the character distribution of $\mJ(\mu,s)$
\[\tr \mJ(\mu, s) \phi  = \int\limits_{\M(\bR)}  \mA_\rho\phi(m) \mu(m) \Delta(m)^{s+1/2} \d m.\]
Here $\mJ(\mu, s)$ is defined (see Section~\ref{section:jacquet}) as the right regular representation on the space of smooth functions $f: \GL_2(\bR) \rightarrow \bC$, which satisfy
 \[ f \left( \sma a& * \\ 0 & b\smz g \right) =       \left| \frac{a}{b}\right|^{s+1/2} \mu_1(a) \mu_2(b) f(g).\]
We compute the algebraic character distribution, which comes in a slightly different normalization than the unitary character distribution. The difference can be read from the Plancherel formula.
\begin{proposition}[Relation of $h_\phi$ to the Jacquetn modules]\label{prop:realjacquet}
Extend $\rho_n$ to a representation of $\O(2) \Z(\bR)$ with algebraic central character $\chi$. Let $\mu =(\mu_1, \mu_2)$ be an algebraic character of $\M(\bR)$. 
Let $\phi \in \mSH(\GL_2(\bR), \rho_n)$. The following identity holds for the character distribution for $s \in \bC$\footnote{The contragedient $\check{\rho}$ is isomorphic to $\rho$.} \footnote{The function $h_\phi$ is entire as the Fourier transform of a compactly supported function.} 
\[\tr \mJ(\mu, s) \tilde{\phi}  = \begin{cases}    h_\phi( \im s), &  \mu_1\mu_2 = \sign^n, \\ 
                                                        0, & \textup{otherwise}.\\
                          \end{cases}\]
\end{proposition}
\begin{proof}
The proof is straightforward and a specialization of the proof of Theorem~\ref{thm:jacquettrace}. We choose the ordinary Lebesgue measure $\dr r$ on $\bR$, and set $\dpr r =\,\dr r/ |r|$ on $\bR^\times$. We let $\O(2)$ carry a unit Haar measure $\d k$. We have fixed the unique Haar measure $\d g$ on $\GL(2,\bR)$, such that
\[ 2\int\limits_{\bR^\times} \int\limits_{0}^\infty \int\limits_{\infty}^\infty \int\limits_{\O(2)}  f\left( \sma z & 0 \\ 0 & z \smz \sma a & 0 \\ 0 & a^{-1} \smz \sma 1 & x \\ 0 & 1 \smz k\right) \d k\,\dr r \dpr a \dpr z\,= \int\limits_{\GL_2(\bR)} f(g) \d g. \]
Let $\chi$ be the central character of $\mu$, then $\Ccinf( \GL_2(\bR),\chi')$ acts on the algebraic representation $\mJ(\mu, s)$ via a section $\tilde{\phi} \in \Ccinf(\GL_2(\bR))$ with
\[ \int\limits_{\Z(\bR)} \overline{\chi'(z)} \tilde{\phi}(zg) \d z = \phi(g)\]
by 
\[ \mJ(\mu,s, \phi) : f \in \mJ(\mu,s) \mapsto \mJ(\mu,s, \phi)(f)  (g) = \int\limits_{\GL_2(\bR)} \tilde{ \phi}(x) f(gx) \d x.\]
Note that this definition is independent of the specific section. We separate the integral according to the Iwasawa decomposition
\begin{align*}  & \int\limits_{\GL_2(\bR)} \tilde{ \phi}(g^{-1}x) f(x) \d x \\ 
&= 2    \int\limits_{\O(2)} \int\limits_{\bR^\times} \int\limits_{0}^\infty \int\limits_{\infty}^\infty   \tilde{\phi}\left( g^{-1} \sma z & 0 \\ 0 & z \smz \sma a & 0 \\ 0 & a^{-1} \smz \sma 1 & x \\ 0 & 1 \smz k\right) \chi(z)  \left| a \right|^{2s+1} \,\dr r \,\dpr a\, \dpr z\,f(k)  \d k. \\
& = 2  \int\limits_{\O(2)}  \int\limits_{0}^\infty \int\limits_{\infty}^\infty   \phi\left( g^{-1} \sma a & 0 \\ 0 & a^{-1} \smz \sma 1 & x \\ 0 & 1 \smz k\right)   \left| a \right|^{2s+1} \,\dr r \, \dpr a \, f(k)  \d k.
\end{align*}
Recall $\mu_1(a) = \mu_2(a) = 1$ for $a > 0$. By the Iwasawa decomposition for $\GL_2(\bR)$, we can identify $\mJ(\mu,s)$ with a subspace of $\mL^2(\O(2))$. This is merely the definition of being an algebraic representation, plus the fact that every representation of $\O(2)$ occurs at most with multiplicity one in $\Res_{\O(2)}\mJ(\mu,s)$. The identification happens on the level of complex vector spaces, and is not true on the level of $\GL_2(\bR)$-representations. Hence the operator $\mJ(\mu, s, \phi)$ acts by a kernel transformation with kernel
\begin{align*}K_\phi &: \O(2) \times \O(2) \rightarrow \bC, \\ K_\phi&(k_1,k_2) = \int\limits_{0}^\infty \int\limits_{\infty}^\infty   \phi\left( k_1^{-1} \sma a & 0 \\ 0 & a^{-1} \smz \sma 1 & x \\ 0 & 1 \smz k_2\right)   \mu_1(a) \mu_2(a) \left| a \right|^{2s+1} \,\dr r\, \dpr a\, \dpr z. \end{align*} 
This operator has continuous kernel, and is a Hilbert-Schmidt operator. The operator can be rewritten as a finite linear combination of positive trace class operators (see the proof of Theorem~\ref{thm:jacquettrace}). The trace is computed
\[ \tr \mJ(\mu,s,\phi) =2 \int\limits_{\O(2)}   \int\limits_{0}^\infty \int\limits_{\infty}^\infty   \phi\left( k^{-1} \sma a & 0 \\ 0 & a^{-1} \smz \sma 1 & x \\ 0 & 1 \smz k\right)   \mu_1(a) \mu_2(a) \left| a \right|^{2s+1} \,\dr r \,\dpr a \, \d k.\]
Write as in Proposition~\ref{prop:Kexp}:
\[           \int\limits_{\O(2)}           \phi\left( k^{-1} g k\right)   \d k =  \sum\limits_{\rho} \phi^\rho(g) \]
for some unique $\phi^\rho \in \mH(G, \rho)$, and we have by the Schur orthogonality relations that the character distribution vanishes
\[ \tr \mJ(\mu,s,\phi^\rho) = 0\]
if the representation $\check{\rho} \cong  \rho$ is not contained in $\Res_{\O(2)} \mJ(\mu, s)$. The definitions otherwise yield 
\begin{align*}   \tr \mJ(\mu,s)\phi^\rho & =    2   \int\limits_{0}^\infty \mA_\rho  \phi^\rho\left(  \sma a & 0 \\ 0 & a^{-1} \smz\right)    \left| a \right|^{2s}  \dpr a.
\end{align*}
If $\phi \in \mSH(\GL_2(\bR), \rho) \subset \mH(\GL_2(\bR), \rho) \subset \Ccinf(\GL_2(\bR), \chi)$, we can further identify 
\begin{align*}
     \tr \mJ(\mu,s)\phi & =     2  \int\limits_{0}^\infty \mA_n  \Phi_{\phi} \left(  a^2 + a^{-2} -2\right)     a^{2s}  \dpr a \\
                               & =     2   \int\limits_{0}^\infty Q_{\phi} \left(  a^2 + a^{-2} -2\right)     a^{2s}  \dpr a \\
                               & \underset{a = \e^{x}} =  2        \int\limits_{-\infty}^\infty Q_{\phi} \left(  \e^{2x} + \e^{-2x} -2\right)    \e^{2xs} \,\dr x \\
                               & =  2  \int\limits_{-\infty}^\infty g_{\phi} \left(  2x \right)    \e^{2xs} \,\dr x \\
                               & =    \int\limits_{-\infty}^\infty g_\phi\left( x \right)    \e^{xs} \,\dr x = h_\phi( \im s). \qedhere
\end{align*}
\end{proof}

\begin{corollary}[Character of DS]\label{cor:realDS}
The character distribution of the discrete series representations vanishes on $\mSH(\GL_2(\bR), \rho_0)$. Let $\phi \in \mSH(\GL_2(\bR), \rho_n)$ for $n >0$, then 
\[ \tr D_k(\mu_1,\mu_2)( \phi)  =\begin{cases}
 h_{\phi}\left( \im \frac{k-1}{2}  \right), & n \geq k, \, n - k \textup{ even}, \\ 
0 , & \textup{otherwise}. \end{cases}\]
\end{corollary}
\begin{proof}
 This follows from the fact that discrete series representations are subquotients of the parabolic inductions, from the formula for their character distributions provided by Proposition~\ref{prop:realjacquet} and from the analysis of their $K$-isotypes. Recall from Theorem~\ref{thm:realKtype} that for $k$ odd
 \[ \Res_{\O(2)}  \Res_{\O(2)} D_k(\mu,\mu) =  \Res_{\O(2)} \mJ\left(\mu, \mu, \frac{k-1}{2}\right) \ominus \bigoplus_{n=0 \atop n \textup{ odd}}^k \rho_n,\]
and for $k$ even
\[   D_k(\mu_1,\mu_2) =  \Res_{\O(2)} \mJ\left(\mu_1, \mu_2, \frac{k-1}{2}\right)   \ominus \left( \mu \circ \det  \oplus \bigoplus_{n=1 \atop n \textup{ even}}^k \rho_n \right),\]
The vanishing assumption follows from Corollary~\ref{cor:charvanish}, and the non-vanishing cases from Proposition~\ref{prop:realjacquet}.
\end{proof}

\begin{defn}
Let $G$ be a locally compact group, and let $\pi$ be an irreducible, unitary representation of $G$ with central character $\chi$. A pseudo coefficient of $\pi$ (modulo character twists) is a function $\phi \in \Ccinf(\GL_2(\bR), \overline{\chi})$ such that for all irreducible $\pi'$
\[ \tr \pi'(\phi) = \begin{cases}  1, &  \chi  \otimes \pi \cong \pi', \textup{ for some one-dimensional representation }\chi , \\  0 , &\textup{otherwise}.\end{cases}\] 
\end{defn}

\begin{corollary}[Construction of pseudo coefficient for DS]\label{cor:DSpseudo}  \mbox{}
 Let $k \geq 2$. For each pair of elements $\phi_k \in \mSH(\GL_2(\bR), \rho_k)$  and $\phi_{k+2} \in \mSH(\GL_2(\bR), \rho_{k+2})$ with $h_{\phi_k} = h_{\phi_{k+2}}$ and $h_{\phi_{k+2}}( \im \frac{k-1}{2} )=1$  , the function 
 \[ \phi_{D,k} =        -   \phi_k + \phi_{k+2}  \]
is a pseudo coefficient of the discrete series representation of weight $k$.
\end{corollary}
\begin{proof}
This follows from the linearity of the character distributions and Corollary~\ref{cor:realDS}. 
\end{proof}
Note that the above corollary solves a problem implicitly posed in \cite{KnightlyLi}*{Remark, page 214}. In a similar vein, this has already been exploited in \cite{Hejhal1}*{page 459}, though in a completely different language. The existence of pseudo coefficients for square integrable representations of reductive Lie groups is known \cite{ClozelDelorme}.

\begin{remark}[Non-existence of pseudo matrix coefficient of PS]
There exists no pseudo coefficient for the principal series representations. A rigorous proof is too expensive here, since we have decided to work with $\mSH(\GL_2(\bR), \rho)$ instead of $\mH(\GL_2(\bR), \rho)$. The argument is however then simple. It follows immediately from the fact that the principal series representations $\mJ(1,1,s)$ all contain the same $K$-types and that the functions of type $h_{\phi}$ are entire.
\end{remark}

\section{The character of one-dimensional representations}
\begin{proposition}\label{prop:realoned}
Let $\chi$ be a one-dimensional algebraic representation of $\GL_2(\bR)$.
We have for $\phi \in \mSH(G, \rho_n)$ with $n \geq 0$ that
\[     \tr \chi(\phi) =   \begin{cases} 
  h_\phi(\im /2), & n = 0, \\
0, & \textup{else}. 
 \end{cases}\]
\end{proposition}
\begin{proof}
The only $K$-type in $\chi$ is $\Res_{\O(2)} \chi$, hence the vanishing result. For the remaining formula, we recognize that
\begin{align*}
        \tr \chi(\phi)  &= \int\limits_{\Z(\bR) \backslash \GL_2(\bR)}  \phi(g) \chi(g) \d g,      \\
                   &=   2\int\limits_{0}^\infty \int\limits_{\bR} \int\limits_{\O(2)}  \phi \left(  \sma a & 0 \\ 0 & a^{-1} \smz \sma 1 & x \\ 0 & 1 \smz k\right) \chi(k) \d k \,\dr x\, \dpr a, \\
                    & = 2  \int\limits_{0}^\infty    a^{-1} \mA_\rho\phi \left(  \sma a & 0 \\ 0 & a^{-1} \smz \right)  \dpr a, \\
                     & = 2  \int\limits_{0}^\infty    a^{-1} Q_\phi \left( a^2 + a^{-2} -2 \right)  \dpr a,  \\
                     & \underset{a =\e^{x/2}} =          \int\limits_{-\infty}^\infty    \e^{-x/2} Q_\phi \left( \e^{x} - \e^{-x}-2 \right) \,\dr x,  \\
                      & =              \int\limits_{-\infty}^\infty    \e^{-x/2} g_\phi \left( x \right) \,\dr x \\ 
                 &  = h_\phi(\im /2) .\qedhere
\end{align*}
\end{proof}
\begin{corollary}\label{cor:realDSoned}
 Let $\chi$ be a one-dimensional algebraic representation of $\GL_2(\bR)$. Let $\phi_{D,k}$ be a pseudo matrix coefficient of the discrete series representation of weight $k$, then
\[ \tr \chi(\phi_{D,k}) = \begin{cases} 1, & k = 2, \\ 0 , & k \geq 3 . \end{cases}\]
\end{corollary}

\section{The identity distribution}\label{section:realuni} 
The following computations can be found in \cite{Hejhal1}*{page 440--441, page 459}, \cite{Hejhal2}*{Eq. 6.35, page 397}. The computations are a specialization of the Plancherel formula for the $\rho$-isotypic component of the right regular representation $\mL^2( \GL_2(\bR) , \chi)$.
\begin{proposition}[The identity distribution]\label{prop:realone}
Let $n \geq 0$. For $\phi \in \mSH(\GL_2(\bR), \rho_n)$, we obtain 
\begin{align*} \phi\left( \sma 1 &0 \\ 0 & 1\smz \right)  &=   \sum\limits_{l=1\atop  n-1 = l \bmod 2}^{n} \frac{l}{4 \uppi} h_\phi\left( \frac{\im l}{2} \right), \\ 
                                                                             & \qquad +  \frac{1}{4 \uppi}  \int\limits_{-\infty}^\infty r h_\phi(r)  \begin{cases}  \tanh(\pi r) \d r, & n\textup{ even},\\ 
                                                                                                                                                       \coth(\pi r) \d r, & n \textup{ odd}.
                                                                                                                                                      \end{cases} 
\end{align*}
\end{proposition}
The appearance of the finite sums is due to the existence of discrete series representations.
\begin{corollary}\label{cor:realDSone}
\[ \phi_{D,k} (1) =   \frac{k-1}{4 \uppi}.\]
\end{corollary}
\begin{proof}[Proof of the proposition]
Appeal to the Abel inversion formula
\[ \phi(1) = \Phi_\phi(0) = \hat{A}_n A_n \Phi_\phi'.\]
The computation works for real $\lambda$. If we apply the formulas for $\widehat{A}_\lambda$, this yields that
\begin{align*} \phi(1) &= - \frac{1}{\uppi}       \int\limits_{-\infty}^\infty \Phi_\phi' \left( \eta^2 \right)  \left( \frac{\sqrt{4+\eta^2}-\eta}{\sqrt{4+\eta^2}  + \eta} \right)^{\lambda/2}  \d \eta  \\ 
                              &\underset{\xi = \e^{y/2} - \e^{-y/2}}= - \frac{1}{\uppi}       \int\limits_{-\infty}^\infty A_n \Phi_\phi' \left(2 \cosh(y) - 2 \right) \left( \frac{2\cosh(y/2)-2\sinh(y/2)}{2\cosh(y/2)  + 2 \sinh(y/2)} \right)^{\lambda/2} \cosh(y/2) \d y  \\
                              & =          - \frac{1}{2 \uppi}          \int\limits_{-\infty}^\infty g_\phi' \left(y \right) \left( \frac{2 \e^{y/2}}{2\e^{-y/2}} \right)^{\lambda/2} \frac{ \cosh(y/2)}{2 \sinh(y)}\d y \\
                              & =                  \int\limits_{-\infty}^\infty g_\phi' \left(y \right)   \frac{\e^{-\lambda y/2}}{\e^{y/2}- \e^{-y/2}} \d y.
\end{align*}
First consider $|\lambda| <1$ \cite{Hejhal2}*{page 397,(C)}, \cite{BMP4}*{page 88(7)}. Apply the Fourier inversion formula 
\begin{align*}
- \frac{1}{2 \uppi}     \int\limits_{-\infty}^\infty g_\phi' \left(y \right)   \frac{\e^{-\lambda y/2}}{\e^{y/2}- \e^{-y/2}} \d y   =  \frac{1}{4 \uppi} \int\limits_{-\infty}^\infty  r h_\phi(r) \frac{\sinh(\uppi r)}{ \cosh(2 \uppi r) + \cos(\pi \lambda)} \d r.
\end{align*}
Note that $g_\phi'$ is an odd function and vanishes at $x=0$. For $\lambda=1$, we encounter somewhat differently \cite{Hejhal1}*{page 440--441}
\begin{align*}
     - \frac{1}{2 \uppi}     \int\limits_{-\infty}^\infty g_\phi' \left(y \right)   \frac{\e^{- y/2}}{\e^{y/2}- \e^{-y/2}} \d y &=       \frac{1}{2 \uppi}     \int\limits_{-\infty}^\infty r h_\phi(r) \coth(\pi r) \d y.
\end{align*}
We create a shortcut to reduce the general case to the case where $0 \leq | \lambda| \leq 1$.
Note that for general $\lambda$, the following identity holds
\begin{align*}
           \int\limits_{- \infty}^\infty g_\phi' \left(y \right)   \frac{\e^{\lambda y/2} - \e^{(\lambda-2) y/2} }{\e^{y/2}- \e^{-y/2}} \d y  &=   \int\limits_{-\infty}^\infty g_\phi' \left(y \right)  \e^{(\lambda-1) y/2} \d y  \\
                                                                                                                                                                             & =    - \frac{\lambda-1}{2}\int\limits_{-\infty}^\infty g_\phi \left(y \right)  \e^{(\lambda-1) y/2} \d y\\
                                                                                                                                                                        &= -  \frac{(\lambda-1) }{2}   h_\phi \left( \im (\lambda-1) /2 \right) . 
\end{align*}
Choose now $|\lambda'| \leq 1$ such that $\lambda - \lambda'$ an integer. Repeating the above procedure yields the result by induction:
\begin{align*}
 - \frac{1}{2 \uppi}          \int\limits_{-\infty}^\infty g_\phi' \left(y \right)   \frac{\e^{-\lambda y/2}}{\e^{y/2}- \e^{-y/2}} \d y = & - \frac{1}{2 \uppi}          \int\limits_{-\infty}^\infty g_\phi' \left(y \right)   \frac{\e^{-\lambda' y/2}}{\e^{y/2}- \e^{-y/2}} \d y\\ 
& + \frac{1}{4 \uppi}  \sum\limits_{l=1 \atop l \textup{ odd} }^{| \lambda|}  \frac{(\lambda-1) }{2}   h_\phi \left( \im (\lambda-1) /2 \right).\qedhere
\end{align*} 
\end{proof}

\section{The parabolic distributions}
The parabolic distribution \cite{Gelbart}*{Proposition 1.2, page 47} is related to the local zeta integral 
\[ \zeta_{\GL_2(\bR)}(\phi, s) := \int\limits_{\bR^\times} \int\limits_{\O(2)}  \phi\left(k^{-1}\sma 1 & x \\ 0 & 1\smz k\right) \left| x \right|^s \d k\, \dpr x.\]
We compute its value and the value of its derivative at $s=1$ for every element in $\mSH(\GL_2(\bR), \rho)$. We have for $\phi \in \mSH(\GL_2(\bR), \rho)$ by definition:
\[  \zeta_{\GL_2(\bR)}(\phi, 1) = \int\limits_{\bR^\times} \int\limits_{\O(2)}  \phi\left(k^{-1}\sma 1 & x \\ 0 & 1\smz k\right) \left| x \right| \d k\, \dpr x = \mA_\rho \phi(1).\]
The computation of its derivative at $s=1$ requires some work, however. Fortunately, equivalent computations can be found in \cite{Hejhal2}.
\begin{proposition}[The local Zeta integral and its derivative at $s=1$]\label{prop:realpara}
For $\phi \in \mSH(\GL_2(\bR), \rho_n)$ with $n\geq 0$, we obtain
\begin{align*}
\frac{\zeta_{\GL_2(\bR)}(\phi, s)}{\zeta_\bR(1)} &= g_\phi(0)
\end{align*}
and 
\begin{align*} \frac{\partial}{\partial s} \Bigg|_{s=1}  \frac{\zeta_{\GL_2(\bR)}(\phi, s)}{\zeta_\bR(s)} & =  \left(   - \frac{\gamma_0}{2} + \frac{\log(\uppi)}{2}   \right) g_\phi(0) +  \frac{1}{4} h_\phi(0)   - \frac{1}{2 \uppi} \int\limits_{\bR} h_\phi(r) \frac{\Upgamma'}{\Upgamma} (1+\im r)\d r  \\
& \qquad  + \int\limits_{0}^\infty  \frac{g_\phi(u)}{\e^{u/2}-\e^{-u/2}} \left( 1 - \cosh\left(\frac{n u}{2}\right) \right) \d u,
\end{align*}
where for $\Re\;z >0$
\[ \Upgamma(z) =\int\limits_{0}^\infty t^z \e^{-t} \d^\times t.\]
 is the standard gamma function and 
\[ \gamma_0 =  \lim\limits_{N \rightarrow \infty} \left(   \sum\limits_{k \leq N} \frac{1}{k} - \log(N) \right) \approx 0.577 \]
is the Euler-Mascheroni constant.
\end{proposition}
\begin{corollary}[Parabolic distributions]\label{cor:realDSpara}
\begin{align*}
  \zeta_{\GL_2(\bR)}  (1, \phi_{D,n}) &=0, \\
\partial_s \Big|_{s=1} \frac{ \zeta_{\GL_2(\bR)}  (s, \phi_{D,n})}{\zeta_\bR(s,1)} &= 1.  
\end{align*}
\end{corollary}
\begin{proof}
Consider Proposition~\ref{prop:realpara} for $n$ and $n-2$. Note that we have identities $h_{\phi_n} = h_{\phi_{n-2}}$ and  $g_{\phi_n} = g_{\phi_{n-2}}$ by definition.
\begin{align*}\partial_s \Big|_{s=1} \frac{ \zeta_{\GL_2(\bR)}  (s, \phi_{D,n})}{\zeta_\bR(s,1)} &=  \int\limits_{0}^\infty \frac{g_{\phi_n}(x)}{\e^{x/2} - \e^{-x/2}} (\cosh(nx/2-x) - \cosh(nx/2))\dr  x\\
                                                                                                                                  &=  \int\limits_{0}^\infty g_{\phi_n}(x) \sinh(x(n-1)/2)\dr  x\\
                                                                                                                                   &=  \int\limits_{-\infty}^\infty g_{\phi_n}(x) \e^{x(n-1)/2}\,\dr r x = h_\phi(\im (n-1)/2) = 4 \uppi.\qedhere
\end{align*} 
\end{proof}
\begin{proof}
We copy the explicit form of $\zeta_\bR$ from Tate's thesis \cite{Tate:Thesis}*{page 317}
\[ \zeta_\bR(s) = \uppi^{-\frac{s}{2}} \Upgamma\left( \frac{s}{2} \right) , \qquad \partial_s  \zeta_\bR(s) = \frac{1}{2} \left( \frac{\Upgamma'}{\Upgamma}\left( \frac{s}{2} \right) - \log(\uppi) \right) \zeta_\bR(s).\]
We compute the special values of the Euler gamma function
\[ \Gamma(1/2) =\sqrt{\uppi}, \qquad \Gamma'(1/2) = - \sqrt{\uppi}( \gamma_0 +2 \log(2)).\]
We evaluate accordingly
\[  \zeta_\bR(1)^{-1} = 1 , \qquad          \partial_s \Bigg|_{s=1}  \zeta_\bR(s)^{-1} = \frac{-\zeta_\bR'(1)}{\zeta_\bR(1)^2} = \frac{\gamma_0 + \log(\uppi)}{2}  + \log(2). \]
The first equation then follows from the definitions
\[  \frac{\zeta_{\GL_2(\bR)}(\phi, 1)}{\zeta_\bR(1)} = \int\limits_{\bR^\times} \int\limits_{\O(2)}  \phi\left(k^{-1}\sma 1 & x \\ 0 & 1\smz k\right) \left| x \right| \d k \dpr x \underset{def}= \mA_\rho \phi(1) \underset{def}= Q_\phi(0) \underset{def}= g_\phi(0).\] 
In the light of \cite{Hejhal2}, we will assume without loss of generality that $\Phi_\phi$ is a real-valued function. The computations are valid for $\lambda$ real-valued. 
 \begin{align*}
\zeta_{\GL_2(\bR)}(\phi, s) = \int\limits_{\bR^\times}  \phi\left( \sma 1 & a \\ 0 & 1 \smz \right) |a|^s \d^\times a &  = 2 \int\limits_{0}^\infty  \phi\left( \sma 1 & a \\ 0 & 1 \smz \right) |a|^s \d^\times a  \\
                   & =    2   \int\limits_{0}^\infty \Phi_\phi(a^2) |a|^s  T_\lambda \left( \sqrt{\frac{4}{a^2+4}} \right) \d^\times a.
\end{align*}
We appeal to Equation~\ref{eq:tn2} for $a>0$
\[ T_\lambda \left( \sqrt{\frac{4}{a^2+4}} \right) = \frac{(2 + \im a)^\lambda + (2 - \im a)^\lambda}{2(4+ a^2)^{\lambda/2}} =\Re \frac{(2 + \im a)^\lambda}{(4+a^2)^{\lambda/2}} = \Re \exp( \im \lambda \arg(2 + \im a)).\]
We end up with
\begin{align*} \int\limits_{\bR^\times}  \phi\left( \sma 1 & a \\ 0 & 1 \smz \right) |a|^s \d^\times a& =       2  \Re  \int\limits_{0}^\infty \Phi_\phi(a^2) |a|^s   \exp( \im \lambda \arg(2 + \im a)) \d^\times a   \\
                                                                                                                                        &\underset{x= a^2}=           \int\limits_{0}^\infty \Phi_\phi(x) |x|^{s/2}    \Re\exp( \im \lambda \arg(2 + \im \sqrt{x})) \dpr x .
 \end{align*}
For the derivative at $s=1$, we obtain 
\begin{align*} R&:=\partial_s\Big|_{s=1} \zeta_{\GL_2(\bR)}(\phi, s) =  \partial_s\Big|_{s=1}   \int\limits_{0}^\infty \Phi_\phi(x) |x|^{s/2}    \Re\exp( \im \lambda \arg(2 + \im \sqrt x)) \dpr x \\
   & =   \frac{1}{2} \int\limits_{0}^\infty \Phi_\phi(x) x^{-1/2}  \log(x)   \Re\exp( \im \lambda \arg(2 + \im \sqrt x))\,\dr x.  
\end{align*}
Mimicking the computations from the \cite{Hejhal2}*{pages 407--408} yields
\begin{align*}
R=\log(2) g_\phi(0) - \frac{1}{\pi} \Re  \int\limits_{0}^\infty Q_\phi'(\eta) \int\limits_{-\uppi/2}^{\uppi/2} \left( \frac{\sinh(w(\eta)) +\im \cos \vartheta}{\cosh(w(\eta))+ \sin \vartheta} \right)^\lambda \log\left( \frac{\cos \phi}{\sinh(w(\eta))} \right) \d \vartheta \d \eta,
\end{align*}
where $w(\eta) := \textup{arcsinh} ( 2 / \sqrt{\eta})$. 
Set
\begin{align*}
  J_1(\eta)&:= \Re   \int\limits_{-\uppi/2}^{\uppi/2} \left( \frac{\sinh(w(\eta)) +\im \cos \vartheta}{\cosh(w(\eta))+ \sin \vartheta} \right)^\lambda \log\left( \cos \phi \right) \d \vartheta, \\
  J_2(\eta)&:=  \Re  \int\limits_{-\uppi/2}^{\uppi/2} \left( \frac{\sinh(w(\eta)) +\im \cos \vartheta}{\cosh(w(\eta))+ \sin \vartheta} \right)^\lambda \log\left( \frac{1}{\sinh(w(\eta))} \right) \d \vartheta    \\&   \underset{\textup{\cite{Hejhal2}*{page 408}}}= -\uppi \log(\e^w + \e^{-w}) .
\end{align*}
The computation continues in this notation on \cite{Hejhal2}*{page 410--411}:
 \begin{align*}
 \gamma_0 g_\phi(0) + R =  & ( \gamma_0 + \log(2)) g_\phi(0) - \frac{1}{\pi} \Re  \int\limits_{0}^\infty Q_\phi'(\eta) \left( J_1(\eta) + J_2(\eta) \right) \d \eta \\
& = -\log(2) g_\phi(0) +  \frac{1}{4} h_\phi(0) - \frac{1}{2 \uppi} \int\limits_{\bR} h_\phi(r) \frac{\Upgamma'}{\Upgamma} (1+\im r)\d r  \\
& \qquad  + \int\limits_{0}^\infty  \frac{g_\phi(u)}{\e^{u/2}-\e^{-u/2}} \left( 1 - \cosh\left(\frac{\lambda u}{2}\right) \right) \d u  .
\end{align*}
These computations sum up to the second equation:
\begin{align*}  \frac{\partial}{\partial s} \Bigg|_{s=1}  \frac{\zeta_{\GL_2(\bR)}(\phi,s)}{\zeta_\bR(s)} &=  \left(   \frac{\gamma_0 + \log(\uppi)}{2}  + \log(2)  \right) g_\phi(0) + R  \\
                                                                                                                                           &=  \left(   - \frac{\gamma_0}{2} + \frac{\log(\uppi)}{2}   \right) g_\phi(0) +  \frac{1}{4} h_\phi(0)   - \frac{1}{2 \uppi} \int\limits_{\bR} h_\phi(r) \frac{\Upgamma'}{\Upgamma} (1+\im r)\d r  \\
& \qquad  + \int\limits_{0}^\infty  \frac{g_\phi(u)}{\e^{u/2}-\e^{-u/2}} \left( 1 - \cosh\left(\frac{\lambda u}{2}\right) \right) \d u . \qedhere
\end{align*}         
\end{proof}
\section{The hyperbolic distributions}\label{sec:realhyper}
We will consider an element $\gamma \in \GL_2(\bR)$. By definition, an element $\gamma$ is hyperbolic if its characteristic polynomial splits into two distinct factors over $\bR$. Then $\gamma$ is conjugate inside $\GL_2(\bR)$ to an element
\[  \sma \alpha & 0 \\ 0 & \beta \smz\]
for $\alpha, \beta \in \bR^\times$ with $\alpha \neq \beta$.
The stabilizer of an element $\gamma$ is the subgroup $\M(\bR) = \sma * & 0 \\0 & * \smz$  of diagonal matrices in $\GL_2(\bR)$. The Arthur trace formula associates two types of distributions to a hyperbolic element (see \cite{Gelbart}*{Proposition 1.1, page 46}), namely, a hyperbolic (orbital) integral for $\phi \in \Ccinf(\GL_2(\bR), \overline{\chi})$
\[ J_\gamma ( \phi ) = \int\limits_{\M(\bR) \backslash \GL_2(\bR)} \phi ( g^{-1}  \gamma g ) \d \dot{g},\]
and a weighted hyperbolic integral for $\phi \in \Ccinf(\GL_2(\bR)$
\[  J_\gamma^H ( \phi) =  \int\limits_{\M(\bR) \backslash \GL_2(\bR)}  \phi ( g^{-1}  \gamma g ) w_H(g) \d \dot{g}.\]
The quotient measure $\textup{d} \dot{g}$ is defined here as the unique right invariant Radon measure on $\M(\bR) \backslash \GL_2(\bR)$, such that
\[ \int\limits_{ \GL_2(\bR)} f(g) \d g  =   \int\limits_{\M(\bR) \backslash \GL_2(\bR)}  \int\limits_{ \M(\bR)} f(mg) \d m \d \dot{g} .\]
Here, $H$ is defined via the Iwasawa decomposition and the modular character 
\[  H(g) = \Delta_{\B(\bR)} (b), \qquad g = bk, \qquad b \in \B(\bR),\, k \in \O(2),\]
and the weight
\[ w_H(g) = H(w_0g ) + H(g), \qquad w_0 = \sma 0 & -1 \\ 1 & 0 \smz.\]
We have the following invariance properties for all $z \in \Z(\bR)$ and $g \in \GL_2(\bR)$
\[ J_\gamma( \phi) =  \chi(z) J_{z\gamma}( \phi) =    J_{g^{-1}\gamma g}( \phi) = J_{\gamma}( \phi^g), \]
 and for $z \in \Z(\bR)$ and $k \in \O(2)$
\[ J_\gamma^H( \phi) =  \chi(z) J_{z\gamma}^H( \phi) = J_{\gamma}^H( \phi^k).\] 
The hyperbolic integral and the weighted hyperbolic orbital integral can be expressed in terms of the function $g_\phi$.
\begin{proposition}[The hyperbolic distributions]\label{prop:realhyper}
Let $\gamma = \sma \alpha & 0 \\0 & 1 \smz$ be a hyperbolic element in $\GL_2(\bR)$, i.e., $\alpha \neq 1$.  Fix $\phi \in \mSH(\GL_2(\bR), \rho_n)$ for $n \geq 0$.
\begin{itemize}
\item 
 The hyperbolic integral and the weighted hyperbolic integral vanish for $\alpha <0$.
 \item  For $\alpha>0$, the hyperbolic integral of $\gamma$ evaluates to
\[ J_\gamma(\phi) = \frac{1}{\cosh( \log \alpha / 2)}    g_\phi( \log |\alpha|/2) .\]
\item   For $\alpha>0$, the hyperbolic integral and the weighted hyperbolic integral vanish if $g_v$ is supported in $[- |\log \alpha |/2, | \log \alpha|/2 ]$.
\end{itemize}
\end{proposition}
Readers interested in a more explicit formula for the weighted hyperbolic integral should manipulate Equation~\ref{eq:weightedhyper} accordingly. For the Weyl law, the vanishing assertion and the next corollary are fully satisfying. Note that the first point would not hold if we had decided to argue with the irreducible representations $1$ and $\sign \circ \det$ in place of the reducible representation $\rho_0$.
\begin{corollary}
For every constant $C>0$, there exists a pseudo coefficient $\phi_{D,k}$ of the discrete series representation $D_k(\mu_1,\mu_2)$ such that
\[ J_\alpha ( \phi_{D,k} ) = J_\alpha^H ( \phi_{D,k} ) = 0 \]
if $\alpha < 0$, or if $\alpha > 0$ and $| \log |\alpha| | > C$.  
\end{corollary}
 \begin{proof}
The hyperbolic integral of $\gamma$ is absolutely convergent \cite{Rao}. We subdivide the proof of this proposition in several lemmas.
\begin{lemma}[Explicit form of $w_H$]\label{lemma:wH}
\[ w_H\left( \sma m_1 & 0  \\  0 & m_2 \smz \sma 1 & x \\ 0 & 1 \smz k \right) = w_H\left(  \sma 1 & |x| \\ 0 & 1 \smz  \right) = \log \left| 1+ |x|^2 \right| .\]
\end{lemma}
\begin{proof}
We have by definition for $b \in \B(\bR)$:
 \[ w_H\left( b k \right) = w_H\left( b \right). \]
We have for real $t$ the following matrix decomposition:
\begin{align*}
   w_0 \sma 1 & t \\ 0 & 1\smz =   \sma 0 & -1 \\ 1 & t \smz =      \sma \frac{1}{\sqrt{1+t^2}} & \frac{-t}{\sqrt{1+t^2} } \\ 0 & \sqrt{1+t^2}\smz \sma \frac{t}{\sqrt{1+t^2} } & \frac{-1}{\sqrt{1+t^2} }\\ \frac{1}{\sqrt{1+t^2} } & \frac{t}{\sqrt{1+t^2} }  \smz.
\end{align*}
We have for $u \in \{ \pm 1 \}$ and positive $t= |x| >0$ the decomposition
\[ w_0 \sma m_1 & u m_1 t \\ 0 & m_2 \smz  = \sma  m_2 & 0 \\ 0 & u m_1 \smz   \sma 0 & -1 \\ 1 & t \smz \sma  u^{-1} & 0 \\ 0 & 1 \smz.\]
We complete the computation
\begin{align*} w_H\left( \sma m_1 & 0  \\  0 & m_2 \smz \sma 1 & x \\ 0 & 1 \smz k \right)  &= \log \Delta_B\left( \sma  m_1 & 0 \\ 0 & m_2 \smz \right)   -  \log  \Delta_B \left( \sma m_2  \frac{1}{\sqrt{1+t^2}} & \frac{-t}{\sqrt{1+t^2} } \\ 0 &u m_1 \sqrt{1+t^2}\smz \right)  \\
                                                                                                                                   & =     - \log \left|  \frac{1}{\sqrt{1+t^2}} \right| + \log \left|  \sqrt{1+t^2} \right| =  \log \left| 1+ |x|^2 \right|.\qedhere
\end{align*}
\end{proof}
\begin{lemma}
  For $\phi \in \mSH(\GL_2(\bR), \rho_k)$, we obtain
 \[ J_\gamma (\phi) = \int\limits_{\N(\bR)} \phi(n^{-1}\gamma n) \d n, \]
and
\[   J^H_\gamma (\phi) = \int\limits_{\N(\bR)} \phi(n^{-1}\gamma n) w_H(n) \d n.\]
 \end{lemma}

\begin{proof}
 We choose the ordinary Lebesgue measure $\dr r$ on $\bR$, and set $\dpr r =\,\dr r/ |r|$ on $\bR^\times$. We let $\O(2)$ carry a unit Haar measure $\textup{d} k$. We have fixed the unique Haar measure $\textup{d} g$ on $\GL(2,\bR)$ and $\M(\bR)$, such that for all $f \in \GL_2(\bR)$
\[ \int\limits_{\M(\bR)} \int\limits_{\N(\bR)} \int\limits_{\SO(2)}  f\left( mn k\right) \textup{d} k \d n \d m = \int\limits_{\GL_2(\bR)} f(g) \d g. \]
The quotient measure is, according to the Iwasawa decomposition, definable via the property: For all continuous, compactly supported functions $F : \M(\bR) \backslash \GL_2(\bR) \rightarrow \bC$, we have that
\[   \int\limits_{\M(\bR) \backslash \GL_2(\bR)}  F\left( g\right) \d \dot{g}       =       \int\limits_{\infty}^\infty \int\limits_{\SO(2)}  F\left(  \sma 1 & x \\ 0 & 1 \smz k\right) \d k \,\dr r.\] 
The (weighted) hyperbolic integral is computed for $\phi \in \mSH(G, \rho_k)$ as
\[ \int\limits_{\M(\bR) \backslash \GL_2(\bR)}  \phi(g^{-1} m  g)  w_j(g) \d \dot{g} =  \int\limits_{\infty}^\infty \int\limits_{\SO(2)}  \phi\left(k^{-1} \sma 1 & -x \\ 0 & 1 \smz m \sma 1 & x \\ 0 & 1 \smz  k\right)   \d k w_j\left(\sma 1 & x \\ 0 & 1 \smz \right)\,\dr r .\]
where the weight \( w_j \) is either constant or $w_H$. Of course, this requires some knowledge about $w_H$ as given by Lemma~\ref{lemma:wH}.  Note that for $\phi \in \mSH(\GL_2(\bR),\rho_k)$, we get the lemma immediately, since 
\[  \int\limits_{\SO(2)}  \phi\left(k^{-1} g  k\right)   \d k    = \phi(g).\qedhere\]
\end{proof}
The computation without weight is simple.
\begin{lemma}[The invariant hyperbolic integral]
  For $\phi \in \mSH(\GL_2(\bR), \rho)$, we obtain for $t >0$ and $+t \neq 1$
 \[ J_\gamma (\phi) = \int\limits_{\N(\bR)} \phi\left(n^{-1}\sma \pm t & 0 \\ 0 & 1 \smz n\right) \d n = \frac{1}{| t^{1/2} \mp t^{-1/2}|}   \tr \rho\left( \sma \pm 1 & 0 \\ 0 & 1 \smz \right) g_\phi( \log t/2). \]
 \end{lemma}
\begin{proof}
We have
 \[ \sma 1 & -x \\ 0 & 1 \smz \sma \pm t & 0 \\ 0 & 1 \smz  \sma 1 & x \\ 0 & 1 \smz  = \sma   \pm t & 0 \\ 0 & 1 \smz       \sma 1 & ( 1 \mp t^{-1}) x \\ 0 & 1 \smz.\]
This yields
\[ \int\limits_{\N(\bR)} \phi\left(n^{-1}\sma \pm t & 0 \\ 0 & 1 \smz n\right) \d n  = \frac{1}{| 1 \mp t^{-1} | } \int\limits_{\bR} \phi\left( \sma \pm t & 0 \\ 0 & 1 \smz  \sma 1 & x \\ 0 & 1 \smz \right) \d x. \]
By definition, the function $\phi \in \mSH(G, \rho)$ satisfies
\[    \phi\left( \sma \pm t & 0 \\ 0 & 1 \smz  \sma 1 & x \\ 0 & 1 \smz \right) =\tr \rho\left( \sma \pm 1  & 0 \\ 0 & 1 \smz \right) \phi\left( \sma \sqrt{t} & 0 \\ 0 & \sqrt{t}^{-1} \smz  \sma 1 & x \\ 0 & 1 \smz \right),\]
 and we can express the integral in terms of the Abel transform
\begin{align*}\int\limits_{\bR} \phi\left( \sma \pm t & 0 \\ 0 & 1 \smz  \sma 1 & x \\ 0 & 1 \smz \right) \d x &= |t|^{-1/2} \mA_\rho \phi\left( \sma \sqrt{t} & 0 \\ 0 & \sqrt{t}^{-1} \smz  \right)\\ &= Q_\phi(t + t^{-1} + 2) = g_\phi( \log t /2).\qedhere\end{align*}
\end{proof}

\begin{lemma}
Assume that $g_\phi$ is supported in $[ - 2\log|\alpha|, 2 \log|\alpha|  ]$, then \[J_\gamma(\phi) =0.\]
\end{lemma}
\begin{proof}
We obtain in the same way as above
\begin{align*} &\int\limits_{\bR} \phi\left( \sma 1 & - x \\ 0 & 1 \smz \sma \pm t & 0 \\ 0 & t^{-1} \smz \sma 1 &  x \\ 0 & 1 \smz \right)  \log( 1 + x^2) \,\dr x \\
& =2   \int\limits_{0}^\infty \phi\left(  \sma  t &  (t \mp t^{-1}) x \\ 0 & t^{-1}  \smz \right) \log( 1+ x^2)  \,\dr x \\
&\underset{\lambda = (t \mp t^{-1})}  = 2   \int\limits_{0}^\infty \Phi_\phi\left( t^2 +t^{-2} +\lambda^2x^2 -2\right) \log( 1+ x^2)   \,\dr x \\
&\underset{c = t^2 +t^{-2}}  = \frac{2}{\lambda}  \int\limits_{0}^\infty \Phi_\phi\left( c + \lambda^2 x^2 -2  \right) \log( 1+ x^2/ \lambda^2)   \,\dr x. \end{align*}
Abel inversion yields
\begin{align*}
                                           =            - \frac{4}{\uppi \lambda}  \int\limits_{0}^\infty     \int\limits_{0}^\infty &Q_\phi'\left(c +  x^2 -2 + \eta^2 \right)   T_\lambda \left( \sqrt{\frac{c +  x^2  + 2+ \eta^2}{c + x^2 +2}} \right)\,\dr \eta   \log( 1+ x^2 / \lambda^2) \, \,\dr x \\
                                          =            - \frac{8}{\uppi^2 \lambda}  \int\limits_{0}^\infty &     Q_\phi'\left(c +  R^2 -2 \right)  \\
                       & \int\limits_{0}^{\uppi/2} T_\lambda \left( \sqrt{\frac{c +  R^2  + 2 }{c + R^2 \sin(\theta)^2+2}} \right)  \log( 1+ R^2 \sin(\theta)^2 / \lambda^2)\,  \d \theta \, 2R\,\dr R  \\
                                          =            - \frac{8}{\uppi^2 \lambda}  \int\limits_{0}^\infty   &   Q_\phi'\left(c +  R -2 \right) \\&\int\limits_{0}^{\uppi/2} T_\lambda \left( \sqrt{\frac{c +  R  + 2 }{c + R \sin(\theta)^2+2}} \right) \d \theta   \log( 1+ R \sin(\theta)^2 / \lambda^2) \, \,\dr R \\
                                           =            - \frac{8}{\uppi^2 \lambda}  \int\limits_{c}^\infty   &   Q_\phi'\left(  R -2 \right) \\
&\int\limits_{0}^{\uppi/2} T_\lambda \left( \sqrt{\frac{R  + 2 }{c \cos(\theta)^2+  R \sin(\theta)^2+2}} \right) \d \theta    \end{align*}
  \begin{align}
            \label{eq:weightedhyper}                         & \qquad \qquad   \log( 1+ R \sin(\theta)^2 / \lambda^2 + c \cos(\theta) / \lambda^2)   \,\dr R    \\
                 \underset{R =\e^x + \e^{-x}}      =         \frac{8}{\uppi^2 \lambda} & \int\limits_{2 \log t}^\infty      g_\phi'\left(  x \right)   \dots \d x = 0.  \qedhere
 \end{align}
\end{proof}
    \end{proof}

\section{The intertwiner and its derivative}
Let $\mu_j$ be an algebraic one-dimensional representation of $\bR^\times$ for $j=1,2$. Let $\mu = (\mu_1, \mu_2)$ be the associated one-dimensional representation of $\M(\bR)$. Set $w_0 = \sma 0 & -1 \\ 1 & 0 \smz$ and $\mu^{w_0} = (\mu_2, \mu_1)$. 

We define the intertwiner by
\begin{align*} \mM(\mu, s) &\colon \mJ(\mu, s) \rightarrow \mJ(\mu^{w_0} , -s),  \\
                    \mM(\mu, s)& f(g) \coloneqq \int\limits_{\N(\bR)} f( w_0 n g) \d n. 
\end{align*}
By the Iwasawa decomposition, the smooth function $F\in \mJ(\mu,s)$ is uniquely determined by its value on $\SO(2)$, since by definition
\[  F\left( \sma a & * \\ 0 & b \smz k\right) = |a/b|^{s+1/2} \mu_1(a) \mu_2(b) F(k).\]
A canonical basis is given by the representation theory of $\SO(2)$ by 
\[ F_n\left( \sma a & * \\ 0 & b \smz k\right) = |a/b|^{s+1/2} \mu_1(a) \mu_2(b) \epsilon_n(k),\]
such that $n$ runs through all non-zero even (odd) integers if $\mu_1 = \mu_2$ ($\mu_1\neq \mu_2$). We want to allow the parameter $s$ to vary, so we will write
\[ F_{\mu, s,n} \in \mJ(\mu, s).\]
The operator $\mM(\mu,s)$ is in this sense a meromorphic function in $s \in \bC$.

\begin{proposition}[\cite{Bump:Auto}*{Proposition 2.3, page 230}]\label{prop:realinter}
In the above notation, we have an identity
\[ \mM(\mu, s) F_{\mu, s, \pm n} =    \im^n  \frac{ \sqrt{\uppi} \Upgamma(s) \Upgamma(s+1/2)}{ \Upgamma(s+1/2+n/2) \Upgamma(  s +1/2-n/2)}F_{\mu^{w_0},-s,\pm n}.\]
\end{proposition}
\begin{proof}
We will demonstrate an elementary derivation of this formula, but the identity can also be derived via Whittaker models \cite{Bump:Auto}.

The operator $\mM(\mu, s)$ is an intertwiner. Every $\SO(2)$ representation occurs with multiplicity one. By Schur's lemma, there exists a unique complex value $\lambda(s, \mu, n) \in \bC$, such that
$$\mM(\mu, s) F_{\mu, s, n} = \lambda(s,\mu,n) F_{\mu^{w_0},-s, n}.$$
\newcommand{\Beta}{\textup{Beta}}
We introduce the Beta-function as a quotient of Euler $\Upgamma$ functions:
\[ \Beta(x,y) \coloneqq \frac{\Upgamma(x)\Upgamma(y)}{ \Upgamma(x+y)}.\]
A short computation gives the exact value for $s>0$ real, and follows for all $s$ by uniqueness of analytic continuation:
\begin{align*}
\mM(\mu,s)F_{\mu, s, \pm n}(1) &\underset{\textup{def.}}= \int\limits_{\bR} F_{\mu, s, \pm n}\left( \sma 0 & -1 \\ 1 & t \smz \right) \,\dr t \\
 & = \int\limits_{\bR} F_{\mu, s, \pm n}\left(   \sma \frac{1}{\sqrt{1+t^2}} & \frac{-t}{\sqrt{1+t^2} } \\ 0 & \sqrt{1+t^2}\smz \sma \frac{t}{\sqrt{1+t^2} } & \frac{-1}{\sqrt{1+t^2} }\\ \frac{1}{\sqrt{1+t^2} } & \frac{t}{\sqrt{1+t^2} }   \smz \right)\,\dr t\\
 & = \int\limits_{\bR}    \frac{(t \pm \im)^n}{(1+  \im t)^{1/2+s}  (1-\im t)^{1/2+s}(t^2+1)^{n/2}} \,\dr t\\
 & =\im^n \int\limits_{\bR}  \frac{1}{(1 \pm \im t)^{+s+n/2+1/2}(1 \mp \im t)^{+s+1/2-n/2}}\dr t\\
 & \underset{\textup{\cite{table}*{8.381, page 909}}}=    \im^n  \frac{ 2\uppi 2^{-2s}}{(2s) \Beta(s+1/2+n/2,s +1/2-n/2)}. 
\end{align*}
Let us now appeal to the Legendre duplication formula
\[ \Upgamma(2z)\Upgamma\left( \frac{1}{2} \right)  =    2^{2z-1} \Upgamma\left( z+ \frac{1}{2} \right)   \Upgamma(z), \qquad \Upgamma\left( \frac{1}{2} \right) = \sqrt{\uppi}.\]
We transform
\begin{align*}  \im^n  & \frac{ \uppi 2^{-2s}  \Upgamma ( 2 s+1  )}{s \Upgamma(s+1/2+n/2) \Upgamma( s +1/2-n/2)}   \\
&            \underset{2s\Upgamma(2s) =\Upgamma(2s+1)}=  \im^n \frac{2 \uppi  2^{-2s} \Upgamma(2s)   }{\Upgamma(s+1/2+n/2) \Upgamma( s +1/2-n/2)} \\
& \underset{\textup{Legendre:} z = s} = \im^n \frac{ \sqrt{\uppi}  2^{-2s+1}  2^{-2s-1} \Upgamma(s+1/2) \Upgamma(s)   }{\Upgamma(s+1/2+n/2) \Upgamma( s +1/2-n/2)} \\
& = \im^n \frac{ \sqrt{\uppi}  \Upgamma(s) \Upgamma(s+1/2)   }{\Upgamma(s+1/2+n/2) \Upgamma( s +1/2-n/2)} . \qedhere \end{align*}
\end{proof}

\section{The elliptic distributions}\label{sec:realell}
Let $\gamma = \sma a & b \\ c & d \smz$ be an elliptic element in $\GL_2(\mathbb{R})$, i.e., the characteristic polynomial is irreducible over $\bR$.
\begin{lemma}\label{lemma:ellR}
Let $\gamma = \sma a & b \\ c & d \smz$ be an elliptic element in $\GL_2(\mathbb{R})$. The determinant of $\gamma$ is positive and $\gamma$ is conjugate inside $\GL_2(\bR)$ to an element
\begin{align*}   \sma \sqrt{\det \gamma} & 0 \\ 0 & \sqrt{\det \gamma} \smz k\end{align*}
for some $k$ in $\SO_2( \bR)$. Its centralizer is isomorphic to $\bR^\times \cdot \SO(2)$.
\end{lemma}
 \begin{proof}
The roots of an irreducible polynomial over $\mathbb{R}$ come in pairs $z, \overline{z}$. 
The characteristic polynomial of $\gamma$ is of degree two, and its roots  $z, \overline{z}$ are precisely the eigenvalues of $\gamma$, i.e., $\det \gamma = z \overline{z} = |z|^2>0$.
Set $z =  x +  \im y = \sqrt{\det \gamma} ( \cos \theta + \im \sin \theta) $, where $\theta = \textup{arctan} \frac{x}{y}$, then $\gamma$ is conjugate to 
\begin{align*} \gamma' =  \sma x & y \\ -y & x \smz =     \sma \sqrt{\det \gamma}& 0 \\ 0 & \sqrt{\det \gamma} \smz \sma \cos \theta & \sin \theta \\ - \sin \theta & \cos\theta \smz,\end{align*}
since the characteristic polynomials of $\gamma$ and $\gamma'$ coincide.
\end{proof}
The distribution associated to an elliptic element is the integral
\[ J_\gamma(\phi) =  \int\limits_{G_\gamma \backslash \GL_2(\bR)} \phi(g^{-1}xg) \d \dot{g} .\]

\begin{proposition}[The elliptic orbital integral]\label{prop:realelliptic}
 Consider the elliptic element
\[ k_\theta = \sma \cos \theta &  \sin \theta \\ -\sin \theta & \cos \theta \smz \in \SO(2).\]
We compute the orbital integral for $\phi \in \mSH(\GL_2(\bR), \rho_n)$ with $n\geq 0$
 \begin{align*}
J_\gamma(\phi) & = \im \frac{\e^{\im ( n -1) \theta}}{|\sin(\theta)|} \int\limits_{0}^\infty g_\phi(u) \frac{\cosh\left( \frac{n+1}{2} u \right) - \e^{\im \theta} \cosh\left( \frac{n-1}{2} u \right) }{\cosh(u) - 1 + 2 \sin^2(\theta)} \d u\\
&=\begin{cases}  \frac{1}{2 |\sin \theta|} \int\limits_{-\infty}^\infty h(r) \frac{\cosh( 2r(\uppi- \theta)) + \e^{\im n \uppi} \cosh(2r \theta)}{ \cosh(2 \uppi r) + 1}\,\dr r\\
                          + \sum\limits_{k=0}^{N} \frac{\im \e^{\im(n -1-2k)\theta}}{|\sin(\theta)|} h_\phi\left( \im \frac{n-1-2k}{2} \right), & n = 2N +2 \textup{ even}, \\
                          \frac{1}{2 |\sin \theta|} \int\limits_{-\infty}^\infty h(r) \frac{\sinh( (\uppi- 2\theta)r) }{ \sinh( \uppi r) }\,\dr r- \frac{\im h(0)}{2 |\sin(\theta)|}\\
                          + \sum\limits_{k=0}^{N} \frac{\im \e^{\im(n -1-2k)\theta}}{|\sin(\theta)|} h_\phi\left( \im \frac{n-1-2k}{2} \right) , &  n = 2N+1 \textup{ odd}. \end{cases}
 \end{align*}
\end{proposition}
\begin{corollary}[For $\phi_{D,k}$]\label{cor:realDSelliptic}
     \[ J_{\textup{ell}, \theta} ( \phi_{D,k} )  =\frac{2 \uppi \im \e^{\im (k-1) \theta}}{k |\sin(\theta)|}.\]
\end{corollary}
\begin{proof}[Proof of the corollary]
Assume the proposition. For odd weight $n \geq 3$, we have that 
\[ J_\gamma(\phi_{D,n})  = J_\gamma(\phi_n) - J_\gamma(\phi_{n-2}) =     \frac{\im \e^{\im (n-1) \theta}}{|\sin(\theta)|} h_\phi\left( \im \frac{n-1}{2} \right) =  \frac{2 \uppi \im \e^{\im (n-1) \theta}}{n |\sin(\theta)|}.\]
For even weight, the kernel satisfies
\[                   \frac{\cosh( 2r(\uppi- \theta)) + \e^{\im n \uppi} \cosh(2r \theta)}{ \cosh(2 \uppi r) + \cos(\uppi n)} =\frac{\cosh( 2r(\uppi- \theta)) + \e^{\im (n-2) \uppi} \cosh(2r \theta)}{ \cosh(2 \uppi r) + \cos(\uppi (n-2))}. \]
The same identity holds 
\[ J_\gamma(\phi_{D,n})   = \frac{\im \e^{\im (n-1) \theta}}{|\sin(\theta)|} h_\phi\left( \im \frac{n-1}{2} \right) =  \frac{2 \uppi \im \e^{\im (n-1) \theta}}{n |\sin(\theta)|}. \qedhere\]                               
\end{proof}
  \begin{proof}[Proof of the proposition]
 Set $G = \GL_2(\bR)$, $K =\O(2)$, $G_0 = \SL_2(\bR)$ and $K_0 = \SO(2)$. 
\begin{lemma} Let $\gamma$ be an element in $\Z(\bR)\SO(2)$, and let $\phi \in \mSH(\GL_2(\bR), \rho_n)$. We have the integral identity:
\begin{align*} & \int\limits_{K_0 Z \backslash G} \phi(g^{-1} \gamma g) \d \dot{g}  \\ & = 4 \uppi   \int\limits_{0}^\infty  \phi\left(  \sma \e^t  & 0 \\ 0 & \e^{-t} \smz  \gamma  \sma \e^{-t}  & 0 \\ 0 & \e^{t} \smz \right)  (\e^{2t} - \e^{-2t}) \,\dr t .
\end{align*}
\end{lemma}
\begin{proof}
 We choose the ordinary Lebesgue measure $\dr r$ on $\bR$, and set $\dpr r =\,\dr r/ |r|$ on $\bR^\times$. We let $\O(2)$ carry a unit Haar measure $\d k$. We have fixed the unique Haar measure $\d g$ on $\GL(2,\bR)$ and $\M(\bR)$, such that for all $f \in \GL_2(\bR)$
\[ 2\int\limits_{\bR^\times} \int\limits_{0}^\infty \int\limits_{\infty}^\infty \int\limits_{\O(2)}  f\left( \sma z & 0 \\ 0 & z \smz \sma a & 0 \\ 0 & a^{-1} \smz \sma 1 & x \\ 0 & 1 \smz k\right) \d k \,\dr r \,\dpr a \,\dpr z\,= \int\limits_{\GL_2(\bR)} f(g) \d g. \]
If we choose a normalized Haar measure on $\SO(2)$, we have that
\[ 2  \int\limits_{\O(2)} f(k)\d k =      \int\limits_{\SO(2)} f(k_0)\d k_0  +     \int\limits_{\SO(2)} f\left( \sma -1 & 0 \\ 0 & 1 \smz k_0\right)\d k_0.\] 
It is therefore clear that
\[ \int\limits_{\GL_2(\bR)} f(g) \d g = \sum\limits_{x \in \{ \pm 1\} }\int\limits_{\bR^\times} \int\limits_{0}^\infty \int\limits_{\infty}^\infty \int\limits_{\SO(2)}  f\left( \sma xz & 0 \\ 0 & z \smz \sma a & 0 \\ 0 & a^{-1} \smz \sma 1 & x \\ 0 & 1 \smz k_0 \right) \d k_0 \,\dr r \,\dpr a \,\dpr z,\]
and
\[ \int\limits_{Z_+ \backslash \GL_2(\bR)} f(g) \d g = \sum\limits_{x \in \{ \pm 1\} }\int\limits_{\bR^\times} \int\limits_{\infty}^\infty \int\limits_{\SO(2)}  f\left( \sma x & 0 \\ 0 & 1 \smz \sma a & 0 \\ 0 & a^{-1} \smz \sma 1 & x \\ 0 & 1 \smz k_0 \right) \d k_0 \,\dr r \,\dpr a .\]
We fix a unique Haar measure on $\SL_2(\bR)$, such that (\cite{DeEc}*{Theorem 11.1.3}) 
\[ \int\limits_{\SL_2(\bR)} g(x) \d x = \int\limits_{\bR^\times} \int\limits_{\infty}^\infty \int\limits_{\SO(2)}  g\left( \sma x & 0 \\ 0 & 1 \smz \sma a & 0 \\ 0 & a^{-1} \smz \sma 1 & x \\ 0 & 1 \smz k_0 \right) \d k_0 \,\dr r \,\dpr a.\]
We have the following decomposition (\cite{DeEc}*{Theorem 11.2.1})
\[     \int\limits_{\SL_2(\bR)} g(x) \d x =  2 \uppi \int\limits_{\SO(2)} \int\limits_{0}^\infty    \int\limits_{\SO(2)}   \phi\left( k_1 \sma \e^t  & 0 \\ 0 & \e^{-t} \smz k_2\right)  (\e^t - \e^{-t}) \,\dr t \d k_1 \d k_2.\]
We accordingly obtain for $\phi \in \mSH(G, \rho)$ that
\begin{align*}
 \int\limits_{K_0 Z \backslash G} \phi(g^{-1} \gamma g) \d \dot{g}  & =   \int\limits_{Z_+ K_0 \backslash G} \int\limits_{K_0} \phi(g^{-1}k^{-1} \gamma k g) \d \dot{g} \\
                                                                                               & =   \int\limits_{Z_+ \backslash G} \phi(g^{-1}k^{-1} \gamma k g) \d \dot{g} \\
                                                                                               =  \sum\limits_{x \in \{ \pm 1 \}}  2 \uppi \int\limits_{\SO(2)} \int\limits_{0}^\infty    \int\limits_{\SO(2)}   & \phi\left( k_1 \sma x\e^t  & 0 \\ 0 & \e^{-t} \smz k_2 \gamma k_2^{-1} \sma x\e^{-t}  & 0 \\ 0 & \e^{t} \smz k_1^{-1} \right)  (\e^t - \e^{-t}) \,\dr t \d k_1 \d k_2 \\
                                                                                               & = 4 \uppi   \int\limits_{0}^\infty  \phi\left(  \sma \e^t  & 0 \\ 0 & \e^{-t} \smz  \gamma  \sma \e^{-t}  & 0 \\ 0 & \e^{t} \smz \right)  (\e^{2t} - \e^{-2t}) \,\dr t. \qedhere
\end{align*}

\end{proof}

\begin{lemma}
For $\theta \in\bR$, define
\[  \kappa_\theta = \sma \cos \theta & \sin \theta \\ - \sin \theta & \cos \theta \smz. \]
 We have that
  \[  \phi\left(  \sma \e^{-t} & 0 \\ 0 & \e^{t}  \smz   \kappa_\theta  \sma \e^{t} & 0 \\ 0 & \e^{-t}  \smz \right) = \Phi_\phi(  \sin(\theta)^2 ( \e^{2t} - \e^{-2t})^2) \tr \rho (\kappa_\alpha), \]
where  
 \[ \frac{\tan \theta}{\tan \alpha} =\cosh(2t).\]
\end{lemma}
\begin{proof}
Let
\[  \kappa_\theta = \sma \cos \theta & \sin \theta \\ - \sin \theta & \cos \theta \smz. \]
We factor
\[ \gamma = \sma \e^{-t} & 0 \\ 0 & \e^{t}  \smz   \kappa_\theta  \sma \e^{t} & 0 \\ 0 & \e^{-t}  \smz   = \sma \cos \theta & \e^{-2t}\sin \theta \\ - \e^{2t}\sin \theta & \cos \theta \smz.\]
Choose $\alpha$ in such away that
\[  \sin \alpha \cos \theta =  \cos \alpha \sin \theta \cosh(2t) , \qquad  \frac{\tan \theta}{\tan \alpha} =\cosh(2t), \]
then 
\begin{align*} \gamma_{\theta, t}:= \sma \cos \theta & \e^{-2t}\sin \theta \\ - \e^{2t}\sin \theta & \cos \theta \smz \sma \cos \alpha & \sin \alpha \\ -\sin \alpha & \cos \alpha \smz 
= \sma * & \sin \alpha \cos \theta + \cos \alpha \sin \theta \e^{-2t} \\   \sin \alpha \cos \theta - \cos \alpha \sin \theta \e^{2t} & * \smz  
\end{align*}
is symmetric. Note that
\[  \gamma^\dagger \gamma =   \sma \cos \theta & -\e^{2t}\sin \theta \\  \e^{-2t}\sin \theta & \cos \theta \smz \sma \cos \theta & \e^{-2t}\sin \theta \\ - \e^{2t}\sin \theta & \cos \theta \smz, \]
so
\[  \tr \gamma^\dagger \gamma -2 =  2\cos^2(\theta) + \sin^2(\theta)  (\e^{4t} + \e^{-4t}) -2 = \sin(\theta)^2 ( \e^{4t} + \e^{-4t} -2).\qedhere\]
\end{proof}
To summarize our progress, we have achieved the analogue of \cite{Hejhal2}*{page 389}:
\begin{lemma}
    \[   \int\limits_{K_0 Z \backslash G} \phi(g^{-1} k_\theta g) \d \dot{g}  =     \frac{ \uppi}{|\sin(\theta)|}  \int\limits_{0}^\infty \e^{\im \lambda \arg( 2 \cos(\theta) + \im \sqrt{t+4 \sin^2(\theta)}}   \frac{\Phi_\phi(s )\,\dr s}{\sqrt{s+ 4 \sin^2(\theta)}} .  \]
\end{lemma}
\begin{proof}
 This proof is an exercise in calculus. We have so far produced 
\begin{align*}& \int\limits_{G_\gamma \backslash G} \phi(g^{-1} \gamma g) \d \dot{g} \\ &=4 \uppi  \int\limits_{0}^\infty \Phi_\phi(    \sin(\theta)^2 (\e^{4t} + \e^{-4t} - 2 ) ) T_\lambda( \cos( \arctan( \tan (\theta) / \cosh(2t))))  (\e^{2t} - \e^{2t}) \,\dr t. \end{align*}
Replace
\[        \cos( \arctan(x)) = \sqrt{1+x^2}^{-1}.\]
Substitution of $y = \cosh(2t)$ yields
 \begin{align*}     & \underset{y = \cosh(2t)}{=}  4 \uppi  \int\limits_{1}^\infty \Phi_\phi(  4\sin(\theta)^2 (y^2 - 1)) T_\lambda\left( 1/\sqrt{ 1+  \tan^2 (\theta) / y^2} \right)\,\dr y \\
                                       & \underset{s = y^2}{=}    4 \uppi  \int\limits_{1}^\infty \Phi_\phi(  4\sin^2(\theta)  (s -1)) T_\lambda\left( 1/\sqrt{ 1+  \tan^2 (\theta) / s} \right) \frac{\dr s}{\sqrt{s}} \\
                                      &=2  \uppi  \int\limits_{0}^\infty \Phi_\phi(4 \sin^2(\theta) s ) T_\lambda\left( 1/\sqrt{ 1+  \tan^2 (\theta) / (s+1)} \right)  \frac{\dr s}{\sqrt{s+1}} \\
                                      & = \frac{ \uppi}{2\sin(\theta)^2}  \int\limits_{0}^\infty \Phi_\phi(s ) T_\lambda\left( 1/\sqrt{ 1+  \tan^2 (\theta) /  (s/4 \sin^2(\theta) +1)} \right)  \frac{\dr s}{\sqrt{s/ 4 \sin^2(\theta) + 1}}  \\ 
                                       &  = \frac{ \uppi}{|\sin(\theta)|}  \int\limits_{0}^\infty \Phi_\phi(s ) T_\lambda\left( 1/\sqrt{ 1+  4 \cos^2 (\theta) /  (s + 4 \sin^2(\theta))} \right)  \frac{\dr s}{\sqrt{s+ 4 \sin^2(\theta)}}  .\label{eq:lasteq}
 \end{align*}
We appeal to~\ref{eq:tn2}
\begin{align*}
&T_\lambda\left( 1/\sqrt{ 1+  4 \cos^2 (\theta) /  (s + 4 \sin^2(\theta))} \right)   = T_\lambda\left( \sqrt{ \frac{s+4 \sin^2(\theta)}{s+4}} \right) \\
 & =  \frac{\left( \sqrt{s+4\sin^2(\theta)} + \sqrt{s+4 \sin^2(\theta) -s-4}\right) - \left( \sqrt{s+4\sin^2(\theta)} - \sqrt{s+4 \sin^2(\theta) -s-4}\right)^\lambda }{2(s+4)^{\lambda/2}} \\
 & =  \frac{\Re \left(   \sqrt{s+4\sin^2(\theta)} + \im 2 \cos(\theta) \right)^\lambda }{|s+4|^{\lambda/2}} \\
 & =    \e^{\im \lambda \arg( 2 \cos(\theta)  +  \im\sqrt{s+4\sin^2(\theta)})}. \qedhere \end{align*}
 \end{proof}
The expression in the last lemma is computed in terms of $g_\phi$ and $h_\phi$ on \cite{Hejhal2}*{page 389--396}. For similar computations, the reader can additionally consult \cite{Kubota}*{page 100--102}, \cite{Hejhal1}*{Remark 9.4, page 449} or \cite{Iwaniec:Spectral}*{Section 10.6, page 163}.
We will omit a repetition of this expensive computation and quote the results:
\begin{lemma}
 \begin{align*}
&   \frac{ \uppi}{|\sin(\theta)|}  \int\limits_{0}^\infty \e^{\im \lambda \arg( 2 \cos(\theta) + \im \sqrt{t+4 \sin^2(\theta)}}   \frac{\Phi_\phi(s )\dr s}{\sqrt{s+ 4 \sin^2(\theta)}}   \\
& = \im \frac{\e^{\im ( \lambda -1) \theta}}{|\sin(\theta)|} \int\limits_{0}^\infty g_\phi(u) \frac{\cosh\left( \frac{\lambda+1}{2} u \right) - \e^{\im \theta} \cosh\left( \frac{\lambda-1}{2} u \right) }{\cosh(u) - 1 + 2 \sin^2(\theta)} \d u\\
&=\begin{cases}  \frac{1}{2 |\sin \theta|} \int\limits_{-\infty}^\infty h(r) \frac{\cosh( 2r(\uppi- \theta)) + \e^{\im \lambda \uppi} \cosh(2r \theta)}{ \cosh(2 \uppi r) + \cos(\uppi \lambda)}\,\dr r\\
                          + \sum\limits_{k=0}^{N} \frac{\im \e^{\im(\lambda -1-2k)\theta}}{|\sin(\theta)|} h_\phi\left( \im \frac{\lambda-1-2k}{2} \right), & 2N+1 < \lambda < 2N+3, \\
                          \frac{1}{2 |\sin \theta|} \int\limits_{-\infty}^\infty h(r) \frac{\sinh( (\uppi- 2\theta)r) }{ \sinh( \uppi r) }\,\dr r- \frac{\im h(0)}{2 |\sin(\theta)|}\\
                          + \sum\limits_{k=0}^{N} \frac{\im \e^{\im(\lambda -1-2k)\theta}}{|\sin(\theta)|} h_\phi\left( \im \frac{\lambda-1-2k}{2} \right) , &  \lambda = 2N+1. \end{cases}
 \end{align*}
\end{lemma}
\begin{proof}
Look at \cite{Hejhal2}. For the first equality, compare the expression at the end of page 389 with the expression at the end of page 393. The second equality is given by equation (6.34a) and (6.34b) on page 396. 
\end{proof}
We have now completed the computation of the elliptic distribution.
\end{proof}
\chapter{Harmonic analysis on $\GL(2,\bC)$}

\section{Haar measure}

We endow $\bC$ with the Haar measure $\textup{d}^+_\bC z$, which is twice the ordinary Lebesgue measure.
The set  \( \{ |z| \leq 1 \} \) has measure $2\pi$. Equivalently, we have for $f \in \Ccinf(\bC)$
\[ \int\limits_{\bC} f(z) \d^+_\bC  z = \int\limits_{\bR} \int\limits_{\bR} f(x +\im y)  2 \d^+_\bR x \d^+_\bR y,\]
where the Haar measure $\textup{d}^+_\bR x$ on $\bR$ is as in the preceding section.
The group $\bC^\times$ is endowed with a norm 
\[  \left| z \right|_\bC = z \overline{z} =  x^2 + y^2, \qquad z = x +\im y,\]
which is the square of the ordinary absolute value. We endow $\bC^\times$ with the Haar measure 
\[ \d^\times_\bC z =   \left| z \right|_\bC^{-1} \d^+_\bC z. \]

The following integral formulas are direct consequences of these selections.
Let $\d r$ and $\d \vartheta$ be the Lebesgue measure on the real line. For all $f \in \Ccinf(\bC), h \in \Ccinf(\bC^\times)$, we have 
\begin{align} \int\limits_{\bC} f(z) \d_\bC^+ z&=    \int\limits_{\bC^\times} f(z) \left|z \right|_\bC \d^\times_\bC z  , \qquad    \left|z \right|_\bC = |z|^2. \\
                                                     &=  2\int\limits_{0}^{2 \uppi}  \int\limits_{0}^{\infty}   f(r \e^{i\vartheta})    r   \d r \d \vartheta \\ 
\int\limits_{\bC^\times} h(z)  \d^\times_\bC z  & = 2 \int\limits_{0}^{2 \uppi} \int\limits_{0}^{\infty}  h(r \e^{i\vartheta})     \frac{\d r}{r} \d \vartheta.
\end{align}
These can be easily verified by computing the measure of the unit circle. We consider the locally compact group $\GL_2(\bC)$, with its closed subgroups
\begin{align*}
 \N(\bC) & \coloneqq \left\{ \sma 1 & x \\ 0 & 1 \smz : x \in \bC \right\}, \\
 \M(\bC) & \coloneqq  \left\{ \sma \alpha & 0 \\ 0 & \beta  \smz : \alpha, \beta \in \bC^\times \right\}, \\
   \Z(\bC) & \coloneqq  \left\{ \sma z & 0 \\ 0 &z  \smz : z \in \bC^\times \right\}, \\
\B(\bC) & \coloneqq  \left\{ \sma \alpha &x \\ 0 & \beta  \smz : \alpha, \beta \in \bC^\times, x \in \bC \right\},
\end{align*}
and its closed compact subgroups $\U(2)$ and $\SU(2)$. Only the group $\B(\bC)$ is not unimodular.

The compact groups $\U(2)$ and $\SU(2)$ are endowed with the unit Haar measures.
We endow $\N(\bC)$ with the Haar measure of $\bC^+$ by identifying
\[ \bC^+ \xrightarrow\cong \N(\bC), \qquad x \mapsto \sma 1 & x \\ 0 & 1 \smz. \]
We endow $\Z(\bC)$ with the Haar measure of $\bC^\times$ via
\[ \bC^\times \xrightarrow\cong \Z(\bC) , \qquad \alpha  \mapsto \sma \alpha & 0 \\ 0 & \alpha  \smz.\]
For the group $\M(\bC)$, we define the Haar measure via $\M(\bC) \cong \bC^\times \times \bC^\times$:
\begin{align*}
 \int\limits_{\M(\bC)} f(m) \d m = \int\limits_{\bC^\times}    \int\limits_{\bC^\times} f\left( \sma z \beta& 0 \\ 0 & z \smz \right) \d^\times_\bC z \d^\times_\bC \beta \qquad \textup{ for } f \in \mL^1(\M(\bC)).
\end{align*}
We fix the unique left invariant Haar measure on $\B(\bC)$ with
\[            \int\limits_{\B(\bC)} f(b) \d_l b =  \int\limits_{\M(\bC)} \int\limits_{\N(\bC)}  f(mn) \d n \d m \qquad \textup{ for } f \in \mL^1(\B(\bC)), \]
and the unique right invariant Haar measure with
\[            \int\limits_{\B(\bC)} f(b) \d_r b =  \int\limits_{\N(\bC)}  \int\limits_{\M(\bC)} f(nm) \d m \d n \qquad \textup{ for } f \in \mL^1(\B(\bC)). \]
The modular character is then given as
\[ \Delta_{\B(\bR)} : \sma a & * \\ 0 & b \smz \mapsto |a/b|.\]
The Iwasawa decomposition $\GL_2(\bC) =\B(\bC) \SU(2)$ yields the existence of a unique Haar measure on $\GL_2(\bC)$, such that
 \begin{align}
 \int\limits_{\GL_2(\bC)} f(g) \d g  & =\int\limits_{\M(\bC)} \int\limits_{\N(\bC)} \int\limits_{\SU(2)}  f\left( mn k\right) \d k \d n \d m \label{eq:complexmnk}\\
                                                & =\int\limits_{\bC^\times} \int\limits_{\bC^\times}\int\limits_{\bC} \int\limits_{\SU(2)}  f\left(\sma z a& 0 \\ 0 & z \smz \sma 1 & x \\ 0 & 1\smz k\right) \d k \d_\bC^+ x \d_\bC^\times z \d_\bC^\times a  \nonumber\\
                                               & =\iiint\limits_{0}^\infty  \iiint\limits_{0}^{2 \uppi}     \int\limits_{\SU(2)}                f\left(\sma \e^{\im(\theta_1+\theta_2)} r_1 r_2& 0 \\ 0 &  \e^{\im\theta_1} r_1 \smz \sma 1 & \e^{\im\theta_3} r_3  \\ 0 & 1\smz k\right) \nonumber \\
& {} \qquad \qquad  \qquad \d k  \frac{8 r_3}{r_1 r_2} \dr r_1 \dr r_2 \dr r_3\d \theta_1 \d \theta_2 \d \theta_3 \nonumber\\
                                                 & = \int\limits_{\Z(\bC)} \int\limits_{0}^{2\uppi} \int\limits_{0}^{\infty}   \int\limits_{\N(\bC)} \int\limits_{\SU(2)}  f\left( z \sma \e^{\im\theta/2} t & 0 \\ 0 &  (\e^{\im\theta/2} t)^{-1} \smz n k\right) \d k \d n \,2 \dpr  t \d \theta \d z. \label{eq:complexmnksuper}
 \end{align}

\section{The compact subgroups $\SU(2)$ and $\U(2)$}
As usual the compact groups $\SU(2)$ and $\U(2)$ carry a probability Haar measure.
\begin{proposition}[All irreducible representations of $\U(2)$]
All irreducible representation of $\U(2)$ are constructed from its canonical representation on $\bC^2$ and the determinant $\det : \U(2) \rightarrow \bC^1$. The representations
         \[ \rho_{n, m} = \textup{Sym}^n(\bC^2) \otimes \textup{det}^{\otimes m}  \qquad m,n \in \mathbb{Z}, n \geq 0\]
         are irreducible representations of $\U(2)$. Every irreducible representation of $\U(2)$ is isomorphic to some $\rho_{n,m}$.
  \end{proposition}
\begin{proof}
The description of all irreducible representations of the compact group $\U(2)$ can be found in \cite{Teleman:Lecturenotes}*{Proposition 22.9}, \cite{Faraut:Lie}, \cite{tomDieck:Lie}. I will outline the argument. Every element in $\U(2)$ is conjugate to an element
\[ \sma z_1 & 0 \\ 0 & z_2 \smz, \qquad z_i \in \bC^1.\]
Two distinct elements, $\sma z_1 & 0 \\ 0 & z_2 \smz$ and $\sma z_1' & 0 \\ 0 & z_2' \smz$, are in the same conjugacy class if and only if $z_1 = z_2'$ and $z_2 = z_1'$. 
The character of $\rho_{n,m}$ is
\begin{align}\label{eq:tre}  \tr \rho_{n,m}  \left(    \sma z_1 & 0 \\ 0 & z_2 \smz   \right) = z_1^{m+n} z_{2}^{m} +    z_1^{m+n-1} z_{2}^{m+1} + \dots + z_1^{m} z_{2}^{m+n}    =    \tr \rho_{n,m}  \left(    \sma z_2 & 0 \\ 0 & z_1 \smz   \right).\end{align}
We want to verify that each representation $\rho_{n,m}$ is irreducible. Moreover, representations of this type should exhaust all irreducible representations of $\U(2)$. We need only to show that the functions $x \mapsto \tr \rho_{n,m}(x)$ give an orthonormal basis for the conjugation invariant functions in $\mL^2(\U(2))$. This is a standard consequence of the Peter-Weyl Theorem.
Clearly the functions $\tr \rho_{n,m}$ span all the symmetric Laurent series $P(x,y) = P(y,x)$. They span the conjugation invariant functions in $\mL^2(\U(2))$ by the Stone-Weierstrass Theorem, yielding the exhaustion part. The Weyl integration formula for $\U(2)$ gives on the conjugation invariant functions $f : \U(2) \rightarrow \bC$ as:
\[ \int\limits_{U(2)} f(u) \d u = \frac{1}{8 \uppi^2} \int\limits_{0}^{2 \uppi} \int\limits_{0}^{2 \uppi}  f\left(\sma \e^{\im \theta_1} & 0 \\ 0 &e^{\im \theta_2} \smz\right) \left|\e^{\im \theta_1} - \e^{\im \theta_2}\right| \d \theta_1 \d \theta_2.\] 
From this, it can be directly seen that the functions $\tr \rho_{n,m}$ form an orthonormal basis for the class function in $\mL^2(\U(2))$. This implies the irreducibility of each $\rho_{n,m}$.
\end{proof}

\begin{proposition}[All irreducible representations of $\SU(2)$]
Let $\rho_0$ be the trivial representation of $\SU(2)$. Let $\rho_1$ be the standard representation of $\SU(2)$ on $\bC^2$, and define for $n\geq 2$
\[ \rho_n = \textup{Sym}^n (\bC^2).\]
The representations $\rho_n$ for $n\geq 0$ exhaust all the irreducible unitary representations of $\SU(2)$.
\end{proposition}
\begin{proof}
Every conjugacy class contains a matrix of the form $\sma z& 0 \\ 0 & z^{-1} \smz$. Two distinct diagonal matrices in $\SU(2)$ are in the same conjugacy class if and only if they are inverses.
The trace of $\rho_n$ is given on a diagonal element
\begin{align}\label{eq:trrhon} \tr \rho_n \left( \sma z & 0 \\0 & z^{-1} \smz \right) = z^{-n} +  z^{-n+2} + \dots  + z^n. \end{align}
These polynomials are dense in the conjugation invariant functions on $\SU(2)$ by the aforementioned classification of conjugacy classes, and by the Stone-Weierstrass Theorem. The Weyl integration formula yields for a conjugation invariant function \cite{Teleman:Lecturenotes}*{Theorem 20.9, page 49}
\[ \int\limits_{\SU(2)} f(u) \d u = \frac{1}{\uppi} \int\limits_{0}^{2 \uppi} f \left( \sma \e^{\im \theta} & 0 \\  0 &  \e^{-\im \theta}  \smz \right) \sin^2(\theta) \d \theta.\]
The functions 
\begin{align} \tr \rho_n  \left( \sma \e^{\im \theta} & 0 \\  0 &  \e^{-\im \theta}  \smz \right) = \frac{\e^{(n+1) \im \theta} - \e^{-(n+1) \im \theta} }{\e^{ \im \theta} - \e^{- \im \theta}}\end{align}
thus give an orthonormal basis. We conclude that $\rho_n$ is irreducible.
\end{proof}
\begin{lemma}\label{lemma:Restrictionsu}
We have, for an irreducible representation $\rho_{n,m}$ of $\U(2)$, that
\[ \Res_{\SU(2)} \rho_{n,m}   =  \rho_{n}.\]
\end{lemma}
\begin{proof}
Recall equation~\ref{eq:tre} for $z \in \bC^1$ and  compare it with the character~\ref{eq:trrhon}: 
\begin{align*} \tr \rho_{n,m}  \left(    \sma z  & 0 \\ 0 &z^{-1} \smz   \right) & = z^{m+n} z^{-m} +    z^{m+n-1} z^{-(m+1)} + \dots + z^{m} z^{-(m+n)}    \\ & = \tr \rho_n \left(    \sma z  & 0 \\ 0 &z^{-1} \smz   \right).\qedhere
\end{align*}
\end{proof}

\section{The representation theory of $\GL(2, \bC)$}\label{section:complexclass}
We now classify all unitary representations of $\GL_2(\bC)$. The same discussion as seen in Section~\ref{section:realclass} applies. Let us repeat it nevertheless. All irreducible, unitary representations of $\GL_2(\bC)$ are given as subquotients or subrepresentations of parabolic inductions by the Casselman submodule theorem.

Consider a one-dimensional representation $\mu : \B(\bC) \rightarrow \bC^\times$, which determines uniquely two one-dimensional representations $\mu_j : \bC^\times \rightarrow \bC^\times$, such that
\[ \mu \left( \sma a & * \\ 0 & b \smz \right) = \mu_1(a) \mu_2(b).\]
\begin{defn}
Let $s \in \bC$. The representation $\mJ(\mu,s) = \mJ(\mu_1, \mu_2,s)$\index{$\mJ(\mu,s) = \mJ(\mu_1, \mu_2,s)$ for $\GL_2(\bC)$} is the right regular representation of $\GL_2(\bC)$ on the space of smooth functions $$f: \GL_2(\bC) \rightarrow \bC,$$ which satisfy
\[ f\left( \sma a & * \\ 0 & b \smz g\right) = \mu\left( \sma a & 0 \\ 0 & b \smz\right)  \left| \frac{a}{b} \right|_\bC^{s+1/2}f(g).  \]
Note that $ \left| \frac{a}{b} \right|_\bC^{s+1/2} =  \left| \frac{a}{b} \right|^{2s+1}.$
\end{defn}
Every one-dimensional representation $\chi : \bC^\times \rightarrow \bC^\times$ can be uniquely decomposed as
\[ \chi( \e^{\im \theta} t ) = \chi_{alg} (\e^{\im \theta}) t^{s_\chi}  , \qquad t \in(0,\infty) \]
for some unique $s_\chi \in \bC$ and some unique one-dimensional representation $\chi_{alg} : \bC^1 \rightarrow \bC$. We say that $\chi$ is \textbf{algebraic} if $s_\chi= 0$.\index{algebraic character case $\bC$} Similarly, we say that
\[ \mu \left( \sma a & * \\ 0 & b \smz \right) = \mu_1(a) \mu_2(b).\]
is algebraic if $\mu_1$ and $\mu_2$ are algebraic. It is sufficient to consider parabolic inductions with algebraic one-dimensional representation $\mu$, since we have an isomorphism
\[ \mJ(\mu_1, \mu_2, s) = \left| \det(\blank) \right|^{s_{\mu_1}/2 + s_{\mu_2}/2} \otimes \mJ( \mu_{1, alg}, \mu_{2, alg} , s_{\mu_1} -s_{\mu_2} ).\]
Generally, we have that
\[ \chi \circ \det \otimes \mJ(\mu, s) \cong \mJ(\mu \cdot  \chi\circ\det|_{\B(\bR)}, s) = \mJ(\mu_1 \chi, \mu_2 \chi, s).\]
The central character of $\mJ(\mu_1, \mu_2,s)$ is given by $\mu_1 \mu_2$. If $\mu$ is algebraic, the central character will be algebraic as well. \textbf{We assume from now on that all one-dimensional representations denoted by $\mu, \mu_1, \dots$ are algebraic.}

\begin{theorem}\label{thm:gl2c}
Every irreducible, unitarizable, smooth, admissible representation of $\GL_2(\bC)$ is isomorphic to either
\begin{enumerate}[font=\normalfont]
     \item a one-dimensional representation,
     \item a principal series representation,
       \begin{enumerate}[font=\normalfont]
        \item  a continuous series representation $\mJ(\mu_1, \mu_2, s)$ for $\Re\; s = 0$,
        \item a complementary series representation $\mJ(\mu, \mu, s)$ for\\ $-1/2 < \Re\; s < 1/2$.
       \end{enumerate}
\end{enumerate}
We have an isomorphism $\mJ(\mu_1, \mu_2, s) \cong \mJ(\mu_2, \mu_1, -s)$. All the other listed representations are non-equivalent.
\end{theorem}
The above theorem is deduced from the representation theory of $\SL_2(\bC)$.  For the classification of the irreducible, smooth, admissible representations of $\GL_2(\bC)$, the reader may consult \cite{JacquetLanglands}*{Chapter 6}. I have no reference for the issue of unitarizability of representations of $\GL_2(\bC)$. The statements for $\SL_2(\bC)$ can be found in \cite{Knapp:Semi},\cite{Wallach1}.

\begin{defn}
Let $s \in \bC$. The representation $\mJ(\mu,s)$\index{$\mJ(\mu,s)$ for $\SL_2(\bC)$} is the right regular representation of $\SL_2(\bC)$ on the space of smooth functions $$f: \SL_2(\bC) \rightarrow \bC,$$ which satisfy
\[ f\left( \sma a & * \\ 0 & a^{-1} \smz g\right) = \mu\left( a\right)  \left| a^2 \right|_\bC^{s+1/2}f(g).  \]
\end{defn}
 
\begin{theorem}[\cite{Knapp:Semi},\cite{Wallach1}] \label{thm:sl2c}
Every irreducible, unitarizable, smooth, admissibile representation of $\SL_2(\bC)$ is isomorphic to either
\begin{enumerate}[font=\normalfont]
     \item a one-dimensional representation,
     \item a principal series representation,
       \begin{enumerate}[font=\normalfont]
        \item  a continuous series representation $\mJ(\mu, s)$ for $\Re\; s = 0$, 
        \item a complementary series representation $\mJ(\mu, s)$ for\\ $-1/2 < \Re\; s < 1/2$ and $\mu^2 = 1$.
       \end{enumerate}
\end{enumerate}
We have an isomorphism $\mJ(\mu, s) \cong \mJ(\mu^{-1}, -s)$. The other listed representations are non-equivalent.
\end{theorem} 
One concept of how to derive the representation theory of $\GL_2(\bC)$ from the representation theory of $\SL_2(\bC)$ is presented in \cite{Knapp:GL2}. Knapp suggests pasting on a one-dimensional representation of $\Z(\bC)$ to an irreducible representation $\pi$ of $\SL_2(\bC)$, which coincides on the restriction to $\SL_2(\bC) \cap \Z(\bC)$. I use a variant of this argument, which is available for any local field. The next lemma clarifies why Theorem~\ref{thm:sl2c} implies Theorem~\ref{thm:gl2c}.

 The continuous series representations are tempered, the complementary series representations are not. The continuous series representations occur in the right regular representation, while the complementary series representations do not. The Ramanujan-Petersson conjecture\index{Ramanujan-Petersson conjecture} asserts that the complementary series representations do not occur as constituents of automorphic forms. Both groups $\SL_2(\bC)$ and $\GL_2(\bC)$ have no square-integrable representations, which is a general feature of reductive Lie groups
without compact Tori \cite{Harish:DiscreteII}.

\begin{lemma}\label{lemma:sl2}
Let $\F$ be a local field. An irreducible, smooth, admissible representation of $\GL_2(\F)$ with unitary central character is unitarizable / square-integrable (mod center) / tempered if and only its restriction to $\SL_2(\F)$ is unitarizable / square integrable (mod center) / tempered. 
\end{lemma}
\begin{proof}
The map $g \mapsto \left| \det(g) \right|_v$ gives a group extension
\[ 1 \rightarrow  G^1 \rightarrow \GL_2(\F) \rightarrow (0, \infty) \rightarrow 1,\]
where the closed, normal subgroup is defined
\[ G^1 = \{ g : |\det g|_F= 1 \}.\]
After twisting by an appropriate unitary character $g \mapsto | \det g|_v^{\im t}$, we may assume that the central character of an irreducible representation $\pi$ of $\GL_2(\F)$  lives on $\F^1$ only. The later unitary / square-integrable / tempered representations of $\GL_2(\F)$ are in one-to-one correspondence with the unitary / square-integrable / tempered representations of $G^1$, since $\Z(\F) \cdot G_1 =\GL_2(\F)$.

Now $G^1$ has a $\SL_2(\F)$ as a normal cocompact subgroup. In the case of induction from and restriction to a cocompact subgroup, all the properties listed above are preserved.  
If an irreducible representation $\pi$ of $G^1$ is unitarizable, then $\Res_{\SL_2(\F)} \pi$ is automatically unitarizable. Assume now that the representation $\sigma = \Res_{\SL_2(\F)} \pi$ is unitarizable. The Mackey induction functor is given for a unique invariant Haar measure $\textup{d} \dot{g}$ on  the compact abelian group $\SL_2(\F) \backslash G^1 \cong \F^1$:
\begin{align*} \Ind_{\SL_2(\F)}^{G^1} \pi = \left\{ \right. f : G^1 &\rightarrow V_\pi : f(g'g) = \pi(g') f(g)  \textup{ for all }  g' \in \SL_2(\F), g\in G^1  \\ 
                                                                                                                      &          \int\limits_{\SL_2(\F) \backslash G^1}  || f(g)||_{V_\pi}^2 \d \dot{g} < \infty \left.\right\}.
\end{align*}
This representation is undoubtedly unitary. It decomposes discretely and contains $\pi$. Hence, $\pi$ is unitarizable. For the consideration of temperedness, we note that every matrix coefficient $m_{v,w}$ of $\pi$ gives a matrix coefficient of $m_{v,w}|_{\SL_2(\F)}$ of $\Res_{\SL_2(\F)} \pi$, and this restriction is of course surjective. Certainly by the cocompactness of $\SL_2(\F) \subset G^1$, the matrix coefficient $m_{v,w} \in \mL^p( G^1)$ is $p$-integrable if and only if $m_{v,w}|_{\SL_2(\F)} \in \mL^p(\SL_2(\F))$ is $p$-integrable.   
\end{proof} 

\begin{theorem}\label{thm:complexKtype}
The set of algebraic characters of $\bC^\times$ is indexed by $q \in \bZ$: 
\[ \epsilon_q : \bC^\times \rightarrow \bC^1, \qquad z \mapsto \left( \frac{z}{|z|} \right)^q.\]
The set of of irreducible unitary representation of $\U(2)$ is given by $\rho_{n,m} = \det^{\otimes m} \otimes \textup{Sym}^{n} \bC^2$ for $m \in \bZ$ and $n \geq 1$ with $\dim(\rho_{n,m}) =  n+1$. 
\begin{align*} 
 \Res_{\U(2)}\Ind_{B(F)}^{\GL_2(F)} (\epsilon_{q_1}, \epsilon_{q_2}, s) &=  \Ind_{\U(1) \times \U(1)}^{\U(2)}  (\epsilon_{q_1}, \epsilon_{q_2}) \\ 
&= \bigoplus_{\substack{ -q_1 - q_2 = 2m +n \\   |q_1 - q_2| \leq n \\ n  =q_2 +q_1 \bmod 2}} \rho_{n,m}.
\end{align*}
\end{theorem}
\begin{proof}
 The first equality is a product of the Iwasawa decomposition, see also Mackey's Induction Restriction Formula \cite{Barut}*{Section 18.1}:
 \[   \Res_{\U(2)}\Ind_{\B(F)}^{\GL_2(F)} (\epsilon_{q_1}, \epsilon_{q_2}, s) \cong \Ind_{\U(1) \times \U(1)}^{\U(2)}  (\epsilon_{q_1}, \epsilon_{q_2}).\]
A standard computation yields the results. First, we rely on the Frobenius reciprocity for compact groups: 
\[  \Hom_{\U(2)} ( \rho_{n,m},  \Ind_{\U(1) \times \U(1)}^{\U(2)} (\epsilon_{q_1}, \epsilon_{q_2}) )  \cong        \Hom_{\U(1) \times \U(1)} ( \Res \rho_{n,m},  (\epsilon_{q_1}, \epsilon_{q_2}) ).\]
The dimension is computed
\[ \dim \Hom_{\U(2)} ( \dots) = \int\limits_{\U(1)} \int\limits_{\U(1)}    \tr \rho_{n,m} \left( \sma z_1 & 0 \\ 0 & z_2 \smz\right) z_1^{-q_1}  z_2^{-q_2} \d z_1 \d z_2.\]
 We have seen that
\[ \tr \rho_{n,m}  \left(    \sma z_1 & 0 \\ 0 & z_2 \smz   \right) = z_1^{m+n} z_{2}^{m} +    z_1^{m+n-1} z_{2}^{m+1} + \dots + z_1^{m} z_{2}^{m+n}.\]
It follows that the dimension of the space of $\U(2)$-intertwiners 
 \[ \dim \Hom_{\U(2)} ( \dots) \leq 1 \]
is at most one. This dimension is exactly one if there exists $0 \leq n_0 \leq n$, such that $-q_1 = m+n_0 $ and $-q_2 = m + n -n_0$.
\end{proof}

Note that \cite{JacquetLanglands}*{Lemma 6.1(ii), page 112} contains a typo. 
\begin{lemma}
For any two integers $q_1, q_2 \in \bZ$, we have that
\[ \Res_{\SU(2)} \Ind_{\B(\bC)}^{\GL_2(\bC)} (\epsilon_{q_1}, \epsilon_{q_2}, s) = \bigoplus_{|q_1-q_2| \leq n \atop  n  =q_2 -q_1 \bmod 2} \rho_n. \]
\end{lemma}
\begin{proof}
This follows immediately from theorem~\ref{thm:complexKtype} and lemma~\ref{lemma:Restrictionsu}:
\begin{align*} 
&\Res_{\SU(2)} \Res_{\U(2)} \Ind_{\B(\bC)}^{\GL_2(\bC)} (\epsilon_{q_1}, \epsilon_{q_2}, s)   \\
                     & =      \bigoplus_{\substack{ q_1 - q_2 = 2m +n \\   |q_1 - q_2| \leq n \\ n  =q_2 -q_1 \bmod 2}} \Res_{\SU(2)} \rho_{n,m}.\qedhere
\end{align*}
\end{proof}

\section{The Abel inversion for $\GL(2, \bC)$}
The Abel inversion for $\SL_2(\bC)$ is available in the literature \cite{Koornwinder:SL2C}. From the preceding section, we know that the representations of $\GL_2(\bC)$ and $\SL_2(\bC)$ are closely related. The field $\bC$ is algebraically closed. After an appropriate twist by a one-dimensional representation, we must only address the irreducible unitarizable representations of $\U(2)$ and $\GL_2(\bC)$, both of which admit a trivial central character. All representations of $\U(2)$ are precisely given by the symmetric tensors for all non-negative integers $n$ of the natural action of $\U(2)$ as endomorphisms on $\bC^2$ and all integer powers of the determinant map $\U(2) \rightarrow \bC^1$, that is,
\[ \rho_{m,n} := \Sym^n(\bC^2) \otimes \det( \blank)^m. \]
This representation has trivial central character if and only if $m=0$. We set\index{$\rho_n$ for $\U(2)$}  
\[ \rho_n = \rho_{0,n}.\]
Be aware that although $\rho_n$ also denotes a representation of $\O(2)$, there is little potential for confusion.

As for $\GL_2(\bR)$, the group $\GL_2(\bC)$ has a non-compact center. All computations are done modulo the center. The irreducible representations of $\Z(\bC) \U(2)$ with trivial central character are the inflations of irreducible representations $\rho_n$ of $\U(2)$ to $\Z(\bC) \U(2)$. The character of $\rho_{n}$ is the Chebyshev polynomial $U_n$\index{$U_n$ Chebyshev polynomial of the second kind} of the \underline{second} kind. Instead of defining them via power series, we define them as the solution to an equation
\begin{align}\label{eq:un} U_n( \cos(x)) = \frac{\sin( (n+1)x)}{\sin(x)}, \qquad U_n(  \cosh(t)) = \frac{\sinh( (n+1)x)}{\sinh(x)}. \end{align}
 For $\tau \in \bR$, the function $U_n$ satisfies:
 \[ \tr \rho_n  \left( \sma \e^{\im \tau} & 0 \\ 0 & \e^{-\im \tau} \smz \right)= \frac{\sin\left( (n+1) \tau\right)}{\sin( \tau)} = U_n(\cos(\tau)).\]
Note that every element in $\U(2)$ is conjugate to a diagonal matrix.

The purpose of this section is to understand the Abel inversion formula. We state the Abel inversion formula on $\SL_2(\bC)$ as suggested in \cite{Koornwinder:SL2C}*{Theorem 6.5, page 431}. Define for this the group
\[ \M_1(\bC) = \left\{ \sma x&0\\0&x^{-1} \smz : x \in \bC^\times \right\}.\]
\begin{thmu}[Main inversion identity for $\SL_2(\bC)$ \cite{Koornwinder:SL2C}]
Consider the irreducible representation $(\rho_n,V_n) = \Sym^n(\bC^2)$ of $\SU(2)$. We define the operator
\begin{align}    \underline{A}_n \colon & \mH(\SL_2(\bC), \rho_n) \rightarrow  \mH(\M_1(\bC), \rho_n) \label{eq:abelcpl}\\ 
\underline{A}_n & f\left( \sma \e^{\im \tau} \e^t & 0 \\0 & \e^{-\im \tau} \e^t\smz \right) \\ & = \frac{1}{2\uppi} \int\limits_{0}^{2 \uppi} \int\limits_t^\infty  f\left( \sma \e^{\im \vartheta} \e^w & 0 \\0 & \e^{-\im \vartheta} \e^{-w}\smz \right)  2 \sinh(2w) K_n(t, \tau, w, \vartheta)   \d \vartheta \d w ,\nonumber\end{align}
for the kernel
\[ K_n( t, \tau, w, \vartheta) \coloneqq    U_n\left( \frac{\cosh(t)}{\cosh(w)} \cos(\vartheta) \cos(\tau) +\frac{\sinh(t)}{\sinh(w)} \sin(\vartheta) \sin(\tau) \right).\]
The following inversion formula applies:
\begin{align*} 
& f\left( \sma \e^{\im \vartheta} \e^w & 0 \\0 & \e^{-\im \vartheta} \e^w\smz \right) \\&  = - \frac{1}{4\uppi \sinh(2w)}  & \int\limits_{0}^{2 \uppi}  \int\limits_w^\infty K_n(t, \tau, w, \vartheta)  \left( \frac{\partial^2}{\partial \tau^2} + \frac{\partial^2}{\partial t^2} \right) \underline{A}_n f\left( \sma \e^{\im \tau} \e^t & 0 \\0 & \e^{-\im \tau} \e^t\smz \right) \d \tau \d t.
\end{align*}
\end{thmu} 
The equivalence of $\underline{A}_n$ to the group theoretic Abel transform $\mA_{\rho_n}$ as introduced in Section~\ref{section:Abel} is given by \cite{Koornwinder:SL2C}*{Theorem 5.2, page 425}. 
The Abel inversion formula as stated above is fairly complicated. I am not aware of any reference in the complex case which computes the local distributions classically such as \cite{Hejhal1}, \cite{Hejhal2}. Only the case of bi-invariant functions has been examined by \cite{Tanigawa}, \cite{Szmidt}, \cite{Bauer1}, \cite{Bauer2} and \cite{Elstrodt}. In this special case, the inversion formula simplifies significantly. Set
\[ \mSH(\GL_2(\bC), 1) \coloneqq \Ccinf( \GL_2(\bC)//\U(2)\Z(\bC)),\]
 and identify
\[   \mSH(\GL_2(\bC), 1) \cong \Ccinf([0, \infty)), \qquad \phi(g) = \Phi\left( \frac{\tr g^\dagger g}{| \det  g|} -2 \right).\]
On this set of functions, the Abel inversion formula is given by \cite{Elstrodt}*{Lemma 3.5.5, page 121}.                 
\begin{defnthm}[Main inversion identity on $\GL_2(\bC)$ --- Bi-invariant case]\label{thm:complexabelinversion}
Define the operators
\[ A_0= \widehat{A}_0 : \Ccinf([0, \infty)) \rightarrow \Ccinf([0, \infty)) ,\]
via kernel transformations
\begin{align*}
   A_0 \Phi(y):= &      \int\limits_{0}^\infty  \Phi\left( y + t\right)  \d t. 
\end{align*}
We have the inversion formula
\[   A_0 \Phi(x)' = - \Phi(x).\] 
\end{defnthm}
\begin{proof}
This follows directly from integration by parts
\[            \int\limits_{0}^\infty  \Phi'\left( y + t\right)  \d t  = - \Phi(y). \qedhere\]
\end{proof}
The analytical difficulties in presenting closed formulas for all distributions, for which one has to appeal to the Abel inversion formula, becomes a formidable task in advanced integral calculus and the theory of special functions.  
These distributions are:
\begin{itemize}
 \item the Plancherel formula
 \item the derivative of the local zeta function
 \item the weighted orbital hyperbolic integral
\end{itemize}
These distributions will be computed only for bi-invariant functions. All other distributions are computed in full generality. Recalling the real situation, I will restrict the analysis to a suitable subspace of test functions given as
\begin{align*}\mSH( \GL_2(\bC), \rho_n) = \Big\{  \phi& : \GL_2(\bC)  \rightarrow \bC : \textup{ smooth, comp.supported } \bmod\Z(\bC): \\
& \phi \left( k_1 \sma z & 0 \\ 0 & z \smz \sma \e^{t} & 0 \\ 0 & \e^{-t} \smz k_2\right)     =   \phi\left(  \sma \e^{t} & 0 \\ 0 & \e^{-t} \smz \right) \frac{\tr \rho_n(k_1k_2)}{\dim \rho_n}\\
 &\qquad\qquad \textup{ for all } k_1, k_2 \in \U(2), z \in \bC^\times, t\geq 0  \Big\} \end{align*}  
  for the irreducible representation $\rho_n = \textup{Sym}^n(\bC^2)$. These subspaces are sufficient for the analysis of all spectral parameters. On this set of test functions, we compute all the remaining distributions
\begin{itemize}
 \item the value of character distribution of the irreducible, unitary representations,
 \item the local zeta function,
 \item the hyperbolic integral.
\end{itemize}
There are no local elliptic distributions, since the field $\bC$ is algebraically closed.

With the provided analysis, only the automorphic representations with unramified principal series representations as constituents at the complex places can be analyzed.
Since $\GL(2, \bC)$ has no discrete series representations, there is less interest in the general computations. The restriction happens mostly  for the sake of brevity and clarity.

As for $\GL_2(\bR)$, we can identify
\[ \Ccinf([0, \infty)) \cong \mSH( \GL_2(\bC),\rho_n)\] 
by setting for $t \geq 0$
\[ \Phi_\phi( \e^{2t} + \e^{-2t} - 2 ) =2 \uppi \phi\left( \sma \e^t & 0 \\ 0 & \e^{-t} \smz \right).\]
The Abel transform simplifies to the following form on this subspace after a suitable shift of coordinates.
\begin{lemma}
The operator $\Ccinf([0, \infty)) \rightarrow \Ccinf([0, \infty))$ given by
\begin{align*}
   A_n \Phi(y):=       \int\limits_{0}^\infty  \Phi\left( y + t\right)       U_n \left( \frac{\sqrt{ t+4}}{\sqrt{t+y+4}} \right) \d t
\end{align*} 
satisfies for all $t \in \bR$ 
  \[  \mA_{\Sym^n(\bC^2)} \phi  \left( \sma \e^t & 0 \\ 0 & \e^{-t}\smz \right) =  \underline{A}_n \phi  \left(  \sma \e^t & 0 \\ 0 & \e^{-t}\smz \right)  =  A_{n} \Phi_\phi( \e^{2t} + \e^{-2t}-2). \]
\end{lemma}
\begin{proof}
Set $\tau = 0$ in Equation~\ref{eq:abelcpl}. This simplifies the transform:
\begin{align*}
 \underline{A}_n  \phi\left( \sma \e^t & 0 \\0 & \e^t\smz \right)  &= \frac{1}{2\uppi}  \int\limits_t^\infty  \phi\left( \sma  \e^w & 0 \\0 &  \e^{-w}\smz \right)  2 \sinh(2w) \\
&  \int\limits_{0}^{2 \uppi} \tr \rho_n\left( \sma \e^{\im \vartheta} & 0 \\0 & \e^{-\im \vartheta} \smz \right) U_n\left(  \frac{\cosh(t)}{\cosh(w)} \cos(\vartheta) \right)  \d \vartheta \d w.
\end{align*}
The identity
\( \int\limits_{0}^{2 \uppi} U_n(\cos(\vartheta)) U_n(\alpha \cos(\vartheta)) \d \vartheta  = 2 \pi U_n(\alpha)\)        
allows us to rewrite the expression for $\underline{A}_n \phi\left( \sma \e^t & 0 \\0 & \e^t\smz \right)$ as
\begin{align*}   
     &   =  2 \int\limits_t^\infty \Phi_\phi\left(  \e^{2w} +  \e^{-2w} - 2 \right)   \sinh(2w)   U_n \left(  \frac{\sqrt{ \e^{2t} + \e^{-2t} -2+4}}{\sqrt{ \e^{2w} + \e^{-2w} -2 +4}} \right) \d w \\
                                                                                                 & \underset{x =  \e^{2w} +  \e^{-2w} - 2 }= \quad      \int\limits_{\e^{2t}+ \e^{2t}-2}^\infty  \Phi_\phi\left( x \right)       U_n \left( \frac{\sqrt{ y+4}}{\sqrt{x+4}} \right) \d x \\
                                                                                                 &  \underset{y =\e^{2t}+ \e^{-2t}-2 }=     \quad   \int\limits_{y}^\infty  \Phi_\phi\left( x \right)       U_n \left( \frac{\sqrt{ y+4}}{\sqrt{x+4}} \right) \d x \\
                                                                                                 & \underset{x = t+y}    =       \int\limits_{0}^\infty  \Phi_\phi\left( y + t\right)       U_n \left( \frac{\sqrt{ y+4}}{\sqrt{t+y+4}} \right) \d t   = : A_n \Phi_\phi(y).
\end{align*}
Lets us now verify that the transform $A_n$ coincides with the abstract Abel transform $\mA_\rho$ on the group level introduced in Section~\ref{section:Abel}.
\begin{align*}
      \mA_{\Sym^n(\bC^2)} \phi\left( \sma \e^t & 0 \\ 0 & \e^{-t} \smz \right) &  = \int\limits_{\bC}               \phi\left( \sma \e^t & z \\ 0 & \e^{-t} \smz \right)     \d z \\ 
                                                                                                          &  \underset{\textup{polar-coord.}}= 2\int\limits_{0}^\infty \int\limits_{0}^{2\uppi}               \phi\left( \sma \e^t & r \e^{\im \theta} \\ 0 & \e^{-t} \smz \right)      r \d r \d \theta\\
                                                                                                           & =      2\int\limits_{0}^\infty \int\limits_{0}^{2\uppi}               \phi\left(\sma \e^{\im \theta}  &0 \\ 0 & 1 \smz  \sma \e^t & r  \\ 0 & \e^{-t} \smz \sma \e^{-\im \theta}  &0 \\ 0 & 1 \smz   \right)      r \d r \d \theta \\
 & \underset{\phi \in \mSH} = 4 \uppi \int\limits_{0}^\infty          \phi\left(   \sma \e^t & r  \\ 0 & \e^{-t} \smz  \right)     r \d r.
\end{align*}
 We write
\[       \phi\left(   \sma \e^t & r  \\ 0 & \e^{-t} \smz \right) = \Phi_\phi( \e^{2t}  + \e^{-2t} + r^2 -2 ) \cdot U_n(\cos\theta_{t,r}),\]
where  in Equation~\ref{eq:diegleichung}, we have computed
\[  \cot \theta_{t,r} = \frac{\e^t + \e^{-t}}{r}, \qquad \cos \textup{arccot} \frac{\e^t + \e^{-t}}{r} = \frac{1}{\sqrt{\frac{r^2}{(\e^t+\e^{-t})^2}+1}}.\]
We therefore obtain
 \begin{align*}   \mA_{\Sym^n(\bC^2)}& \phi\left( \sma \e^t & 0 \\ 0 & \e^{-t} \smz \right)     = 4 \uppi \int\limits_{0}^\infty          \phi\left(   \sma \e^t & r  \\ 0 & \e^{-t} \smz   \right)   r \d r  \\
                                                                                                                              & =  2 \uppi\int\limits_{0}^\infty          \Phi_\phi\left(   \e^{2t}+\e^{-2t} + r^2 -2  \right)  U_n\left(  \frac{1}{\sqrt{\frac{r^2}{(\e^t+\e^{-t})^2}+1}} \right) 2r \d r \\
                                                                                                               & \underset{t=r^2}= 2 \uppi \int\limits_{0}^\infty          \Phi_\phi\left(   \e^{2t}+\e^{-2t} -2  + t \right)  U_n\left(  \frac{1}{\sqrt{\frac{t}{\e^{2t}+\e^{-2t} +2}+1}} \right)  \d t \\
                                                                                                                 &=  2 \uppi \int\limits_{0}^\infty          \Phi_\phi\left(   \e^{2t}+\e^{-2t} -2  + t \right)  U_n\left(  \sqrt{\frac{\e^{2t}+\e^{-2t} +2}{t+\e^{2t}+\e^{-2t} +2}} \right)  \d t \\
                                                                                                                 & \underset{y =  \e^{2t}+\e^{-2t} -2}= 2 \uppi \int\limits_{0}^\infty          \Phi_\phi\left(   y + t \right)  U_n\left(  \sqrt{\frac{y+4}{t+y+4}} \right)  \d t = A_n \Phi_\phi( y).\qedhere
 \end{align*}
\end{proof}

\newcommand{\EGM}{\textup{EGM}}
We introduce new notation following \cite{Hejhal1}*{page 15--16},\cite{Hejhal2}*{page 385--386}, which differs slightly from \cite{Elstrodt}:\footnote{
We provide a comparison with \cite{Elstrodt}*{Lemma 3.5.5, page 121}, whose functions will be given a superscripct $\EGM$ for the authors of \cite{Elstrodt} Elstrodt, Grunewald, and Mennicke.
\begin{align*}
 k^\EGM(t) &= \frac{1}{2 \uppi} \Phi_\phi(2t-2),\\
      Q^\EGM(u)  &=  \frac{1}{2} Q_\phi(2u-2),\\
      g^\EGM(x)     & =  \frac{1}{2} g_\phi(2x), \\
       h^\EGM(1+t^2) & = \frac{1}{4} h_\phi\left(\frac{t}{2}\right).
\end{align*}
In particular, the last identity differs significantly. Classically, one expresses the spectral side of the Selberg trace formula in terms of the eigenvalues of the Laplace-Beltrami/Casimir operator, whereas I want to express it in terms of the parameters of the principal series representation.}
\begin{defn}[$Q_\phi$, $g_\phi$ and $h_\phi$]\label{defn:complexfunctions}
Let $\rho_n$ be the irreducible representation $\Sym^n(\bC^2)$ of $\Z(\bC) \U(2)$. Define for $\phi \in \mSH(\GL_2(\bR), \rho)$
\begin{itemize}
 \item  the smooth, compactly supported function $\Phi_\phi \in \Ccinf([0, \infty))$
       \[  \Phi_\phi\left( \e^{2x} + \e^{-2x} - 2  \right) \coloneqq 2\uppi  \phi\left( \sma \e^{x} & 0 \\ 0 &\e^{-x} \smz \right),\]
 \item  the smooth, compactly supported function $Q_\phi \in \Ccinf([0, \infty))$ via\index{$Q_\phi= A_n \Phi_\phi$} 
       \[ Q_\phi(x) \coloneqq  A_n \Phi_\phi(x), \qquad Q_\phi'(x) = -\Phi_\phi(x), \]
 \item  the smooth, even,  compactly supported function $g_\phi \in \Ccinf(\bR)^{\textup{even}}$ via\index{$g_\phi(x)  = Q_\phi\left( \e^{x} + \e^{-x} - 2 \right)$}\footnote{The different normalization --- compare with the real situation $g^\bR_\phi(x) = Q^\bR_\phi( \e^x +\e^{-x}-2)$ --- has been chosen, because of the form of the modular character of $\Delta_{\B(\bC)} (b) = \left| b_{1,1}/b_{2,2} \right|_\bC =   \left| b_{1,1}/b_{2,2} \right|^2$ and its relation to the parabolic inductions.}  
             \[  g_\phi(x)  \coloneqq Q_\phi\left( \e^{x/2} + \e^{-x/2} - 2 \right),\]
 \item  the even, entire function $h_\phi$, which is a Schwartz function on every line $\im y + \bR$, via \index{$h_\phi(x) = \int\limits_{\bR} g_\phi(u) \e^{\im r u} \d u$.}            
$$h_\phi(x) \coloneqq \int\limits_{\bR} g_\phi(u) \e^{\im r u} \d u.$$
\end{itemize}
\end{defn}

\section{The character of the infinite-dimensional representations}
Let $\mu = (\mu_1, \mu_2)$ be a pair of algebraic one-dimensional representations, i.e., $\mu_j : \bC^1 \rightarrow \bC^1$. 
Define
\[ \mu : \B (\bC) \rightarrow \bC, \qquad \mu\left(\sma a & * \\0 & b \smz \right) =  \mu_1(a) \mu_2(b).\]

\begin{proposition}[Relation of $h_\phi$ to the principal series representation]\label{prop:complexjacquet}
Let $\rho$ be an irreducible representation $\U(2) \Z(\bC)$ with trivial central character $\chi$. Let $\mu =(\mu_1, \mu_2)$ be an algebraic character of $\M(\bC)$. 
Let $\phi \in \mSH(\GL_2(\bC), \rho)$, then the character distribution of the parabolic induction is computed for $s \in \bC$ as
\[\tr \mJ(\mu, s) \phi = \begin{cases}    h_\phi( \im s), &  \overline{\rho} \subset \Res_{\Z(\bC) \U(2)} \mJ(\mu,\im t), \\ 
                                                        0, & \overline{\rho} \not\subset \Res_{\Z(\bC) \U(2)} \mJ(\mu, \im t).\\
                          \end{cases}\]
\end{proposition}
\begin{proof}
The proof is straightforward and a paraphrasing of the proof of Proposition~\ref{prop:realjacquet}. We prefer to compute the character in terms of a section $\tilde{\phi} \in \Ccinf(\GL_2(\bC))$. We obtain the following formula (Theorem~\ref{thm:jacquettrace}) for the character distribution of $\mJ(\mu,s)$, and for $\tilde{\phi} \in \Ccinf(\GL_2(\bC))$
\[\tr \mJ(\mu, s, \tilde{\phi})  = \int\limits_{\B(\bC)}  \int\limits_{\SU(2)} \phi(k^{-1}bk)\mu(b) \d k \, \Delta(b)^{s+1/2} \d b.\]
 This distribution factors through the space $\Ccinf(\GL_2(\bC), \overline{\mu_1 \mu_2})$  via Equation~\ref{eq:complexmnksuper}. For the element
\[  \phi(g) = \int\limits_{\Z(\bC)}                            \tilde{\phi}(zg) \d z,\]
 we obtain the formula
\[\tr \mJ(\mu, s, \phi)  = 4 \int\limits_{0}^{2\uppi} \int\limits_{0}^{\infty}  \int\limits_{\N(\bC)} \int\limits_{\SU(2)}   \phi\left( k^{-1}  \sma \e^{\im\theta/2} t & 0 \\ 0 &  (\e^{\im\theta/2} t)^{-1} \smz n  k\right) |t^2|_\bC^{s+1/2}\, \dpr t \d n \d \theta. \]
Assume that $\mu_1 \mu_2=1$. If $\phi \in \mSH(\GL_2(\bC),  \overline{\rho})$ with $\overline{\rho} \not\subset \Res_{\Z(\bC)\U(2)} \mJ(\mu,\im t)$, then the character distribution vanishes as in Corollary~\ref{cor:charvanish} due to general facts. Otherwise, we have that
\begin{align*}\tr \mJ(\mu, s, \phi)  & = 8 \uppi \int\limits_{0}^{\infty} A_\rho \phi\left(  \sma t & 0 \\ 0 &  t^{-1} \smz \right) t^{4s}\,  \dpr t \\
                                                & \underset{t = \e^{x/4}} =  \int\limits_{-\infty}^{\infty} \overbrace{Q_\phi\left( \e^{x/2}  + \e^{-x/2} -2 \right)}^{=g_\phi(x)} \e^{xs} \, \dr x =  h_\phi( \im s). \qedhere
\end{align*}
\end{proof} 
\section{The character of the one-dimensional representations}
\begin{proposition}\label{prop:complexoned}
Let $\chi$ be a one-dimensional algebraic representation of $\GL_2(\bC)$. We have for $\phi \in \mSH(\GL_2(\bC), \rho)$ that
\[     \tr \chi(\phi) =   \begin{cases} 
  h_\phi(\im /2), & \chi|_{\Z(\bC) \U(2)} =\rho , \\
0, & \textup{else}. 
 \end{cases}\]
\end{proposition}
\begin{proof}
The only $K$-type in $\chi$ is $\Res_{\Z(\bC) \U(2)} \chi$, hence the vanishing results. For the remaining formula, we have by Equation~\ref{eq:complexmnksuper}
\begin{align*}
        \tr \chi(\phi)  &= \int\limits_{\Z(\bC) \backslash \GL_2(\bC)}  \phi(g) \chi(g) \d g      \\
                   &=   4 \int\limits_{0}^{\infty}  \int\limits_{\N(\bC)}   \phi\left(  \sma  t & 0 \\ 0 &  t^{-1} \smz n  \right)  \dpr t \d n \\
                    & =  8 \uppi \int\limits_{0}^{\infty}  t^{-2}  \mA_\rho \phi\left(  \sma  t & 0 \\ 0 &  t^{-1} \smz \right)  \dpr t \\
               &  \underset{t = \e^{x/4}}{=}   \int\limits_{-\infty}^{\infty}  e^{-x/2}  g_\phi(x)  \dpr t  = h_\phi( \im /2) .\qedhere
\end{align*}

\end{proof}

\section{The identity distribution}\label{section:complexuni}
We compute the Plancherel formula for $\Ccinf(\GL_2(\bC)//\Z(\bC) \U(2))$. Equivalent computations can be found in \cite{Elstrodt}*{page 306}. 
\begin{proposition}[The Plancherel formula]\label{prop:complexone}
For $\phi \in \Ccinf(\GL_2(\bC))$, we obtain
\begin{align} \phi\left( \sma 1 &0 \\ 0 & 1\smz \right)   =      \frac{1}{8\uppi^2}    \int\limits_{-\infty}^\infty   h_\phi(r) r^2 \d r.
\end{align}
\end{proposition}
\begin{proof}
Appeal to the Abel inversion formula
\[ \phi(1) = \Phi_\phi(0) = - \frac{1}{2 \uppi}  Q_\phi'(0).\]
The chain rule gives us 
\begin{align}\label{eq:der} g_\phi'(x) = Q_\phi'(\e^{x/2}+ \e^{-x/2} - 2)  \sinh(x/2) = A\Phi_\phi'(\e^{x/2}+ \e^{-x/2} - 2)  \sinh(x/2). \end{align}
Since $\sinh(0) =0$, the later expression does not allow direct conclusions about the relation between $g_\phi'(0)$ and $A_0 \Phi_\phi'(0)$. In the limit, however, we have
\[ \phi(1) = \frac{1}{2 \uppi} \lim\limits_{x \rightarrow 0+}   \frac{g_\phi'(x)}{\sinh(x/2)}.\]
The Fourier inversion formula yields that
\begin{align*}
g_\phi' \left(y \right) &= \frac{1}{2 \uppi} \int\limits_{-\infty}^\infty - \im r h_\phi(r) \e^{-\im r y} \d y \\ 
                             &=  \frac{1}{ \uppi}  \int\limits_{0}^\infty   r h_\phi(r) \sin(r y) \d r.
\end{align*}
Collecting these pieces together, we get the formula
\begin{align*} 2 \uppi \phi(1) &= -   \lim\limits_{x \rightarrow 0+}   \frac{g_\phi'(x)}{2\sinh(x)} \\ 
                              & =     \frac{1}{\uppi}    \int\limits_{0}^\infty   r h_\phi(r) \lim\limits_{x \rightarrow 0+}\frac{\sin(r x/2)}{\sinh(x)} \d r \\
                              & =     \frac{1}{2\uppi}    \int\limits_{0}^\infty   h_\phi(r) r^2 \d r \\
                              & =      \frac{1}{4\uppi}    \int\limits_{-\infty}^\infty   h_\phi(r) r^2 \d r.\qedhere
\end{align*}
\end{proof}

\section{The parabolic distributions}
We compute the local zeta integral at $s=1$ for $\mSH(\GL_2(\bC), \Sym^n(\bC^2))$  and its derivative at $s=1$ only for $\Ccinf(\GL_2(\bC)//\Z(\bC) \U(2))$. These values are significant for an explicit version of the Arthur trace formula (see \cite{Gelbart}*{Proposition 1.2, page 47}). The computations on the pages 301--302 in \cite{Elstrodt} are therefore relevant.
\begin{proposition}[The local Zeta integral and its derivative at $s=1$]\label{prop:complexpara}
For $\phi \in \mSH(\GL_2(\bC), \Sym^n(\bC^2))$, we obtain
\begin{align*}
\frac{1}{\zeta_\bC(1)}\int\limits_{\bC}  \phi\left( \sma 1 & a \\ 0 & 1 \smz \right)  \d a &= g_\phi(0).
\end{align*}
For $\phi \in\Ccinf(\GL_2(\bC)//\Z(\bC) \U(2))$, we have that  
\begin{align*} & \frac{\partial}{\partial s} \Bigg|_{s=1}  \frac{1}{\zeta_\bC(s)} \int\limits_{\bC^\times}  \phi\left( \sma 1 & a \\ 0 & 1 \smz \right) \left| a \right|_\bC^{s} \d^\times a \\
 &  =   \frac{h_\phi(0)}{4}  + \left(   \log(2 \uppi)  -  \gamma_0 \right) g_\phi(0)     -      \frac{1}{ \uppi}    \int\limits_0^{\infty}   h(t)   \frac{\Upgamma'}{\Upgamma}(1-2 \im t)  \dr t.
\end{align*}
\end{proposition}
\begin{proof}
The local zeta function $\zeta_\bC(s)$ is computed in Tate's Thesis \cite{Tate:Thesis}*{page 319}
\[ \zeta_\bC(s) = (2 \uppi)^{1-s} \Upgamma(s).\]
We have
\[ \zeta_\bC(1)  = 1\]
and
\[  \partial_s \Big|_{s=1} \zeta_\bC(s) = -\log(2 \uppi) - \gamma_0.\]
The first identity is evidenced merely by checking the definitions
\begin{align*}\int\limits_{\bC}  \phi\left( \sma 1 & a \\ 0 & 1 \smz \right)  \d a = & \mA_\rho \phi(1) \\ 
                                                                                                             = &  A_n \Phi_\phi(0) =Q_\phi(0) = g_\phi(0).
\end{align*}
The second identity is more difficult to demonstrate. Assume that $\phi$ is $\Z(\bC) \U(2)$-bi-invariant:
\begin{align*} 
   \int\limits_{\bC^\times}  \phi\left( \sma 1 & a \\ 0 & 1 \smz \right) \left| a \right|_\bC^{s} \d^\times a &=
	\int\limits_0^{2 \uppi }   \int\limits_0^{\infty}   \phi\left( \sma 1 & \e^{\im \theta}r \\ 0 & 1 \smz \right) r^{2s} \frac{2\d r \d \theta}{r} \\
	&=  2 \uppi \int\limits_0^{\infty}   \phi\left( \sma 1 & r \\ 0 & 1 \smz \right) r^{2s} \frac{2\d r \d \theta}{r} \\
 &=  \int\limits_0^{\infty}   \Phi_\phi\left(  r^2 \right) r^{2s} \frac{2\d r }{r} \\
  &\underset{t=r^2}=  \int\limits_0^{\infty}   \Phi_\phi\left(  r \right) r^{s} \frac{\d r}{r}.
\end{align*}
Set $\Re\;s>1$. We use the Abel inversion formula
\begin{align*}
\int\limits_0^{\infty}   \Phi_\phi\left(  r \right) r^{s} \frac{\d r}{r} & =  - \int\limits_0^{\infty} Q_\phi'\left(  r \right) r^{s} \frac{\d r}{r}  \\ 
                                                                                       & \underset{r = \e^{x/2} + \e^{-x/2} -2} =  -   \int\limits_0^{\infty}  A_0 \Phi_\phi'\left( \e^{x/2} + \e^{-x/2} -2 \right) \sinh(x/2)  \left( \e^{x/2} + \e^{-x/2} -2\right)^{s-1} \dr x \\
                                                                                        & \underset{~\ref{eq:der}} =   -  \int\limits_0^{\infty}  g_\phi'(x)  \left( \e^{x/2} + \e^{-x/2} -2\right)^{s-1} \dr x.
\end{align*}
Derivation with respect to $s$ yields that
\begin{align*}  -    \partial_s  \int\limits_0^{\infty}  g_\phi'(x) &  \left( (\e^{x/4} - \e^{-x/4})^2 \right)^{s-1}  \dr x  \\ & =    -        \int\limits_0^{\infty}  g_\phi'(x)  \left( (\e^{x/4} - \e^{-x/4})^2 \right)^{s-1}  2 \log( \e^{x/4}- \e^{-x/4} ) \,\dr x. \end{align*}
We write
\[ 2 \log( \e^{x/4}- \e^{-x/4} ) =  \frac{x}{2} + 2 \log( 1 -\e^{-x/2}).\]
We evaluate at $s=1$, then
\begin{align*}      -      \partial_s\Big|_{s=1}  \int\limits_0^{\infty}  g_\phi'(x)  \left( (\e^{x/4} - \e^{-x/4})^2 \right)^{s-1}  \dr x   & = -      \int\limits_0^{\infty}     \frac{x}{2}  g_\phi'(x)  \d x  \\
                                                                                                                                                                          & \qquad-  2  \int\limits_0^{\infty}  g_\phi'(x)     \log(1- \e^{-x/2} )  \,\dr x. 
\end{align*}
The first integral is computed via integration by parts and Fourier inversion
\[   -  \int\limits_0^{\infty}     \frac{x}{2}  g_\phi'(x)  \d x =    \int\limits_0^{\infty} \frac{1}{2} g_\phi(x) \d x =  \frac{h_\phi(0)}{4}.\]
The second integral is calculated by appealing to \cite{Elstrodt}*{Equation 5.17, page 302}, 
\[  \im t \int\limits_{0}^\infty \e^{\im x t} \log(1-\e^{-x/2}) \d x=  \frac{\Upgamma'}{\Upgamma}(1- 2 \im t) + \gamma_0.\]
We apply the Fourier inversion formula
\[ g_\phi'(x) = \frac{1}{2 \uppi} \int\limits_{-\infty}^\infty \im t h_\phi(t) \e^{\im x t}\d x.\]
This yields at $s=1$
\begin{align*}- 2  \int\limits_0^{\infty} &   g_\phi'(x)  \left( (\e^{x/4} - \e^{-x/4})^2 \right)^{s-1}  \log( 1-\e^{-x/2} ) \, \dr x \\
                                                    & =\frac{-1}{ \uppi}  \int\limits_0^{\infty}  \int\limits_{-\infty}^\infty \im t h_\phi(t) \e^{\im x t}    \log( 1-\e^{x/2} ) \dr t \dr x \\
 &\qquad =               \frac{-1}{\uppi}  \int\limits_0^{\infty} h_\phi(t)  \left(  \frac{\Upgamma'}{\Upgamma}(1-2\im t) + \gamma_0  \right) \,\dr t  \\
&\qquad\qquad =        -  2 \gamma_0 g_\phi(0) -      \frac{1}{ \uppi}    \int\limits_0^{\infty}   h_\phi(t)   \frac{\Upgamma'}{\Upgamma}(1-2\im t)  \,\dr t.
\end{align*}
Finally, we obtain by the multiplication rule 
\begin{align*}& \frac{\partial}{\partial s} \Bigg|_{s=1}   \frac{1}{\zeta_\bC(s)} \int\limits_{\bC^\times}  \phi\left( \sma 1 & a \\ 0 & 1 \smz \right) \left| a \right|_\bC^{s} \d^\times a \\
& =-  \frac{\zeta_{\GL_2(\bC)} ( 1, \phi) \zeta_\bC'(1)}{ \zeta_\bC(1)^2}  + \frac{\zeta_{\GL_2(\bC)} ( 1, \phi) }{\zeta_\bC(1)}        \\
&=  - \left(  - \log(2 \uppi) - \gamma_0 \right) g_\phi(0) +        \frac{h_\phi(0)}{4}               -   2 \gamma_0 g_\phi(0) -      \frac{1}{ \uppi}    \int\limits_0^{\infty}   h(t)   \frac{\Upgamma'}{\Upgamma}(1-2 \im t) \, \dr t\\
& =  \left(   \log(2 \uppi)  -  \gamma_0 \right) g_\phi(0)   + \frac{h_\phi(0)}{4}   -      \frac{1}{ \uppi}    \int\limits_0^{\infty}   h(t)   \frac{\Upgamma'}{\Upgamma}(1-2 \im t)\,  \dr t.  \qedhere
\end{align*}

\end{proof}

\section{The hyperbolic distributions}\label{sec:complexhyper}
We consider an element $\gamma \in \GL_2(\bC)$. By definition, $\gamma$ is a hyperbolic element if its characteristic polynomial splits into two distinct factors over $\bC$, that is, $\gamma $ is conjugate to a diagonal matrix
\[  \sma \alpha & 0 \\ 0 & \beta \smz\]
for $\alpha, \beta \in \bC^\times$ with $\alpha \neq \beta$.

The stabilizer of the element $\gamma$ is the subgroup $\M(\bC) = \sma * & 0 \\0 & * \smz$  of diagonal matrices in $\GL_2(\bC)$. The Arthur trace formula  \cite{Gelbart}*{Proposition 1.1, page 46}, \cite{GelbartJacquet}*{page 214} associates to a hyperbolic element two types of distributions, i.e., an orbital integral for $\phi \in \Ccinf(\GL_2(\bC), \overline{\chi})$
\[ J_\gamma ( \phi ) = \int\limits_{\M(\bC) \backslash \GL_2(\bC)} \phi ( g^{-1}  \gamma g ) \d \dot{g},\]
and a weighted orbital integral for $\tilde{\phi} \in \Ccinf(\GL_2(\bC)$
\[  J_\gamma^H ( \phi) =  \int\limits_{\M(\bC) \backslash \GL_2(\bC)}  \phi ( g^{-1}  \gamma g ) w_H(g) \d \dot{g}.\]
We have the following invariance properties for all $z \in \Z(\bC)$ and $g \in \GL_2(\bC)$
\[ J_\gamma( \phi) =  \chi(z) J_{z\gamma}( \tilde{\phi}) =    J_{g^{-1}\gamma g}( \phi) = J_{\gamma}( \phi^g), \]
 and for $z \in \Z(\bC)$ and $k \in \U(2)$.

\begin{proposition}[The hyperbolic orbital integral]\label{prop:complexhyper}
Set $\gamma = \sma \alpha & 0 \\ 0 & 1 \smz $  for $\alpha \in \bC^\times -\{1 \}$. Consider the irreducible representation $\rho_{n,m}$ of $\Z(\bC) \U(2)$. 

 For each $\phi \in \mSH(\GL_2(\bC), \rho)$ the orbital integral of $\gamma$ evaluates to
 \[ J_\gamma(\phi) = \frac{\left| \alpha\right|_\bC^{-1}}{\left| 1 -\alpha^{-1} \right|_\bC} e^{\im m \arg \alpha/2} U_{n}\left(  \cos(\arg \alpha/2) \right)    g_\phi( \log |\alpha| ).\]
For each  $\phi \in \Ccinf(\GL_2(\bC)//\Z(\bC) \U(2))$, the weighted orbital integral gives
\[ J_\gamma^H(\phi) =  \frac{1}{\sqrt{|\alpha|_\bC}| 1 - \alpha^{-1}|_\bC} \int\limits_{2 \log | \alpha |_\bC}^\infty  g_\phi(x)  \frac{\e^{x/2} - \e^{-x/2}}{\e^{x/2} + \e^{-x/2} + \sqrt{|\alpha|_\bC}| 1 - \alpha^{-1}|_\bC- | \alpha|_\bC - |\alpha|_\bC^{-1} }  \d^+_\bR r.\]
\end{proposition}
 \begin{corollary}
  If $\alpha = \e^{\im \theta}$, then
\[  J^H_{\sma \alpha & 0 \\ 0 & 1 \smz}(\phi) =\frac{1}{ \sin^2( \theta/2)} \int\limits_{0}^\infty  g_\phi(x)    \frac{\e^{x/2} - \e^{-x/2}}{\e^{x/2} + \e^{-x/2} + 2 \cos( \theta)}  \d^+_\bR r.\]
 \end{corollary}
\begin{proof}
The orbital integral of $\gamma$ is absolutely convergent \cite{Rao}. We subdivide the proof of this theorem in several lemmas.
\begin{lemma}[Explicit form of $w_H$]\label{lemma:wH2}
\[ w_H\left( \sma m_1 & 0  \\  0 & m_2 \smz \sma 1 & x \\ 0 & 1 \smz k \right) = w_H\left(  \sma 1 & |x| \\ 0 & 1 \smz  \right) = \log \left| 1+ |x|^2 \right|_\bC .\]
\end{lemma}
This is double the value than in the real case (Lemma~\ref{lemma:wH}), and the computation is identical.  
\begin{proof}
We have by definition for $b \in \B(\bC)$, $k \in \U(2)$:
 \[ w_H\left( m b k \right) = w_H\left( b \right). \]
It is sufficient to prove the equality for $\B(\bR)$.  We can recycle the proof of Lemma~\ref{lemma:wH}. We have for $t$ real
\begin{align*}
   w_0 \sma 1 & t \\ 0 & 1\smz =   \sma 0 & -1 \\ 1 & t \smz =      \sma \frac{1}{\sqrt{1+t^2}} & \frac{-t}{\sqrt{1+t^2} } \\ 0 & \sqrt{1+t^2}\smz \sma \frac{t}{\sqrt{1+t^2} } & \frac{-1}{\sqrt{1+t^2} }\\ \frac{1}{\sqrt{1+t^2} } & \frac{t}{\sqrt{1+t^2} }  \smz.
\end{align*}
We have for $u \in \{ \pm 1 \}$ and $t= |x| >0$ the decomposition
\[ w_0 \sma m_1 & u m_1 t \\ 0 & m_2 \smz  = \sma  m_2 & 0 \\ 0 & u m_1 \smz   \sma 0 & -1 \\ 1 & t \smz \sma  u^{-1} & 0 \\ 0 & 1 \smz.\]
So this yields for 
\begin{align*} w_H\left( \sma m_1 & 0  \\  0 & m_2 \smz \sma 1 & x \\ 0 & 1 \smz k \right)  &= \log \Delta_B\left( \sma  m_1 & 0 \\ 0 & m_2 \smz \right)   -  \log  \Delta_B \left( \sma m_2  \frac{1}{\sqrt{1+t^2}} & \frac{-t}{\sqrt{1+t^2} } \\ 0 &u m_1 \sqrt{1+t^2}\smz \right)  \\
                                                                                                                                   & =     - \log \left|  \frac{1}{\sqrt{1+t^2}} \right|_\bC + \log \left|  \sqrt{1+t^2} \right|_\bC\\
& =  \log \left| 1+ |x|^2 \right|_\bC= 2 \log( 1 + |x|^2). \qedhere
\end{align*}
\end{proof}
The Iwasawa decomposition allows us to rewrite the orbital integral as an integral over the subgroup $\N(\bC)$. 
\begin{lemma}\label{lemma:hwint}
  For $\phi \in \mSH(\GL_2(\bC), \rho)$, we obtain
 \[ J_\gamma (\phi) = \int\limits_{\N(\bC)} \phi(n^{-1}\gamma n) \d n, \]
and
\[   J^H_\gamma (\phi) = \int\limits_{\N(\bC)} \phi(n^{-1}\gamma n) w_H(n) \d n.\]
 \end{lemma}
\begin{proof}
Set $M= \M(\bC)$. The (weighted) orbital integral on $G = \GL_2(\bC)$ is computed for $\phi \in \mSH(G, \rho)$ as
\[ \int\limits_{M \backslash G}  \phi(g^{-1} m  g)  w_j(g) \d \dot{g} ,\]
where we set
\[ w_j (g) = \begin{cases} 1 , & j=0, \\ w_H, & j = 1. \end{cases}\]
By Lemma~\ref{lemma:wH2}, we have that $w_j(mnk) = w_j(n)$. According to Equation~\ref{eq:complexmnk}, we have that 
\begin{align*} \int\limits_{M \backslash G}  \phi(g^{-1} m  g)  w_j(g) \d \dot{g}  & =  \int\limits_{\N(\bC)}  \int\limits_{\SU(2)}  \phi(k^{-1} n^{-1} m n k)  w_j(nk) \d n \d k \\
                                                                                                                & \underset{\SU(2)-\textup{inv}}=             \int\limits_{\N(\bC)}   \phi( n^{-1} m n )  w_j(n) \d n  .\qedhere
\end{align*}
\end{proof}

\begin{lemma}[The unweighted hyperbolic integral]
 For $\phi \in \mSH(\GL_2(\bC), \rho)$, we have for 
\[ J_{\sma \alpha & 0 \\ 0 & 1 \smz}(\phi) = \frac{\left| \alpha\right|_\bC^{-1/2}}{\left| 1 -\alpha^{-1} \right|_\bC}  U_{n}\left(  \cos(\arg \alpha/2) \right)    g_\phi(2 \log \left|\alpha\right|_\bC ).\]
\end{lemma}
\begin{proof}
After the preceding lemma, the proof follows essentially from the definition of $g_\phi$.  Set $\alpha = \e^{\im \theta} t^2$ for $t>0$. According to the last lemma, we write
\begin{align*}  J_\gamma(\phi)  & = \int\limits_{\bC} \phi\left(  \sma 1 & -n \\ 0 & 1 \smz \sma \alpha & 0 \\ 0 &1 \smz  \sma 1 & n \\ 0 & 1 \smz \right)  \d^+_\bC n \\ 
                                            & = \int\limits_{\bC}  \phi\left(  \sma \alpha & 0\\ 0 & 1 \smz  \sma 1 &  n - n/\alpha \\ 0 & 1 \smz \right)  \d^+_\bC n\\
                                            & =\tr \rho \left( \sma \e^{\im \theta} & 0 \\0 & 1 \smz\right) \frac{1}{ \left| 1 - \alpha^{-1} \right|_\bC} \int\limits_{\bC}  \phi\left(  \sma t & 0\\ 0 & t^{-1} \smz  \sma 1 &  n \\ 0 & 1 \smz \right)  \d^+_\bC n\\
                                            & = \tr \rho \left( \sma \e^{\im \theta} & 0 \\0 & 1 \smz\right) \frac{\left| t \right|_\bC^{-1}}{ \left| 1 - \alpha^{-1} \right|_\bC}    A_\phi \phi\left(  \sma t & 0\\ 0 & t^{-1} \smz  \right). 
\end{align*}
By definition, we have that $A_\phi \phi\left(  \sma t & 0\\ 0 & t^{-1} \smz  \right) = g_\phi\left( 2 \log \left| \alpha \right|_\bC\right)$.
\end{proof}

\begin{lemma}[The weighted hyperbolic integral --- Bi-invariant case]\mbox{}
 For $\phi \in \Ccinf(\GL_2(\bC)//\Z(\bC)\U(2))$, we have
\[ J^H_{\sma \alpha & 0 \\ 0 & 1 \smz}(\phi) =\frac{2}{\lambda^2} \int\limits_{2 \log | \alpha|_\bC}^\infty  g_\phi(x)  \frac{\sinh(x/2) }{\lambda^2+\e^{x/2} + \e^{x/2} -( t^2 + t^{-2}) }  \d^+_\bR r.\]
\end{lemma}
\begin{proof}
After the preceding lemma, the proof follows essentially from the definition of $g_\phi$. Set $\alpha = \e^{\im \theta} t^2$ for $t>0$. According to Lemma~\ref{lemma:hwint}, we write
\begin{align*}  J^H_\gamma(\phi)  & =  2 \int\limits_{\bC} \phi\left(  \sma 1 & -n \\ 0 & 1 \smz \sma \alpha & 0 \\ 0 &1 \smz  \sma 1 & n \\ 0 & 1 \smz \right) \log( 1+ |n|^2)  \d^+_\bC n \\ 
                                            & =  2 \int\limits_0^{2 \uppi } \int\limits_{0}^\infty  \phi\left(  \sma t^2 & 0\\ 0 & 1 \smz  \sma 1 &  (1-1/\alpha) r \e^{\im \psi} \\ 0 & 1 \smz \right)  \log( 1+ r^2) 2 r \d^+_\bR r \d \psi  \\
                                            & =  4 \uppi \int\limits_{0}^\infty  \phi\left(  \sma t& t| 1 - 1 / \alpha| r\\ 0 & t^{-1} \smz  \right)  \log( 1+ r^2) 2 r \d^+_\bR r  \\
                                            &  \underset{\lambda \coloneqq  t| 1 - 1 / \alpha|} =   2 \int\limits_{0}^\infty  \Phi_\phi( t^2 + t^{-2}  + \lambda^2  r^2 -2 )  \log( 1+ r^2) 2 r \d^+_\bR r.
\end{align*}
We refer to the Abel inversion formula and complete the computation with:
\begin{align*}
                                        &  \underset{c =t^2 + t^{-2}}=  -2 \int\limits_{0}^\infty  Q_\phi'( c + \lambda^2  r -2 )  \log( 1+ r)  \d^+_\bR r \\
                                             &      = \frac{2}{\lambda^2} \int\limits_{0}^\infty  Q_\phi( c + \lambda^2  r -2 )  \frac{1}{1+r}  \d^+_\bR r     = \frac{2}{\lambda^2} \int\limits_{ c}^\infty  Q_\phi(   r -2 )  \frac{\d^+_\bR r}{\lambda^2+r -c }   \\
                                             & \underset{r = \e^{x/2} + \e^{-x/2}}    =    \frac{2}{\lambda^2} \int\limits_{2 \log | \alpha|_\bC}^\infty  g_\phi(x)  \frac{\sinh(x/2) }{\lambda^2+\e^{x/2} + \e^{-x/2} -( t^2 + t^{-2}) }  \d^+_\bR r. \qedhere
\end{align*}
 \end{proof}
This completes the proof of the proposition.\end{proof}

\section{The intertwiner and its derivative}
Let $\mu_j$ be an algebraic one-dimensional representation of $\bR^\times$ for $j=1,2$. Let $\mu = (\mu_1, \mu_2)$ be the associated one-dimensional representation of $\M(\bC)$. Set $w_0 = \sma 0 & -1 \\ 1 & 0 \smz$ and $\mu^{w_0} = (\mu_2, \mu_1)$. 

We define the intertwiner
\begin{align*} \mM(\mu, s) &\colon \mJ(\mu, s) \rightarrow \mJ(\mu^{w_0} , -s),  \\
                    \mM(\mu, s)& f(g) \coloneqq \int\limits_{\N(\bR)} f( w_0 n g) \d n. 
\end{align*}
By the Iwasawa decomposition, the smooth function $F\in \mJ(\mu,s)$ is uniquely determined by its value on $\SU(2)$, since by definition
\[  F\left( \sma a & * \\ 0 & b \smz k\right) = |a/b|^{s+1/2} \mu_1(a) \mu_2(b) F(k).\]
The $n$th symmetric tensor representation $\Sym^n(\bC)$ can be realized as the action on the space of complex homogeneous polynomials in two variables of degree $n$
\[ k \in \SU(2) : P([X,Y]) \mapsto P([X,Y]k).\]
A canonical basis is given by the representation theory of $\SU(2)$ by 
\[ F_{\mu, s, m,k} \left( \sma a & * \\ 0 & b \smz k\right) = |a/b|^{2s+1} \mu_1(a) \mu_2(b) X^m Y^k, \qquad m+k =n.\]
We have restricted our attention to bi-invariant test functions. For this, it turns out that we are only required to understand the action of $\mM(\mu, s)$ on $F_{0,0}$ for $\mu$ being the trivial representation. 
\begin{proposition}\label{prop:complexinter}
In the above notation, we have an identity
\[ \mM(1,1, s) F_{1,s, 0,0} =   \frac{1}{2s} F_{1,s,0,0} .\]
\end{proposition}
\begin{proof}
The operator $\mM(\mu, s)$ is an intertwiner. Every $\SU(2)$ representation occurs with multiplicity one. By Schur's Lemma, there exists a unique complex value $\lambda(s, \mu, n) \in \bC$ such that
$$\mM(\mu, s) F_{\mu, s, n} = \lambda(s,\mu,n) F_{\mu^{w_0},-s, n}.$$
A short computation gives the exact value for $s>0$ real, and follows for all $s$ by uniqueness of analytic continuation:
\begin{align*}
\mM(1,s)F_{1, s, 0,0}(1) &\underset{\textup{def.}}= \int\limits_{\bC} F_{1, s,0,0}\left( \sma 0 & -1 \\ 1 & z \smz \right)  \d_\bC^+ z \\
 & \underset{z = t \e^{\im \theta}}= \frac{1}{\uppi} \int\limits_{0}^{2 \uppi}\int\limits_{0}^\infty F_{1, s,0,0}\left( \sma \e^{\im \theta} & 0 \\ 0 & 1 \smz \sma 0 & -1 \\ 1 & t \smz \sma1 & 0 \\ 0 & \e^{-\im \theta} \smz \right) \d \theta t \dr t \\
 &= \int\limits_{0}^\infty F_{1, s,0,0}\left(  \sma 0 & -1 \\ 1 & t \smz  \right) 2 t \dr t \\
 & =  \int\limits_{0}^\infty F_{1, s,0,0} \left(   \sma \frac{1}{\sqrt{1+t^2}} & \frac{-t}{\sqrt{1+t^2} } \\ 0 & \sqrt{1+t^2}\smz \sma \frac{t}{\sqrt{1+t^2} } & \frac{-1}{\sqrt{1+t^2} }\\ \frac{1}{\sqrt{1+t^2} } & \frac{t}{\sqrt{1+t^2} }   \smz \right) 2 t \dr t\\
 & = \int\limits_{0}^\infty   \frac{2t}{( 1+ t^2)^{2s+1}}  \dr t \\
 & \underset{r = t^2+1}= \int\limits_{1}^\infty   r^{-2s-1}  \dr t = \frac{1}{2s}.\qedhere
\end{align*}
\end{proof}

\chapter{Harmonic analysis of $\GL(2)$ over a~non-archimedean~field}
Let $\F$ be a local, non-archimedean field, i.e., a finite extension of the rational $p$-adic numbers $\mathbb{Q}_p$ or the Laurent series in one variable over a finite field.

\section{Haar measure}
Let $\o$ be the ring of integers of $\F$. Let $\p$ be the maximal ideal of $\o$. We fix a generator $\w$ of $\p$, called the uniformizer.

 Let $\Fq$ denote the residue field $\o / \p$, with $q$ elements and of characteristic $p$. We scale the valuation, such that
\[ \left| \w \right| = q^{-1}.\]
 \newcommand{\dv}{\;\textup{d}_v^+}   \newcommand{\dvx}{\;\textup{d}_v^\times}
We endow the groups $\o$, $\o^\times$, and their subgroups $\p^k$, $1+\p^k$ for $k \geq 1$ with unit Haar measures. We endow $\F$ / $\F^\times$ with the unique Haar measures $\dv$ / $\dvx$ such that the open, compact subgroups $\o$ / $\o^\times$ have unit measure.  These normalizations are assumed in \cite{GelbartJacquet}*{Section 7.A, page 241}. Note that these choices coincide with those in the real and complex situation, when we define $\d_v^\times x = \frac{\zeta_v(1) \d^+_v x}{\left| x \right|_v}$ for all valuations. Here, $\zeta_v$ is the local zeta function of $\F$. 

We also define for negative integers $d$ the open, compact subgroup of the additive group $\F$ 
\[ \p^d = \{ x : |x|_v \leq q^{-d}\}.\]
This set has measure $q^{d}$ in $\F$.  Let $k \geq n$. Since $\p^k$ has index $q^{k-d}$ inside $\p^d$, we obtain the integral identity
\[ q^{k-n} \int\limits_{\p^d} f(x) \d x =     \sum\limits_{x \in \p^d \bmod \p^k} \int\limits_{\p^n} f(xy) \d y \qquad f \in \mL^1(\p^n).\]
Observe that $\o^\times = \o - \p$, therefore,
\[ \int\limits_{\F} f(x) \dv x = (1-q^{-1}) \cdot \int\limits_{\o^\times} f(u) \d u \qquad f \in \mL^1(\o^\times).\]
The set $\w^{-n} \o^\times$ has measure $q^{n} - q^{n-1} = q^{n}(1-q^{-1})$ in $\F$. 
The decomposition $\F = \coprod\limits_{n \in \bZ} \o^\times \w^n$ yields that
\[ \int\limits_{\F} f(x) \dv x =  \sum\limits_{k=\infty}^\infty q^{-k} ( 1- q^{-1})   \int\limits_{\o^\times} f(u \w^k) \d x.\]

We consider the locally compact group $\GL_2(\F)$, with its closed subgroups
\begin{align*}
 \N(\F) & \coloneqq \left\{ \sma 1 & x \\ 0 & 1 \smz : x \in \F \right\}, \\
 \M(\F) & \coloneqq  \left\{ \sma \alpha & 0 \\ 0 & \beta  \smz : \alpha, \beta \in \F^\times \right\}, \\
   \Z(\F) & \coloneqq  \left\{ \sma z & 0 \\ 0 &z  \smz : z \in \F^\times \right\}, \\
\B(\F) & \coloneqq  \left\{ \sma \alpha &x \\ 0 & \beta  \smz : \alpha, \beta \in \F^\times, x \in \F \right\}.
\end{align*}
 Only the group $\B(\F)$ is not unimodular.
Define the following pro-finite groups:
\begin{align*}
\GL_2(\o) & := \left\{ \sma a &b \\ c& d \smz: a,b,d,c \in \o, ad-bc \in \o^\times \right\} \\
\B(\o) & := \left\{ \sma a &b \\ 0 & d \smz: a,d \in \o^\times, b \in \o \right\} \\
\M(\o) & := \left\{ \sma a &0 \\ 0 & d \smz: a,d \in \o^\times \right\} \\
\N(\o)      &  :=  \left\{ \sma 1 & b  \\0 & 1 \smz: b \in \o \right\}
\end{align*} 
The group $\GL_2(\F)$ / $\N(\F)$ / $\M(\F)$ / $\Z(\F)$ / $\B(\F)$ admits a unique (left invariant) Haar measure, such that
$\GL_2(\o)$ / $\N(\o)$ / $\M(\o)$ / $\Z(\o)$ / $\B(\o)$ carry a unit measure. All compact groups are always endowed with a probability Haar measure.  

The following integral identities result:
\begin{align*}
 \int\limits_{\M(\F)} j(m) \d m = \int\limits_{\Z(\F)} \int\limits_{\F^\times} j \left( z \sma y & 0 \\ 0 & 1 \smz \right)\d z \dvx y \qquad  j \in \mL^1(\M(\F)),\\
 \int\limits_{\B(\F)} f(b) \d b = \int\limits_{\M(\F)} \int\limits_{\N(\F)}  f(mn) \d m \d n \qquad f \in \mL^1(\B(\F)),\\
  \int\limits_{\GL_2(\F)} h(g) \d b = \int\limits_{\B(\F)}  \int\limits_{\GL_2(\o)}  h(bk) \d m \d n \qquad h \in \mL^1(\GL_2(\F)).
\end{align*}

We say that a one-dimensional representation $\chi : \F^\times \rightarrow \bC^\times$ is algebraic if $\chi(\w)=1$.

\section{The compact, open subgroups}
The representation theory of $\GL_2(\o)$ is far more complicated than that of a compact Lie group. The classification of all irreducible representations of $\GL_n(\o)$ for $n\geq 3$ is not known. The case $\GL_2(\o)$ is well understood \cite{Stasinski:Smooth}.

To classify all irreducible representations of $\GL_2(\o)$ would, in our context, be excessive, since they only interest us as a tool to classify, parametrize, and construct the smooth, admissible representations of $\GL_2(\F)$.
For our purposes, it is only necessary to construct and understand a certain subset of these and a certain subset of those of the Iwahori subgroup
\begin{align*} \Gamma_0(\p) \coloneqq \left\{ \sma a &b \\ c& d \smz: b \in \o, c \in \p, a,d \in \o^\times \right\}.\end{align*}
Due to the work of Bushnell and Kutzko \cite{BushnellKutzko:GLNopen}, we have sufficient control over the representation theory of $\GL_n(\o)$ to construct all the smooth, admissible representations of $\GL_n(\F)$ explicitly. The theory depends only on the residue characteristic of $\F$, and this in a fairly uniform manner. 
\begin{remark}[$\GL(2)$ is easier than $\SL(2)$]
Incidentally, this is one of the main reasons why we went with $\GL(2)$, and not $\SL(2)$. Note that we have a semi-direct product
\[ \PSL_n(\F) \rightarrow \PGL_n(\F) \rightarrow \F^\times / ( \F^\times)^n.\]
The group $\F^\times / ( \F^\times)^n$ depends on the explicit structure of $\F$ if $n$ divides the residue characteristic of $\F$.
\end{remark}
The only finite-dimensional representations are the one-dimensional representations. There are two mutually disjoint sets of infinite-dimensional, irreducible, smooth, admissible representations of $\GL_2(\F)$: 
\begin{itemize}
 \item The parabolic inductions and their subquotients (=Steinberg representations). Here, we have to understand for every one-dimensional representation $\mu : \B(\F) \rightarrow \bC^\times$ the representation
      \[ \Res_{\GL_2(\o)} \Ind_{\B(\F)}^{\GL_2(\F)} \mu = \Ind_{\B(\o)}^{\GL_2(\o)}  \Res_{\B(\o)} \mu.\]
 \item The irreducible, supercuspidal representations. These are isomorphic to {compactly} induced representations of maximal compact-mod-center subgroups of~$\GL_2(\F)$. This can be examined in several steps
          \begin{itemize}
          \item Classify the maximal  compact-mod-center subgroups  of $\GL_2(\F)$: These are the subgroups $\GL_2(\o) \Z(\F)$ and the normalizer of $\Gamma_0(\p)$.
          \item Relate the representation theory of $\GL_2(\o)$ and $\Gamma_0(\p)$ to the representation of the maximal compact-mod-center subgroups. 
          \item Identify and construct all the irreducible representations of $\GL_2(\o)$ and $\Gamma_0(\p)$, such that the related irreducible representations $\{ \rho \}$ of the maximal compact-mod-center subgroups $K$ give all the irreducible supercuspidal representations via
                                    \[ \ind_K^{\GL_2(\F)} \rho.\]
          \end{itemize}         
\end{itemize}
I deal with the two cases disjointly. This is not entirely necessary, as Bushnell and Kutzko's theory allows us to view all the irreducible, smooth, admissible representations in the same framework. 

The representation theory of profinite groups reduces to that of finite groups. Every irreducible, topological representation of a profinite group has a finite index kernel. We will frequently utilize results in the representation theory of finite groups within the context of profinite groups.

\subsection{Parabolically induced representations of $\GL(2,\o)$}
Let $\o$ be the ring of integers of a non-archimedean field $\F$, and let $\p$ be its unique maximal ideal.

We will now analyze the representations of $\GL_2(\o)$ associated with the parabolic inductions. This was previously obtained in \cite{Casselman:Restr}*{Theorem 1} and \cite{Silberger:PGL2}*{Theorem 3.3, page 58}. 

\begin{lemma}\label{lemma:rhopara}
  For an irreducible representation $\rho$ of $\GL_2(\o)$, the following are equivalent:
\begin{enumerate}[font=\normalfont]
 \item the trivial representation of $\N(\o)$ is contained in the restriction $\Res_{\N(\o)} \rho$,
 \item there are two one-dimensional unitary characters $\mu_j : \o^\times \rightarrow\bC^1$ such that $\rho$ is contained in the induced representation  \( \Ind_{\B(\o)}^{\GL_2(\o)} (\mu_1, \mu_2) \) , i.e., in the right regular representation on smooth functions
      \[ f\left( \sma a & b \\ 0 & d\smz  k \right) = \mu_1(a) \mu_2(d) f(k).\] 
\end{enumerate}
\end{lemma}
\begin{proof}
 If the following intertwiner ring is non-zero
\[ \Hom_{\N(\o)} ( 1, \Res_{\N(\o)} \rho) \neq \{ 0 \}, \]
then, via the Frobenius reciprocity theorem, the same is true for:
 \[ \Hom_{\GL_2(\o)} (  \Ind_{\N(\o)}^{\GL_2(\o)} 1, \rho) \neq \{ 0 \}.\]
Induction by steps produces
\[     \Ind_{\N(\o)}^{\GL_2(\o)} 1  =  \Ind_{\B(\o)}^{\GL_2(\o)}  \Ind_{\N(\o)}^{\B(\o)} 1.\]
 Let $\iota_{\M(\o)}^{\B(\o)}$ denote the inflation functor and $ \widehat{\M(\o)}$ the discrete Pontryagin dual of $\M(\o)$, then we have that
 \[ \Ind_{\N(\o)}^{\B(\o)} 1 \cong\bigoplus\limits_{ \chi \in \widehat{\M(\o)}} \iota_{\M(\o)}^{\B(\o)} \chi . \qedhere\]
\end{proof}
\begin{defn}
An irreducible unitary representation $\rho$ of $\GL_2(\o)$ is 
\begin{enumerate}[font=\normalfont]
 \item supercuspidal, if $\Res_{\N(\o)} \rho$ does not contain the trivial representation, 
 \item parabolically induced, if $\Res_{\N(\o)} \rho$ does contain the trivial representation.
\end{enumerate}
\end{defn}
Let us motivate the term ``supercuspidal'' in this context.
\begin{theorem}[\cite{Bushnell:Induced}*{Theorem 1 suppl., page 111}, \cite{Mautner:Spherical2}]\label{thm:propsc}
Let $\rho$ be a supercuspidal representation of $\GL_2(\o)$, and let $\chi$ be a one dimensional representation of $\Z(\F)$ such that the central characters of $\rho$ and $\chi|_{\Z(\o)}$ coincide. The compactly induced representation
\[ \ind_{\GL_2(\o) \Z(\F)}^{\GL_2(\F)} \rho \]
decomposes as a finite sum of irreducible, supercuspidal representations of $\GL_2(\F)$. The converse holds as well.
\end{theorem}
At this point, we turn our attention to the understanding of parabolic inductions, for which only the parabolically induced representations of $\GL_2(\o)$ are important.
We aim for a simple notation. Every parabolically induced representation 
\[ \rho \subset    \Ind_{B(\o)}^{\GL_2(\o)} (\mu_1, \mu_2) \]
is brought, after a twist by a one-dimensional representation, into the form
\[ \rho \otimes  \mu_2^{-1} \circ \det  \subset \Ind_{B(\o)}^{\GL_2(\o)} (\mu_1 \mu_2^{-1}, 1).\]
We will only classify irreducible subrepresentations of 
\[ J_\o(\mu) \coloneqq \Ind_{B(\o)}^{\GL_2(\o)} (\mu, 1).\] 
\begin{defn}
The conductor of $\mu : \o^\times \rightarrow \bC^1$ is the largest ideal $\p^N$ such that $\mu|_{1+\p^N} = 1$. We write $\cond(\mu) = \p^N$. 
\end{defn}
Define a family of open, compact, normal subgroups in $\GL_2(\o)$ for all $R \geq 1$
\begin{align*}
     \Gamma ( \p^R) & =\left\{ \sma a & b \\ c & d \smz : a,d = 1 \bmod \p^R, b,c \in \p^R \right\}, \\
      \Gamma_0 ( \p^R) & =\left\{ \sma a & b \\ c & d \smz : a,d \in \o^\times, b \in \o,c \in \p^R \right\}.
\end{align*}

\begin{lemma}
Let $R \geq 1$, the representation $\mJ_\o(\mu)$ then has a $\Gamma(\p^R)$-invariant vector if and only if $\p^R \subset \cond(\mu)$.
\end{lemma}
\begin{proof}
 The group $\Gamma(\p^R)$ is normal. If $b \in \Gamma(\p^R) \cap \B(\o)$, we have that
\[ f(kb) = f(bk') = \mu(b_{11}) f(k') , \qquad b = \sma b_{1,1} & * \\ 0 & * \smz,\]
so the condition is clearly necessary. In the case that $\p^R \subset \cond(\mu)$, we can consider the function
\[ f: x \mapsto \begin{cases} \mu(x_{1,1}) , & \; x = \sma x_{1,1} & * \\ * & * \smz  \in \Gamma_0(\p^R),\\ 0, & \textup{else}. \end{cases}\]
This is a $\Gamma(\p^R)$-invariant vector, so the condition is here sufficient as well. 
\end{proof}

\begin{lemma}
Let $\p^R \subset \cond(\mu)$ for $R \geq 1$. The subrepresentation of $\Gamma(\p^R)$-invariant vectors of $J_\o(\mu)$ is isomorphic to the representation
\[ \Ind_{\Gamma_0(\p^R)}^{\GL_2(\o)} \mu, \qquad \mu : \sma a &b \\ c & d \smz \in \Gamma_0(\p) \mapsto \mu(a).\]
\end{lemma}
\begin{proof}
 The group $\Gamma(\p^R)$ is normal, thus
\[     \left(    \Ind_{\B(\o)}^{\GL_2(\o)} \mu  \right)^{\Gamma(\p^R)} \cong \iota^{\GL_2(\o)}    \Ind_{\B(\o/\p^R)}^{\GL_2(\o/\p^R)} \mu \cong  \Ind_{\Gamma_0(\p^R)}^{\GL_2(\o/\p^R)} \mu,\]
where $\iota^{\GL_2(\o)}$ is the inflation functor of $\GL_2(\o/\p^R)$ representations to $\GL_2(\o)$ representations. 
\end{proof}
\begin{lemma}\label{lemma:triv}
The trivial representation is contained in $J_\o(1)$ with simple multiplicity. No other one-dimensional representation that factors through the determinant map is contained in $J_\o(\mu)$ for each one-dimensional representation $\mu$ of $\o^\times$.
\end{lemma}
\begin{proof}
This is a straightforward consequence of Frobenius reciprocity and the preceding lemma. Let $R \geq 1$. Let $\mu$ and $\tilde{\mu}$ be one-dimensional representations of $\o^\times$. We have isomorphisms of endomorphism rings:
\begin{align*} \Hom_{\GL_2(\o)} (  \tilde{\mu} \circ \det ,  \Ind_{\Gamma_0(\p^R)}^{\GL_2(\o)}  \mu ) & \cong  \Hom_{\Gamma_0(\p^R)} (  \tilde{\mu} \circ \det ,  \mu ) \\ &\cong \begin{cases} \bC, & \mu = \tilde{\mu} = 1, \\ \{0 \}, & \textup{else}.\end{cases} \qedhere \end{align*}
\end{proof}
As usual, let $\ominus$ denote the orthogonal difference.
\begin{defn}\label{defn:rhomup}
Consider a one-dimensional unitary representation $\mu$ of $\o^\times$ with conductor $\p^N$. Define representations of $\GL_2(\o)$ by
\[ \rho(\mu) = \rho(\mu, \p^N) := \begin{cases} 1, & \mu = 1, \\  \Ind_{\Gamma_0(\p^N)}^{\GL_2(\o)} \mu, & \mu \neq 1, \end{cases} \]
and, by induction, for $R>N$
\[   \rho(\mu, \p^{R}) := \Ind_{\Gamma_0(\p^{R})}^{\GL_2(\o)} \mu \ominus \left( \bigoplus_{k=N}^{R-1}\rho(\mu, \p^k) \right).\]
\end{defn}
For our purposes, it is sufficient to understand $\rho(\mu)$ and $\rho(1,\p)$. However, the reader may find it beneficial to analyze all of them.
\begin{theorem}\label{thm:globpara}
The representation $\rho(\mu, \p^R)$ is irreducible.
\end{theorem}
We need the Mackey induction restriction formula (short: MIR).
\begin{theorem}[Mackey induction restriction formula \cite{Weintraub}]\label{thm:MIR}
Let $K$ and $H$ be closed subgroups of a profinite group $G$. Assume that $K$ has finite index in $G$. Let $\rho$ be a finite-dimensional representation of $K$. Then
\[ \Res_{H} \Ind_K^{G}        \rho = \bigoplus_{\gamma \in H \backslash G / K} \Ind_{\gamma K \gamma^{-1} \cap H}^H    \Res_{\gamma K \gamma^{-1} \cap H}  \left( x \mapsto \rho( \gamma^{-1} x \gamma) \right).\]
\end{theorem}

\begin{proof}[Proof of Theorem~\ref{thm:globpara}]
Set $R \geq 1$, since the case $R=0$ is solved by Lemma~\ref{lemma:triv}. The endomorphism ring satisfies according to Frobenius reciprocity and the MIR:
\begin{align}
 &\Endo_{\GL_2(\o)} ( \Ind_{\Gamma_0(\p^R)}^{\GL_2(\o)} \mu) \\
&\underset{\textup{Frob.rec.}}\cong  \Hom_{\Gamma_0(\p^R)} ( \Res_{\Gamma_0(\p^R)} \Ind_{\Gamma_0(\p^R)}^{\GL_2(\o)} \mu, \mu) \\
                                                                                        &{}\quad \underset{\textup{MIR}}\cong  \bigoplus\limits_{\gamma \in \GL_2(\o) // \Gamma_0(\p^R)} \Hom_{ \Gamma_0(\p^R)} \left( \Ind_{\Gamma_0(\p^R) \cap \gamma^{-1} \Gamma_0(\p^R) \gamma}^{\Gamma_0(\p^R)} \mu( \gamma \blank \gamma^{-1}), \mu\right) \\
             & \underset{\textup{Frob.rec.}}\cong \bigoplus\limits_{\gamma \in \GL_2(\o) // \Gamma_0(\p^R)} \Hom_{\Gamma_0(\p^R) \cap \gamma^{-1} \Gamma_0(\p^R) \gamma} \left(  \mu( \gamma \blank \gamma^{-1}), \mu\right).
\end{align}
Now the coset space $\GL_2(\o) // \Gamma_0(\p^R)$ is computed from the Bruhat decomposition over the residue field \cite{BushnellHenniart:GL2}*{page 44}                    
\[ \GL_2(\o) = \Gamma_0(\p) \,\amalg \Gamma_0(\p) w_0 \Gamma_0(\p), \qquad w_0 = \sma 0 & -1 \\ 1 & 0 \smz,\]
and the Iwahori decomposition \cite{BushnellHenniart:GL2}*{(7.3.1), page 52}
\[ \Gamma_0(\p)  = w_0 \N(\p) w_0 \cdot \M(\o) \cdot \N(\o) = \N(\o) \cdot  \M(\o) \cdot  w_0 \N(\p) w_0.\]
A set of representatives for $\Gamma_0(\p^R)$-double coset is given by
\[   \GL_2(\o) = \coprod\limits_{k=0}^R \Gamma_0(\p^R) w_k \Gamma_0(\p^R),\]
where $w_0$ is defined as above, and for $k>0$, define $w_k = \sma 1 & 0 \\ \w^k & 1  \smz$. Set
\[ d_k = \dim_\bC    \Hom_{\Gamma_0(\p^R) \cap w_k^{-1} \Gamma_0(\p^R) w_k} \left(  \mu( w_k \blank \w_k^{-1}), \mu\right) \in \{ 0,1 \}.\]
We observe that $d_0=1$ if and only if $\mu = 1$ and $d_k = 1$ for all $k \geq 1$. The dimension is therefore computed  
  \[  D_R = \dim_\bC \Endo_{\GL_2(\o)} ( \Ind_{\Gamma_0(\p^R)}^{\GL_2(\o)} \mu)  = \begin{cases} R+1, & \mu = 1, \\  R, & \mu \neq 1. \end{cases}\]
We see that $D_{R+1} - D_{R}=1$, which completes the proof.
\end{proof}

\begin{theorem}\label{thm:globpara2}
The representations $\rho(\mu, \p^R)$ and $\rho(\tilde{\mu}, \p^{\tilde{R}})$ are isomorphic if and only if $R = \tilde{R}$ and $\mu$ and $\tilde{\mu}$ coincide on $1+ \p^R$.
\end{theorem}
\begin{proof}
The implication  $R = \tilde{R}$ follows, because the dimension of the representations must coincide. If $R \geq 1$, then the cardinality of the coset space
\( \Gamma_0(\p^R) \backslash \GL_2(\o) \) is precisely $q^R(1+q^{-1})$ for $R\geq 1$, with $q$ being the cardinality of the residue field. The proof of the remaining claim is identical to that of the preceding theorem. We obtain an isomorphism
\begin{align*}
 &\Endo_{\GL_2(\o)} ( \Ind_{\Gamma_0(\p^R)}^{\GL_2(\o)} \mu_1, \Ind_{\Gamma_0(\p^R)}^{\GL_2(\o)} \mu_2 ) \\ 
 &=    \bigoplus\limits_{k=0}^R \Hom_{\Gamma_0(\p^R) \cap w_k^{-1} \Gamma_0(\p^R) \gamma} \left(  \mu_1( \gamma \blank \gamma^{-1}), \mu_2\right).
\end{align*}
 Set
\[  d_k  =\dim_\bC \Hom_{\Gamma_0(\p^R) \cap w_k^{-1} \Gamma_0(\p^R) \gamma} \left(  \mu_1( \gamma \blank \gamma^{-1}), \mu_2\right) \in \{0, 1\}.\]
We have that $d_0=1$ if and only if $\mu_1 = \mu_2 =1$. Additionally, we know that $d_k=1$ for $k \geq 0$ if and only if $\mu_1$ and $\mu_2$ coincide on $1+\p^k$.
\end{proof}

\subsection{Unramified cuspidal types}
We define a neighborhood base of open, normal subgroups of $\GL_2(\o)$
\[    \Gamma(\p^N) = \left\{ \sma a & b \\  c & d \smz: a,d \in 1+\p^N ,  b,c \in \p^N \right\}.\]
We define the level of a finite-dimensional representation $\rho$ as the integer
\[ \ell(\rho) \coloneqq \min\{ N : 1 \subset \Res_{\Gamma(\p^{N+1})} \rho\}. \]
Because $\Gamma(\p^{N+1})$ is normal, we have that 
\[\Gamma(\p^{\ell(\rho)+1}) \supset \ker(\rho) \supset \Gamma(\p^{\ell(\rho)}).\]
We say that $\rho$ is minimal if 
\begin{align}\label{eq:minimal}\ell(\rho) \leq \ell( \rho  \otimes \chi \circ \det) \end{align}
for all one-dimensional representations $\chi: \o^\times \rightarrow \bC$.
\begin{defn}
An unramified cuspidal type is an irreducible representation $\rho$ of $\GL_2(\o)$, such that
\begin{itemize}
 \item the representation $\rho$ is minimal,
 \item the compactly induced representation $\ind_{\Z(\F)\GL_2(\o)}^{\GL_2(\F)} \chi \rho$ is irreducible for all one-dimensional representations $\chi:\Z(\F) \rightarrow \bC^\times$, which coincide with $\rho$ on $\Z(\o)$. 
\end{itemize}
\end{defn}
\begin{corollary}[of Theorem~\ref{thm:propsc}]
An unramified cuspidal type is a supercuspidal representation of $\GL_2(\o)$.
\end{corollary} 
The converse is not true. The adjective ``unramified'' in the above definition suggests that the unramified cuspidal types are constructed from unramified quadratic extensions of $\F$. Note that this differs significantly from the meaning of the adjective ``unramified'' for a principal series representation. A principal series representation is called unramified if it admits a $\GL_2(\o)$-invariant vector.
\paragraph{Case 1: $\ell(\rho)=0$}
The cuspidal types of level zero and the associated irreducible supercuspidal representations are called depth-zero.
\begin{thm}[\cite{BushnellHenniart:GL2}*{Section 6.4, page 47 and Section 15.5, page 107}]\label{thm:scone}
Let $\Fq$ be the residue field of $\F$. We have a one-to-one correspondence between:
\begin{enumerate}[font=\normalfont]
 \item the unramified cuspidal types of level zero, 
 \item the supercuspidal representations of $\GL_2(\Fq)$, and
 \item the regular quadratic characters of $\Fq$, i.e., one-dimensional representations of the (unique up-to-isomorphism) quadratic extension $E$ of $\Fq$, which are not fixed under the action of the Frobenius map
\[ \textup{Frob} : x \mapsto x^q.\]
\end{enumerate} 
\end{thm}
Let us explain this correspondence. 

From (1) to (2) and back: By definition, an unramified cuspidal type of level zero factors through the group $\GL_2(\o) /\Gamma(\p)= \GL_2(\Fq)$. On the other hand, every representation of $\GL_2(\Fq)$ can be inflated to a representation of $\GL_2(\o)$.

From (2) to (3) and back: The group $E^\times$ embeds into $\GL_2(\Fq)$ as the centralizer of an element whose characteristic polynomial is irreducible. Let $\rho$ be a supercuspidal representation of $\GL_2(\F)$, then there exists a one-dimensional representation \[ \theta \neq \theta \circ \textup{Frob} \] of $\E^\times$, such that
\begin{align}   
&\rho \cong \Ind_{\Z\N(\Fq)}^{\GL_2(\Fq)} \psi_\p \cdot \theta|_{\Z(\Fq)}  \ominus        \Ind_{E^\times}^{\GL_2(\Fq)} \theta,  \\
&\tr \rho(e) = - ( \tr \theta(e) + \tr \theta( \textup{Frob}(e)),             \qquad e \in E^\times - \Z(\Fq).
\end{align}
Here, $\psi_\p$ is an arbitrary nontrivial character of $\N(\bF_q)$.

\subsubsection{Clifford theory}
If the level is larger than one, the irreducible representations of $\GL_2(\o)$ are best classified via the neighborhood base of open, normal subgroups $\{ \Gamma(\p^N) : N\geq 1\}$. 
Clifford theory allows a parametrization of the irreducible representations of a finite group in terms of its normal subgroups.
\begin{theorem}[Clifford's Theorem]\label{thm:clifford}
Let $G$ be a finite group. Let $H$ be a normal subgroup of $G$. The groups $G$ resp. $G/H$ act on the irreducible representations of $H$ via conjugation on $H$.
\begin{enumerate}[font=\normalfont]
 \item Let $\rho$ be an irreducible representation of $G$, then the restriction \( \Res_{H} \rho \) contains precisely one $G$-orbit of irreducible representations of $H$.
 \item Let $\psi$ be an irreducible representation of $H$, and let $G_\psi$ be its stabilizer in $G$. We have a one-to-one correspondence between irreducible representations $\rho_0$ of $G_\psi$, contained in $\Ind_{H}^{G_\psi} \psi$, and irreducible representations $\rho$ of $G$, contained in $\Ind_{G_\psi}^{G} \psi$. The correspondence is given by
\[ \rho_0 \mapsto \rho = \Ind_{G_\psi}^{G} \rho_0.\]
\end{enumerate}
\end{theorem}
A reference can be found in \cite{Isaacs}*{Theorem 6.2, pg.76 and Theorem 6.11, pg.82}. 

We want to apply Clifford theory to the irreducible representations of $\GL_2(\o)$ of level $n \geq 1$, we follow \cite{Stasinski:Smooth}. The quotient $\Gamma(\p^{n}) / \Gamma(\p^{n+1})$ is abelian and isomorphic to the endomorphism ring $M_2(\Fq)$ of $\Fq \oplus \Fq$ via
\[ \iota_n \colon \Gamma(\p^{n}) / \Gamma(\p^{n+1}) \xrightarrow\cong M_2(\Fq), \qquad x \mapsto (x-1)/\w^{n} \bmod \p.\]
According to Clifford theory, the restriction
\[ \Res_{\Gamma(\p^{n})} \rho \]
decomposes into a $\GL_2(\o)$-orbit of one-dimensional representations 
\[ \psi: \Gamma(\p^n) \rightarrow \bC, \qquad \psi|_{\Gamma(\p^{n+1})} = 1.\]
The Pontryagin dual $\widehat{M_2(\Fq)}$ is canonically isomorphic to $M_2(\Fq) = \Gamma(\p^{n}) / \Gamma(\p^{n+1})$ via
\[      M_2(\Fq)  \xrightarrow{\cong}    \widehat{M_2(\Fq)}, \qquad x \mapsto \psi_x, \qquad \psi_x(y) \coloneqq \psi_{\p} \circ \tr \left( x \cdotp \right).\]
The orbit space of the conjugation action of $\GL_2(\o)$ on  $\widehat{ \Gamma(\p^{n-1}) / \Gamma(\p^n)} $ is isomorphic to the orbit space of the action of $\GL_2(\Fq)$ on the ring $M_2(\Fq)$ by conjugation with inverse elements.

We will only classify the unramified classical types via Clifford theory. A classification of all irreducible representations of $\GL_2(\o)$ is provided in \cite{Stasinski:Smooth}.

\begin{defn}
The stratum of a minimal, irreducible representation $\rho$ of $\GL_2(\o)$ is a one-dimensional representation 
\[ \psi \colon \Gamma( \p^{\ell(\rho)} ) \rightarrow \bC^1, \]
contained in the restriction
\[ \psi \subset \Res_{\Gamma(\p^{\ell(\rho)})} \rho.\]
\end{defn}

\begin{thm}
A minimal, irreducible representation $\rho$ of $\GL_2(\o)$ of level $\ell(\rho) \geq 1$ is  an unramified cuspidal type if and only if there exists an elliptic element $x_\rho$ in $\M_2(\Fq)$, such that $\psi = \psi_{x_\rho} \circ \iota_{\ell(\rho)}$ is a stratum of $\rho$. Let
$G_\psi = \GL_2(\o)_\psi$. In this case, there exists a unique irreducible representation $\Lambda_{\rho}$ of $G_\psi$ contained in $\Ind_{\Gamma(\p^{\ell(\rho})}^{G_\psi} \psi$ with
 \[ \rho \cong \Ind_{G_\psi}^{\GL_2(\o)} \Lambda_\rho.\]
\end{thm}
\begin{proof}
 Combine the classification theorem on page 108 with the definition on page 99 in \cite{BushnellHenniart:GL2}. The rest follows from Clifford's Theorem.
\end{proof}

\paragraph{Case 2: $\ell(\rho)$ odd}
Let $\rho$ be a cuspidal type of odd level $\ell(\rho) = 2n-1$. Here, the associated representation $\Lambda_\rho$ of
\[ G_\psi = \Gamma(\p^{n-1}) \left(  \o[ \beta ] \right)^\times , \qquad \beta \bmod p =x_\rho\]
 is one-dimensional.

\begin{theorem}
All cuspidal types of level $2n-1$ with stratum $\psi$ are associated to an elliptic conjugacy class $x$ in $\GL_2(\Fq)$ with $\beta \in \o$, such that $\beta \bmod p = x$ are in one-to-one correspondence with characters
\[ \theta :\left(  \o[ \beta ] \right)^\times \rightarrow \bC^1,\]
coinciding with $\tilde{\psi}$ on $ \left(  \o[ \beta ] \right)^\times   \cap \Gamma(\p^{n-1})$, where $\tilde{\phi}$ is an extension $\phi$ to $\Gamma(\p^{n-1})$. For every cuspidal type $\rho$ of level $2n+1$, there exists a unique character $\theta$ as above, such that
\[ \rho \cong       \Ind_{\Gamma(\p^{n-1}) \left(  \o[ \beta ] \right)^\times }^{\GL_2(\o)}       \theta \cdot \tilde{\psi}.\]
The construction is independent of the extension $\tilde{\phi}$ being considered.
\end{theorem}
\begin{proof}
 According to Clifford theory, we have a one-to-one correspondence of cuspidal types of level $2n-1$ with irreducible sub-representations $\Ind_{\Gamma(\p^{\ell(\rho)})}^{\Gamma(\p^{\ell(\rho)}) \left(  \o[ \beta ] \right)^\times} \psi$.
We can extend $\psi$ to a one-dimensional representation $\tilde{\psi}$ of $\Gamma(\p^{n-1})$, since
\[          \Gamma(\p^{n-1}) / \Gamma(p^{2n-2}) \]
is abelian. Let $\tilde{\psi}' $ be another such extension. We have an isomorphism
\[      \Ind_{\Gamma(\p^{n-1})}^{\Gamma(\p^{n-1}) \left(  \o[ \beta ] \right)^\times} \tilde{\psi}   \cong  \Ind_{\Gamma(\p^{n-1})}^{\Gamma(\p^{n-1}) \left(  \o[ \beta ] \right)^\times} \tilde{\psi}'.\]
This follows by applying Clifford theory to the functor $\Res_{\Gamma(\p^{n-1})}$ and from the observation that $\tilde{\psi}$ and $\tilde{\psi}'$ must be in the same $\GL_2(\o)$-conjugacy class.
 Since the quotient
\[ \Gamma(\p^{n-1}) \left(  \o[ \beta ] \right)^\times  / \Gamma(\p^{\ell(\rho)})  \]
is abelian, the final result is derived from the Fourier theory for abelian groups.
\end{proof}

\paragraph{Case 3: $\ell(\rho)$ even}
Let $\rho$ be a cuspidal type of even level $\ell(\rho) =2n$. In this case, the associated representation $\Lambda_\rho$ of
\[ G_\psi = \Gamma(\p^{n-1}) \left(  \o[ \beta ] \right)^\times , \qquad \beta \bmod p =x_\rho\]
 is $q$-dimensional.

\begin{theorem}
The cuspidal types of level $2n$, associated to an elliptic conjugacy class $x$ in $\GL_2(\Fq)$ with $\beta \in \o$, such that $\beta \bmod p = x$, are in one-to-one correspondence with characters
\[ \theta :\left(  \o[ \beta ] \right)^\times \rightarrow \bC^1,\]
coinciding with $\psi$ on $ \left(  \o[ \beta ] \right)^\times   \cap \Gamma(\p^{n-1})$. The representation $\chi \cdotp \psi$ admits a unique extension $\Lambda_\theta$ to $\Gamma(\p^{\ell(\rho)+1}) \left(  \o[ \beta ] \right)^\times$ with
\[ \Res_{ \left(  \Gamma(\p^{2n+1}) \o[ \beta ] \right)^\times }^{\Gamma(\p^{\ell(\rho)+1}) \left(  \o[ \beta ] \right)^\times}    \Lambda_\theta = q  \cdot \theta.\]
For every cuspidal type $\rho$ of level $2n$, there exists as above a unique character $\chi$, such that
\[ 
  \rho \cong       \Ind_{\Gamma(\p^{\ell(\rho)+1}) \left(  \o[ \beta ] \right)^\times }^{\GL_2(\o)}      \Lambda_\theta .\]
\end{theorem}
\begin{proof}
The argument for extension of $\psi$ to $\Gamma(\p^{n-1})$ functions as before. Since the quotient
\[ \Gamma(\p^{\ell(\rho)}) \left(  \o[ \beta ] \right)^\times  / \Gamma(\p^{\ell(\rho)}) \]
is not abelian, the induced representation
\[  \Ind_{\Gamma(\p^{n-1})}^{\Gamma(\p^{n-1}) \left(  \o[ \beta ] \right)^\times} \tilde{\psi}  \]
is far more difficult to decompose. We refer to \cite{BushnellHenniart:GL2}*{Lemma, page 109} and \cite{Stasinski:Smooth}*{Section 4.2, page 4422}.
\end{proof}

\newcommand{\NGp}{\textup{N}\!\Gamma\!_0(\p)}
\subsection{Ramified cuspidal types}
Ramified cuspidal types are certain irreducible representations of $\Gamma_0(\p)$. Recalling Clifford theory, we define a neighborhood base of open, normal subgroups of $\Gamma_0(\p)$ for $k \geq 0$
\begin{align*} U^{2k} &= \Gamma_0^1(\p^{k+1}, \p^k) \coloneqq \left\{ \sma \alpha & \beta \\ \gamma & \delta \smz : \alpha , \delta \in 1 + \p^k, \beta \in \p^k, \gamma \in \p^{k+1} \right\}, \\
U^{2k+1} &= \Gamma_1^0(\p^k, \p^{k-1}) \coloneqq \left\{ \sma \alpha & \beta \\ \gamma & \delta \smz : \alpha , \delta \in 1 + \p^k, \beta \in \p^{k-1}, \gamma \in \p^k \right\}, \\ 
                                     \end{align*}
with $U^k \supset U^{k+1}$. There are conceptual considerations for these decisions \cite{BushnellHenniart:GL2}*{Section 12, page 86}.   We have that
\[ \Gamma(\p^{k}) \supset U^{2k} \supset U^{2k+1} \supset \Gamma(\p^{k+1}).\]
 The normalizer $\NGp$ of $\Gamma_0(\p)$ in $\GL_2(\F)$ is the semi-direct product
\[ \NGp = \Gamma_0(\p) \rtimes \langle w_\p \rangle, \qquad w_\p \coloneqq \sma 0 & 1 \\ \w & 0 \smz .\]
Every maximal, open subgroup of $\GL_2(\F)$ which is compact modulo the center is conjugate to either $\Z(\F) \GL_2(\o)$ or $\NGp$.
The level of an irreducible representation $\rho$ of $\Gamma_0(\p)$ is defined as the integer or half-integer
\[ \ell(\rho) = \min\{ n/2 : 1 \subset \Res_{U^{n+1}} \rho \}.\]
We say that $\rho$ is minimal if
\begin{align}\label{eq:minimal2} \ell(\rho) \leq \ell(\rho \otimes \chi \circ \textup{det}) \end{align}
for all one-dimensional representations $\chi : \o^\times \rightarrow \bC^1$. The irreducible representations of $\Gamma_0(\p)$ are related to the irreducible representations of $\NGp$ in the following fashion:
\begin{theorem}\label{thm:redKeU0} \mbox{}
The irreducible representations $\rho$ of $\NGp$ with an algebraic central character are isomorphic to  either
\begin{enumerate}[font=\normalfont]
 \item one of the two non-isomorphic extensions $\rho_1$ and $\rho_2$ of the representations $\infl_{\Gamma_0(\p)}^{\Z(\F)\Gamma_0(\p)} \;\rho_0$ of an irreducible representation $\rho_0$ of $\Gamma_0(\p)$ with $\rho_0 = \rho_0^{w_\p}$, and, in this instance,
  \[  \Res_{\Gamma_0(\p)} \rho_j =  \rho_0 . \]
The extensions $\rho_1$ and $\rho_2$ become isomorphic after a twist by a one-dimensional representation $\chi = \left| \det( \cdotp) \right|_v^{\uppi \im / \log (q_v)}$ of $\NGp$ with order two:
\[ \rho_1 \cong \rho_2 \otimes \chi\]
 \item an induced-inflated representation 
\[ \rho \cong  \Ind^{\NGp} \infl_{\Gamma_0(\p)}^{\Z(\F)\Gamma_0(\p)} \; \rho_0 \]
 of an irreducible representation $\rho_0$ of $\Gamma_0(\p)$, with $\rho \neq \rho^{w_\p}$.
\end{enumerate} 
\end{theorem}
\begin{proof}
We appeal to the Mackey machine for the group extension of $H = \Z(\F) \Gamma_0(\p)$ by $ \langle w_\p \rangle /\langle w_\p^2 \rangle$. Note that conjugation by $w_\p^2 = \sma \w & 0 \\  0 & \w\smz$ acts trivially via conjugation on representations with  $\Z(\F^1)$-central character. Let $\rho_0$ be an irreducible representation of $H$ inflated from  $\Gamma_0(\p)$. Then there are two classes of orbits in $H$: 
\begin{enumerate}[font=\normalfont]
\item If $\rho_0 = \rho_0^{w_\p}$, then the stabilizer of $\rho$ is the full group $\NGp$. The quotient is cyclic, hence the existence of a representation $\rho_{1}$ of $\NGp$ with $\rho = \Res \rho_1$ is granted.
\begin{theorem}[\cite{Isaacs}*{Corollary 11.22, page 186}]\label{thm:extcyclic}
Let $H$ be a normal subgroup of a finite group $G$. Let $\rho_0$ be an irreducible representation of $H$, and let $H_0$ be the stabilizer of $\rho_0$ in $G$.

If $H_0 / H$ is a cyclic group, there is an extension of $\rho_0$ to $H_0$.
\end{theorem}
In this special case, the induced representation $\Ind^{\NGp}_{H} \rho_0$ is easily decomposed:
\begin{theorem}[Gallagher's Theorem \cite{Isaacs}*{Corollary 6.17, pg.85}]\label{thm:gallagher}
Consider a chain $H \subset H_0 \subset G$ of finite groups, such that $H$ is normal in $G$. Let $\rho_0$ be an irreducible representation of $H$ which has an extension $\rho_1$ to $H_0$. Thus,  
 \[ \Ind_{H}^{H_0} \rho = \bigoplus_{\kappa \in \textup{Irr}(H_0/H) } \rho_1 \otimes \infl^{H_0} \kappa.\]                                                                                                                                                                                      
\end{theorem}
We consequently obtain
\[ \Ind^{\NGp}_{H} \rho_0  = \bigoplus\limits_{\chi_0 \in \widehat{\NGp/H}} \chi_0 \otimes \rho_1 = \rho_1 \oplus \left( \rho_1 \otimes \chi\right).\]
\item If $\rho \neq \rho^{w_\p}$, then the stabilizer of $\rho$ is $H$, and \(  \Ind^{\NGp}_{H}   \rho\) is irreducible. \qedhere
\end{enumerate} 
\end{proof}

\begin{defn}[Ramified cuspidal types]
A ramified cuspidal type is a representation $\rho$ of $\Gamma_0(\p)$, such that for each extension $\tilde{\rho}$ to $\NGp$, the compactly induced representation
\[ \ind_{\NGp}^{\GL_2(\F)} \rho \]
is an irreducible, supercuspidal representation. 
\end{defn}
We want to apply Clifford theory once again.
\begin{defn}
A stratum of an irreducible representation $\rho$ of $\Gamma_0(\p)$ is a one-dimensional representation
\[ \psi : U^{\ell(\rho)} / U^{\ell(\rho)+1} \rightarrow \bC^1\]
such that
\[ \psi \subset \Res_{U^{\ell(\rho)}} \rho.\]
A stratum of a ramified cuspidal type is called a ramified simple stratum.
\end{defn}
 According to \cite{BushnellHenniart:GL2}*{Proposition 13.1(1), page 96}, a ramified cuspidal type has a half integer level $\ell$, and there exists a unique element $\beta \in \Gamma_0(\p)$ for each stratum $\psi$ of $\rho$  with
\[ \psi  = \psi_{\beta,  2 \ell} , \qquad \psi_{\beta,  2 \ell}   (x) = \tr \left( \w_p^{-2 \ell} \beta (x-1) \right).\]
For any two strata of $\rho$, the associated elements $\beta$ are conjugate in $\Gamma_0(\p)$. The element $\w_p^{2 \ell} \beta$ is elliptic and has odd valuation determinant.  
The stabilizer of $\psi_\beta$ in $\NGp$ is given by \cite{BushnellHenniart:GL2}*{Section 15.3, page 106}
 \[   \F[ \beta ]^\times U^{(2 \ell +1)/2} \]
and consequently in $\Gamma_0(\p)$ by
 \[   \o[ \beta ]^\times U^{\ell + 1/2}. \]
\begin{theorem}
Let $\rho$ be a ramified cuspidal type with stratum $\psi_\beta$.
\begin{itemize}
 \item The level of $\rho$ is a half integer. 
 \item There exist two non-isomorphic extensions to $\NGp$, such as those given in point $(1)$ of Theorem~\ref{thm:redKeU0}.
 \item For each extension of $\tilde{\psi}_\beta$ of $\psi_\beta$ to $U^{\ell(\rho)+1/2}$, there is a unique one-dimensional representation $\theta$ of $\o[\beta]^\times$ with
         \[ \theta|_{\o[\beta]^\times \cap U^{\ell(\rho)+1/2}}    =   \tilde{\psi}_\beta    |_{\o[\beta]^\times \cap U^{\ell(\rho)-1/2}} \]
         and
        \[ \rho = \ind_{\o[\beta]^\times \cap U^{\ell(\rho)-1/2}}^{\Gamma_0(\p)}    \psi_\beta    \cdot  \theta.\]
\end{itemize} 
\end{theorem}
\begin{proof}
  The assumption about the level is given as Proposition 13.1(2) on \cite{BushnellHenniart:GL2}*{page 96}. The second assertion can be found as Proposition 15.7(1) on \cite{BushnellHenniart:GL2}*{page 109}. The extension to $U^{\ell(\rho)+1/2}$ is possible, because the quotient
\[           U^{\ell(\rho)+1/2} / U^{2\ell(\rho) +1}\]
is abelian. The decomposition of the induced representation
\[        \Ind^{\o[\beta]^\times \cap U^{\ell(\rho)+1/2}}_{U^{\ell(\rho)+1/2}}    \tilde{\psi}_\beta \]
 follows as claimed, since the quotient $\o[\beta]^\times \cap U^{\ell(\rho)+1/2} / U^{2\ell(\rho) +1}$ is abelian as well. The last assertion is thus obtained by Clifford theory. 
\end{proof}

\section{The representation theory of $\GL(2,\F)$}\label{section:padicclass}
The discussion of Section~\ref{section:realclass} applies as well. However, we observe an additional feature here. The group $\GL_2(\F)$, asopposed to the groups $\GL_2(\bR)$ and $\GL_2(\bC)$, also admits supercuspidal representations, i.e., representations with a compactly supported matrix coefficient.
Only the non-supercuspidal irreducible, unitary representations can be realized as subrepresentations or subquotients of parabolic inductions. 

Consider a one-dimensional representation $\mu : \B(\F) \rightarrow \bC^\times$. It determines uniquely two one-dimensional representations $\mu_j : \F^\times \rightarrow \bC^\times$, such that
\[ \mu \left( \sma a & * \\ 0 & b \smz \right) = \mu_1(a) \mu_2(b).\]
\begin{defn}
The representation $ \mJ(\mu_1, \mu_2,s)$ is the right regular representation of $\GL_2(\F)$ on the space of smooth functions $f: \GL_2(\F) \rightarrow \bC$ which satisfy
\[ f\left( \sma a & * \\ 0 & b \smz g\right) = \mu\left( \sma a & 0 \\ 0 & b \smz\right)  \left| \frac{a}{b} \right|_v^{s+1/2} f(g).  \]
\end{defn}
Every one-dimensional representation $\chi : \F^\times \rightarrow \bC^\times$ can be uniquely decomposed as
\[ \chi( u \w^n) = \chi_{alg} (u) q^{s_\chi n}  , \qquad u \in \o^\times, n \in \bZ \]
for some unique $s_\chi \in \bC\bmod 2 \uppi i /\log(q)$, with $\chi_{alg}$ being a one-dimensional representation of $\o^\times$.  We say that $\chi$ is algebraic if $s_\chi= 0$.
Similarly, we say that
\[ \mu \left( \sma a & * \\ 0 & b \smz \right) = \mu_1(a) \mu_2(b)\]
is algebraic if $\mu_1$ and $\mu_2$ are algebraic. It is sufficient to consider only parabolic inductions with algebraic one-dimensional representation $\mu$ of $\B(\F)$, since we have an isomorphism
\[ \mJ(\mu_1, \mu_2, s) = \left| \det(\blank) \right|^{s_{\mu_1}/2 + s_{\mu_2}/2} \otimes \mJ( \mu_{1, alg}, \mu_{2, alg} , s_{\mu_1} -s_{\mu_2} ).\]
Generally, we have that
\[ \chi \circ \textup{det} \otimes \mJ(\mu_1, , s) \cong \mJ(\mu \cdot  \chi\circ\textup{det}|_{\B(\F)}, s) = \mJ(\mu_1 \chi, \mu_2 \chi, s).\]
The central character of $\mJ(\mu_1, \mu_2,s)$ is given by $\mu_1 \mu_2$. \textbf{We assume from this point on that all one-dimensional representations denoted by $\mu, \mu_1, \dots$ are algebraic.}

The parabolic inductions are decomposed by the following theorem:
\begin{theorem}[\cite{BushnellHenniart:GL2}*{Section 9.5}]  \mbox{}
\begin{enumerate}[font=\normalfont]
 \item The representation $\mJ(\mu_1, \mu_2, s)$ is irreducible if and only if either 
\begin{itemize}
 \item $\mu_1 \neq \mu_2$, or 
 \item $\mu_1 = \mu_2$ and $s \neq \pm 1/2$.
\end{itemize}
 \item The representation $\mJ(\mu, \mu, -1/2)$ contains a one-dimensional irreducible subrepresentation $\mu \circ \textup{det}$.
 \item The representation $\mJ(\mu, \mu, 1/2)$ contains an infinite-dimensional irreducible subrepresentation $\textup{St}(\mu)$ of co-dimension $1$.  
\end{enumerate}
\end{theorem}

Now let us state the classification of unitarizable irreducible representations of $\GL_2(\F)$;
\begin{theorem}
Every smooth, unitarizable representation of $\GL_2(\F)$ for a non-archimedean field $\F$ with algebraic central character $\chi$ is isomorphic to 
\begin{enumerate}[font=\normalfont]
     \item a one-dimensional representation $\mu \circ \textup{det}$ if $\mu^2 = \chi$, 
     \item an irreducible, supercuspidal representation with central character $\chi$,
     \item a continuous series representation,
       \begin{enumerate}[font=\normalfont]
        \item  a principal series representation $\mJ(\mu_1, \mu_2, s)$ with  $\Re s = 0$ and $\mu_1,\mu_2$ algebraic with $\mu_1 \mu_2 = \chi$, 
        \item a complementary series representation $\mJ(\mu, \mu, s)$ for $0 < \Re s < 1/2$ with $\mu^2 = \chi$,
       \end{enumerate}
    \item a Steinberg (or special) representation $\textup{St}(\mu)$ if $\mu^2 =\chi$.
\end{enumerate}
We have an isomorphism $\mJ(\mu_1, \mu_2, s) \cong \mJ(\mu_2, \mu_1, -s)$. All the other listed representations are non-equivalent.
\end{theorem}
The unitarizability of the principal series representations is addressed in \cite{Bump:Auto}*{Theorem 4.6.7, page 511}, and that of the supercuspidal representations in \cite{Bump:Auto}*{Theorem 4.8.1, page 523}. The Steinberg representation is square-integrable \cite{Laumon1}, i.e., automatically unitarizable.

\subsection{The supercuspidal representations of $\GL(2, \F)$}
\begin{theorem}[\cite{Bushnell:Induced}*{Theorem 1 suppl., page 111}]\label{thm:rhocuspidal}
Let $\F$ be a non-archimedean field. Let $K$ be a maximal compact-mod-center subgroup of $\GL_2(\F)$. Let $\rho$ be an irreducible representation of $K$. 
The following assertions are equivalent:
\begin{enumerate}[font=\normalfont]
 \item the representation $\rho$ is a supercuspidal representation of $K$, i.e., $\Res_{\N(\F) \cap K} \rho$ does not contain the trivial representation for any unipotent subgroup $\N$ of $\GL(2)$;
 \item the representation $\Ind_K^{\GL_2(\F)} \rho$ decomposes into a finite sum of irreducible supercuspidal representations.
\end{enumerate}
\end{theorem}

 \begin{lemma}
Let $\F$ be a non-archimedean field. Every maximal compact-mod-center subgroup of $\GL_2(\F)$ is conjugate to the normalizer of either the standard compact open subgroup $\GL_2(\o)$ 
\[ K_1 := \Z(\F) \GL_2(\o) \]
or the Iwahori subgroup $\Gamma_0(\p)$
\[ K_2  = \Gamma_0(\p) \rtimes \langle \sma 0 & -1 \\ \w & 0 \smz\rangle.\]
\end{lemma}

\begin{theorem}[\cite{Kutzko:SuperGL2}*{Theorem, page 43}]
Let $\F$ be a non-archimedean field. Let $\pi$ be a supercuspidal representation of $\GL_2(\F)$. There exists a maximal compact-mod-center subgroup $K$ of $\GL_2(\F)$ and an irreducible representation $\rho$ of $K$, such that
\[ \pi \cong \Ind_{K}^{\GL_2(\F)} \rho.\]
\end{theorem}

\begin{corollary}
A supercuspidal representation of $\GL_2(\F)$ is infinite-dimensional. It is unitarizable if and only if the central character is unitary.   
\end{corollary}
\begin{proof}
The representation is infinite-dimensional, since the coset space $\GL_2(\F) / K$ is discrete and infinite. It is unitarizable as the induction of a unitarizable representation. 
\end{proof}

\begin{proposition}\label{prop:scKtype}
Let $K$ be a maximal compact-mod-center subgroup of $\GL_2(\F)$, such that \( \pi \cong \Ind_{K}^{\GL_2(\F)} \rho\) is supercuspidal and irreducible. 
Every irreducible, smooth, admissible representation $\pi_0$ of $\GL_2(\F)$ which contains $\rho$, i.e.,
\[ \rho \subset \Res_K \pi_0,\]
contains it at most with multiplicity one, and is isomorphic to $\pi$.
\end{proposition}
\begin{proof}
The Frobenius reciprocity theorem \cite{BushnellHenniart:GL2}*{page 18} yields that
\[ 0 \neq  \Hom_{K} ( \rho, \Res_K \pi_0 ) = \Hom_{\GL_2(\F)} ( \pi,  \pi_0 ).\]
Schur's Lemma \cite{BushnellHenniart:GL2}*{page 21} demonstrates that the dimension is at most one, since $\pi_0$ and $\pi$ are irreducible.
\end{proof}

Let $\pi$ be an irreducible, supercuspidal representation of $\GL_2(\F)$. We define the level
\[ \ell(\pi) \coloneqq \min \left\{ \min \{ N/2:  1 \subset \Res_{U^{N-1}} \pi  \} , \min \{ N:  1 \subset \Res_{\Gamma(\p^{N-1})} \pi  \} \right\}.\]
We also say that $\pi$ is minimal if
\[ \ell(\pi) \leq \ell(\pi \otimes \chi \circ \det) \]
for all one-dimensional representations $\chi$ of $\F^\times$. 
\begin{theorem}[\cite{BushnellHenniart:GL2}*{Chapter 15}]
 An irreducible, minimal supercuspidal representation contains only one isomorphism class of cuspidal types. It contains a ramified cuspidal type if and only if $\ell(\pi)$ is a half-integer.
An irreducible, minimal supercuspidal representation with central character $\chi$ is unitarizable  if and only if $\chi$ is unitarizable, and contains only cuspidal types whose central characters coincide with $\chi$ on $\Z(\o)$.
The three types of minimal supercuspidal representations are: 
 \begin{itemize}
  \item an irreducible unramified supercuspidal representation \[ \ind_{\Z(\o) \GL_2(\o)}^{\GL_2(\F)} \rho_\pi \] 
	for a unique unramified cuspidal type $\rho_\pi$ with $\Res_{\Z(\o)} \rho_\pi = \chi$
     \item an irreducible ramified supercuspidal representation \[ \ind_{\NGp}^{\GL_2(\F)}  \rho_\pi^+ \] for a unique ramified cuspidal type $\rho_\pi^+$ with $\Res_{\Z(\o)} \rho_\pi^+ = \chi$
     \item an irreducible ramified supercuspidal representation  \[ \ind_{\NGp}^{\GL_2(\F)}  \rho^-_\pi , \qquad \rho_\pi^- \coloneqq \rho \cdot ( -1)^{v(\det( \cdotp))} \]
for a unique ramified cuspidal type $\rho_\pi$ with $\Res_{\Z(\o)} \rho_\pi = \chi$ 
\end{itemize}
\end{theorem} 

\subsection{The parabolic inductions of $\GL(2,\F)$}
As always, we aim for a simple notation. Twisting by a one-dimensional character yields that
\[ \mu_2^{-1}\circ \det \otimes  \mJ(\mu_1, \mu_2, s) \cong  \mJ(\mu_1 \mu_2^{-1}, 1, s) =: \mJ(\mu_1 \mu_2^{-1}, s).\]
We must classify the $\GL_2(\o)$-types of the Jacquet-modules. Recall Definition~\ref{defn:rhomup}.
\begin{proposition}[The $K$-types of parabolic inductions]\label{prop:KtypeQp}
Assume that $\mu$ has conductor $\p^N$. We have the following decomposition
\begin{align*} \Res_{\GL_2(\o)}  \mJ(\mu, s) & =  \mJ_\o(\mu) =\bigoplus_{R=N}^\infty \rho( \mu, \p^R),\\
                     \Res_{\GL_2(\o)}  \St(1) &= \bigoplus_{R=1}^\infty \rho( 1, \p^R).
\end{align*}
Conversely, if for a smooth, admissible, infinite-dimensional, irreducible representation $\pi$ 
\[ \rho(\mu, \p^N) \subset \pi,\]
then $\pi \cong \mJ(\mu,s)$ is a parabolic induction for $s \neq \pm 1/2$.
Additionally, if for a smooth, admissible, infinite-dimensional, irreducible representation $\pi$ 
\[ \rho(1, \o) \not\subset \pi, \qquad \rho(1,  \p) \subset \pi, \]
then $\pi \cong\St(1)$ is a special representation.
\end{proposition}
\begin{proof}
By the Iwasawa decomposition, we have that
\[ \Res_{\GL_2(\o)} \mJ(\mu,s) = \mJ_\o(\mu).\]
The decomposition of $\mJ_\o(\mu)$ follows from Lemma~\ref{lemma:rhopara} and Theorem~\ref{thm:globpara}.
Likewise, we know that $\St(1)$ is the subquotient of $\mJ(1, 1/2)$ by the trivial one-dimensional representation \cite{BushnellHenniart:GL2}*{Equation 9.10.4, page 69}, so
\[ \Res_{\GL_2(\o)} \St(1)  =   \Res_{\GL_2(\o)} \mJ(1,1/2) \ominus 1.\]
The converse statement then follows from Theorem~\ref{thm:globpara2}, the classification theorem, and the fact that neither of the representations $\rho(\mu, \cond(\mu))$ or $\rho(1, \p)$ can occur in the restriction of any supercuspidal representation.

To elaborate: we note that the stratum of a supercuspidal representation is split for $\rho(\mu, \cond(\mu))$ with $\mu \neq 1$, or scalar for $\rho(1, \o)= 1$, or non-fundamental for $\rho(1, \p)$.  No split stratum can be included in an irreducible supercuspidal representation \cite{BushnellHenniart:GL2}*{Cor., page 98}.
Also, no supercuspidal representation has a $\GL_2(\o)$-invariant vector, i.e., it cannot contain the trivial representation. The case $\rho(1,\p)$ must be checked by hand. First, if a supercuspidal representation contains $\rho(1, \p)$, it will contain a $\Gamma_0(\p)$-invariant vector.
Certainly $\rho(1, \p)$ cannot be a $K$-type in an unramified supercuspidal representation, since it has a $\Gamma(\p)$-invariant representation. A ramified supercuspidal representation with a $\Gamma(\p)$-invariant vector is isomorphic to the compact induction of an inflation to $\Z(\F)\GL_2(\o)$ from a supercuspidal representation of $\GL_2(\o/\p)$. 
Now assume that $\pi$ is a ramified supercuspidal and admits a $\Gamma_0(\p) = \fU^0$-invariant vector. This is not possible \cite{BushnellHenniart:GL2}*{Section 15.5, page 107}.
\end{proof}

\section{The Abel transform and construction of test function}
At this point, we will adopt a completely different approach than in the case of the Lie groups $\GL(2, \bR)$ and $\GL(2, \bC)$. Let $\overline{K}$ be a maximal compact-mod-center subgroup, and let $\rho$ be a unitary, finite-dimensional representation of $\overline{K}$. Define, as usual,
\begin{align*} \mH(\GL_2(\F), \rho)    &=  \{ \tr \Psi : \Psi \in \mmH( \GL_2(\F), \rho ) \}, \\
                      \mmH(\GL_2(\F), \rho)                               &   = \Big\{ \Psi : \GL_2(\F) \rightarrow \Endo_{\bC}(V_\rho) \textup{ smooth} : \\ 
                                                                                   & {} \qquad \qquad \qquad \Psi(k_1 x k_2) = \rho(k_1) \Psi(x) \rho(k_2) \textup{ for all } k_j \in \overline{K} \Big\}.
\end{align*}

The Abel transform 
\begin{align*} \mA_\rho : \mH(\GL_2(\F), \rho) &\rightarrow \mH(\M(\F), \rho|_{\M(\F)}), \\ &\mA_\rho  \phi(m) \rightarrow \Delta_{\B(\F)}(m)^{1/2} \int\limits_{\N(\F)} \phi(mn) \d n\end{align*}
is equally important in the harmonic analysis, but it fails to be an isomorphism if $\rho$ does not factor through the determinant. In particular, no Abel inversion formula is available. The space $\mH(\GL_2(\F), \rho)$ is easily understood because $\GL_2(\F) // \overline{K}$ is discrete. In fact, we have a Cartan decomposition
\[ \GL_2(\F) = \bigoplus_{ k\in \bZ} \GL_2(\o) \Z(\F) \sma \w^k & 0 \\ 0 & 1 \smz  \GL_2(\o).\]
In order to construct test functions which separate and distinguish all $\overline{K}$-equivalence classes, we have to analyze carefully the $K$-type decompositions and the character distributions of the irreducible unitary representations. This section will concern itself primarily with this analysis.

Recall that two irreducible, unitary representations of $\GL_2(\F)$ are $\GL_2(\o)$-equivalent if their restrictions to $\GL_2(\o)$ are isomorphic. The results from the previous section demonstrate that there are essentially three different $\GL_2(\o)$-equivalence classes of infinite-dimensional, irreducible representations:
\begin{enumerate}[font=\normalfont]
 \item the  $\GL_2(\o)$-equivalence classes of principal series representations associated to a fixed one-dimensional representation of $\M(\o)$, which contains a one parameter-family of non-isomorphic, irreducible, unitary representations;
 \item the $\GL_2(\o)$-equivalence class of a special or Steinberg representation, which has only one isomorphism class of irreducible representations;
 \item the $\GL_2(\o)$-equivalence class of an irreducible, unramified/ramified supercuspidal representation, which has one/ two isomorphism class/es of irreducible representations.   
\end{enumerate}
Since the $\GL_2(\o)$-equivalence class of a ramified supercuspidal representation has two isomorphism class/es of irreducible representations, we have to look upon them from another maximal compact-modulo-center subgroup, i.e., the normalizer of the Iwahori subgroup.

A direct consequence of this is the following observation:
\begin{theorem}[Distinguished set of test functions]\label{thm:disttest}
Set $G=\GL_2(\F)$.
\begin{itemize}
 \item Let $\mu : \o^\times \rightarrow \bC^1$ be a one-dimensional representation, and let $\pi$ be an irreducible, unitary infinite-dimensional representation, such that the character distribution
        \( \phi \mapsto \tr \pi(\phi) \) does not vanish on $\mH(G, \overline{\rho(\mu, \cond(\mu))})$. Then $\pi \cong \mJ(\mu,s)$.
 \item Consider elements $\phi_0 \in \mH(G,1)$ and $\phi_1 \in \mH(G, {\rho(1,\p^j)})$, such that 
\[ \mA_{1} \phi_0 =    \mA_{\{\rho(1,\p)\}} \phi_1.\]
In this case, if $\tr \pi ( \phi_0 - \phi_1) \neq 0$ for some irreducible, unitary, infinite-dimensional representation of $G$, the representation is a Steinberg representation
\[ \pi \cong \St(1).\]
Moreover, there exist such functions $\phi_1$ and $\phi_0$ with $\tr \St(1) ( \phi_0-\phi_1) \neq 0$.
 \item  Let $(\rho, \overline{K})$ be a cuspidal type. If $\tr \pi$ does not vanish on the algebra $\mH(G, \overline{\rho})$ for some irreducible, unitary, infinite-dimensional representation of $G$, we can conclude 
\[ \pi\cong \ind_{\overline{K}}^G \rho.\]
\end{itemize}
\end{theorem}
\begin{proof}
This follows immediately from the $K$-type classification in Propositions~\ref{prop:KtypeQp} and ~\ref{prop:scKtype}. The second assertion also requires the formula for the character distribution of the parabolic induction (Theorem~\ref{thm:jacquettrace}) for $\mH(\GL_2(\F), \overline{\rho})$
\[ \tr \mJ(\mu, \o) \phi = \tr \mu( \mA_{\overline{\rho}} \phi) .\qedhere\]
\end{proof}

\begin{defn}\label{defn:phirho}
Let $\overline{K}$ be a maximal compact-mod-center subgroup of $\GL_2(\F)$, and let $\rho$ be an irreducible representation of $\overline{K}$. We define a non-trivial element in $\mH(\GL_2(\F), \rho)$ 
\[ \phi_\rho : \GL_2(\F) \rightarrow \bC , \qquad x \mapsto \begin{cases} \tr \rho(x), & x \in \overline{K}, \\ 0, & x \notin \overline{K}. \end{cases}.\] 
\end{defn}
\begin{lemma}\label{lemma:unitG}
The element $\phi_\rho$ is the unit of the algebra $\mH(\GL_2(\F), \rho)$.
\end{lemma}
\begin{proof}
This is a direct consequence of the Schur orthogonality relations for characters.
\end{proof}
\begin{lemma}[The identity distribution of $\phi_\rho$]\label{lemma:iddistr}
\[ \phi_\rho(1) = \tr \rho(1) = \dim(\rho).\]
\end{lemma}
\begin{lemma}[The Abel transform of $\phi_\rho$]\label{lemma:Abel}
Let $\rho$ be an irreducible representation of $\Z(\F) \GL_2(\o)$. We have that
\[  \mA_\rho \phi_\rho(1) =  \dim_\bC \Hom_{\N(\o)}( 1, \rho).\]
Moreover, $\mA_\rho$ is supported on $\M(\o)\Z(\F)$.
\end{lemma}
\begin{proof}
Since $\phi_\rho$ is supported on $\GL_2(\o)\Z(\F)$, we have that
\[ \phi(mn)= 0 \]
for all $m \in \M(\F) - \Z(\F) \M(\o)$ and all $n \in \N(\F) - \N(\o)$. So $\mA_\rho$ is supported on $\M(\o)\Z(\F)$. Furthermore, we have that
\[ \mA_\rho \phi_\rho(1) = \int\limits_{\N(\o)} \tr \rho(n) \d n,\]
and the result follows by the Schur orthogonality relations.
\end{proof}

\begin{lemma}[The Abel transform --- support $\sma \w^l& 0 \\ 0 & 1 \smz$]\label{lemma:abelsupport}
Fix any integer $k$. Let $\phi \in \mH(\GL_2(\F), \rho)$ be a function with support on
\[    \GL_2(\o) \Z(\F) \sma \w^k& 0 \\ 0 & 1 \smz \GL_2(\o).\]
The Abel transform is supported on 
\[  \coprod\limits_{\substack{l = -k\\ k = l \bmod 2 }}^k   \sma \w^{l} &0 \\ 0 & 1 \smz\M(\o)\Z(\F).\] 
\end{lemma}
This is directly seen in the following matrix decomposition:
\begin{lemma}\label{lemma:abeldec}
For any non-negative integer, we have that
\[  \GL_2(\o) \sma \w^k& 0 \\ 0 & 1 \smz \Z(\F) \GL_2(\o)  =   \coprod\limits_{\substack{l = -k\\ k = l \bmod 2 }}^k  \Z(\F)  \sma \w^{l} &0 \\ 0 & 1 \smz \N(\p^{l-k})  \GL_2(\o).\]
\end{lemma}
\begin{proof}
The disjointness of the sets on the right-hand side is immediate. A case-by-case analysis is required in order to confirm that every element on the left-hand side can be brought in the form suggested by the coset decomposition on the right. The Bruhat decomposition over the residue field asserts:
\[ \GL_2(\o) = \Gamma_0(\p) \amalg \Gamma_0(\p) w_0 \Gamma_0(\p).\]
The Iwahori decomposition warrants
\[   \Gamma_0(\p) =   \B(\o) \cdot w_0\N(\p)w_0, \qquad  \Gamma_0(\p) w_0 \Gamma_0(\p) = \N(\o) w_0 \B(\o).\]
The Levi decomposition indicates 
\[ \B(\o) =  \N(\o)  \M(\o).\]
The group $\M(\o)$ commutes with $\M(\F)$ and can be absorbed in the right $\GL_2(\o)$, i.e.,
\[           \M(\o) \sma \w^k& 0 \\ 0 & 1 \smz \Z(\F)\GL_2(\o) =  \sma \w^k& 0 \\ 0 & 1 \smz \Z(\F)\GL_2(\o).\]
We have that
\[  \B(\o) \cdot w_0\N(\p)w_0    \sma \w^k& 0 \\ 0 & 1 \smz \GL_2(\o)  = \N(\o)    \sma \w^k& 0 \\ 0 & 1 \smz \GL_2(\o) =  \sma \w^k& 0 \\ 0 & 1 \smz \N(\p^{-k})\GL_2(\o),\]
as well as
\[  \N(\o) w_0 \B(\o)    \sma \w^k& 0 \\ 0 & 1 \smz \GL_2(\o) =  w_0 \N(\o)^{w_0}  \sma \w^k& 0 \\ 0 & 1 \smz \N(\p^{-k})\GL_2(\o) .\]
The later set of cosets is not yet in the desired form, so for $x,y \in \o$ with $y =u\w^l$ for $v(y) =l \in \{0 ,   k\}$, we decompose
\[            w_0 \sma 1 & 0 \\ x & 1 \smz  \sma \w^k & y\\0 & 1 \smz   = \sma -x \w^k & -1 \\ \w^k  & y \smz  = \sma x\w^k  -\w^k/y & 1 \\ 0 & y \smz \underbrace{\sma 1 & 0\\ \w^k/y & 1\smz}_{\in \GL_2(\o)}.\]
Finally, we transform this matrix into the suggested form
\begin{align*} \sma x\w^k  -\w^k/y & 1 \\ 0 & y \smz&      = \sma x\w^k / y - \w^k / y^2  & y^{-1} \\ 0 & 1\smz  y \\ 
&=   \underbrace{  \sma x\w^k / y - \w^k / y^2  & 0 \\ 0 & 1\smz}_{\in \sma \w^{k-2l} & 0 \\ 0 &1\smz \M(\o)} \underbrace{\sma 1 & (x \w^k - \w^k /y)^{-1} \\ 0 & 1\smz}_{\in \N(\p^{l-k})} y.  \qedhere
\end{align*}
\end{proof}

\begin{remark}
For most purposes, we are well-served with test functions supported on $\Z(\F) \GL_2(\o)$. However, if we want to distinguish between non-isomorphic, $\overline{K}$-equivalent principal series representations, we must to choose test functions which are not supported on $\Z(\F) \GL_2(\o)$. From a computational perspective, Lemma~\ref{lemma:abelsupport} and Lemma~\ref{lemma:abeldec} suggest that we should try to work with test functions supported on
\[    \GL_2(\o) \Z(\F) \sma \w & 0 \\ 0 & 1 \smz \GL_2(\o).\]
Each function in $\mH(\GL_2(\F), \rho(\mu))$ with such support can separate the principal series representations. Simultaneously, the Abel transform of such functions has small support and also the integrals along the unipotent subgroup $\N(\F)$ with weights are more manageable.
\end{remark}

\subsection{Pseudo coefficients for the square-integrable representations}
The existence of pseudo coefficients for square integrable representations is a general result \cite{Kazhdan:Pseudo}. It is particularly convenient, that we can find pseudo-coefficients for the square-integrable representations of $\GL_2(\F)$ which are supported on maximal compact-mod-center subgroups of $\GL_2(\F)$. We give a constructive proof.
\begin{lemma}
Let $(\rho, \overline{K})$ be a cuspidal type, then all elements of $\mH(\GL_2(\F), \rho)$ are matrix coefficients of $\rho$. 
\end{lemma}
\begin{remark}[Suggested Interpretation of theorem]
The above lemma states that for supercuspidal representations, there is essentially no other choice for pseudo coefficients than matrix coefficients of $\rho$. To spell this out carefully would require much care and more analysis.
\end{remark}
\begin{proof}[Proof of the lemma]
 Set $G=\GL_2(\F)$. The Schur isomorphism yields an isomorphism
\[ \mH(G, \rho) \cong \Hom_{G}( \ind_{\overline{K}}^{G} \rho) \otimes \Endo(V_\rho). \]
Because $\ind_{\overline{K}}^{G} \rho$ is irreducible and supercuspidal as the contragredient of an irreducible and supercuspidal representation $\ind_{\overline{K}}^{G} \rho$, we have 
\[       \Hom_{G}( \ind_{\overline{K}}^{G} \overline{\rho}) \cong \bC\]
by Schur's Lemma. The lemma follows by comparing the dimensions.
\end{proof}

\begin{defnthm}[Pseudo coefficient of the supercuspidal representation]\label{defn:testsc}
Let $\pi$ be an irreducible, supercuspidal, unitary representation with cuspidal type $(\overline{K}, \rho)$. The function
\[ \phi_\pi : = \begin{cases}   \tr \rho(x), & \overline{K} = \Z(\F) \GL_2(\o) \textup{ and } x \in \overline{K}, \\ 
                  \frac{1+q}{2} \tr \rho(x), & \overline{K}=N\Gamma_0(\p)\textup{ and } x \in \overline{K}, \\                
                    0, & \textup{otherwise}
                \end{cases}  \]
is a pseudo coefficient of $\pi$.
\end{defnthm}
\begin{proof}
The vanishing property follows from Corollary~\ref{cor:superpseudo}. We obtain by Theorem~\ref{thm:frobdc}
\[      \tr \pi ( \tr \rho) =\frac{1}{\vol_{\Z(\F) \backslash \GL_2(\F)} ( \overline{K} )} .\]                                                                                                                                          
The Haar measure of $\GL_2(\F)$ was normalized in such a way that $\GL_2(\o)$ has unit measure. The Iwahori subgroup $\Gamma_0(\p)$ has index $q+1$ in $\GL_2(\o)$, and measure $(q+1)^{-1}$. The group $\Z(\F)\Gamma_0(\p)$ has index two in $\N\Gamma_0(\p)$. Let $\textup{d}_0 g / \textup{d}_1 g$ be the Haar measure of $\PGL_2(\F)$, such that  $\Z(\F) \GL_2(\o) / \N\Gamma_0(\p)$ admits a unit Haar measure. We compare the volumes of the set $\Z(\F)\Gamma_0(\p)$ and obtain the relation
\[ (q+1)  \d_0 g  = 2 \d_1 g.\] 
We have 
\[ \vol_{\Z(\F) \backslash \GL_2(\F)} ( \Z(\F)\GL_2(\o) ) = 1\] and  
\[ \vol_{\Z(\F) \backslash \GL_2(\F)} ( \N\Gamma_0(\p) ) =  2 /(q+1).\qedhere\]
\end{proof}

\begin{defnthm}[Pseudo coefficient of the Steinberg representation]\label{defn:testst}
 Define the smooth, compactly supported function $\phi_1$ as the characteristic function
 of $\GL_2(\o)\Z(\F)$, and $\phi_{\rho(1, \p)}$ as in Definition~\ref{defn:phirho}.
The character distribution of each irreducible, unitary, infinite-dimensional
 representation --- except for the Steinberg representation $\St_1$ --- vanishes on
\[ \phi_{\St} :=   \phi_{\rho(1, \p)} -  \phi_1.\]
We have that
\[  \tr \St_1 \phi_{\St}= 1. \]
\end{defnthm}
\begin{proof}
The above observations are an immediate consequence of Theorem~\ref{thm:disttest}, Lemma~\ref{lemma:unitG} and a computation 
\[ \dim_\bC \Hom_{\N(\o)}( 1, \rho(1, \p)) = 1.\]
Let us demonstrate the computation. Note that
\[ \Ind_{\Gamma_0(\p)}^{\GL_2(\o)} 1 = \rho(1,\p) \oplus 1.\] 
By general properties of class functions $\GL_2(\o) \rightarrow \bC$, we have that
\[ \tr \Ind_{\Gamma_0(\p)}^{\GL_2(\o)} 1 = \tr \rho(1,\p) \oplus \tr 1.\] 
 The Mackey induction restriction formula yields that
\[ \Res_{\N(\o)} \Ind_{\Gamma_0(\p)}^{\GL_2(\o)} 1 = \bigoplus\limits_{\gamma \in \N(\o) \backslash \GL_2(\o) /\Gamma_0(\p)}   \Res_{\N(\o) \cap \Gamma_0(\p)^\gamma}^{\N(\o)} 1.\]
whereas the Bruhat decomposition over the residue field \cite{BushnellHenniart:GL2}*{page 44}  and the Iwahori decomposition \cite{BushnellHenniart:GL2}*{(7.3.1), page 52} produce
 \[ \N(\o) \backslash \GL_2(\o) /\Gamma_0(\p)  = \Gamma_0(\p)  \amalg \N(\o) w_0 \B(\o), \qquad w_0 :=\sma 0 & -1 \\ 1 & 0\smz.\]
The result follows because
\[ \Res_{\N(\o)} \Ind_{\Gamma_0(\p)}^{\GL_2(\o)} 1 \cong 1 + \Ind_{\N(\p)}^{\N(\o)} 1 .\]
Thus, the Abel transforms have been computed $\mA_\rho \phi_{\rho(1,\p)} = 1$, and $\mA_\rho \phi_1 = 1$.
\end{proof}

\begin{lemma}
The Abel transform of both $\phi_\pi$ for $\pi$ supercuspidal and $\phi_\St$ vanishes. 
\end{lemma}                       
\begin{proof}
The Abel transform of $\phi_{\St}$ vanishes by construction. The Abel transform of $\phi_\pi$ vanishes by Theorem~\ref{thm:rhocuspidal}.
\end{proof}

\subsection{Distinguished vectors for the principal series representations}
\begin{defnthm}[Distinguished vectors for the principal series representation]\label{defn:testprinc}
For any one-dimensional representation $\mu$ of $\o^\times$, define
\[ \phi_\mu = \frac{\phi_{\overline{\rho(\mu)}}}{\tr \mA_{\rho(\mu)} \phi_{\overline{\rho(\mu)}}(1) }.\]
We have that
\[ \tr \pi(\phi_\mu ) = 0 \]
for all irreducible, unitary representations $\pi$, with the exception of the parabolic induction $\mJ(\mu,s)$:
\[ \tr \mJ(\mu,s) \left( \phi_\mu \right) = 1.\] 
\end{defnthm}
\begin{proof}
This follows from the formula for the character distribution of the parabolic induction~\ref{thm:jacquettrace}, the support of the Abel transfrom~\ref{lemma:abelsupport} and the $K$-type classification~\ref{prop:KtypeQp}.
\end{proof}

For various computations, it is important to understand the normalization factors.
\begin{lemma}\label{lemma:abelphimu}
$\tr \mA_{\rho(\mu)} \phi_{\overline{\rho(\mu)}}(1) = \dim_\bC \Hom_{\N(\o)} ( \rho(\mu), 1)= 1+ q^{\lfloor N/2 \rfloor} .$
\end{lemma}
\begin{proof}
 We appeal to the Mackey Induction Restriction Formula
\[ \Res_{\N(\o)}   \Ind_{\Gamma_0(\p^N)}^{\GL_2(\o)} \mu = \bigoplus\limits_{\gamma \in \N(\o) \backslash \GL_2(\o)/\Gamma_0(\p^N)} \Ind_{\N(\o) \cap \Gamma_0(\p^N)^\gamma}^{\N(\o)} \mu^\gamma. \]
We give a set of coset representatives obtained from the Bruhat decomposition over the residue field \cite{BushnellHenniart:GL2}*{page 44}  and the Iwahori decomposition \cite{BushnellHenniart:GL2}*{(7.3.1), page 52}
\[     \N(\o) \backslash \GL_2(\o)/\Gamma_0(\p^N) \cong \left\{w_0,  \sma 1 & 0\\ x& 1 \smz : x \in \p \bmod \p^N \right\}. \]
The contribution of the element $\w_0$ is given bu
\[ \Ind_{\N(\o) \cap \Gamma_0(\p^N)^{w_0}}^{\N(\o)} \mu^{w_0} = \Ind_{\N(\p^N)}^{\N(\o)}1.\]
The contribution of an element $\gamma =   \sma 1 & 0\\ x& 1 \smz$  with $x \in \p^R$ can be read from the Iwahori decomposition as well as
\[  \sma 1 & 0\\ -x& 1 \smz \sma a & b \\ c \w^N & d \smz \sma 1 & 0 \\ x & 1 \smz = \sma a & b \\ -ax + c\w^N  & -bx +d \smz        \sma 1 & 0 \\ x & 1 \smz = \sma a  + bx& b \\ -ax + c\w^N - bx^2 + dx  & -bx +d \smz  \]
 and provides us with
\[         \Ind_{\N(\o) \cap \Gamma_0(\p^N)^\gamma}^{\N(\o)} \mu^\gamma    = \Ind_{\N(\p^R)}^{\N(\o)} \left( b \mapsto \mu( 1 +bx)\right).\]  
It contains the trivial representation of $\N(\p^R)$ if and only if $R=N$.
\end{proof}

\subsection{The unramified Hecke operator --- Separation of the unramified principal series representations}
\newcommand{\HTo}{\mathbb{T}_{\w}}
Let
\( \phi_\w \) be the characteristic function of the set
\[ \GL_2(\o) \Z(\F) \sma \w&0\\0& 1 \smz \GL_2(\o).\]
\begin{defnthm}[The (unramified) Hecke operator]\label{defn:uHecke}
Define the unramified Hecke operator
\[ \HTo(g) = q^{-3/2} \phi_{\w}(g) \qquad \in \Ccinf(\GL_2(\F)//\Z(\F)\GL_2(\o)).\]
For every, irreducible, unitary representation $\pi$ of $\GL_2(\F)$,
\[  \pi( \HTo) = \begin{cases} q^{s} + q^{-s}, & \pi \cong \mJ(\mu,s), \\ q^{1/2}+q^{-1/2}, & \pi \textup{ trivial}, \\ 0, & \textup{otherwise}.\end{cases}\]
\end{defnthm}
\begin{proof}
Certainly all character distributions of representations without $\GL_2(\o)$-invariant vector vanish on $\phi_{\mu,\w}$. The only irreducible unitary representations with a $\GL_2(\o)$-invariant vector are the trivial representation and the unramified principal series representations.
The formula for the character distribution of a parabolic induction is
\[ \tr \mJ( \mu, s, \phi) = \int\limits_{\F^\times} \mA \phi \left( \sma a & 0 \\ 0 & 1 \smz \right) |a|^s \d^\times a.\]
The character distribution is easily computed  for $\phi \in \Ccinf(\GL_2(\F) // \GL_2(\o) \Z(\F))$ as
\begin{align*} \tr \textup{triv}(\phi) \underset{def.}= \int\limits_{\Z(\F) \backslash \GL_2(\F)} \phi(g) \d g \\ 
\underset{\textup{Iwasawa-dec.}}{=} \int\limits_{\F} \mA_1 \phi\left( \sma a & 0 \\ 0 & 1 \smz \right) |a|^{1/2} \d^\times a. \end{align*}
The remaining conclusion results from the following lemma:
\begin{lemma}\label{lemma:phiw}
The \( \mA_1 \phi_\w : \M(\F) \rightarrow \bC \) is supported on 
\[ \M(\o) \Z(\F) \sma \w & 0 \\ 0 & 1 \smz \amalg  \M(\o) \Z(\F) \sma \w^{-1} & 0 \\ 0 & 1 \smz,\]
and satisfies
\[  \mA_1 \phi_\w \left( \sma \w & 0 \\ 0 & 1 \smz \right) =  \mA_1 \phi_\w \left( \sma \w^{-1} & 0 \\ 0 & 1 \smz \right) = q^{3/2}.\]
\end{lemma}
As a consequence of a matrix decomposition:
\begin{lemma}\label{lemma:heckecoset}
We have a coset decomposition
\[  \GL_2(\o) \sma \w & 0 \\ 0 & 1 \smz\Z(\F) \GL_2(\o) = \coprod\limits_{x \bmod \p}     \sma \w & x \\ 0 & 1 \smz \Z(\F)     \GL_2(\o)    \amalg   \sma 1 &  0 \\ 0 & \w  \smz \Z(\F)    \GL_2(\o).\]
\end{lemma}
\begin{proof}[Proof of the coset decomposition]
This coset decomposition follows from the Bruhat decomposition over the residue field \cite{BushnellHenniart:GL2}*{page 44}  and the Iwahori decomposition \cite{BushnellHenniart:GL2}*{(7.3.1), page 52}
\begin{align*}\GL_2(\o) \sma \w & 0 \\ 0 & 1 \smz \GL_2(\o) & \underset{\textup{Bruhat}}= \Gamma_0(\p)   \sma \w & 0 \\ 0 & 1 \smz     \GL_2(\o) \amalg  \Gamma_0(\p) w_0 \Gamma_0(\p)   \sma \w & 0 \\ 0 & 1 \smz     \GL_2(\o) \\
                                                                                       &  \underset{\textup{Iwahori} } =\left(  \coprod\limits_{x \bmod \p}     \sma \w & x \\ 0 & 1 \smz     \GL_2(\o)  \right)   \amalg   \bigcup\limits_{x \bmod \p}  \sma 1 & 0 \\ x & \w \smz     \GL_2(\o) \\
                                                                                       &   = \left( \coprod\limits_{x \bmod \p}     \sma \w & x \\ 0 & 1 \smz     \GL_2(\o)  \right)   \amalg   \sma 1 &  0 \\ 0 & \w  \smz     \GL_2(\o) \\ 
& {}\qquad\amalg \bigcup\limits_{x \in \o^\times \bmod \p}       \sma1&  0 \\0 & \w   \smz  \sma1 & 0 \\  x/\w  & 1 \smz   \GL_2(\o).
\end{align*}
A matrix computation
\[    \sma1 & 0 \\  x/\w  & 1 \smz  = \sma \w/x  & 1 \\ 0 & x/\w \smz \sma 0 & -1 \\ 1 & \w/x \smz \]
yields both 
\[ \sma1&  0 \\0 & \w   \smz  \sma1 & 0 \\  x/\w  & 1 \smz   \GL_2(\o) =     \sma \w  & x \\ 0 & 1 \smz \GL_2(\o),\]
and the coset decomposition. 
\end{proof}
\begin{proof}[Proof of Lemma~\ref{lemma:phiw}]
The coset decomposition proves the claim about the Abel transform's support.
Now let us compute the value of the Abel transform of $\phi_\w$ explicitly
\begin{align*} \mA_1 \phi_\w \left( \sma \w & 0 \\ 0 & 1 \smz \right)  &= q^{1/2} \int\limits_{\N(\F)} \phi\left( \sma \w & 0 \\ 0 & 1 \smz n \right) \d n \\
                                                                      & = q^{1/2} \int\limits_{\p^{-1}}  \d n = q^{3/2}.\qedhere\end{align*} 
                                                                      \end{proof}
This completes the proof of the main theorem as well.\end{proof}

\subsection{The ramified Hecke operator --- Separation of the ramified principal series representations}
\newcommand{\HT}{\mathbb{U}_{\mu, \w}}
Let $\mu$ be a one-dimensional representation of $\o^\times$ with conductor $\p^N$ for some $N\geq 1$. We say that the principal series $\mJ(\mu,s)$ is ramified. Similarly to the unramified situation in the preceding subsection, we want to construct a test function $\phi_0 \in \Ccinf(\GL_2(\F), \overline{\mu})$, such that $\phi_0$ has vanishing character distribution for all unitary, irreducible representations except for all the parabolic inductions of type $\mJ(\mu,s)$. Additionally, we will separate every two non-isomorphic principal series representations 
\[ \tr \mJ(\mu, s, \phi_0) \neq \tr \mJ(\mu,s', \phi_0), \qquad s \neq s' \]
via reconstructing the complex value $s$. This can all be done via the theory of ramified Hecke operators which is much more complicated than their unramified counterpart. 

As we will see, the choice $\HT = \HTo$ for $\mu=1$ replicates the discussion of the preceding subsection.

Consider $\overline{\mu}$ as a one-dimensional representation of \newcommand{\ZG}{\Gamma(\mu)}
\[ \ZG =\Gamma_0(\p^N)\Z(\F) \] 
by setting
\[ \overline{\mu} \left( \sma a &* \\ * &* \smz \right) = \overline{\mu(a)}.\] 
Define a smooth, non-vanishing function $\phi_{\mu, \p}: \GL_2(\F) \rightarrow \bC$, which is supported on the set
\[ \ZG \sma \w & 0 \\ 0 & 1 \smz \ZG, \]
and satisfies
\[ \phi_{\mu, \w}(k_1 g zk_2)  = \overline{\mu}(k_1zk_2)    \phi_{\mu, \w}(g).\]
The function $\phi_{\mu, \w}$ distinguishes and separates the principal series representations of type $\mJ(\mu,1,s)$, but fails to be $\GL_2(\o)$ conjugation invariant.
Note that by the Iwahori decomposition \cite{BushnellHenniart:GL2}*{(7.3.1), page 52}
\[ \Gamma_0(\p^N) = \sma 1& 0 \\ \p^N & 1 \smz \cdot \sma * & 0 \\ 0 & * \smz \cdot \sma 1 & * \\ 0 & 1 \smz =   \sma 1 & * \\ 0 & 1 \smz           \cdot \sma * & 0 \\ 0 & * \smz \cdot      \sma 1& 0 \\ \p^N & 1 \smz ,\]
 it is straightforward to compute
\[      \ZG\sma \w & 0 \\ 0 & 1 \smz \ZG = \coprod\limits_{x \in \o \bmod \p} \sma \w & x \\ 0 & 1 \smz \ZG =\coprod\limits_{x \in \p^N \bmod \p^{N+1}}  \ZG \sma \w &  0\\ x & 1 \smz .\]
It follows for the integral
\begin{align}\label{eq:raster}  \nonumber
& \int\limits_{\F}  \phi_{\mu, \w} \left( \sma x& 0\\0 & 1 \smz \sma 1 & t \\ 0 & 1 \smz \right) \d^+ t\\  &=      \int\limits_{\F}  \phi_{\mu, \w} \left( \sma 1 & x t \\ 0 & 1 \smz \sma \w&0 \\0 & 1 \smz  \right) \d^+ t  \\
               & =  \begin{cases} q^{-1}  \phi_{\mu, \w} \left( \sma x& 0\\0 & 1 \smz \right), & v(x) = \pm 1, \\ 0 , & \textup{else}. \end{cases} 
\end{align}

\begin{defnthm}[The ramified Hecke operator]\label{defn:rHecke}
For every function $\phi_{\mu,\w}$ with the above properties,
\[ \phi_{\mu, \w}^K(x) = \int\limits_{\GL_2(\o)} \phi_{ \mu, \w} (k^{-1}xk) \d k \qquad \in \mH(\GL_2(\F), \rho(\overline{\mu})).\]
We define the unramified Hecke operator
\[   \HT(x) := \frac{\phi_{\mu, \p}^K(x)}{\mA_{\rho( \overline{\mu})} \phi_{\mu, \p}^K \left( \sma \w& 0 \\ 0 & 1 \smz\right)}.\]
It does not depend on the specific non-vanishing function $\phi_{\mu, \p}$ given. It is supported on the set
\[ \GL_2(\o)  \Z(\F) \sma \w & 0 \\ 0 & 1 \smz \GL_2(\o), \]
and satisfies for all unitary, irreducible representations $\pi$ of $\GL_2(\F)$
\[ \tr  \pi (\HT) = \begin{cases} q^{s}+q^{-s}, & \pi \cong \mJ(\mu,s), \\ 
                              0, & \pi \not\cong \mJ(\mu,s).
                    \end{cases}\]
\end{defnthm}
 For normalization and uniqueness purposes, we have normalized the function by the value of its Abel transform at $\sma \w& 0 \\ 0 & 1 \smz$.
 \begin{proof}
First, we have that
\[ \mA_{\rho( \overline{\mu})} \phi_{\mu, \p}^K \left( \sma \w& 0 \\ 0 & 1 \smz\right) \neq 0.\]
To this end, we must prove the following lemma:
\begin{lemma}[Frobenius Character Formula]\label{lemma:frobcharacetr}
Consider $\phi_{\mu,\w}$ as above, then
\begin{align*} \phi_{\mu,\w}^{K} \left( x \right) & = \frac{1}{[ \GL_2(\o) :\Gamma_0(\p^N)]} \sum\limits_{\substack{ \gamma \in  \GL_2(\o)/ \Gamma_0(\p^N)}} \phi_{\mu,\w}  \left( \gamma^{-1}x \gamma \right) \\
& = \frac{1}{q^N(1+q^{-1})}\sum\limits_{\alpha \in \p \bmod \p^N } \phi_{\mu,\w}  \left( \sma 1 & 0 \\  -\alpha & 1 \smz x \sma 1 & 0\\ \alpha & 1 \smz \right) \\
&{} \qquad +  \frac{1}{q^N(1+q^{-1})} \sum\limits_{\beta \in \o \bmod \p^N } \phi_{\mu,\w}  \left(  w_0^{-1}\sma 1 & -\beta \\   & 1 \smz  x \sma 1 & \beta \\  & 1 \smz w_0  \right).
\end{align*}
\end{lemma}
\begin{proof}
Note that the function $\phi_{\mu,\w}$ is $\Gamma_0(\p^N)$-conjugation invariant. If we endow $\GL_2(\o)$ and $\Gamma_0(\p^N)$ with unit Haar measures $\d k$ and $\d k_0$, then the quotient integral formula yields for $f \in\mL^1(\GL_2(\o))$:
\[      \int\limits_{\GL_2(\o)} f(k)\d k = \frac{1}{[ \GL_2(\o) : \Gamma_0(\p^N)]} \sum\limits_{\gamma \in \GL_2(\o)/ \Gamma_0(\p^N)  }   \int\limits_{\Gamma_0(\p^N)}  f( k_0 \gamma ) \d k_0.\]
This follows by the uniqueness of the quotient measure and by verifying the equation for constant functions only. We obtain the first equation
\[ \phi_{\mu,\w}^{K} \left( x \right)  = \frac{1}{[ \GL_2(\o) :\Gamma_0(\p^N)]} \sum\limits_{\substack{ \gamma \in  \GL_2(\o) / \Gamma_0(\p^N) }} \phi_{\mu,\w}  \left( \gamma^{-1}x \gamma \right).\]
The quotient 
\[ \GL_2(\o)  /  \Gamma_0(\p^N) =   \coprod\limits_{\alpha \in \p \bmod \p^N} \sma 1 & 0\\ \alpha & 1 \smz \Gamma_0(\p^N)  \amalg \coprod\limits_{\beta \in \o \bmod \p} \sma 1& \beta \\ 0 & 1 \smz w_0  \Gamma_0(\p^N)\]  
is computed, as always, from the Bruhat decomposition over the residue field \cite{BushnellHenniart:GL2}*{page 44}  and the Iwahori decomposition \cite{BushnellHenniart:GL2}*{(7.3.1), page 52}. The cardinality is exactly $q^N + q^{N-1}$ for $N \geq 1$.
\end{proof}
The immediate consequence is:
\begin{corollary}
 $ \phi_{\mu,\w}^{K} \left( \sma \w& 0 \\ 0 & 1\smz \right) =        \frac{1}{q^N(1+q^{-1})}               \phi_{\mu,\w} \left( \sma \w& 0 \\ 0 & 1\smz \right).$ 
\end{corollary}
\begin{proof}
 \begin{align*} \phi_{\mu,\w}^{K} \left( \sma \w& 0 \\ 0 & 1\smz \right) &  = \frac{1}{q^N(1+q^{-1})}\sum\limits_{\alpha \in \p \bmod \p^N } \phi_{\mu,\w}  \left( \sma\w & 0 \\ 0& 1 \smz \sma 1 & 0\\ \alpha- \alpha \w & 1 \smz \right) \\
&{} \qquad +  \frac{1}{q^N(1+q^{-1})} \sum\limits_{\beta \in \o \bmod \p^N } \phi_{\mu,\w}  \left(  \sma 1 & 0 \\   \beta -\w \beta & 1 \smz   \sma 1 & 0 \\ 0 & \w \smz  \right)   \\
& =    \frac{1}{q^N(1+q^{-1})}   \phi_{\mu,\w} \left( \sma \w& 0 \\ 0 & 1\smz \right) .   \qedhere\end{align*}
\end{proof}
Equation~\ref{eq:raster} implies the non-vanishing of the Abel transform and indeed provides the exact value.
\begin{corollary}
\(      \mA_{\rho( \overline{\mu})} \phi_{\mu, \p}^K \left( \sma x& 0 \\ 0 & 1 \smz\right) = \frac{q^{1/2}}{q^N(1+q^{-1})} \phi_{\mu,\w} \left( \sma x& 0 \\ 0 & 1\smz \right).\)
\end{corollary}
\begin{lemma}      $ \phi_{\mu,\w}^{K} \in \mH(\GL_2(\F),  \rho(\mu) ).$
\end{lemma}
This is a result of a more general statement.
\begin{lemma}[Nesting of Hecke algebras]
 Let $G$ be a locally compact group, $K$ a compact subgroup of $G$, and $K'$ a compact subgroup of $K$.  Let $\rho$ be a finite-dimensional unitary representation of $K'$.
 We have that 
\[ \phi \mapsto \phi^{K} \]
 maps
\[ \mH(G, K', \rho) \twoheadrightarrow \mH(G,K, \Ind_{K'}^K \rho).\]
\end{lemma}
\begin{proof}
For any finite-dimensional representation $\rho$ of $K'$, the function $\tr \rho$ is defined on $K'$, and can be extended to $G$ by zero off $K$. Similar conventions are assumed for $K$.
\begin{align*}
 \int\limits_K \int\limits_{K'} \int\limits_{K'}  &\phi(k_1kxk^{-1}k_2) \tr \rho(k_1^{-1}k_2^{-1}) \d k_1 \d k_2 \d k'  \\
                                                              & = \int\limits_K \int\limits_K  \int\limits_{K'} \int\limits_{K'}  \phi(k_1k_0^{-1}k^{-1}x k k_0k_2) \tr \rho(k_1^{-1}k_2^{-1}) \d k_1 \d k_2 \d k \d k_0  \\
                                                                                                                                                          &  = \int\limits_K  \int\limits_{K'} \int\limits_{K'}  \phi(k_1k^{-1}x kk_2)  \int\limits_K  \tr \rho(k_0^{-1} k_1^{-1}k_2^{-1} k_0)\d k_0 \d k_1 \d k_2 \d k.
\end{align*}
We appeal to the Frobenius character formula for compact groups
\[  \int\limits_K  \tr \rho(k_0 k_1^{-1}k_2^{-1} k_0)\d k_0      = \vol(K' \backslash K) \tr \Ind_{K'}^K  \rho(k_1^{-1}k_2),\]
and we eliminate one of the integrals
\begin{align*}      &     \int\limits_K  \int\limits_{K'} \int\limits_{K'}  \phi(k_1k^{-1}x kk_2)    \vol(K' \backslash K) \tr \Ind_{K'}^K  \rho(k_1^{-1}k_2)  \d k_1 \d k_2 \d k  \\ & \underset{k_2 \mapsto k_1^{-1}k_2}{ =}      \int\limits_K  \int\limits_{K'}   \phi(k x k^{-1}k_2)    \vol(K' \backslash K) \tr \Ind_{K'}^K  \rho(k_2^{-1})  \d k_2 \d k\\
&     \underset{k_2 \mapsto k_1^{-1}k_2}{ =}      \int\limits_K  \int\limits_{K'}   \phi(k x k_2)    \vol(K' \backslash K) \tr \Ind_{K'}^K  \rho(k_2^{-1}k)  \d k_2 \d k   . \qedhere
\end{align*}

\end{proof}
The identities for the character distributions remain to be shown. It is equivalent for an irreducible representation $\pi$ of $\GL_2(\F)$ to have a $\Gamma_0(\p)$-isotype $\overline{\mu}$ and to have a $\GL_2(\o)$-isotype $\overline{\rho(\mu)} = \rho(\overline{\mu})$. This can be proven by the associativity of the restriction function and the Frobenius reciprocity law:
\[ \Hom_{\GL_2(\o)} \left( \Res_{\GL_2(\o)} \pi , \Ind_{\Gamma_0(\p^N)}^{\GL_2(\o)} \overline{\mu}\right) \cong \Hom_{\Gamma_0(\p)} \left( \Res_{\Gamma_0(\p^N)}\Res_{\GL_2(\o)} \pi ,  \overline{\mu}\right).\]
By the invariance of character distributions, we undoubtedly achieve
\[ \tr \pi( \phi_{\mu, \w} ) = \tr \pi( \phi_{\mu, \w}^K ).\]
And because the only irreducible, unitary representations with non-zero $\GL_2(\o)$-isotype $\overline{\rho(\mu)}$ are the irreducible, unitary principal series representation $\mJ(\mu,s)$, the character distribution vanishes on all other $\GL_2(\o)$-equivalence classes. Since $\phi_{\mu, \w}^K$ is a $\GL_2(\o)$-conjugation-invariant function, and by the described vanishing properties, it must be an element of $\mH(\GL_2(\F), \rho(\overline{\mu}))$. Now according to the formula for the character distribution of a parabolic induction, we obtain
\[ \mJ(\mu, s, \HT) = \int\limits_{\sma * & 0 \\ 0 & 1 \smz} \frac{\mA_{\rho(\mu)} \phi^K_{\mu,\w}(m)}{\mA \phi_{\mu, \p}^K\left( \sma \w& 0 \\ 0 & 1 \smz\right)} \mu(m) \Delta(m)^{s+1/2} \d m.\]
We note that $\phi_{\mu,w}$ is supported on 
\[ \Gamma_0(\p^N) \Z(\F) \sma \w& 0 \\ 0 & 1 \smz \Gamma_0(\p^N),\]
and $\phi_{\mu, \w}^K$ is supported on
\[ \GL_2(\o) \Z(\F) \sma \w& 0 \\ 0 & 1 \smz \GL_2(\o).\]
The Abel transform 
\[ \mA_{\rho(\overline{\mu})} \phi_{\mu, w}(m) = \delta_B(m)^{-1/2} \int\limits_{\N(\F)} \phi_{\mu, w}(mn) \d n , \qquad m \in \M(\F), \]
is supported on
\[      \M(\o) \Z(\F) \sma \w& 0 \\ 0 & 1 \smz \M(\o) \cup \M(\o) \Z(\F) \sma \w^{-1} & 0 \\ 0 & 1 \smz \M(\o)\]
according to the coset decomposition
\[  \GL_2(\o) \sma \w & 0 \\ 0 & 1 \smz\Z(\F) \GL_2(\o) = \coprod\limits_{x \bmod \p}     \sma \w & x \\ 0 & 1 \smz \Z(\F)     \GL_2(\o)    \amalg   \sma 1 &  0 \\ 0 & \w  \smz \Z(\F)    \GL_2(\o). \]
We know from the general properties of the Abel transform from Section~\ref{section:Abel} that the Abel transform is invariant under conjugation of the Weyl group
\[   \mA_{\rho(\overline{\mu})} \phi_{\mu, w} \left( m\right) =\mA_{\rho(\overline{\mu})} \phi_{\mu, w}(w_0^{-1} mw_0) \qquad w_0 \coloneqq \sma 0 & -1 \\ 1 & 0 \smz. \]
According to the normalization, this gives us
\[   \mA_{\rho(\overline{\mu})} \HT\left(\sma \w^k & 0 \\ 0 & 1 \smz \right) =  \begin{cases} q^{-k/2}, & k =\pm 1, \\ 0 & \textup{else}.\end{cases}\]
The formula for the trace of the parabolic induction yields now
\[       \mJ(\mu, s, \HT)   =  q^s + q^{-s}.\qedhere\]
 \end{proof}

\section{The character of the infinite-dimensional representations}
The values of the spectral distribution have been determined in Definition-Theorem~\ref{defn:testsc},~\ref{defn:testst} and~\ref{defn:testprinc}.

\section{The character of the one-dimensional representations}
\begin{proposition}\label{prop:padiconed}
Let $\chi$ and $\mu$ be algebraic one-dimensional representations of $\F^\times$. Let $J_\chi$ be the character distribution of $\chi \circ\det\!$. Let $\pi$ be a supercuspidal representation of $\GL_2(\F)$
\begin{align*}
    J_\chi (\phi_\mu) &= \begin{cases}1,  & \chi =\mu=1, \\
                                                 0, & \textup{otherwise},
                      \end{cases} \\
  J_\chi(\HT) & = \begin{cases} q^{1/2} + q^{-1/2} ,  & \chi =\mu=1, \\
                                                 0, & \textup{otherwise},
                      \end{cases} \\
    J_\chi (\phi_\pi) &= 0 ,\\
 J_\chi(\phi_{\St})& = \begin{cases} -1 ,  & \chi =\mu=1, \\
                                                 0, & \textup{otherwise}.
                      \end{cases} \\
\end{align*}
\end{proposition}

\section{The identity distribution}
The computation of the identity distribution has already been partially introduced in Lemma~\ref{lemma:iddistr}. However, we must be more precise.
\begin{proposition}\label{prop:padicone}
Let $\mu$ be a one-dimensional representation of $\o^\times$ with conductor $\p^N$, and let $\pi$ be a supercuspidal representation with cuspidal type $( \overline{K}, \rho)$. 
We obtain
\begin{align*}
\phi_{1} (1) &= 1, \\
\phi_\mu (1) &= \frac{q^{N}+ q^{N-1}}{q^{\lfloor N/2\rfloor} + 1}, \\
\phi_{\St} (1) & = q-1,\\
\phi_{\pi} (1) & = \dim(\rho),  \qquad \pi \textup{ unramified}, \\
\phi_{\pi} (1) & = \frac{1+q}{2}\dim(\rho),  \qquad \pi \textup{ ramified, and} \\
          \HT(1) & = 0.
\end{align*}
\end{proposition}
 \begin{proof}
The Hecke operator is not supported on $\GL_2(\o)$ and satisfies $\HT(1) = 0$ by definition. From  Lemma~\ref{lemma:iddistr} and Lemma~\ref{lemma:abelphimu}, we immediately obtain 
\begin{align*}
   \phi_\mu (1) &= \frac{\dim \Ind_{\Gamma_0(\p^N)}^{\GL_2(\o)} \mu}{q^{\lfloor N/2 \rfloor} + 1}, \\
\phi_{\St} (1) & = \dim \rho(\mu, \p) - 1.\end{align*} 
For finite-dimensional representations, we have \( \tr \rho (1)  =  \dim \rho.\)
 \end{proof}

\section{The parabolic distributions}
The local $\GL(2,\F)$ zeta integral is given for $\phi \in \Ccinf(\GL_2(\F))$, and $s \in \bC$ as the value
\[ \zeta_{\GL_2(\F)}(s, \phi) \coloneqq \int\limits_{\F^\times} \int\limits_{\GL_2(\o)} \phi\left(k^{-1}\sma 1 & x\\ 0 & 1 \smz k\right) \left|x \right|_v^s \d k \d^\times x.\]
Special values of this function contribute to the Arthur trace formula. At $s=1$, the Abel transform is evaluated at the identity, i.e., for $\phi  \in \mH(\GL_2(\F), \rho)$:
\[ \frac{\zeta_{\GL_2(\F)}(1, \phi)}{\zeta_v(1)} = \mA_\rho \phi(1),\]
and the derivative of $\zeta_{\GL_2(\F)}$ is
\[    \frac{\partial}{\partial s} \Big|_{s=1} \frac{\zeta_{\GL_2(\F)}(s, \phi) }{ \zeta_v(s)},\]
where $\zeta_v(s)$ is the local zeta function of $\F$ \cite{Tate:Thesis}
\[ \zeta_v(s) = \left( 1-q^{-s} \right)^{-1}, \qquad \zeta_v'(s) = \frac{-\log(q) q^{-s}}{(1-q^{-s})^{2}}.\] 
The computations for our set of test functions are presented in the next proposition.
\begin{proposition}\label{prop:padicpara}
We have the following special values for
\begin{align*}
\zeta_{\GL_2(\F)}(1, \phi_{\St}) &=0, \\
\zeta_{\GL_2(\F)}(1, \phi_{\pi}) &=0, \qquad  \pi \textup{ supercuspidal}, \\
\frac{\zeta_{\GL_2(\F)}(1, \phi_{\mu})}{\zeta_v(1)} & =1,  \qquad \mu: \o^\times \rightarrow \bC^1, \\
\zeta_{\GL_2(\F)}(1, \HT) & =0, 
\end{align*}
and    
\begin{align*}
 \frac{\partial}{\partial s} \Big|_{s=1} \frac{\zeta_{\GL_2(\F)}(s, \phi_{\St}) }{ \zeta_v(s)} &=\log(q)(1-q^{-1})  , \\
 \frac{\partial}{\partial s} \Big|_{s=1} \frac{\zeta_{\GL_2(\F)}(s, \phi_{\pi})}{ \zeta_v(s)} & = \log(q)(1-q^{-1}) \sum\limits_{k=0}^\infty  q^{-k} \dim \Hom_{\N(\p^k)}( \rho_\pi, 1), \\
& {}\qquad  \pi \textup{ supercuspidal}, \\
 \frac{\partial}{\partial s} \Big|_{s=1} \frac{\zeta_{\GL_2(\F)}(s, \phi_{1})}{ \zeta_v(s)} & =0 , \\
 \frac{\partial}{\partial s} \Big|_{s=1} \frac{\zeta_{\GL_2(\F)}(s, \phi_{\mu}) }{ \zeta_v(s)}  & =  \log(q) \left( (N+1/2) +\frac{-N+3/2}{q}  + \frac{1}{2q^2} \right), \\  
                                             & {} \qquad  \mu: \o^\times \rightarrow \bC^1 \textup{ non-trivial}, \; \cond(\mu) = \p^N, \\
\frac{\partial}{\partial s} \Big|_{s=1} \frac{\zeta_{\GL_2(\F)}(s, \HT) }{ \zeta_v(s)} & =0.
\end{align*}
\end{proposition}
 \begin{proof}
 The first equalities are easily concluded from prior observations and the equation
\[ \zeta_{\GL_2(\F)}(1, \phi) = (1-q^{-1})^{-1} \mA_\rho \phi(1).\]
The Abel transform vanishes for $\phi_{\St}$ and $\phi_{\pi}$. The Abel transform of $\phi_\mu$ is by definition normalized to $\mA_{\rho(\mu)} \phi_{\mu}(1) = 1$. 
The Abel transform of $\HT$ is not supported on $\M(\o)\Z(\F)$. 

The third equality requires more work. The quotient rule of differential calculus tells us that
\begin{align*} \frac{\partial}{\partial s} \Big|_{s=1} \frac{\zeta_{\GL_2(\F)}(s, \phi) }{ \zeta_v(s)} &= \zeta_{\GL_2(F)}'(1, \phi) ( 1- q^{-1}) + \log(q) q^{-1} ( 1-q^{-1})  \zeta_{\GL_2(F)}(1, \phi).
\end{align*}
For $\phi = \phi_1$, the value is zero because $\zeta_{\GL_2(\F)}(s, \phi) = \zeta_v(s)$. 

The function $\HT$ is not supported on $\GL_2(\o)\Z(\F)\N(\F)\GL_2(\o)$, since $\w$ is not a square. This is evidenced because the decomposition in Lemma~\ref{lemma:heckecoset} provides that the support of $\HT$ is contained in the set
\[ \Z(\F) \sma \w & 0 \\ 0 & 1 \smz \N(\F) \GL_2(\o)  \cup  \Z(\F) \sma \w^{-1} & 0 \\ 0 & 1 \smz \N(\F) \GL_2(\o).\]
Looking at the valuation of the determinant yields the result. The valuation of the determinant is even on $\GL_2(\o)\Z(\F)\N(\F)\GL_2(\o)$ and odd on the support of $\HT$. 

Assume that $\phi$ is equal to the trace of the irreducible representation $\rho$ of $\Z(\F)\GL_2(\o)$ on $\GL_2(\o)\Z(\F)$ and to zero elsewhere. The derivative at $s=1$ is computed as a sum
\begin{align*} \frac{\partial}{\partial s}\Big|_{s=1} \zeta_{\GL_2(\F)}(s, \phi)        &= \int\limits_{\F^\times}  \tr \phi\left(\sma 1 & x\\ 0 & 1 \smz \right) \log\left|x \right|   \d^\times  x\\
& =\log(q)\sum\limits_{k=0}^\infty  kq^{-k} \int\limits_{\o^\times } \tr \phi\left(\sma 1 & x\\ 0 & 1 \smz \right) \d^\times x \\
& =\log(q) \sum\limits_{k=0}^\infty  q^{-k} \int\limits_{\p^k} \tr \phi\left(\sma 1 & x\\ 0 & 1 \smz \right) \d^\times x \\
& \underset{\textup{Schur Orth.}}=\log(q) \sum\limits_{k=0}^\infty  q^{-k} \dim \Hom_{\N(\p^k)}( \rho, 1).
\end{align*}
Similarly, if $\rho$ is an irreducible representation of the normalizer of the Iwahori subgroup $K$, and $\phi$ equals the trace of $\rho$ on $K$ and is zero off $K$, then essentially the same computation yields  
\begin{align*} \frac{\partial}{\partial s} \zeta_{\GL_2(\F)}(s, \phi)  \Big|_{s=1} & = \log(q) \sum\limits_{k=-1}^\infty  q^{-k} \dim \Hom_{\N(\p^k)}( \rho, 1).
\end{align*}
From Theorem~\ref{thm:rhocuspidal}, we deduce
\[   \dim \Hom_{\N(\p^{-1})}( \rho, 1)  =0.\]
None of these expressions vanish, because we are dealing with irreducible representation of pro-finite (modulo the center) groups. In the case of the ramified and the unramified supercuspidal representations, the proposition gives only the above formulas.
For the Steinberg representations and the principal series representations, the proposition provides more.
For the former, we write
\[ \phi_{\St} =  \phi_{\rho(1, \p)} - \phi_{1}, \]
where $\phi_1$ is the characteristic function of $\GL_2(\o) \Z(\o)$  and  $\phi_{\rho(1, \p)}$ the trace of $\rho(1, \p) = \Ind_{\Gamma_0(\p)}^{\GL_2(\o)} 1 \ominus$ the trivial representation on $\Z(\F) \GL_2(\o)$ and zero off  $\Z(\F) \GL_2(\o)$.   The above formulas directly yield
\begin{align*} \frac{\partial}{\partial s} \zeta_{\GL_2(\F)}(s, -\phi_{1})  \Big|_{s=1} =  -\log(q) \left( 1 + q^{-1} +q^{-2}  + \dots \right) = \frac{-\log(q)}{1-q^{-1}}. \end{align*}
The Mackey Restriction Induction formula yields
\begin{align*}   \dim & \Hom_{\N(\p^k)}( \Res_{\N(\p^k)} \Ind_{\Gamma_0(\p^N)}^{\GL_2(\o)} 1 , 1)    \\ 
&\underset{\textup{Mackey}}  = \sum\limits_{\gamma \in \N(\p^k) \backslash \GL_2(\o) / \Gamma_0(\p) }     \dim \Hom_{\N(\p^k)}( \Ind_{\Gamma_0(\p)^\gamma \cap \N(\p^k)}^{\N(\p^k)} 1 , 1)\\
&\underset{\textup{Frob.rec.}}  = \sum\limits_{\gamma \in \N(\p^k) \backslash \GL_2(\o) / \Gamma_0(\p) }     \dim \Hom_{\Gamma_0(\p)^\gamma \cap \N(\p^k)}(  1 , 1).
\end{align*}
The dimension is equal to the cardinality of the coset
\[  \N(\p^k) \backslash \GL_2(\o) / \Gamma_0(\p)  = \begin{cases} 2, & k = 0, \\ 1+q, &  k >0, \end{cases}\]
which gives us
\[     \dim \Hom_{\N(\p^k)}( \Res_{\N(\p^k)} \rho( 1, \p) , 1) =     \begin{cases} 1, & k = 0, \\ q, &  k >0, \end{cases} \]
and hence
\[ \frac{\partial}{\partial s} \zeta_{\GL_2(\F)}(s, \phi_{\rho(1, \p)})  \Big|_{s=1}  = \frac{\log(q)}{1-q^{-1}} +\frac{\log(q)}{1-q^{-1}}  \sum\limits_{k=1}^\infty q q^{-k} =    \log(q) + \frac{\log(q)}{(1+q^{-1})} .\]
Finally, the linearity of the zeta integral yields
\[ \frac{\partial}{\partial s} \zeta_{\GL_2(\F)}(s, \phi_{\St})  \Big|_{s=1}  =   \log(q)  .\]
For the principal series representation, we must compute for $\mu: \o^\times \rightarrow \bC$, with conductor $\p^N$ for $N >0$, the dimension  via the Mackey Induction Restriction formula
\begin{align*}  \dim & \Hom_{\N(\p^k)}( \Res_{\N(\p^k)} \Ind_{\Gamma_0(\p^N)}^{\GL_2(\o)} \mu , 1)   \\
&  \underset{\textup{Mackey}}  =   \sum\limits_{\gamma \in \N(\p^k) \backslash \GL_2(\o) / \Gamma_0(\p^N) }   \dim \Hom_{\N(\p^k)}( \Ind_{\Gamma_0(\p^N)^\gamma \cap \N(\p^k)}^{\N(\p^k)} \mu^\gamma , 1 )\\
&  \underset{\textup{Frob.rec.}}  =   \sum\limits_{\gamma \in \N(\p^k) \backslash \GL_2(\o) / \Gamma_0(\p^N) }   \dim \Hom_{\Gamma_0(\p^N)^\gamma \cap \N(\p^k)}( \mu^\gamma , 1 ),\\
&  \underset{\textup{conj.}}  =   \sum\limits_{\gamma \in \N(\p^k) \backslash \GL_2(\o) / \Gamma_0(\p^N) }   \dim \Hom_{\Gamma_0(\p^N) \cap \N(\p^k)^{\gamma^{-1}}}( \mu, 1 ),\\
&  \underset{\textup{Bruhat-Iwahori}}=    \sum\limits_{\alpha \in \p \bmod \p^N}   \dim \Hom_{\Gamma_0(\p^N) \cap \sma 1 & 0 \\ - \alpha & 1 \smz \N(\p^k) \sma 1 & 0 \\  \alpha & 1 \smz}( \mu, 1 )\\
& {}\qquad +  \sum\limits_{\beta \in \o \bmod \p^{\max\{N,k\}}}   \dim \Hom_{\Gamma_0(\p^N) \cap w_0^{-1}\sma 1 & -\beta \\  & 1 \smz \N(\p^k) \sma 1 & \beta \\  0 & 1 \smz w_0}( \mu, 1 ),\\ 
\end{align*} 
according to the coset decomposition:
\begin{align*} \N(\p^k) \backslash \GL_2(\o) / \Gamma_0(\p^N)  &  = \coprod\limits_{\alpha \in \p \bmod \p^N} \N(\p^k) \sma 1 & 0 \\ \alpha & 1 \smz \Gamma_0(\p^N) \\ 
                                                                                            & {}\qquad        \amalg    \coprod\limits_{\beta \in \o \bmod \p^{ \max\{N,k\}}}  \N(\p^k) \sma 1 & \beta \\ 0& 1 \smz w_0 \Gamma_0(\p^N)
\end{align*}
The coset decomposition is derived, as usual, via both the Bruhat decomposition over the residue field \cite{BushnellHenniart:GL2}*{page 44}  and the Iwahori decomposition \cite{BushnellHenniart:GL2}*{(7.3.1), page 52}. For the element $\gamma = \sma 1 & 0 \\ \alpha & 1 \smz$, the dimension is read via
\[ \gamma^{-1} \sma 1 & x \\ 0& 1 \smz \gamma = \sma 1+x \alpha&* \\ -\alpha^2 x & * \smz, \qquad \Gamma_0(\p^N) \cap \N(\p^k)^{\gamma^{-1}} = \N( \p^{\min\{ N-2v(\alpha), k \}}) \]
and evaluated to
\begin{align*} \dim \Hom_{\Gamma_0(\p^N) \cap \N(\p^k)^{\gamma^{-1}}}( \mu, 1 )  & = \begin{cases} 1, & \mu|_{1+ \p^{k + v(\alpha)}}=1, \; k + 2 v(\alpha) \geq N, \\
                                                                         0 , & \textup{otherwise},
                                                                       \end{cases} \\
                             &= \begin{cases} 1, & v(\alpha) \geq N-k, \\
                                                                         0 , & \textup{otherwise}.
                                                                       \end{cases}   \end{align*}
 This evaluates the first summand
\begin{align*}  \sum\limits_{\alpha \in \p \bmod \p^N} &  \dim \Hom_{\Gamma_0(\p^N) \cap \sma 1 & 0 \\ - \alpha & 1 \smz \N(\p^k) \sma 1 & 0 \\  \alpha & 1 \smz}( \mu, 1 ) \\&  = \# \{ \p^{N-k} \cap \p / \p^{N} \} = \begin{cases} q^k, & N-k \geq 1, \\ q^{N-1}, & N-k \leq 1.\end{cases}\end{align*}
For the element $\gamma = \sma 1 & \beta \\ 0 & 1 \smz w_0$, the dimension is understood via 
     \[ \gamma^{-1} \sma 1 & x \\ 0  & 1 \smz \gamma  = \sma 1&0 \\  x & 1 \smz,\]
 and evaluated as
   \[ \dim \Hom_{\Gamma_0(\p^N) \cap \N(\p^k)^{\gamma^{-1}}}( \mu, 1 ) = 1,\]
thus evaluating the second summand to
\begin{align*}    \sum\limits_{\beta \in \o \bmod \p^{\max\{N,k\}}} &  \dim \Hom_{\Gamma_0(\p^N) \cap w_0^{-1}\sma 1 & -\beta \\  & 1 \smz \N(\p^k) \sma 1 & \beta \\  0 & 1 \smz w_0}( \mu, 1 ) \\ 
                                                                                             &                                                       = \begin{cases} q^k, & N - k \geq 0, \\ q^{N}, & N-k \leq 0. \end{cases}
\end{align*}
In this way, we have established for $\mu \neq 1$
\[        \dim  \Hom_{\N(\p^k)}( \Res_{\N(\p^k)} \Ind_{\Gamma_0(\p^N)}^{\GL_2(\o)} \mu , 1)    = \begin{cases} 2 q^k, & N-k \geq 1, \\  q^{N}+q^{N-1}, &  N-k \leq 0. \end{cases}  \]
and
\begin{align*} \frac{\partial}{\partial s} \zeta_{\GL_2(\F)}(s, \phi_{\rho(\mu)})  \Big|_{s=1}  & = \log(q)  \sum\limits_{k=0}^{N-1}  \frac{2q^k}{q^{k}}  + \log(q)\sum\limits_{k=N}^{\infty}  \frac{q^N( 1 +q^{-1})}{q^{k}} \\
& = 2 N \log(q)   + \log(q)  \frac{1+q^{-1}}{1-q^{-1}}.
\end{align*}
The results follow by linearity and the definition $\phi_\mu = \phi_{\rho(\mu)} /C$ with
\[ C  = \mA_{\rho(\mu)} \phi_{\rho(\mu)}(1) = 2\]
for $\mu \neq 1$, and a short calculation
\begin{align*} & \zeta_{\GL_2(\F)}'(1, \phi_{\rho(\mu)}) (1-q^{-1}) + \log(q) q^{-1}(1-q^{-1}) \zeta_{\GL_2(\F)}(1, \phi_{\mu}) 
\\  & = N  \log(q) ( 1- q^{-1}) +   \frac{\log(q)}{2} ( 1 + q^{-1} ) + \log(q) q^{-1}( 1-q^{-1}) \\
                &                         = \log(q) \left( (N+1/2) +\frac{-N+3/2}{q}  + \frac{1}{2q^2} \right).\qedhere                                                                                                      
\end{align*}
\end{proof}

\section{The hyperbolic distributions}
Similarly to the real and complex discussion, we consider an element $\gamma \in \GL_2(\F)$. By definition, an element $\gamma$ is hyperbolic if its characteristic polynomial splits into two distinct factors over $\F$, or alternatively, if $\gamma$ is conjugate to an element
\[  \sma \alpha & 0 \\ 0 & \beta \smz\]
for $\alpha, \beta \in \F^\times$ with $\alpha \neq \beta$.
The stabilizer of $\sma \alpha & 0 \\ 0 & \beta \smz$ is the subgroup $\M(\F) = \sma * & 0 \\0 & * \smz$ of diagonal matrices in $\GL_2(\F)$. The Arthur trace formula associates to a hyperbolic element two types of distributions \cite{Gelbart}*{Proposition 1.1, page 46}; a hyperbolic orbital integral for $\phi \in \Ccinf(\GL_2(\F), \overline{\chi})$
\[ J_\gamma ( \phi ) = \int\limits_{\M(\F) \backslash \GL_2(\F)} \phi ( g^{-1}  \gamma g ) \d \dot{g},\]
and a weighted hyperbolic orbital integral for $\phi \in \Ccinf(\GL_2(\F)$
\[  J_\gamma^H ( \phi) =  \int\limits_{\M(\F) \backslash \GL_2(\F)}  \phi ( g^{-1}  \gamma g ) w_H(g) \d \dot{g}.\]
The quotient measure $\textup{d} \dot{g}$ is defined here as the unique right invariant Radon measure on $\M(\F) \backslash \GL_2(\F)$, such that
\[ \int\limits_{ \GL_2(\F)} f(g) \d g  =   \int\limits_{\M(\F) \backslash \GL_2(\F)}  \int\limits_{ \M(\F)} f(mg) \d m \d \dot{g} .\]
Here $H$ is defined via the Iwasawa decomposition, the modular character 
\[  H(g) = \Delta_{\B(\F)} (b), \qquad g = bk, \qquad b \in \B(\F), k \in \GL_2(\o),\]
and the weight
\[ w_H(g) = H(w_0g ) + H(g), \qquad w_0 = \sma 0 & -1 \\ 1 & 0 \smz.\]
We have invariance properties for all $z \in \Z(\F)$ and $g \in \GL_2(\F)$:
\[ J_\gamma( \phi) =  \chi(z) J_{z\gamma}( \phi) =    J_{g^{-1}\gamma g}( \phi) = J_{\gamma}( \phi^g), \]
 and for $z \in \Z(\F)$ and $k \in \GL_2(\o)$
\[ J_\gamma^H( \phi) =  \chi(z) J_{z\gamma}^H( \phi) = J_{\gamma}^H( \phi^k).\] 
\begin{proposition}[The hyperbolic distributions]\label{prop:padichyper}  \mbox{}
 Consider the hyperbolic element $\gamma$ in $\GL_2(\F)$ of the form
\[ \gamma = \sma m & 0 \\ 0 & 1 \smz.\]
The hyperbolic distributions evaluate to
\begin{align*}
 J_\gamma(\phi_\St) & = J_\gamma^H(\phi_\St)  =0 , \\
 J_\gamma(\phi_\pi) & = J_\gamma^H(\phi_\pi)  = 0, \qquad \pi \textup{ supercuspidal}, \\
J_\gamma(\phi_\mu) & =\begin{cases} \frac{ \mu(m)  }{\left| 1 - m \right|_v} , & m \in \o,  \\ 0 , & \textup{else}, \end{cases}\qquad \mu : \o^\times \rightarrow \bC^1,\\
J_\gamma^H(\phi_\mu) & = 0, \\
J_\gamma( \HT) &= \begin{cases}  q^{-1/2}  \mu(m), &\gamma \in \sma \w^{\pm 1}& 0\\ 0 & 1\smz \M(\o) \Z(\F), \\ 
                                                                       0 , & \gamma \notin \sma \w^{\pm 1}& 0\\ 0 & 1\smz \M(\o) \Z(\F),\end{cases} \\
 J_\gamma^H( \HT)& = \begin{cases}  \frac{2\log(q) }{q-1}  \mu(m), & \gamma \in \sma \w^{\pm 1}& 0\\ 0 & 1\smz \M(\o) \Z(\F), \\ 
                                                                       0 , & \gamma \notin \sma \w^{\pm 1}& 0\\ 0 & 1\smz \M(\o) \Z(\F). \end{cases}
\end{align*}
\end{proposition}
 \begin{proof}
The orbital integral of $\gamma$ is absolutely convergent \cite{Rao}. We divide the proof of this theorem into several lemmas.
\begin{lemma}[Explicit form of $w_H$]\label{lemma:wHp}
\[ w_H\left( \sma m_1 & 0  \\  0 & m_2 \smz \sma 1 & t \\ 0 & 1 \smz k \right) = -2\log \max\{1, \left| t \right|\}.\]
\end{lemma}
\begin{proof}
By definition, we have for $b \in \B(\F)$
 \[ w_H\left( b k \right) = w_H\left( b \right). \]
The result for $t \in \o$ is clear. We have for $t \notin \o$ the matrix decomposition
\begin{align}
   w_0 \sma 1 & t \\ 0 & 1\smz =   \sma 0 & -1 \\ 1 & t \smz =      \sma t^{-1} &-1 \\ 0 &t \smz\sma 1 & 0 \\ t^{-1} & 1 \smz.
\end{align}
We complete the computation with
\begin{align}
 w_H\left( \sma m_1 & 0  \\  0 & m_2 \smz \sma 1 & x \\ 0 & 1 \smz k \right)  &= \log \Delta_B\left( \sma  m_1 & 0 \\ 0 & m_2 \smz \right)   +  \log  \Delta_B \left( \sma m_2 t^{-1} & *  \\ 0 & m_1 t \smz \right)  \\
                                                                                                                                   & =      \log |m_1/m_2| + \log|m_2/(m_1 t^2 )| = -2 \log \left| t \right| . \qedhere
\end{align}
\end{proof}
\begin{lemma}
  For $\phi \in \mH(\GL_2(\F), \rho)$, we obtain
 \[ J_\gamma (\phi) = \int\limits_{\N(\F)} \tr \phi(n^{-1}\gamma n) \d n, \]
and
\[   J^H_\gamma (\phi) = \int\limits_{\N(\F)} \phi(n^{-1}\gamma n) w_H(n) \d n.\]
 \end{lemma}
\begin{proof}
 We choose the ordinary measure $\d r$ on $\F$, and set $\d^\times r = \d r/ |r|$ on $\F^\times$. We assign to $\GL_2(\o)$ a unit Haar measure $\d k$. We have fixed the unique Haar measure $\d g$ on $\GL(2,\F)$ and $\M(\F)$, such that for all $f \in \mL^1((\GL_2(\F))$  we have
\[ \int\limits_{\F^\times} \int\limits_{\F^\times} \int\limits_{\F} \int\limits_{\GL_2(\o)}  f\left( \sma z & 0 \\ 0 & z \smz \sma a & 0 \\ 0 & 1 \smz \sma 1 & x \\ 0 & 1 \smz k\right) \d k \d^+ r \d^\times a \d^\times z = \int\limits_{\GL_2(\F)} f(g) \d g. \]
and for $h \in \mL^1(\M(\F))$
 \[ \int\limits_{\F^\times} \int\limits_{\F^\times}  h\left( \sma z & 0 \\ 0 & z \smz \sma a & 0 \\ 0 & 1 \smz  \right)  \d^\times a \d^\times z = \int\limits_{\M(\F)} h(m) \d m. \]
According to the Iwasawa decomposition, the quotient measure is definable via the property that for all continuous, compactly supported functions $F : \M(\F) \backslash \GL_2(\F) \rightarrow \bC$, we have 
\[   \int\limits_{\M(\F) \backslash \GL_2(\F)}  F\left( g\right) \d \dot{g}       =       \int\limits_{\F} \int\limits_{\GL_2(\o)}  F\left(  \sma 1 & x \\ 0 & 1 \smz k\right) \d k \d^+ r.\] 
The (weighted) orbital integral is computed for $\phi \in \mH(G, \rho)$ as
\[ \int\limits_{\M(\F) \backslash \GL_2(\F)} \tr \phi(g^{-1} m  g)  w_j(g) \d \dot{g} =  \int\limits_{\F} \int\limits_{\GL_2(\o)} \tr \phi\left(k^{-1} \sma 1 & -x \\ 0 & 1 \smz m \sma 1 & x \\ 0 & 1 \smz  k\right)   \d k w_j\left(\sma 1 & x \\ 0 & 1 \smz \right) \d^+ r ,\]
where the weight \( w_j \) is either constant or $w_H$. Of course, this requires some knowledge of $w_H$ as given by Lemma~\ref{lemma:wHp}.  By restricting the result to elements $\phi \in \mH(\GL_2(\F),\rho)$, which satisfy by definition
\[  \int\limits_{\GL_2(\o)} \tr \phi\left(k^{-1} g  k\right)   \d k    = \tr \phi(g),\] 
we have proven the claim.
\end{proof}
\begin{corollary}
If $\phi \in \Ccinf(\GL_2(\F))$ is supported on $\Z(\F) \GL_2(\o)$, then the weighted hyperbolic orbital integral vanishes.
\end{corollary}
\begin{proof}
The weight $w_H$ as a function $\N(\F)\rightarrow \bR$ is supported on $\N(\F) - \N(\o)$. For $m \neq 1$, we have that for $x \in \N(\o)$ 
\[ \sma 1 & -x \\ 0 & 1 \smz \sma m_1 & 0 \\ 0 & m_2 \smz \sma 1 &  x \\ 0 & 1\smz =  \sma m_1 & 0 \\ 0 & m_2 \smz \sma 1 & x-m_2 x/m_1 \\ 0 & 1  \smz \in \GL_2(\o)\Z(\F),\]
if and only if
\[ \sma m_1 &  0 \\ 0 & m_2 \smz \in \Z(\F) \M(\o).\qedhere\]
\end{proof}
This proves the vanishing assumption of the weighted hyperbolic orbital integral in all but the case of $\phi_\pi$, where $\pi$ is a ramified supercuspidal representation. This remaining case is covered by a similar argument.
\begin{corollary}
If $\phi \in \Ccinf(\GL_2(\F))$ is supported on the normalizer $N\Gamma_0(\p)$ of $\Gamma_0(\p)$ in $\GL_2(\F)$, then the weighted hyperbolic orbital integral vanishes.
\end{corollary}
\begin{proof}
We have that
\[   N\Gamma_0(\p) =  \Gamma_0(\p) \rtimes \sma 0 & -1 \\ \w & 0 \smz. \]  
The weight $w_H$ as a function $\N(\F)\rightarrow \bR$ is supported on $\N(\F) - \N(\o)$. For $m \neq 1$, we have that for $x \in \N(\o)$ 
\[ \sma 1 & -x \\ 0 & 1 \smz \sma m_1 & 0 \\ 0 & m_2 \smz \sma 1 &  x \\ 0 & 1\smz =  \sma m_1 & 0 \\ 0 & m_2 \smz \sma 1 & x-m_2 x/m_1 \\ 0 & 1\smz \in \Gamma_0(\p)\rtimes \sma 0 & -\w \\ 0 & 1\smz ,\]
if and only if
\[ \sma m_1 &  0 \\ 0 & m_2 \smz \in \Z(\F) \M(\o).\qedhere\]\end{proof}
\begin{lemma}[The invariant orbital integral]
  For $\phi \in \mH(\GL_2(\F), \rho)$ and $\gamma \in \M(\F)$, we obtain for $t >0$ and $+t \neq 1$
 \[ J_\gamma (\phi) = \frac{1}{\left| \tr \sma 1 & \\  & -1\smz \gamma \right|} \mA_\rho \tr \phi( \gamma). \]
 \end{lemma}
\begin{proof}
Set $\gamma  = \sma m_1& 0 \\ 0 & m_2 \smz.$ Consider the matrix  
 \[ \sma 1 & -x \\ 0 & 1 \smz \sma m_1 & 0 \\ 0 & m_2 \smz  \sma 1 & x \\ 0 & 1 \smz  =\sma m_1 & m_1 x - m_2 x \\ 0 & m_2 \smz ,\]
which yields
\begin{align*}\int\limits_{\N(\F)} \phi\left(n^{-1}\sma \pm t & 0 \\ 0 & 1 \smz n\right) \d n  &= \frac{1}{|m_1 - m_2 | } \int\limits_{\F} \tr  \phi\left( \sma m_1 & m_1 x - m_2 x \\ 0 & m_2 \smz\right) \d x \\ 
&= \frac{1}{|m_1 - m_2 | } \mA_\rho \tr \phi(\gamma). \qedhere\end{align*}
\end{proof}
\begin{corollary}
\[ J_\gamma( \HT) = \begin{cases}  q^{-1/2}  \mu(m), & \gamma = \sma m \w^{\pm 1} & 0 \\ 0 & 1 \smz \Z(\F), \\ 
                                                                       0 , & \gamma \notin \sma \w^{\pm 1}& 0\\ 0 & 1\smz \M(\o) \Z(\F). \end{cases}\] 
\end{corollary}

What remains to be shown is the weighted hyperbolic integral of the Hecke operators. Fortunately, we have efficiently controlled their support.
\begin{lemma}
\[ J_\gamma^H( \HT) = \begin{cases}  \frac{2\log(q) }{q-1}  \mu(m), & \gamma = \sma m \w^{\pm 1} & 0 \\ 0 & 1 \smz \Z(\F), \\ 
                                                                       0 , & \gamma \notin \sma \w^{\pm 1}& 0\\ 0 & 1\smz \M(\o) \Z(\F). \end{cases}\] 
\end{lemma}
\begin{proof}
The support of $\HT$ is precisely
\[ \sma \w &  0 \\ 0 & 1 \smz \N(\p^{-1}) \Z(\F) \GL_2(\o)    \amalg \sma \w^{-1} &  0 \\ 0 & 1 \smz \N(\p^{-1}) \Z(\F)\GL_2(\o),\]
and that of its Abel transform is
\[ \sma \w &  0 \\ 0 & 1 \smz \M(\o)\Z(\F)    \amalg \sma \w^{-1} &  0 \\ 0 & 1 \smz\M(\o)\Z(\F) .\]
The vanishing assertion follows directly from support considerations. Now let
\[ \gamma =         \sma m \w^{\pm 1} & 0 \\ 0 & 1 \smz .\]
We compute
\begin{align*}
  J^H_\gamma (\HT) & =2  \int\limits_{\F} \phi\left( \sma 1 & -t \\ 0 & 1 \smz  \sma m \w^{- 1} & 0 \\ 0 & 1 \smz  \sma 1 & t \\ 0 &1\smz \right) \log \{ |t| , 1 \} \d^+t \\ 
                              & =   \frac{2\log(q)q^{-1}}{1-q^{-1}} \int\limits_{\o^\times} \phi\left(   \sma m \w^{-1} & 0 \\ 0 & 1 \smz  \sma 1 & u \w (1- m^{-1}\w^{+ 1})  \\ 0 &1\smz \right) \d^+u \\
                              & =   \frac{2\log(q)}{q-1} \int\limits_{\o^\times} \phi\left(   \sma m \w^{- 1} & 0 \\ 0 & 1 \smz  \sma 1 & u\w \\ 0 &1\smz \right) \d^+u \\
                              & =     \frac{2\log(q)}{q-1} \int\limits_{\o^\times} \phi\left(  \sma 1 & u \\ 0 &1\smz \sma m \w^{- 1} & 0 \\ 0 & 1 \smz  \sma 1 & \w \\ 0 & 1 \smz \sma  u^{- 1} & 0 \\ 0 & 1 \smz  \right) \d^+u \\
                              & =     \frac{2\log(q) }{q-1} \phi\left(    \sma m \w^{-1} & 0 \\ 0 & 1 \smz \right)   =     \frac{2\log(q) }{1-q^{-1}}  \mu(m),
\end{align*}
and similarly 
\begin{align*}
    J^H_\gamma (\HT) & =2  \int\limits_{\F} \phi\left( \sma 1 & -t \\ 0 & 1 \smz  \sma m \w^{+ 1} & 0 \\ 0 & 1 \smz  \sma 1 & t \\ 0 &1\smz \right) \log \{ |t| , 1 \} \d^+t \\ 
                              & =   \frac{2\log(q)}{q-1} \int\limits_{\o^\times} \phi\left(   \sma m \w^{+1} & 0 \\ 0 & 1 \smz  \sma 1 & u (\w -  m^{-1})  \\ 0 &1\smz \right) \d^+u \\
                              & =   \frac{2\log(q)}{q-1} \int\limits_{\o^\times} \phi\left(   \sma m \w^{+1} & 0 \\ 0 & 1 \smz  \sma 1 & u \\ 0 &1\smz \right) \d^+u \\
                              & =      \frac{2\log(q) }{q-1}  \mu(m).  \qedhere
\end{align*}
\end{proof}
\end{proof}

\section{The intertwiner and its derivative}
Let $\mu$ be an algebraic one-dimensional representation of $\bF^\times$ for $j=1,2$. Let $(\mu,1)$ be the associated one-dimensional representation of $\M(\F)$. Set $w_0 = \sma 0 & -1 \\ 1 & 0 \smz$ and $\mu^{w_0} = (1, \mu)$. 

We define the intertwiner
\begin{align*} \mM(\mu,1, s) &\colon \mJ(\mu,1, s) \rightarrow \mJ(1, \mu , -s),  \\
                    \mM(\mu,1, s)& f(g) \coloneqq \int\limits_{\N(\F)} f( w_0 n g) \d n. 
\end{align*}
By the Iwasawa decomposition, a smooth function $F\in \mJ(\mu,s)$ is uniquely determined by its value on $\GL_2(\o)$, since by definition
\[  F\left( \sma a & * \\ 0 & b \smz k\right) = |a/b|_v^{s+1/2} \mu_1(a) \mu_2(b) F(k).\]
A canonical basis is provided by the representation theory of $\GL_2(\o)$. Because $\mM(\mu,s)$ is an intertwiner, it restricts to an intertwiner on the $\rho$-isotype
\[ \mM(\mu,1, s)^{\rho} \coloneqq \mM(\mu,1, s) \Big|_{ \mJ(\mu,1, s)^\rho} :  \mJ(\mu,1, s)^\rho \rightarrow  \mJ(1, \mu, -s)^\rho.\]
For our later applications, it is enough to compute it for the representation $\rho(\mu) \coloneqq \rho(\mu,1, \cond(\mu))$ of $\GL_2(\o)$.
\begin{proposition}\label{prop:padicinter}
The $\GL_2(\o)$-intertwiner $\mM(1,1,s)^{\rho(1)}$ acts by the scalar
\[ \frac{\zeta_v(2s)}{\zeta_v(2s+1)}.\]
For $\cond(\mu) \neq \o$, the $\GL_2(\o)$-intertwiner $\mM(\mu,1,s)^{\rho(\mu)}$ acts by the scalar one.
\end{proposition}
\begin{proof}
Every $\GL_2(\o)$-isotype has at most single multiplicity and the $\rho(\mu)$-isotype has multiplicity one in the representation $\mJ(\mu,1, s)$. 
 By Schur's Lemma, the intertwiner acts as a scalar --*- however, there is an implicit identification
\[ m(\mu, 1)\colon \rho(\mu, 1, \p^N) \xrightarrow\cong \rho(1, \mu, \p^N), \qquad m(\mu,1) f(k) \coloneqq \int\limits_{\N(\o)} f(w_0nk)\d n \]
occuring.  Let $\p^N$ be the conductor of $\mu$.  We begin with $N = 0$, since the computation is straightforward. 
Consider the element $F_{1, 1,s} \in \mJ(1,1,s)$, such that
\[ F_{1, 1,s} |_{\GL_2(\o)} =    \mathds{1}_{\GL_2(\o)}.\]
We have an identity
\[ \mM(1, 1,s)       F_{1, 1,s} |_{\GL_2(\o)} =  \lambda    F_{1,1,-s}.\]
We compute the scalar $\lambda$:
\begin{align*}
     \mM(\mu, 1,s)   &    F_{\mu, 1,s}  \left( w_0 \right)  = 1 + (1-q^{-1}) \sum\limits_{k >0} q^{k} \int\limits_{\o^\times} F_{1, 1, s} \left( \sma 1 & 0 \\ -u \w^{-k} & 1 \smz \right) \d^\times u \\
                                                                            & =  1+(1-q^{-1})  \sum\limits_{k >0} q^{k} \int\limits_{\o^\times} F_*  \left( \sma u^{-1} \w^k & 1\\ 0 &  -u \w^{-k} \smz \sma 0 & 1 \\ 1 & -u^{-1} \w^k \smz \right) \d^\times u \\
                                                                             & = 1 + (1-q^{-1})  \frac{q^{-2s}}{1-q^{-2s}} =  \frac{1- q^{-2s-1}}{1-q^{-2s}} = \frac{\zeta_v(2s)}{\zeta_v(2s+1)}.
\end{align*}
Now let $N>0$. For any coset
\[ \Gamma_0(\p^N) \backslash \GL_2(\o), \]
we define a vector in $\rho(\mu_1, \mu_2, \p^N)$
\[ f_{\mu_1, \mu_2 ,  \Gamma_0(\p^N)\gamma} \left( \sma a & * \\ 0 & b \smz \right)  =  \begin{cases} \mu(k\gamma^{-1}), & k \in \Gamma_0(\p^N) \gamma , \\ 0 , & \textup{otherwise}, \end{cases}\]
 which is supported on $\Gamma_0(\p^N)\gamma$. 
\begin{lemma}
 The isomorphism $m(\mu,1)$ acts by
\[      m(\mu, 1)  f_{\mu, 1,  \Gamma_0(\p^N) } = q^{-N} \sum\limits_{x \in \o \bmod \p^N} f_{1, \mu, \Gamma_0(\p^N) w_0 \sma 1 & x \\ 0 & 1\smz    }.\] 
\end{lemma}
\begin{proof}
 The element $w_0 n k$ lies in $\Gamma_0(\p^N)$ if and only if $ k$ lies in $n w_0 \Gamma_0(\p^N)$. The Iwahori decomposition \cite{BushnellHenniart:GL2}*{(7.3.1), page 52} yields                                                                  
\[               \Gamma_0(\p^N) \backslash \N(\o) w_0 \Gamma_0(\p^N) =\coprod\limits_{x \in \o \bmod \p^N}  \Gamma_0(\p^N) w_0    \sma 1 & x \\ 0 & 1\smz.\]
The function is a linear combination
\[        m(\mu, 1)  f_{\mu, 1,  \Gamma_0(\p^N) }             =               \sum\limits_{x \in \o \bmod \p^N}  \alpha_x f_{1, \mu, \Gamma_0(\p^N)w_0 \sma 1 & x \\ 0 & 1\smz    }.\]
We compute the values $\alpha_x$
\begin{align*}    m(\mu, 1)  f_{\mu, 1,  \Gamma_0(\p^N) } \left(      w_0    \sma 1 & x \\ 0 & 1\smz \right) &= \int\limits_{\N(\o)} f_{\mu, 1,  \Gamma_0(\p^N) } \left( w_0 n w_0   \sma 1 & x \\ 0 & 1\smz \right) \d n \\
                                                                                                                                                       & =    q^{-N} \sum\limits_{n \in \N(\o)/\N(\p^N)} f_{\mu, 1,  \Gamma_0(\p^N) } \left( w_0 n w_0   \sma 1 & x \\ 0 & 1\smz \right) \\
                                                                                                                                                        & = q^{-N}. \qedhere\end{align*}
\end{proof}
\begin{corollary}
 $   m(\mu, 1) m(1, \mu) = 1.$
\end{corollary}
By Schur's Lemma, there exists only one intertwiner up to a complex scalar. So on the $\rho$-isotype
\( \mM(\mu, 1,s) \) and $m(\mu,1)$ differ only by a scalar. 

Consider the element $F_{\mu, 1,s} \in \mJ(\mu,1,s)$, such that
\[ F_{\mu, 1,s} |_{\GL_2(\o)} =      f_{\mu, 1,  \Gamma_0(\p^N) }.\]
We have 
\[ \mM(\mu, 1,s)       F_{\mu, 1,s} |_{\GL_2(\o)} =  \lambda    q^{-N} \sum\limits_{x \in \o \bmod \p^N} f_{1, \mu, \Gamma_0(\p^N)  w_0 \sma 1 & x \\ 0 & 1\smz   }.\]

We compute the scalar $\lambda =1$:
\begin{align*}
  \lambda q^{-N} &=  \mM(\mu, 1,s)       F_{\mu, 1,s}  \left( w_0 \right) \\
& =  \int\limits_{\N(\o)} F_{\mu,1,s} \left( w_0n w_0 \right) \d n + \sum\limits_{k >0} q^{k} \int\limits_{\o^\times} F_{\mu, 1,s} \left( \sma 1 & 0 \\ -u \w^{-k} & 1 \smz \right) \d^\times u \\
                                                                            & =  q^{-N} + \sum\limits_{k >0} q^{k} \int\limits_{\o^\times} \underbrace{ F_{\mu,1,s} \left( \sma u^{-1} \w^k & 1\\ 0 &  -u \w^{-k} \smz \underbrace{\sma  & 1 \\ 1 & -u^{-1} \w^k \smz}_{\notin \Gamma_0(\p^N)} \right)}_{=0} \d^\times u \\
                                                                            & = q^{-N} . \qedhere
\end{align*}
\end{proof}

Additionally, we have to treat the irreducible representation $\rho(1, \p)$, i.e., the Steinberg representation of $\GL_2(\o/\p)$.
\begin{proposition}[The Steinberg representation]\label{prop:intersteinberg}
The $\GL_2(\o)$-intertwiner $\mM(1,1,s)^{\rho(1,\p)}$ acts by the scalar
\[ \frac{\zeta_v(2s)}{\zeta_v(2s-1)}.\] 
\end{proposition}
\begin{proof}
Set $w_0 =\sma 0 & -1\\ 1 & 1\smz$. Let us first project an element
\[ \mathds{1}_{\Gamma_0(\p)} \in \Ind_{\Gamma_0(\p)}^{\GL_2(\o)} 1 \]
onto $\rho(1, \p)$ by considering
\begin{align*} f(x) &= \mathds{1}_{\Gamma_0(\p)}(x)  - \int\limits_{\GL_2(\o)}   \mathds{1}_{\Gamma_0(\p)}(gx)     \d g   \\
                        f  & =  \mathds{1}_{\Gamma_0(\p)}  - \frac{1}{1+q} \mathds{1}_{\GL_2(\o)}.
\end{align*}
The intertwiner 
\[ m(1, \p) f(g) = \int\limits_{\N(\o)} f (w_0  n g) \d n\]
is up to a complex scalar the only intertwiner on $\rho(1, \p)$ by Schur's Lemma. 
 By the Bruhat decomposition over the residue field
\[ \GL_2(\o) = \Gamma_0(\p) \amalg \Gamma_0(\p) w_0 \Gamma_0(\p),\]
we have that
\[ m(1, \p) f (g) = \begin{cases} \frac{1}{(1+q)q}, &x \in \Gamma_0 (\p) w_0 \Gamma_0(\p), \\   - \frac{1}{1+q} , & x \in \Gamma_0(\p) . \end{cases}\]
The operator $m(1,\p)$ is an isometry. Choose $F_s \in \mJ(1,1,s)$ such that
\[ F_s |_{\GL_2(\o)} =f.\]
Consequently, $\mM(1,1,s)$ acts on the $\rho$-isotype by
\[ \mM(1,1,s) F_s(w_0) = \frac{1}{q(q+1)} \lambda_s \]
We compute $\lambda_s$:
\begin{align*}
 \mM(1, 1,s)   &    F_{s}  \left( w_0 \right)  = \frac{1}{(1+q)q} + (1-q^{-1}) \sum\limits_{k >0} q^{k} \int\limits_{\o^\times} F_s \left( \sma 1 & 0 \\ -u \w^{-k} & 1 \smz \right) \d^\times u \\
                                                                            & = \frac{1}{(1+q)q}  + (1-q^{-1}) \sum\limits_{k >0} q^{k} \int\limits_{\o^\times} F_s \left( \sma u^{-1} \w^k & 1\\ 0 &  -u \w^{-k} \smz \sma 0 & 1 \\ 1 & -u^{-1} \w^k \smz \right) \d^\times u \\
                                                                             & =\frac{1}{(1+q)q}  - \frac{(1-q^{-1})}{1+q}  \frac{q^{-2s}}{1-q^{-2s}} \\
                                                                             &  =\frac{1}{(1+q)q}  \left( 1 - (q-1)  \frac{q^{-2s}}{1-q^{-2s}} \right) \\
                                                                        &  =\frac{1}{(1+q)q} \frac{1- q^{1-2s} }{1-q^{-2s}} = \frac{1}{(1+q)q}  \frac{\zeta_v(2s)}{\zeta_v(2s-1)}. \qedhere
\end{align*}
\end{proof}

\section{The elliptic distributions}
In this section, I will present a method for computing the orbital elliptic integrals for $\GL(2)$.  This method differs slightly from the real situation, but the central idea remains the same. We will require a decomposition of the Haar measure that corresponds to the Cartan decomposition. The computations become easier when we first make use of the Iwahori subgroup $\Gamma_0(\p)$ in place of $\GL_2(\o)$, and deduce the results for $\GL_2(\o)$ \cite{IwahoriMatsumoto:Bruhat} from there.

Compare the situation with the real case in Section~\ref{sec:realell}. An element of $\GL_2(\F)$ is elliptic if its characteristic polynomial is irreducible over $\F$. As usual, we are only interested in results modulo the center. 
 \begin{lemma}\label{lemma:form}
Let $\gamma$ be an elliptic element in $\GL_2(F)$ for a non-archimedean field (or, a Henselian field) $F$ with characteristic polynomial
\begin{align} \det(X - \gamma) = X^2  - \tr( \gamma) X + \det(\gamma), \qquad \det( \gamma )= u \w^n, \; u \in \o^\times.\end{align}
The element $\gamma$ is conjugate in $\GL_2(\F)$ to the element
\begin{align} \w^{\lfloor n/2\rfloor } \sma 0 & 1 \\ -1 & 0 \smz \begin{cases}  \sma 1 & t \\ 0 & u \smz, & n \textup{ even}, \\ \sma 1 & t \\ 0 & u\w^{1} \smz, & n \textup{ odd}, \end{cases} \qquad t = \tr( \gamma)  / \w^{\lfloor n/2\rfloor } \in \o. \end{align}
\end{lemma}
\begin{defn}[Unramified elliptic elements]
We say that an elliptic element $\gamma \in \GL_2(\F)$ is unramified if $\left| \det \gamma \right|_v =q^{2n}$ for some integer $n$, and else, we call it ramified.
\end{defn}
\begin{proof}[Proof of the lemma]
If $\gamma$ is elliptic, then $\lambda \gamma$ is elliptic for all $\lambda \in F^\times$. Note that $\det( \w^{-\lfloor n/2\rfloor } \gamma ) = u w^k$ for $k=0$ if $n$ is even or $k=1$ if $n$ is odd. The theory of the canonical rational form \cite{BushnellHenniart:GL2}*{Section 5.3, page 44} gives that $\gamma$ is conjugate to
\[ \sma 0 & -\det(\gamma) \\ 1 & \tr \gamma \smz.\]
and consequently, that $\gamma'  = \w^{-\lceil n/2\rceil  } \gamma$ is conjugate to an element
\begin{align} \sma 0 & - u \w^k \\ 1 & t \smz = \sma 0 & -1 \\ 1 & 0 \smz \sma 1 &t  \\0&  u \w^k\smz.\end{align}
By Hensel's Lemma \cite{Neukirch}*{Corollary 6.5, page 147}, an irreducible polynomial $\alpha_n X^n + \alpha_{n-1} \dots + \alpha_0$ satisfies $|\alpha_j| \leq \max \{ |\alpha_n|, |\alpha_0| \}$. 
The statement of the lemma follows then from the irreducibility of the characteristic polynomial, as it implies that $t \in \o$. 
\end{proof}
 Let $G$ be $\GL_2(\F)$, and let $G_\gamma$ be the centralizer of $\gamma$ in $G$. If $\gamma$ is elliptic, then $G_\gamma$ is compact modulo the center $Z=\Z(\F)$. The Haar measure of the group $Z$ has been fixed in the case of $\GL_2(\F)$, and is the counting measure in the $\SL_2(\F)$. Let the quotient $Z \backslash G_\gamma$ carry the unique right $G_\gamma$-invariant probability measure $\d q$. We fix the unique Haar measure $\d h$ on $G_\gamma$, such that
\[ \int\limits_{G_\gamma} f(h)\d h =    \int\limits_{Z \backslash G_\gamma} \int\limits_{Z} f(zq)\d z \d q, \qquad f \in \mL^1(G_\gamma).\]
 The orbital integral is then defined as a distribution $\Ccinf(G) \rightarrow \bC$ given by
\[ J_\gamma(\phi) := \int\limits_{G_\gamma \backslash G} \phi(g^{-1}\gamma g)  \d \dot{g},\]
where $\d \dot{g}$ is the unique Radon measure on the homogeneous space $G_\gamma \backslash G$, such that
\[ \int\limits_{G_\gamma \backslash G} \int\limits_{G_\gamma}  f(hg) \d h \d \dot{g} =  \int\limits_{G} f(g) \d g , \qquad f \in \mL^1(G). \]

With this, we have introduced enough notation that the main result of this section can be stated:
\begin{proposition}[The elliptic distributions]\label{prop:padicelliptic}
Consider the function $\phi_\mu$ for a one-dimensional representation $\mu$ of $\o^\times$ with conductor $\p^N$, the functions $\phi_{\St}$ and $\phi_\pi$ for a supercuspidal representation $\pi$.
\begin{itemize}
 \item Let $\gamma$ be an unramified elliptic element, then
\begin{align*} J_\gamma( \phi_{\mu}) &= \begin{cases}1, & \mu = 1, \\ 
                                                         0 , & \mu \neq 1,                 %    (1+q^{-1})\sum\limits_{\substack{ x  \bmod \p^N \\ x^2 + x \tr \gamma +\det \gamma \equiv 0 \bmod \p^N}}  \mu(-x-\tr \gamma)
                                        \end{cases}\\
  J_\gamma( \phi_{\St}) & =-1,  \\                                      %    (1+q^{-1}) \left(  \card \left\{ x \bmod \p :  x^2 + x \tr \gamma +\det \gamma \equiv 0 \bmod \p\right\} - 2 \right)
 \\ J_\gamma(\phi_\pi) &= \begin{cases} \phi_\pi(\gamma), &  \pi \; \textup{unramified supercuspidal}, \\ 
                                                                                                                                                  0, & \pi \; \textup{ramified supercuspidal},
                                                                                                                                 \end{cases}\\
 J_\gamma(\HTo) &=J_\gamma(\HT)  = 0. 
\end{align*}
 \item    Let $\gamma$ be a ramified elliptic element, then
\begin{align*}J_\gamma( \phi_{\mu}) &= 0, \qquad \qquad \qquad \qquad  J_\gamma( \phi_{\St}) = 0, \\  J_\gamma(\phi_\pi) &= \begin{cases} 0 , &  \pi \; \textup{unramified supercuspidal}, \\ 
                                                                                                                                                  \frac{2}{1+q} \phi_\pi(\gamma), & \pi \; \textup{ramified supercuspidal},
                                                                                                                                 \end{cases} \\
 J_\gamma(\HTo) & = 2  ,\\
J_\gamma(\HT) & =    0 \qquad \textup{for } \mu \neq 1.
\end{align*}
\end{itemize}
\end{proposition}

%\begin{proposition}[More precise formulas for $\phi_\pi$]\label{prop:padicellipticcusp}
%Let $\gamma$ be a unramified / ramified elliptic element in $\GL_2(\F)$, and let $\pi$ be an unramified %/ ramified supercuspidal representation of $\GL_2(\F)$ with cuspidal type $K, \rho(\theta \psi_\gamma)$. %We have the following formula for the elliptic distribution
%\[  \phi_\pi(\gamma_0)   = \begin{cases} C_\pi \left(   \theta(\lambda_1) + \theta(\lambda_2) \right), &  \gamma \sim \gamma_0, \\  0, &\textup{else},\end{cases} \]
%where $\lambda_1, \lambda_2$ are the eigenvalues of $\gamma_0$, and $C_\pi$ is a negative constant given as
%\[ C_\pi = - \dim( \rho(\theta \psi_\gamma)) \begin{cases} 1 , & \pi, \gamma  \textup{ unramified},  \\ 2 /( 1+q) , & \pi, \gamma \textup{ ramified}. \end{cases}\] 
%\end{proposition}
%Let us derive Proposition~\ref{prop:padicellipticcusp} from Proposition~\ref{prop:padicelliptic}. 
%\begin{proof}
% The Frobenius character formula asserts
%\end{proof}

\newcommand{\Waff}{W_{\textup{aff}}} 

There are various equivalent approaches towards such a computation. The Bruhat-Tits building plays a central role in \cite{Kottwitz:GL3}, \cite{Langlands:BaseGL2}, \cite{Rogawski:Thesis}, and lattices in \cite{KnightlyLi}*{page 394}. A general argument for pseudo coefficients of a square-integrable representation for characteristic zero is given in \cite{Kazhdan:Pseudo}*{Theorem K}. For a connection between the building of $\GL(n)$ or rather $\SL(n)$ with lattices, the reader may consider \cite{Abramenko-Brown} or \cite{Garrett:Buildings}. We avoid introducing large parts of these theories and the relevant notation, as we will only use the resulting measure decomposition.

However, some useful details are required. Define $w_0=\sma 0 & -1 \\ 1 & 0 \smz$ with $w_0^2 = -1$ and
\begin{align} w_{\p} =
 \sma 0 & -1  \\ \w& 0 \smz,  \qquad  w_\p^2 = -\w. \end{align}

Note that the group \[ \Waff := \langle w_0 ,w_\p\rangle / w_*^2 = 1 \] 
generated by the set $S  =  \{ w_0, w_\p\}$ modulo the center is a Coxeter group with the relations $w_0^2 = 1$ and $w_\p^2 = 1$. The group $\Waff$ is a set of representatives for $\GL_2(\F) //\Z(\F) \Gamma_0(\p)$. 
We define the Iwahori subgroup as 
\begin{align} I :=  \Gamma_0(\p),  \qquad I_0 := I \cdot \Z(\F). 
\end{align} 

We define the operator
\begin{align}\label{eq:iwahori} \phi \mapsto \phi^I, \qquad \phi^I(x) = \int\limits_{I} \phi(i^{-1}x i) \d i.\end{align}
 The proposition is a conclusion of the following theorem.
\begin{theorem}\label{thm:ell}
Consider an elliptic element 
\[ \gamma = \sma 0 & -\det \gamma \\ 1 & \tr \gamma \smz\]
 in $G=\GL_2(\F)$. Let $Z$ be the center of $G$, and $G_\gamma$ be the centralizer of $\gamma$. 
For any element $\phi \in \Ccinf(G, \chi)$, the elliptic orbital integral is given as
\begin{align} \int\limits_{G_\gamma \backslash G} \phi(g^{-1} \gamma g) \d g =  \sum\limits_{x \in \Waff} \mu_G(IxI) \phi^I(x^{-1} \gamma x) .\end{align}
\end{theorem}

\begin{proof}[Proof of Proposition~\ref{prop:padicelliptic}]
Let $d\in \{0, 1\}$. Assume first that $\phi$ is supported either on $\Z(\F) \GL_2(\o)$ or on the normalizer of the Iwahori subgroup and that
\[ \phi(zg) = \chi(z) \phi(g), \qquad z \in \Z(\F).\] For simplicity, we may assume that $\gamma$ is of the form suggested by Lemma~\ref{lemma:form}, i.e.,
\[             \gamma=     w_0 \sma 1 &t  \\0&  u \w^d\smz, \qquad w_1 \coloneqq w_\p.\]
Let
\[  w^k\coloneqq \sma \w^{k} & 0 \\ 0 & 1\smz.\]
Assume that the element
\[ w^{-k}   w_0 \sma 1 &t  \\0&  u \w^d\smz w^k = w_0 \sma \w^k &t  \\0&  u \w^{d-k}\smz = w_d  \w^{-k} \sma \w^{2k} &t \w^k \\0&  u \smz \]
lives in the support of $\phi$, then we have that $k=0$. Consequently, we have that
\begin{align*} J_{\gamma} (\phi) =&  \phi(\gamma) + (1+q)\vol( I w_d I)  \phi(w_d^{-1} \gamma w_d), \\                                              
                                  \vol( I w_d I)  =& \begin{cases}  q/(1+q), & d= 0,\\ 1/(1+q), & d =1. \end{cases}
\end{align*}
The distribution vanishes if
\begin{itemize}
 \item  $\phi$ is supported on the normalizer of the Iwahori, and $\gamma$ is unramified, or
 \item $\phi$ is supported on $\Z(\F) \GL_2(\o)$, and $\gamma$ is ramified.
\end{itemize}
Let us now show that for any one-dimensional representation $\mu : \o^\times \rightarrow \bC^1$ of conductor $\p^R \supset \p^N$ and $N \neq 0$ and any unramified elliptic elements $\gamma$
\[ \tr \Ind_{\Gamma_0(\p^N)}^{\GL_2(\o)}  \mu (\gamma)      = 0.\]
This is a direct result of the Frobenius Character Formula:
\begin{align*}
\tr \Ind_{\Gamma_0(\p^N)}^{\GL_2(\o)}  \mu (\gamma) & \underset{\textup{formula}} = \sum\limits_{\substack{ x \in \GL_2(\o) / \Gamma_0(\p^N) \\ x^{-1} \gamma x \in \Gamma_0(\p^N)}} \mu(x^{-1} \gamma x ) \\
                                                                             & \underset{\textup{Iwahori-dec.}}   = \sum\limits_{\substack{ x = \sma 1 & 0 \\ t &  1 \smz \\ t  \in \p \bmod \p^N \\ x^{-1} \gamma x \in \Gamma_0(\p^N)}} \mu(x^{-1} \gamma x ) \\
             & \qquad +      \sum\limits_{\substack{ x = \sma 1 &  t \\ &  1 \smz w_0 \\ t  \in \o \bmod \p^N \\ x^{-1} \gamma x \in \Gamma_0(\p^N)}} \mu(x^{-1} \gamma x )  .
\end{align*}
The first summand vanishes; for $t \in \p$
\[ \sma 1 & 0 \\ - t  & 1 \smz  \sma 0 & -\det \gamma \\ 1 & \tr \gamma \smz \sma 1 & 0 \\ t & 1 \smz = \sma 0 & -\det \gamma t \\ 1 & t \det \gamma + \tr \gamma \smz \sma 1 & 0\\ t & 1 \smz = \sma  * & * \\ 1+ \p & * \smz \]
is never an element of $\Gamma_0(\p^N)$ for $N\geq 1$. 
The second summand does not vanish in general; for $t\in \o$, we have that
\[ w_0^{-1} \sma 1 & -t \\ 0 & 1 \smz  \sma 0 & -\det \gamma \\ 1 & \tr \gamma \smz \sma 1 & t \\ 0 & 1 \smz w_0 =   w_0^{-1} \sma - t & - t^2 -\det(\gamma) t - t \tr \gamma \\ 1 & t + \tr \gamma \smz w_0 = \sma - t - \tr \gamma &  * \\  - t^2 -\det(\gamma)  - t \tr \gamma  & t \smz.\]
This yields the result. In particular, the equality $- t^2 -\det(\gamma)  - t \tr \gamma  = 0 \bmod \p^N$ implies that  $- t - \tr \gamma \in \o^\times$, since 
\[        - t^2 -\det(\gamma) - t \tr \gamma  = 0 \bmod \p^N \]
if and only if
\[         t^2  - t \tr \gamma +\det(\gamma) = 0 \bmod \p^N. \]
By Hensel's Lemma, the characteristic polynomial of $\gamma$ 
\[      t^2  - t \tr \gamma +\det(\gamma) \]
it thus reducible over $\o$, which contradicts the fact that $\gamma$ is elliptic.

For the Steinberg representation, we write
\[ \phi_{\St} = \mathds{1}_{\Z(\F) \Gamma_0(\p)}^{\GL_2(\o)}-  \mathds{1}_{\Z(\F) \GL_2(\o)}\]
and note that
\[ 0 =  J_\gamma( \mathds{1}_{\Z(\F) \Gamma_0(\p)} )  = J_\gamma( \mathds{1}_{\Z(\F) \Gamma_0(\p)}^{\GL_2(\o)}) .\]

We have demonstrated the identities for $\phi_{\St}$, $\phi_\pi$, and $\phi_\mu$.

A Hecke operator $\HT$ is supported on 
\[ \GL_2(\o) \sma \w& 0 \\ 0 & 1 \smz \Z(\F) \GL_2(\o) .\]
Thus, it necessarily vanishes on the distribution of the unramified elliptic element, since the distribution is supported on matrices whose determinant has even valuation.
Assume that $\gamma$ is a ramified elliptic element of the form 
\[ \gamma = \sma 0 & u \w^{-1}\\1 & t \smz , \qquad u \in \o^\times, t \in \o,\]
and consequently
\begin{align*}J_\gamma (\HT) &= \sum\limits_{ w \in \GL_2(\F) //\Z(\F) I}  \vol(IxI) \HT(x^{-1}\gamma x)\\
                                           & \underset{ \GL_2(\o)-\textup{inv.}}=     \sum\limits_{ x \in \Waff / \langle w_0 \rangle}  C_w \HT(x^{-1}\gamma x) \\
                                            & = \sum\limits_{r \in \bZ}   C_r \HT\left(  \sma \w^r & 0 \\ 0 & 1  \smz   \sma 0 & u \w^{-1}\\1 & t \smz \sma \w^{-r} & 0 \\0 & 1 \smz \right) \\
                                            & \underset{\textup{supp} \HT = \GL_2(\o) \Z(\F) \sma \w & 0 \\ 0 & 1 \smz \GL_2(\o)}=   C_0   \HT\left(    \gamma  \right) +        C_1   \HT\left(     \sma 0 & u \\\w^{-1} & t \smz  \right).
\end{align*}
For $\mu=1$ and $\HT = \HTo$, we have that $\HTo(\gamma) = 1$ and 
\[   \HTo\left(     \sma 0 & u \\ \w^{-1} & t \smz  \right) =\HTo\left(     \sma 0 & u \w \\ 1 & t \w \smz  \right)= 1.\]
Consequently, the distribution evaluates to
\[J_\gamma ( \phi) = \vol(I) + \vol(I w_0 I) + \vol(I w_d I) + \vol(I w^1 I) = 2.\]

 For $\mu \neq 1$, both expression vanish.  We argue with $\phi_{\mu, \w}$ and Lemma~\ref{lemma:frobcharacetr}. Set $\cond(\mu) = \p^N$. We have that $\HT(\gamma)$ is a linear combination of  
\[ \phi_{\mu, \w} \left( \sma 1 & 0 \\ -\alpha & 1 \smz \sma 0 & -D \\ 1 & T \smz  \sma 1 & 0 \\ \alpha & 1 \smz  \right)     = \phi_{\mu, \w}   \left( \sma - \alpha D  &D \\ 1 + \alpha T +\alpha^2 D & T +\alpha D \smz \right)\]
for a set of representatives $\alpha \p \bmod \p^N$, and a linear combination of
\begin{align*} \phi_{\mu, \w} \left( w_0^{-1} \sma 1 & -\beta \\ 0 & 1 \smz \sma 0 & -D \\ 1 & T \smz \sma 1 & \beta \\ 0 & 1 \smz w_0 \right)    & =    \phi_{\mu, \w} \left(     w_0^{-1} \sma  - \beta &-\beta^2 -D-\beta T \\ 1 & T+\beta \smz  w_0   \right) \\ 
& =  \phi_{\mu, \w}   \left( \sma -\beta - T  & -1 \\ \beta^2 + \beta T +  D & -\beta \smz \right) \end{align*}
for a set of representative $\beta \in \o \bmod \p^N$. The function $\phi_{\mu, \w}$ is supported on the set
\[ \Gamma_0(\p^N) \sma \w & 0 \\ 0 & 1 \smz \Gamma_0(\p^N) .\]
We have two cases
\begin{enumerate}[font=\normalfont]
 \item  Set $T = t$ and $D =u\w^{-1}$:  Since $ 1 + \alpha \tr \gamma +\alpha^2 \det(\gamma)  \in 1+ \p$ for $\alpha \in \p$, we have that
\[    \phi_{\mu, \w} \left( w_0 \sma 1 & -\beta \\ 0 & 1 \smz \gamma \sma 1 & \beta \\ 0 & 1 \smz w_0 \right)   = 0.\]
Since  $ \beta^2 + \beta \tr \gamma +  \det(\gamma)  \in \p^{-1} - \o$ for $\beta \in \o$, since $\det(\gamma)  \in \p^{-1}$ and $\tr(\gamma) \in \o$ by assumption, we have that
\[ \phi_{\mu, \w}   \left( \sma \beta + \tr \gamma  & -1 \\ \beta^2 + \beta \tr \gamma +  \det(\gamma)  & \beta \smz \right)=0. \]
\item Set $T = \w t$ and $D = u \w$: Then $\phi_{\mu, \w}(\dots) \neq 0$ is not possible, since $\alpha u \w + T \in \p$, and $\beta, T + \beta \in \o^\times$.  \qedhere
\end{enumerate}
\end{proof}

\begin{proof}[Proof of the theorem]
Let $C_M = N_M(F)$ be the normalizer of the Levi subgroup $M(F)$. The group $C_M$ together with $I$ gives a BN pair in the sense of \cite{Abramenko-Brown}*{Definition 6.55, page 319}.
As a conclusion \cite{Abramenko-Brown}*{Definition 6.55, page 319}, we have a disjoint decomposition
\begin{align*} G = \coprod\limits_{x \in \Waff}  I_0 x I_0.\end{align*}
By the discreteness of the quotient $G//I_0$, this results in a decomposition of the Haar measure $\mu_G$ (see Lemma~\ref{lemma:measdc})
\begin{align}\label{eq:measdeaa}\int\limits_{Z \backslash G}\phi(g) \d g =     \sum\limits_{x \in \Waff} \mu_G(IxI) \int\limits_{I} \int\limits_{I} \phi(i_1 x i_2) \d i_1 \d i_2.\end{align}
We assume without loss of generality that $\gamma$ is given in the form of Lemma~\ref{lemma:form}. Now we consider $\phi \in \Ccinf(G(F)/Z(F))^{I}$, that is, 
\[ \phi(i^{-1}xi) = \phi(x), \qquad x \in G, i \in I.\]
The elliptic integral can be translated into an integral over the group:
\begin{align*} \int\limits_{G_\gamma \backslash G} \tr \phi(g^{-1}\gamma g) \d g & =    \int\limits_{Z \backslash G} \tr \phi(g^{-1}\gamma g)   \d g    \\
& \underset{\textup{Eq.~\ref{eq:measdeaa}}}= \sum\limits_{x \in \Waff} \mu_G(IxI)  \int\limits_{I} \tr \phi(x^{-1} i^{-1} \gamma i x) \d i .\end{align*}
The Iwahori decomposition yields that
\[ \mu_G(I) = \frac{1}{ 1+ q^{-1}} .\]
The theorem follows from the next lemma.
\end{proof}
\begin{lemma}
Let $\phi \in \Ccinf(\GL_2(\F))^I$. For each element $x \in \langle w_0, w_\p \rangle$ and every elliptic element $\gamma$, we have the identity
\begin{align} \int\limits_{I} \tr \phi(x^{-1} i^{-1} \gamma i x) \d i   = \phi(x^{-1} \gamma x).\end{align}
\end{lemma}
\begin{proof}
We write $N^w  = \sma 1 & 0 \\ * & 1\smz$. The Iwahori decomposition  \cite{BushnellHenniart:GL2}*{(7.3.1), page 52} gives for every permutation $\sigma$ of three elements the isomorphism                          
\begin{align}\label{eq:iwahoridec} N^w(\p) \times M(F) \cap I \times N(\o) \rightarrow I , \qquad (i_1, i_2, i_3) \mapsto i_{\sigma(1)}   i_{\sigma(2)}     i_{\sigma(3)}. \end{align}
The Iwahori decomposition~\ref{eq:iwahoridec} results in a measure decomposition 
\begin{align*} \int_{I} f(i) \d i & =      \int_{N^w(\p)}  \int_{N(\o)} \int_{M(\o)}    f(i_1 i_2 m) \d m\d i_1 \d i_2 \\ &  =\int_{N(\o)}  \int_{N^w(\p)}  \int_{M(\o)}    f(i_2i_1  m)\d m \d i_2  \d i_1.\end{align*}
Note that the computations for 
\begin{align*}  x = \sma w^{-r-e}  & 0 \\ 0 & w^{r} \smz \end{align*}
will verify the computations for $xw$ as well. I distinguish between whether or not $2r +e$ is positive. Modulo the center, we may assume that $e \in \{ 0 , 1\}$.  For $\phi \in \Ccinf(G)^K$, we encounter for the nontrivial case $-2r +e \neq 0$
\begin{align}
  &  \int\limits_{I} \phi(x^{-1} i^{-1} \gamma i x) \d i  \nonumber    \\
                                                                    &  = \begin{cases}\int_{N^w(\p)}  \int_{N(\o)}    \int_{M(\o)} \phi(x^{-1} m^{-1} i_1^{-1} i_2^{-1} \gamma i_2 i_1  m x)  \d m \d i_2 \d i_1, &   2r + e> 0, \\
                                                                          \int_{N(\o)}  \int_{N^w(\p)} \int_{M(\o)}  \phi(x^{-1} m^{-1} i_2^{-1} i_1^{-1}  \gamma i_1 i_2 mx) \d m\d i_1 \d i_2 ,   &   2r +e<0, \\
                                                                     \end{cases}  \nonumber \\
                                                                    &  = \begin{cases}\int_{N^w(\p)}    \phi(x^{-1} i_2^{-1} \gamma i_2    x)   \d i_2, &  2r+e > 0, \\
                                                                          \int_{N(\o)} \phi(x^{-1}  i_1^{-1}  \gamma i_1  x)  \d i_1  ,   &   2r +e<0, \\
                                                                     \end{cases}       \label{eq:hallori}
\end{align}
In both cases,  we note that $i_j^{-1} w \sma 1 & t \\ 0 & u \w^k \smz i_j  \in w \sma \o^\times & \o \\ \o & \w^k \o^\times \smz$ for $j=1,2$. This can be seen in the basic matrix manipulations
\begin{align}\label{eq:nw}& \sma 1 & 0  \\-y  & 1 \smz        w \sma 1 & t \\ 0 & u \w^k \smz    \sma 1 &  0 \\ y & 1 \smz =                  w  \sma 1 & t +u\w^k y \\ 0 & u \w^k \smz    \sma 1 &  0 \\ y & 1 \smz = w \sma 1 +  yt +u\w^k y^2 & t +u\w^k y \\ yu \w^k &  u \w^k \smz,   \\ 
\label{eq:n} & \sma 1 &  -z \\ 0 & 1 \smz        w \sma 1 & t \\ 0 & u \w^k \smz    \sma 1 &  z \\ 0 & 1 \smz    =      w     \sma 1 &  0 \\ z & 1 \smz  \sma 1 & t+z \\ 0 & u \w^k \smz =     w     \sma 1 &  t+z \\ x & tz+z^2 + u\w^k \smz .\end{align}
It can be deduced from the fact that the polynomials $P_\gamma(-Z) = tZ+Z^2 + u\w^k$ and  $Y^2 \; P_\gamma(-Y^{-1}) = 1 +  Yt +u\w^k Y^2$ are irreducible, because the characteristic polynomial $P_\gamma$ of $\gamma$ is irreducible. 
We compute, for $r \neq 0$  and 
\begin{align} i_j^{-1} w \sma 1 & t \\ 0 & u \w^k \smz i_j = w \sma a & b \\  c & d \smz, \end{align}
the affine Bruhat decomposition explicitly:
\begin{align}
 &\sma \w^{-r-e} & 0 \\ 0 & \w^{r} \smz^{-1} w \sma a & b \\  c & d \smz   \sma \w^{-r-e} & 0 \\ 0 & \w^{r} \smz = w    \sma a \w^{-2r-e} & b \\  c & d \w^{+2r+e} \smz  \nonumber  \\
& =\begin{cases}   w\sma 1 & 0 \\  \w^{2r+e} c/a& 1 \smz \sma a & 0 \\ 0 & \det(\gamma)/a \smz \sma w^{-2r-e}  & 0 \\ 0 & w^{2r+e} \smz   \sma 1 & \w^{2r+e}b/a  \\ 0 & 1 \smz , & 2r+e >0, \\ 
                            w\sma 1 & \w^{-2r-e} b/d \\ 0 & 1 \smz \sma \det(\gamma)/d & 0 \\ 0 & d \smz \sma w^{-2r-e}  & 0 \\ 0 & w^{2r+e} \smz   \sma 1 & 0 \\   \w^{-2r-e}c/d& 1 \smz, & 2r+e < 0 . 
   \end{cases}     \label{eq:deckgk}  
\end{align}
The Equation~\ref{eq:hallori} results for $2r+e>0$ in
\begin{align*}
& \int\limits_{N(\o)} \phi\left(x^{-1} i_2^{-1} \gamma i_2 x \right)  \d i_2 \underset{eq.~\ref{eq:n}}{=}   \int\limits_{\o}    \phi\left(x^{-1}  w\sma 1 & t +z \\ z & P_\gamma(-z)  \smz x \right)  \d z              \\
                                                                                             &\underset{eq.~\ref{eq:deckgk}}{=}   \int\limits_{\o}    \phi\left(  w\sma 1 & 0 \\ \w^{2r+e}z & 1  \smz\sma 1 & 0 \\ 0 & \det(\gamma) \smz \sma w^{-2r-e}  & 0 \\ 0 & w^{2r+e} \smz \sma 1 & \w^{2r+e}(t +z) \\0& 1  \smz \right)  \d z \\
                                                                                             &\quad  =   \int\limits_{\o}    \phi\left( \sma 1 &  -\w^{2r+e} z\\ 0 & 1  \smz w\sma 1 & 0 \\ 0 & \det(\gamma) \smz \sma w^{-2r-e}  & 0 \\ 0 & w^{2r+e} \smz\sma 1 & \w^{2r+e}t  \\0& 1  \smz  \sma 1 & \w^{2r+e}z   \\0& 1  \smz \right)  \d z \\ 
                                                                                               & \underset{I-inv.}{=} \phi\left( w \sma \w^{-2r-e} & t \\ 0 & \w^{2r+e} \det(\gamma) \smz\right) =     \phi\left(x^{-1} \gamma x \right) 
 \end{align*}
 and for    $2r+e<0$
\begin{align*}
 &\int\limits_{N(\p)} \phi\left(x^{-1} i_1^{-1} \gamma i_1 x \right)  \d i_1\\
 & \underset{eq.~\ref{eq:nw}}{=}   \int\limits_{\o}    \phi\left(x^{-1}  w \sma 1 +  yt +u\w^k y^2 & t +u\w^k y \\ yu \w^k &  u \w^k \smz x \right)  \d t              \\
                                                                                             &\underset{eq.~\ref{eq:deckgk}}{=}   \int\limits_{\o}    \phi\left(  w\sma 1 & (t(u\w^k)^{-1}+y) \w^{-2r-e}  \\  0 & 1  \smz \sma 1 & 0 \\ 0 & u \w^k \smz \sma w^{-2r-e}  & 0 \\ 0 & w^{2r+e} \smz \sma 1 & 0 \\y \w^{-2r-e}& 1  \smz \right)  \d t \\
                                                                                               &=   \int\limits_{\o}    \phi\left( \sma 1 & 0 \\ - y \w^{-2r-e} & 1\smz w \sma 1 &  t (u\w^k)^{-1} \w^{-2r-e}  \\  0 & 1  \smz  \sma 1 & 0 \\ 0 & \det(\gamma) \smz \sma w^{-2r-e}  & 0 \\ 0 & w^{2r+e} \smz \sma 1 & 0 \\y \w^{-2r-e}& 1  \smz \right)  \d t \\ 
                                                                                               & \underset{I-inv.}{=} \phi\left( w \sma \w^{-2r-e} & t \\ 0 & \w^{2r+e} \det(\gamma) \smz\right) =     \phi\left(x^{-1} \gamma x \right).\qedhere\end{align*}
\end{proof}

\section{An easy example: Depth-zero supercuspidal representations}
 
\begin{defn}
An irreducible, supercuspidal representation with central algebraic character is of depth zero if it admits a $\Gamma(\p)$-invariant vector. 
\end{defn}
Isomorphism classes of depth-zero supercuspidal representations are in one-to-one correspondence with isomorphism classes of cuspidal representations of $\GL_2(\o / \p)$. More precisely, this is realized as follows: Let $\tilde{\rho}$ be a cuspidal representation of $\GL_2(\o / \p)$, then the representation
\[  \pi_{\tilde{\rho} } = \Ind_{\Z(\F)\GL_2(\o)}^{\GL_2(\F)} \infl_{\GL_2(\o/\p)}^{\GL_2(\o)} \tilde{\rho}\]
is a depth-zero supercuspidal representation. Every depth-zero supercuspidal representation contains the inflation of a unique cuspidal representation of $\GL_2(\o/\p)$ with simple multiplicity. 
\begin{proposition}[The parabolic distributions for depth-zero supercuspidal pseudo coefficients]
Let $\pi$ a depth-zero supercuspidal representation, then
\[ \zeta_{\GL_2(\F)}'(1, \phi_{\pi})  = \log(q)(q-1).\]
\end{proposition}
\begin{proof}
 According to Proposition~\ref{prop:padicpara}, we have the formula
\[ \zeta_{\GL_2(\F)}'(1, \phi_{\pi})  = \log(q)(1-q^{-1}) \sum\limits_{k=0}^\infty  q^{-k} \dim \Hom_{\N(\p^k)}( \rho, 1).\]
At this point, the representation $\rho$ is an inflation of a cuspidal representation of $\GL_2(\o/\p)$, and has no $\N(\o)$-invariant vector
\[          \dim \Hom_{\N(\o)}( \rho, 1) = 0.\]
Because every vector is $\Gamma(\p)$-invariant, we have for all $k \in \mathbb{Z}$
\[       \dim \Hom_{\N(\p^k)}( \rho, 1) = \dim(\rho) = q^2 -q.\]
See below for this dimension formula. We obtain a telescopic sum
\[            \log(q)(1-q^{-1}) \sum\limits_{k=1}^\infty  q^{-k} (q^2 - q)  =   \log(q)(q-1)\qedhere.\] 
\end{proof}
The cuspidal representations of $\GL_2(\o / \p)$ are in one-to-one correspondence with Galois orbits of regular multiplicative characters of the quadratic extension $E$ of $\o / \p$. A multiplicative character of $E$ is regular if it is not fixed by the Galois group of $E$ over $\o/\p$.
This Galois group is generated by $\textup{Frob}: x \mapsto x^q$. Two distinct multiplicative characters, $\theta$ and $\theta_0$, are in the same Galois orbit if and only if
\[ \theta \circ \textup{Frob}  =\theta_0.\]
The correspondence works as follows \cite{BushnellHenniart:GL2}*{Section 6.4, page 47}: fix a non-trivial additive character $\psi : \o / \p \rightarrow \bC^1$ and an embedding of $E^\times \subset \GL_2(\o/\p)$. Given a multiplicative character $\theta$ of $E$, the representation
\[ \rho_\theta := \Ind_{\Z\N(\o/\p)}^{\GL_2(\o/\p)} \psi \cdot \theta|_\Z  \ominus \Ind_{E^\times}^{\GL_2(\o/\p)} \theta\]
is both irreducible and cuspidal representation of $\GL_2(\o/\p)$. In particular, this implies \cite{Lang:Algebra}*{page 714}
\begin{align*} 
\dim(\rho_\theta) &= \frac{\# \GL_2(\o / \p)}{\# \Z\N(\o/\p)} -   \frac{\# \GL_2(\o/\p)}{\# E^\times} 
\\ 
&= \frac{q(q+1)(q-1)^2}{q(q-1)} -  \frac{(q^2-1)(q^2-q)}{q^2-1}= q^2 -q.
 \end{align*}
The isomorphism class of the representation is independent of the embedding and character. We have an isomorphism:
\[  \rho_\theta \cong \rho_{\theta_0} , \qquad  \theta \circ \textup{Frob}  =\theta_0.\]
We write
\[ \pi_\theta  := \Ind_{\Z(\F)\GL_2(\o)}^{\GL_2(\F)} \infl_{\GL_2(\o/\p)}^{\GL_2(\o)} \rho_\theta.\]
\begin{proposition}[The elliptic distributions for depth-zero supercuspidal pseudo coefficients]
 Let $\pi_\theta$ be an depth-zero supercuspidal representation as described above, then 
\[J_\gamma(\phi_\pi) = \begin{cases} 0, & \gamma \textup{ ramified elliptic}, \\ - \theta(\lambda_1) - \theta( \lambda_2), &  \gamma \textup{ unramified elliptic}, \end{cases} \]
 where $\lambda_1$ and $\lambda_2$ are the roots of the characteristic polynomial of $\gamma$ modulo $\p$.
\end{proposition}
\begin{proof}
By our definition of $\phi_\pi$ and Proposition~\ref{prop:padicelliptic}, we have the vanishing result for ramified elliptic elements, and for an unramified elliptic element $\gamma$, we have the equation 
\[    J_\gamma(\phi_\pi) = \tr \rho(\gamma).\]
The formula follows from \cite{BushnellHenniart:GL2}*{Theorem 6.4, page 47} and in particular  \cite{BushnellHenniart:GL2}*{Equation 6.4.1, page 48}
\[ \tr \rho_\theta( y) =   - \theta(y) - \theta( \textup{Frob}(y)), \qquad y  \in E-\Z(\o/\p).\]
Hensel's Lemma implies that $\gamma \bmod \p$ is elliptic in $\GL_2(\o /\p)$ as well, and conjugate to an element in $E$, specifically
\[ \tr \rho( \gamma) = \tr \rho_\theta(\gamma \bmod \p). \qedhere\] 
\end{proof}

\begin{appendix} 

\printindex\label{index}

 \end{appendix}
\end{document}